\documentclass[oneside,11pt]{amsart}
\usepackage[]{geometry}                
\usepackage{graphicx}
\usepackage{amsmath}
\usepackage{amssymb}
\usepackage{amsthm}
\usepackage{mathtools}
\usepackage{mathscinet}
\usepackage{tikz-cd}
\usepackage{dsfont}
\usepackage{color}
\usepackage{enumitem}
\usepackage{tensor}
\usepackage{todonotes}
\usepackage{mathrsfs}
\usepackage{float}

\usepackage[]{hyperref}
\hypersetup{
	pdftitle={Joint Equidistribution of CM Points},
	pdfauthor={Ilya Khayutin},
	bookmarksnumbered=true,     
	bookmarksopen=true,         
	bookmarksopenlevel=1,  
	bookmarksdepth=3,     
	colorlinks=true,
	citecolor=blue,            
	pdfstartview=Fit,           
	pdfpagemode=UseOutlines,    
	pdfpagelayout=TwoPageRight
}

\usepackage[protrusion=true]{microtype}

\usepackage[cal=boondoxo]{mathalfa} 

\usepackage[vvarbb]{newtxmath}

\makeatletter
\let\orgdescriptionlabel\descriptionlabel
\renewcommand*{\descriptionlabel}[1]{%
  \let\orglabel\label
  \let\label\@gobble
  \phantomsection
  \protected@edef\@currentlabel{#1}%
  \let\label\orglabel
  \orgdescriptionlabel{#1}%
}
\def\paragraph{
	\@startsection{paragraph}{4}
	\z@{.5\linespacing\@plus.7\linespacing}{-.5em}%
	{\normalfont\itshape}}
\makeatother

\DeclareFontFamily{U}{mathx}{\hyphenchar\font45}
\DeclareFontShape{U}{mathx}{m}{n}{
      <5> <6> <7> <8> <9> <10>
      <10.95> <12> <14.4> <17.28> <20.74> <24.88>
      mathx10
      }{}
\DeclareSymbolFont{mathx}{U}{mathx}{m}{n}
\DeclareFontSubstitution{U}{mathx}{m}{n}
\DeclareMathAccent{\widecheck}{0}{mathx}{"71}

\renewcommand{\restriction}{\mathord{\upharpoonright}}
\newcommand{\dif}{\,\mathrm{d}}
\renewcommand{\sc}{\mathrm{sc}}

\DeclareMathOperator{\ent}{h}
\DeclareMathOperator{\vol}{vol}
\DeclareMathOperator{\disc}{disc}
\DeclareMathOperator{\Nr}{Nr}
\DeclareMathOperator{\Nrd}{Nrd}
\DeclareMathOperator{\Trd}{Trd}
\DeclareMathOperator{\Tr}{Tr}

\DeclareMathOperator{\Proj}{Proj}
\DeclareMathOperator{\Spec}{Spec}
\DeclareMathOperator{\sign}{sign}
\DeclareMathOperator{\Gal}{Gal}

\DeclareMathOperator{\End}{End}

\DeclareMathOperator{\Span}{Span}
\DeclareMathOperator{\WR}{Res^E_{\mathbb{Q}}}
\DeclareMathOperator{\Img}{Im}

\DeclareMathOperator{\Li}{Li}
\DeclareMathOperator{\ctr}{ctr}
\DeclareMathOperator{\inv}{inv}
\DeclareMathOperator{\Denom}{\mathfrak{d}}
\DeclareMathOperator{\Ideals}{\mathcal{J}}
\DeclareMathOperator{\Princp}{\mathscr{P}}
\DeclareMathOperator{\idl}{idl}
\DeclareMathOperator{\Pic}{Pic}
\DeclareMathOperator{\red}{red}
\DeclareMathOperator{\cbd}{cbd}
\DeclareMathOperator{\ch}{\upchi}

\newcommand{\Nrml}{\operatorname{N}}
\newcommand{\Cent}{\operatorname{Z}}
\newcommand{\Cor}{\operatorname{Cor}}
\newcommand{\Stab}{\operatorname{Stab}}
\newcommand{\RO}{\operatorname{RO}}
\newcommand{\ord}{\operatorname{ord}}
\newcommand{\Ad}{\operatorname{Ad}}
\newcommand{\Cl}{\mathrm{Cl}}
\newcommand{\meas}{\mathrm{m}}
\renewcommand{\det}{\operatorname{det}} 
\newcommand{\relmiddle}[1]{\mathrel{}\middle#1\mathrel{}}

\DeclareMathOperator{\Gm}{\mathbb{G}_m}

\newlength{\faktorheight}
\newcommand*{\dfaktor}[3]{
 \settototalheight{\faktorheight}{\ensuremath{#1}}%
  \raisebox{-0.5\faktorheight}{\ensuremath{#1}}
   \backslash
   \settototalheight{\faktorheight}{\ensuremath{#2}}%
  \raisebox{0.5\faktorheight}{\ensuremath{#2}}
  \slash
   \settototalheight{\faktorheight}{\ensuremath{#3}}%
  \raisebox{-0.5\faktorheight}{\ensuremath{#3}}
}

\newcommand*{\lfaktor}[2]{
 \settototalheight{\faktorheight}{\ensuremath{#1}}%
  \raisebox{-0.5\faktorheight}{\ensuremath{#1}}
	\backslash
   \settototalheight{\faktorheight}{\ensuremath{#2}}%
  \raisebox{0.5\faktorheight}{\ensuremath{#2}}
}

\newcommand*{\faktor}[2]{
 \settototalheight{\faktorheight}{\ensuremath{#1}}%
  \raisebox{0.5\faktorheight}{\ensuremath{#1}}
  \slash
   \settototalheight{\faktorheight}{\ensuremath{#2}}%
  \raisebox{-0.5\faktorheight}{\ensuremath{#2}}
}

\newcommand*{\cyclic}[1]{
	\faktor{\mathbb{Z}}{#1\mathbb{Z}}
}

\newtheorem{thm}{Theorem}
\numberwithin{thm}{section} 
\newtheorem*{thm*}{Statement of Theorem}
\newtheorem*{thm-quote}{Theorem}
\newtheorem{lem}[thm]{Lemma}
\newtheorem{prop}[thm]{Proposition}
\newtheorem{cor}[thm]{Corollary}
\newtheorem{corprf}[thm]{Corollary of Proof}
\newtheorem{conj}[thm]{Conjecture}
\theoremstyle{definition}
\newtheorem{defi}[thm]{Definition}

\theoremstyle{remark}
\newtheorem{remark}[thm]{Remark}

\title{Joint Equidistribution of CM Points}
\author[I. Khayutin]{Ilya Khayutin}
\address{Department of Mathematics, Princeton University, Princeton, NJ 08544, USA}
\address{School of Mathematics, Institute for Advanced Study, Princeton, NJ 08540, USA}

\begin{document}

\begin{abstract}
We prove the mixing conjecture of Michel and Venkatesh for toral packets with negative fundamental discriminants and split at two fixed primes; assuming all splitting fields have no exceptional Landau-Siegel zero.
As a consequence we establish for arbitrary products of indefinite Shimura curves the equidistribution of Galois orbits of generic sequences of CM points all whose components have the same fundamental discriminant; assuming the CM fields are split at two fixed primes and have no exceptional zero.

The joinings theorem of Einsiedler and Lindenstrauss applies to the toral orbits arising in these results. Yet it falls short of demonstrating equidistribution due to the possibility of intermediate algebraic measures supported on Hecke correspondences and their translates.
 
The main novel contribution is a method to exclude intermediate measures for toral periods.
The crux is a geometric expansion of the cross-correlation between the periodic measure on a torus orbit and a Hecke correspondence, expressing it as a short shifted convolution sum. The latter is bounded from above generalizing the method of Shiu and Nair to polynomials in two variables on smooth domains.
\end{abstract}
\maketitle
\bgroup
\hypersetup{linkcolor = blue}
\tableofcontents
\egroup

\section{Introduction}
\subsection{The Mixing Conjecture of Michel and Venkatesh}
Let $Y$ be a complex modular curve. Each CM point\footnote{In the setting of modular curves CM points are classically called Heegner points. We follow the terminology of \textit{CM point} to differentiate between the points on the modular curve -- which are the subject of this manuscript -- and the corresponding point on a modular elliptic curve, which we do not discuss.} on $Y$ is an element of a finite \emph{packet} of CM points all of which have CM by the same quadratic order $\Lambda$ and form a single orbit under $\Pic(\Lambda)$.  For each integer $i$ let $\mathcal{P}_i\subset Y$ be a packet of CM points with discriminant $D_i<0$. Denote by $\mu_i$ the normalized counting measure on $\mathcal{P}_i$.  By a theorem of Duke \cite{Duke} and Iwaniec \cite{Iwaniec} and its generalizations inter alios by \cite{Chelluri,Michel04,Zhang}
we know that if $|D_i|\to_{i\to\infty}\infty$ then  $\mu_i\xrightarrow[i\to\infty]{\mathrm{weak}-*}\meas_Y$ where $\meas_Y$ is the normalized Haar measure on $Y$. 

Michel and Venkatesh \cite{MichelVenkateshRev} have conjectured a variant of the following.
\begin{conj}[Mixing Conjecture]\label{conj:mixing}
Let $\mathcal{P}_i\subset Y$ be a sequence of packets of CM points as above. Each $\mathcal{P}_i$ is a principal homogeneous space of $\Pic(\Lambda_i)$ where $\Lambda_i$ is the CM order of the points in $\mathcal{P}_i$.
For each $i\in\mathbb{N}$ fix $\sigma_i\in\Pic(\Lambda_i)$ and define
\begin{equation*}
\mathcal{P}_i^\mathrm{joint}\coloneqq
\left\{\left(z,\sigma_i.z\right) \mid z\in\mathcal{P}_i \right\}\subset Y\times Y
\end{equation*}
Denote by $\mu_i^\mathrm{joint}$ the normalized counting measure supported on $\mathcal{P}_i^\mathrm{joint}$. 

Set 
\begin{equation*}
\mathfrak{N}_i=\min_{\substack{\mathfrak{a} \subseteq \Lambda_i \textrm{ invertible ideal}\\ \mathfrak{a}\in \sigma_i}} \Nr \mathfrak{a}
\end{equation*}
If $\mathfrak{N}_i\to_{i\to\infty}\infty$ then $\mu_i^\mathrm{joint}$ converge weak-$*$ to $\meas_Y\times\meas_Y$.
\end{conj}

Using the reciprocity map of class field theory this conjecture implies a special case of the following well-known conjecture about equidistribution of Galois orbits of special points on products of modular curves.

\begin{conj}
Let $X$ be a finite product of complex modular curves. Let $\{x_i\}_i$ be a sequence of special points on $X$, i.e.\ each coordinate of $x_i$ is a CM point. Denote by $\mu_i$ the normalized counting measure on the finite Galois orbit of $x_i$.

If the sequence $\{x_i\}_i$ has finite intersection with any proper special subvariety\footnote{This condition implies that the size of the Galois orbit of $x_i$ tends to infinity as $i\to\infty$ because of Brauer-Siegel and the fact that a special point is by itself a special subvariety.} of $X$ then then $\{\mu_i\}_i$ converges weak-$*$ to the uniform probability measure on $X$.
\end{conj}

The latter conjecture implies the Andr\'e-Oort conjecture for products of modular curves which has been settled by Pila \cite{Pila}, see also \cite{Andre,Edixhoven2,EdixhovenArbitrary}. The Andr\'e-Oort conjecture in this setting states that the sequence  $\{x_i\}_i$  above must be Zariski dense in $X$. The Pila-Zannier strategy which is behind the recent breakthroughs on the Anrdr\'e-Oort conjecture \cite{Pila,PilaTsimerman, Tsimerman} does not seem to shed light on the question of equidistribution of 
Galois orbits.

\subsection{Summary of Results}\addpenalty{-1000}
\subsubsection{Results for Torus Orbits}
We present a proof of the mixing conjecture of Michel and Venkatesh conditional on several significant assumptions. Most importantly, we assume that all the CM fields $E_i$ are split at two fixed primes $p_1$, $p_2$ and have no exceptional Landau-Siegel zero.

Einsiedler, Lindenstrauss, Michel and Venkatesh \cite{ELMVCubic} have defined the notion of a toral packet.   To each CM point or a closed geodesic on a modular curve corresponds an order $\Lambda$ in a quadratic field $E/\mathbb{Q}$. A toral packet is a generalization of the notion of a single $\Pic(\Lambda)$-orbit of a CM point or a closed geodesic.

Let $\mathbf{G}$ be a form of $\mathbf{PGL}_2$ defined over $\mathbb{Q}$. Fix a compact-open subgroup $K_f<\mathbf{G}(\mathbb{A}_f)$ and consider the double quotient
\begin{equation*}
\widetilde{Y}\coloneqq\dfaktor{\mathbf{G}(\mathbb{Q})}{\mathbf{G}(\mathbb{A})}{K_f}
\end{equation*}
The action of the real group $\mathbf{G}(\mathbb{R})$ on $\widetilde{Y}$ induces an isomorphism between $\widetilde{Y}$ and a disjoint union of finitely many locally homogeneous spaces
\begin{equation*}
\widetilde{Y}\simeq \bigsqcup_{\delta\in \mathbf{G}(\mathbb{Q})\backslash \mathbf{G}(\mathbb{A}_f) \slash K_f} \lfaktor{\Gamma_\delta}{\mathbf{G}(\mathbb{R})}
\end{equation*}
where $\Gamma_\delta\coloneqq\mathbf{G}(\mathbb{Q})\cap \delta K_f \delta^{-1}$ is a congruence lattice in $\mathbf{G}(\mathbb{R})$. In the simplest case when $\mathbf{G}=\mathbf{PGL}_2$ and $K_f$ is a maximal compact subgroup the space $\widetilde{Y}$ has a single component and can be identified with $\lfaktor{\mathbf{PGL}_2(\mathbb{Z})}{\mathbf{PGL}_2(\mathbb{R})}$.

A toral packet $\mathcal{P}$ is a finite collection of orbits in $\widetilde{Y}$ of a real torus $H<\mathbf{G}(\mathbb{R})$ which is a projection of a single adelic torus orbit; see \S\ref{sec:homogeneous-sets} for an exact definition. The set of $H$-orbits in $\mathcal{P}$ has a natural structure as a principal homogeneous space 
for a finite abelian group $C$ which is a quotient of the id\`ele class group of an associated quadratic field $E/\mathbb{Q}$. Moreover, there is an order $\Lambda<E$ attached to $\mathcal{P}$ and a canonical surjective homomorphism $C\to\Pic(\Lambda)$. If $K_f$ is maximal then this homomorphism is an isomorphism.
To each packet one can attach a discriminant $D\in\mathbb{R}$ that measures the arithmetic complexity of the packet, cf. \S\ref{sec:discriminant-local}. This discriminant is a product of $\disc(\Lambda)$ and an archimedean contribution.

The following is a special case of Theorem \ref{thm:joinings-adelic}.
\begin{thm}\label{thm:main-intro}
Let $\mathbf{G}$ be a form of $\mathbf{PGL}_2$ defined over $\mathbb{Q}$ and let $K_f<\mathbf{G}(\mathbb{A}_f)$ be a compact-open subgroup. Fix a \emph{compact} real torus $K_\infty<\mathbf{G}(\mathbb{R})$ and denote
\begin{equation*}
Y\coloneqq \faktor{\widetilde{Y}}{K_\infty}= \dfaktor{\mathbf{G}(\mathbb{Q})}{\mathbf{G}(\mathbb{A})}{K_\infty\times K_f}\simeq \bigsqcup_{\delta\in \mathbf{G}(\mathbb{Q})\backslash \mathbf{G}(\mathbb{A}_f) \slash K_f} \dfaktor{\Gamma_\delta}{\mathbf{G}(\mathbb{R})}{K_\infty}
\end{equation*}

Let $\mathcal{P}_i\subset\widetilde{Y}$ be a sequence of toral packets invariant under $K_\infty$ with a fundamental discriminant $D_i<0$. By abuse of notation denote by $\mathcal{P}_i$ the projection of the packet to a finite
set of points in $Y$. The set $\mathcal{P}_i\subset Y$ is a principal homogeneous space for an abelian group $C_i$.

Let $\Lambda_i<E_i$ be the order in an imaginary quadratic field $E_i/\mathbb{Q}$ attached to $\mathcal{P}_i$ and denote by $\sigma\mapsto [\sigma]$ the homomorphism $C_i\to\Pic(\Lambda_i)$. For every $i\in\mathbb{N}$ choose some $\sigma_i\in C_i$ and let $\mu_i^\mathrm{joint}$ be the Borel probability measure on $Y\times Y$ defined as the normalized counting measure on the set
\begin{equation*}
\mathcal{P}_i^\mathrm{joint}\coloneqq \left\{(z,\sigma_i.z) \mid z\in \mathcal{P}_i \right\}\subset Y\times Y
\end{equation*}

Fix two primes $p_1$, $p_2$ and assume for all $i\in\mathbb{N}$
\begin{enumerate}
\item the primes $p_1$, $p_2$ are split in $E_i$;
\item the Dedekind $\zeta$-function of $E_i$ has no exceptional Landau-Siegel zero.
\end{enumerate}

Set 
\begin{equation*}
\mathfrak{N}_i=\min_{\substack{\mathfrak{a} \subseteq \Lambda_i \textrm{ invertible ideal}\\ \mathfrak{a}\in [\sigma_i]}} \Nr \mathfrak{a}
\end{equation*}
and assume $\mathfrak{N}_i\to_{i\to\infty}\infty$.
Then any weak-$*$ limit point of $\left\{\mu_i^\mathrm{joint}\right\}_i$ is a convex combination of the uniform probability measures on the connected components of $Y\times Y$.
\end{thm}
The theorem above is a special case of Theorem \ref{thm:joinings-adelic} which together with Proposition \ref{prop:residual} describes completely under the assumptions above the analogues weak-$*$ limit points in the adelic quotient
\begin{equation*}
\left[\left(\mathbf{G}\times\mathbf{G}\right)(\mathbb{A})\right]\coloneqq
\lfaktor{\left(\mathbf{G}\times\mathbf{G}\right)(\mathbb{Q})}{\left(\mathbf{G}\times\mathbf{G}\right)(\mathbb{A})}
\end{equation*}
The conclusion of Theorem \ref{thm:main-intro} is the best possible in this setting. In particular, one cannot expect equidistribution on all of $Y\times Y$ because the joint packets $\mathcal{P}_i^\mathrm{joint}$ can avoid completely some connected components of $Y\times Y$. This phenomena can appear already for $\mathbf{G}=\mathbf{PGL}_2$ whenever $K_f$ is non-maximal. This behavior is intimately related to the limit of the averages of the residual spectrum over $\mathcal{P}_i^\mathrm{joint}$ and it is easy to compute exactly which limit measures exactly occur using Proposition \ref{prop:residual}. 

We establish also a generalization of Theorem \ref{thm:main-intro} to $n$-fold products $Y^{\times n}$ -- Theorem \ref{thm:joinings-adelic-n}. This generalization follows from the $2$-fold result and an auxiliary application of the Einsiedler-Lindenstrauss joining theorem \cite[Corollary 1.5]{ELJoinings}.
\subsubsection{Results for Galois Orbits}
Theorem \ref{thm:main-intro} and its $n$-fold generalization Theorem \ref{thm:joinings-adelic-n} imply through the reciprocity map of class field theory a theorem about equidistribution of Galois orbits of special point in products of indefinite Shimura curves.
\begin{thm}
Let $\mathbf{G}$ be a form of $\mathbf{PGL}_2$ defined over $\mathbb{Q}$ and split over $\mathbb{R}$. Let $X$ be a finite product of indefinite Shimura curves relative to $\mathbf{G}$. 
Assume $\{x_i\}_i$ is a sequence of special points on $X$ such that all coordinates have CM by the same \emph{maximal} order. Denote by $\mu_i$ the normalized counting measure on the finite Galois orbit of $x_i$.

Fix two primes $p_1$, $p_2$ and denote by $E_i$ the CM field of $x_i$. Assume that for all
$i\in\mathbb{N}$
\begin{enumerate}
\item the primes $p_1$, $p_2$ split in $E_i$;
\item the Dedekind $\zeta$-function of $E_i$ has no exceptional Landau-Siegel zero.
\end{enumerate}
If the sequence $\{x_i\}_i$ has a finite intersection with any proper special subvariety then any weak-$*$ limit point of $\{\mu_i\}_i$ is a convex combination of the uniform probability measures on the connected components of $X$.
\end{thm}

\subsection{Previous Results}
Conjecture \ref{conj:mixing} has been proved by Ellenberg, Michel and Venkatesh \cite{EMV} under the assumption of a single fixed split prime $p_1$ and if the following holds
\begin{equation}\label{eq:emv-condition}
\exists\eta>0\; \forall i\gg 1\colon \mathfrak{N}_i \ll |D|^{1/2-\eta}
\end{equation}
The proof in \cite{EMV} used minor assumptions and applied verbatim only when $\mathbf{G}$ was ramified at infinity and $\gcd\left(\mathfrak{N}_i,p_1\right)=1$. The  assumption on $\mathbf{G}(\mathbb{R})$ can be removed using \cite{Kh17} and the condition $\gcd\left(\mathfrak{N}_i,p_1\right)=1$ can be relaxed by restricting the range of $\eta$ in \eqref{eq:emv-condition} depending on the best available bounds towards the Ramanujan Conjecture for\footnote{A. Venkatesh has described to me an alternative proof of the mixing conjecture assuming \eqref{eq:emv-condition} by directly deducing from an appropriate version of Linnik's Basic Lemma that any limit measure must have maximal entropy for the diagonal toral flow at $p_1$ on $\left[(\mathbf{G}\times\mathbf{G})(\mathbb{A})\right]$. This does not rely on a spectral gap and circumvents completely the difficulties arising when $p_1\mid \mathfrak{N}_i$ in the original argument of \cite{EMV}.} $\mathbf{SL}_2$. 

Condition \eqref{eq:emv-condition} is essential to the method of \cite{EMV} and fails for the majority of possible twists $[\sigma_i]\in\Pic(\Lambda_i)$. The proof strategy of Ellenberg, Michel and Venkatesh is to use \eqref{eq:emv-condition} to find for each $i$ a Hecke correspondence containing the packet $\mathcal{P}_i$ and whose volume is small compared to the volume of $\mathcal{P}_i$. In this favorable situation they  use an effective version of Linnik's method using an explicit spectral gap for the Hecke operator at the split prime $p_1$ to deduce that the counting measure on $\mathcal{P}_i$ is approximately equidistributed in the ambient Hecke correspondence. The proof then concludes using the equidistribution of Hecke correspondences in $Y\times Y$.

The analogues questions for function fields in finite characteristic has been studied by Shende and Tsimerman \cite{ShendeTsimerman}. In the finite characteristic setting additional tools are available. Shende and Tsimerman translate the analogues of Duke's theorem and the mixing conjecture to questions about point counting on (singular) varieties. These can be addressed using the Grothendieck-Lefschetz  trace formula. They present a proof of Duke's theorem in finite characteristic using this method and a partial result towards the mixing conjecture. For the latter question they succeed in equating the pertinent higher cohomology groups but the necessary bound on the dimension of the lower cohomology groups is conjectural.

\subsection{Measure Rigidity}
Linnik has proved Duke's theorem about equidistribution of a sequence packets of CM points on the complex modular curve assuming that there is a fixed prime $p$ which splits in all the CM fields in the sequence \cite{LinnikBook}. In this proof Linnik used his ``ergodic method'' to bootstrap a weak bound on the self-correlation of the periodic measure on a toral packet in intermediate scales to full equidistribution using a dynamical argument. It is this dynamical argument where the assumption of a fixed split prime is used.

Einsiedler, Lindenstrauss, Michel and Venkatesh \cite{ELMVPeriodic,ELMVCubic,ELMVPGL2} have introduced a variant of Linnik's ``ergodic method'' which fits into the framework of homogeneous dynamics. The assumption of a fixed split prime $p$ implies that the adelic, or $S$-arithmetic, periodic measures corresponding to the packets in the sequence are all invariant under a \emph{split} $p$-adic torus. Moreover, the self-correlation bound in the form of Linnik's basic lemma implies that any weak-$*$ limit must have \emph{maximal} entropy with respect to the action of any element in the split $p$-adic torus. Linnik's theorem now follows from the classification of measures of maximal entropy with respect to the action of a semi-simple $p$-adic group element which generates an unbounded subgroup. The latter one is straightforward if one uses the relation between entropy and leafwise measures \cite{MargulisTomanov,Pisa}.

The approach of Einsiedler, Lindenstrauss, Michel and Venkatesh has significant ramifications when combined with the modern methods of measure rigidity for toral actions \cite{QuantumLindenstrauss,EKL,EntropySArithmetic,ELJoinings}. Although measure rigidity requires further splitting assumptions it can imply strong equidistribution results based on weaker arithmetic input compared to methods of harmonic analysis. A main example is the analogue of Linnik's theorem for maximal tori in $\mathbf{PGL}_3$ \cite{ELMVCubic} where equidistribution is deduced by verifying Weyl's equidistribution criterion only for a \emph{small} part of the spectrum.

This strategy is also the starting point for our proof of Theorem \ref{thm:main-intro} and its generalizations. The assumption that two primes $p_1$,$p_2$ are split in all the CM fields in the sequence is required for the joinings theorem of Einsiedler and Lindenstrauss \cite{ELJoinings} to apply. This measure rigidity result concurrently with Linnik's or Duke's theorem for equidistribution of packets in rank $1$ implies that any possible weak-$*$ limit measure of periodic measures on joint toral packet must be algebraic, i.e.\ it is a convex combination of uniform measures and some translates of Hecke correspondences. It is these translates of Hecke correspondence that we need to discard using the genericity assumption $\mathfrak{N}_i\to\infty$ in the conjecture of Michel and Venkatesh \ref{conj:mixing}.

\subsection{Cross-Correlation}\label{intro:cross-correlation}
The main novelty is our method to demonstrate that each limit point of the sequence of measures $\left\{\mu_i^\mathrm{joint}\right\}$ in Theorem \ref{thm:main-intro} is singular to any convex combination of intermediate measures allowed by the joinings theorem of Einsiedler and Lindenstrauss. 

The rudiments of our approach can be described in a general setting. Consider a locally compact $G$-space $X$ where $G$ is a second countable locally compact topological group.
Suppose $\mu$ and $\nu$ are Borel measures on $X\times X$ and denote by $\meas_G$ some fixed Haar measure on $G$. We are interested in the case when $\nu$ is a periodic measure for 
the diagonal subgroup $G^\Delta<G\times G$, i.e.\ there is some $x_0\in X\times X$ such that $\Stab_{G^\Delta}(x_0)$ is a lattice in $G^\Delta$ and
$\nu$ is the $G^\Delta$-invariant probability measure supported on the closed orbit $G^\Delta.x_0$. For any compact subset $C\subset X$ if we take a small enough symmetric identity neighborhood $B\subset G$ then $\nu$-almost every $x\in C\times C$ satisfies
\begin{equation}\label{eq:B-meas-demo}
\nu\left((B\times B).x\right) \asymp \meas_G(B)
\end{equation}

In order to show that the measures $\mu$ and $\nu$ are singular we can consider for each compact $C\subset X\times X$ and for a compact symmetric identity neighborhood $B\subset G$ the cross-correlation quantity
\begin{equation*}
\widetilde{\Cor}_C[\mu,\nu](B)\coloneqq \mu\times \nu \left(\left\{(x,y)\in C\times C \mid y\in (B\times B).x \right\}\right)
\end{equation*}
We call $B$ the \emph{test neighborhood} of the cross-correlation.
Assume we are able to establish that
\begin{equation}\label{eq:cor-demo-bound}
\widetilde{\Cor}[\mu,\nu](B)\ll \meas_G(B)^{1+\rho}
\end{equation}
for some $\rho>0$, for a family of compact subsets $C$ exhausting $X$ and for a family of identity neighborhoods $B\subset G$ with arbitrary small Haar measure. Then  the estimates \eqref{eq:B-meas-demo} and \eqref{eq:cor-demo-bound} imply $\nu\perp \mu$. 

The first observation when studying the cross-correlation between two algebraic measures on an adelic quotient is that it is bounded above by a relative trace of the automorphic kernel with test function $\mathbb{1}_{B\times B}$.
In our setting the relative trace that arises is for the double quotient 
\begin{equation*}
\dfaktor{\mathbf{G}^\Delta}{\mathbf{G}\times\mathbf{G}}{\mathbf{T}^\Delta}
\end{equation*}
where $\mathbf{T}<\mathbf{G}$ is a maximal torus defined over $\mathbb{Q}$ and anisotropic over $\mathbb{R}$ embedded diagonally $\mathbf{T}^\Delta<\mathbf{G}\times\mathbf{G}$. This relative trace has a geometric expansion and the main difficulty is bounding the sum of the relative orbital integrals. We require an upper bound which is optimal up to a uniform multiplicative constant. 

\subsection{Invariants and Integral Ideals}\label{intro:arithmetic-invariants}
Denote by $\Lambda<E$ the order attached to a fixed toral packet. Proposition \ref{prop:intersection-to-invariants} is a fundamental result where we show that a relative orbital integral is bounded in terms of the number of pairs of integral ideals in $\Lambda$ whose norms satisfy an additive relation. A. Venkatesh has pointed out that this bears a similarity to the calculation of heights in the proof of the Gross-Zagier Theorem \cite[\S3]{GrossZagier}.

The construction of these integral ideals can be described in an elementary fashion. Assume we are in the setting of the modular curve $Y_0(1)$. Let $\Lambda=\mathcal{O}_E$ be the maximal order in an imaginary quadratic field $E/\mathbb{Q}$ with discriminant $D<0$. Fix a twist $[\sigma] \in \Cl(E)$. The joint packet in $Y_0(1)\times Y_0(1)$ is the set $\left\{\left(H_{[I]},H_{[I\sigma]}\right) \mid [I]\in \Cl(E)\right\}$, where $H_{[I]}\in Y_0(1)$ is the CM point attached to the ideal class $[I]$. For simplicity we only discuss how to show non-accumulation on the diagonal $Y_0(1)^\Delta\hookrightarrow Y_0(1)\times Y_0(1)$.

Let $B_\delta\subset \mathbf{PGL}_2(\mathbb{R})$ be the identity neighborhood of radius $\delta>0$. The cross-correlation between the joint packet and the diagonal with test neighborhood $B_\delta$ is a weighted count of the number of points $\left(H_{[I]},H_{[I\sigma]}\right)$ such that the hyperbolic distance $d(H_{[I]},H_{[I\sigma]})$ is less then $2\delta$. The weight is a continuous decreasing function of $d(H_{[I]},H_{[I\sigma]})$ that vanishes when $d(H_{[I]},H_{[I\sigma]})=2\delta$. For simplicity, consider the unweighted quantity 
\begin{equation*}
\left|\left\{\left(H_{[I]},H_{[I\sigma]}\right) \mid d(H_{[I]},H_{[I\sigma]})\leq 2\delta,\; [I]\in \Cl(E) \right\}\right|
\end{equation*}
For each element in the set above we write an explicit expression for the CM points in the standard fundamental domain in $\mathbb{H}\subset\mathbb{C}\setminus\mathbb{R}$
\begin{equation*}
H_{[I]}=\frac{-b+i\sqrt{|D|}}{2\mathfrak{N}},\; H_{[I\sigma]}=\frac{-b'+i\sqrt{|D|}}{2\mathfrak{N}'}
\end{equation*}
where $I=\left\langle\mathfrak{N}, \frac{-b+i\sqrt{|D|}}{2} \right\rangle, I'=\left\langle\mathfrak{N}', \frac{-b'+i\sqrt{|D|}}{2} \right\rangle \subset E \subset \mathbb{C}$ are the primitive fractional ideals in the classes $[I]$ and $[I\sigma]$ respectively. Consider the elements
\begin{align*}
x&=\mathfrak{N}\frac{-b'+i\sqrt{|D|}}{2}-\mathfrak{N}'\frac{-b+i\sqrt{|D|}}{2}\in I\cdot I'\\
y&=\mathfrak{N}\frac{-b'-i\sqrt{|D|}}{2}-\mathfrak{N}'\frac{-b+i\sqrt{|D|}}{2}\in I\cdot \tensor[^\sigma]{I}{}'\\
\end{align*}
and define
\begin{align*}
\mathcal{O}_E\supset \mathfrak{a}&=y/(I\cdot \tensor[^\sigma]{I}{}')\in [\sigma] \mod \Cl(E)\\
\mathcal{O}_E\supset \mathfrak{b}&=x/(I\cdot I')\in [I^{-2}\sigma^{-1}] \mod \Cl(E)
\end{align*}
A simple calculation shows that
\begin{align*}
\Nr(\mathfrak{a})&=\frac{(\mathfrak{N}b'-\mathfrak{N}'b)^2+|D|(\mathfrak{N}+\mathfrak{N}')^2}
{4\mathfrak{N}\mathfrak{N}'}\\
\Nr(\mathfrak{b})&=\frac{(\mathfrak{N}b'-\mathfrak{N}'b)^2+|D|(\mathfrak{N}-\mathfrak{N}')^2}
{4\mathfrak{N}\mathfrak{N}'}=\frac{|D|}{4}\left(\cosh(d(H_{[I]},H_{[I\sigma]}))-1\right)\ll |D|\meas(B_{2\delta})\\
\Nr(\mathfrak{a})&-\Nr(\mathfrak{b})=\frac{\Nr(y)-\Nr(x)}{\mathfrak{N}\mathfrak{N}'}=|D|
\end{align*}
This construction demonstrates the relation between the mass of the joint packet in a neighborhood of the diagonal and counting pairs of \emph{integral} ideals satisfying and additive norm relations and whose norms are bounded by a multiple of $|D|$. To establish this relation formally we need to check how close is the map $\inv\left(H_{[I]},H_{[I\sigma]}\right)=\left(\mathfrak{a},\mathfrak{b}\right)$ to being injective. The most serious problem with injectivity arises if $\mathfrak{b}=0$. In our special case it is easy to see that this happens only if $[I]=[I\sigma]\Leftrightarrow [\sigma]=e$. This situation is excluded by the assumption $\mathfrak{N}_i\to\infty$ in Conjecture \ref{conj:mixing}. The full strength of this assumption is needed to establish non-concentration on any translate of a Hecke correspondence.

If $\mathfrak{b}\neq 0$ it turns out that injectivity can fail, in a mild way, only at the ramified primes $p\mid D$. This is the essence of Proposition \ref{prop:fiber-of-invariants}. This lack of injectivity is compensated by the fact that $\mathfrak{b}$ is restricted to a fixed class in $\Cl(E)/\Cl(E)^2$. The analysis of the fibers of the map $\inv$ and taking into account the restriction modulo $\Cl(E)^2$ produces significant technical complications. These can be avoided if one assumes that $|D|$ is prime.

The full expression for the cross-correlation is a weighted sum over elements in the joint packet that are contained in a $\delta$-neighborhood of the diagonal. The weight is easily seen to be bounded by $\ll \meas(B_{2\delta})$.

The author has not arrived at the construction of the invariants $\left(\mathfrak{a},\mathfrak{b}\right)$ through this calculation. Rather the ideals $\mathfrak{a}$,$\mathfrak{b}$ arose naturally in a geometric expansion of a relative trace. The classical interpretation above is due to the referee.
\subsection{Shifted-Convolution Sums}
We now describe how to bound the cross-correlation with test neighborhood $B_\delta\subset\mathbf{G}(\mathbb{R})$ between a joint toral packet and a fixed translate of a Hecke correspondence, e.g.\ the diagonal.

The final outcome of the arithmetic analysis of the relative orbital integrals in term of pairs of integral invertible $\Lambda$-ideals satisfying an additive norm relation is that the cross-correlation is bounded by an expression proportional to a shifted convolution sum which is roughly of the form
\begin{equation}\label{eq:shifted-conv-intro}
\mathscr{S}\coloneqq \sum_{0<x-|D|\leq \kappa \meas(B_{2\delta})|D|} g(x)f(x-|D|)
\end{equation} 
where $D=\disc(\Lambda)$, $f(x)$ is the multiplicative function which counts the number of invertible integral $\Lambda$-ideals of norm $x$; and $g$ counts the same ideals as $f$ but with the additional restriction that they belong to the \emph{fixed} Picard class $[\sigma]$. The class $[\sigma]\in\Pic(\Lambda)$ is the Picard class of a single twist in Theorem \ref{thm:main-intro}. For the sake of simplicity we consider for now the real number  $\kappa>0$ as a universal constant. We have neglected the non-injectivity of the invariant map as discussed above and the restrictions modulo $\Pic(\Lambda)^2$.

It is not difficult to show that if we extend the range of summation in $\mathscr{S}$ then the asymptotic mean value is $\asymp \frac{\rho}{\sqrt{|D|}}$ where $\rho$ is the residue at $1$ of the Dedekind $\zeta$-function of $E$. In order to complete the proof we need to show a comparable, up to a fixed constant, upper bound in the extremely short range $\kappa \meas(B_{2\delta})|D|$. Unfortunately, the various methods from harmonic analysis to estimate shifted convolution sums are of no use in this short range of summation.

We proceed instead using a sieve. Let $q(x,y)$ be the \emph{reduced} primitive integral binary quadratic form corresponding to the class $[\sigma]^{-1}\in\Pic(\Lambda)$. Denote by $\mathscr{E}\subset \mathbb{R}^2$ the elliptical annulus of area $2\pi\meas(B_{2\delta})\kappa\sqrt{|D|}$ defined by
\begin{equation*}
\mathscr{E}\coloneqq \left\{(x,y)\in \mathbb{Z}^2 \mid |D|<q(x,y)\leq (1+\meas(B_{2\delta}) \kappa) |D| \right\}
\end{equation*}
The sum $\mathscr{S}$ is tautologically equal to
\begin{equation*}
\sum_{(x,y)\in \mathscr{E}\cap\mathbb{Z}^2} f(q(x,y)-|D|)
\end{equation*}
The latter is a sum of a multiplicative function over the values of a $2$-variable polynomial. We generalize the method of Shiu \cite{Shiu} and Nair \cite{Nair} to polynomials in $2$-variables on smooth domains in order to deduce a bound of the form
\begin{equation*}
S\ll A(\mathscr{E})(\log |D|)^{-1}\sum_{a \ll |D|} \frac{f(a)}{a}
\end{equation*}
where $A(\mathscr{E})$ is the area of the ellipse $\mathscr{E}$. In order to derive an upper bound of the correct order of magnitude for the logarithmic sum above we need to assume the lack of an exceptional zero. 

It is important to mention that the sieve method fails when the ellipse $\mathscr{E}$ is distorted too much. Fortunately, this is exactly the case when the proof method of Ellenberg, Michel and Venkatesh \cite{EMV} applies.

The approach to bounding $\mathscr{S}$ using a sieve is inspired by the work of Bourgain, Sarnak and Rudnick \cite{BSR}. Sieve methods have been fruitfully applied to shifted convolution sums in other contexts as well. Holowinsky has used a related argument in his work on holomorphic QUE \cite{Holowinsky,HolowinnskySound}. P. Michel has pointed out to the author that in the scenario considered by Holowinsky the shifted convolutions arise from the $L$-functions of symmetric squares of holomorphic forms which are known not to have an exceptional zero; in contrast to the case of $\mathscr{S}$ above.

\subsection{Further Discussion}
\subsubsection{Archimedean versus \texorpdfstring{$p$}{p}-adic Cross-correlation}
In the exposition above we have presented a method to show that joint packets of CM points do not accumulate on the diagonal diagonal using a cross-correlation quantity that uses an archimedean test  neighborhood $B_\delta\subset \mathbf{G}(\mathbb{R})$. In the actual proof we shall use a non-archimedean neighborhood at one of the primes, say $p_1$, where all the tori in the sequence were assumed to be split. 

This modification is necessary because the measure rigidity argument does \emph{not} imply that any weak-$*$ limit of $\mu_i^\mathrm{joint}$ in Theorem \ref{thm:main-intro} is a \emph{countable} convex combination of algebraic measures. We may not reduce to a countable collection of possible ergodic components because the normalizer of a diagonally embedded rank $1$ torus $\mathbf{T}^\Delta<\mathbf{G}\times \mathbf{G}$ contains the subgroup $\mathbf{T}\times\mathbf{T}$ which is much bigger then $\mathbf{T}^\Delta$.

Assume for simplicity that $Y$ in Theorem \ref{thm:main-intro} is connected. If we restrict to the archimedean setting, then the possible obstructions to equidistribution are all periodic orbits of the form $$[\delta \mathbf{G}^\Delta(\mathbb{R})^+\xi_{\mathbb{R}})]\subset \lfaktor{\Gamma}{\mathbf{G}(\mathbb{R})}\times \lfaktor{\Gamma}{\mathbf{G}(\mathbb{R})}$$ where $\mathbf{G}(\mathbb{R})^+$ is the real image of the isogeny from the simply connected cover $\mathbf{G}^{\sc}\to\mathbf{G}$, $\delta\in\left(\mathbf{G}\times\mathbf{G}\right)(\mathbb{Q})$ and $\xi_{\mathbb{R}}\in \left(\mathbf{G}\times\mathbf{G}\right)(\mathbb{R})$ is \emph{any} element.

To see how this creates a problem in the argument, notice that the contradiction in \S\ref{intro:cross-correlation} has used the fact \eqref{eq:B-meas-demo} that $\nu\left((B\times B).x\right) \asymp \meas(B)$ for the Haar measure $\meas$ on $\mathbf{G}(\mathbb{R})$. While this is true if $\nu$ is a countable combination of algebraic measures supported on translates of $\mathbf{G}^\Delta(\mathbb{R})^+$, it can be wrong for uncountable families. In particular, such an uncountable convex combination can even be absolutely continuous with respect to the Haar measure $\meas\times \meas$ (even if we fix $\delta$ above). This phenomenon is analogous to the statement that an uncountable combination of Lebesgue measures on $1$-dimensional lines in $\mathbb{R}^2$ can have an arbitrary dimension in the interval $[1,2]$.

To overcome this difficulty we use instead the cross-correlation for a non-archimedean neighborhood $B\subset\mathbf{G}(\mathbb{Q}_{p_1})$. Let $S=\{\infty,p_1\}$. There is a canonical lift of the each measure $\mu_i^\mathrm{joint}$ to a probability measure on a fixed $S$-arithmetic homogeneous space $\lfaktor{\Gamma_S}{\mathbf{G}(\mathbb{Q}_S)}$. Each lift is a finite combination of periodic measures for the  diagonal embedding of the torus $K_\infty\times A_{p_1}<\mathbf{G}(\mathbb{Q}_S)$, where $A_{p_1}<\mathbf{G}(\mathbb{Q}_{p_1})$ is a split torus independent of the index $i$. The measure rigidity argument now implies that the obstruction to equidistribution is a convex combination of algebraic measures supported on $\left[\delta\mathbf{G}^\Delta (\mathbb{Q}_S)^+(\xi_{\mathbb{R}},\xi_{p_1})\right]$ with $\xi_{\mathbb{R}}\in \left(\mathbf{G}\times\mathbf{G}\right)(\mathbb{R})$, $\xi_{p_1}\in \left(\mathbf{G}\times\mathbf{G}\right)(\mathbb{Q}_{p_1})$ and $\delta$ is a rational element. In this setting there is an additional restriction\footnote{A shadow of this condition appears in the archimedean setting as well, the element $\delta$ in the archimedean case cannot be an arbitrary rational point and is restricted at the primes $p=p_1,p_2$. It is not clear how to put this information to good use.} on $\xi_{p_1}=(\xi_{p_1}^1,\xi_{p_1}^2)$ that $(\xi_{p_1}^1)^{-1}\xi_{p_1}^2\in A_{p_1}$. This restriction appears in the $S$-arithmetic setting because each periodic measure in the sequence of packets was invariant under a fixed \emph{split} torus $A_{p_1}^\Delta$. This additional piece of information allows us to rule out accumulation on convex combinations of an uncountable family of algebraic measures. 

Let $\nu$ be any convex combination of algebraic measures supported on closed orbits of the form $\left[\delta\mathbf{G}^\Delta (\mathbb{Q}_S)^+(\xi_{\mathbb{R}},\xi_{p_1})\right]$ , which are all invariant under $A_{p_1}^\Delta$. The gist of the argument is that for every $a\in A_{p_1}^\Delta$ the metric entropy $\ent_\nu(a)$ is a convex combination of the metric entropies on individual periodic measures. There is a relation between metric entropy and self-correlations which implies for a suitable $p_1$-adic identity neighborhood $B$ that $\nu\left((B\times B).x\right)$ cannot \emph{on average} decay as $\meas(B)^{1+\rho}$ for any $\rho>0$ . This is enough to conclude the necessary contradiction.

\subsubsection{The Assumption on the Conductor}
In Theorem \ref{thm:main-intro} we have assumed the discriminants $D_i$ are all fundamental. The slightly more general version in Theorem \ref{thm:joinings-adelic} allows non-trivial conductors $\mathcal{f}_i$ but they should be uniformly bounded.
The difficulty with removing this assumption is that if we allow a non-trivial conductor $\mathcal{f}$  then the shifted convolution sum in \eqref{eq:shifted-conv-intro} is $\asymp  \kappa \meas(B_{2\delta}) \rho \sqrt{|D|}\mathcal{f}$. The extra factor of $\mathcal{f}$ appears because the shift $|D|$ is divisible by $\mathcal{f}$ and at primes dividing the shift there is no decoupling between the arithmetic functions $f$ and $g$ in $\mathscr{S}$. The best bound we can expect for $\mathscr{S}/|\Pic(\Lambda)|$ is proportional to $\mathcal{f}$ and tends to $\infty$ if $\mathcal{f}$ is unbounded. Such a bound is useless for our purposes.

In an upcoming work the author will explain how to overcome this problem by refining the map attaching a pair of integral ideals to orbital integrals. The new map will not be valued just in pairs of integral ideals but will carry additional information.
\subsection{Organization of the Paper}
In \S\ref{sec:preliminaries} we define the basic notions we work with in the rest of the paper. 

In \S\ref{sec:results} we present the main theorems in adelic terms and prove some auxiliary propositions. 

In \S\ref{sec:rigidity} we apply the joinings theorem of Einsiedler and Lindenstrauss to the problem at hand. 

In \S\ref{sec:B-coordinates} we review and prove basic facts about explicit representations of quaternion algebras in coordinate form. We also describe representatives `in lowest terms' for elements of the projective group of units of a quaternion algebra over local fields. 

In \S\ref{sec:GIT} we construct the double quotient $\dfaktor{\mathbf{G}^\Delta}{\mathbf{G}\times\mathbf{G}}{\mathbf{T}^\Delta}$ using GIT and study its properties over $\mathbb{Q}$. This variety is essential for the geometric expansion of the relative trace appearing later on.

In \S\ref{sec:Hecke} we study basic properties of the intermediate measures arising as obstructions to equidistribution.

Section \S\ref{sec:cross-correlation} is a key part of this paper where we study the cross-correlation between a periodic toral measure and a translated Hecke correspondence using a relative trace. Most importantly, we demonstrate the relation between this relative trace and shifted-convolution sums. This requires interpretation of the non-archimedean relative orbital integrals as intersection numbers and explicit parametrization of the relevant intersections using arithmetic invariants.

In \S\ref{sec:sieving} we generalizes the results of \cite{Shiu,Nair} to sums of multiplicative functions along values of polynomials in $2$-variables on smooth domains. This section may be of independent interest.

In \S\ref{sec:proof} we combine all of the previously developed tools to a proof of the main theorem.

In Appendix \ref{appndx:principal-genus} we review the classical principal genus theory for quadratic orders and provide complete proofs in a form useful to us. These results are necessary in translating the shifted-convolution sums that arise from the relative trace into sums of multiplicative functions over values of polynomials.

In Appendix \ref{appndx:conics} we do routine calculations of the number of points on some singular conics over $\cyclic{N}$. These are necessary to translate the upper-bound on the cross-correlation we have after applying the sieve method into a sum treatable using analytic number theory.

\subsection{Acknowledgments}
It is a pleasure to thank Peter Sarnak for numerous fruitful and enlightening discussions on this project and for the observation that the results for toral orbits are relevant to the equidistribution of Galois orbits on products of modular curves. I am indebted to Elon Lindenstrauss for his continuous encouragement and his interest in this project. I am grateful to Akshay Venkatesh for valuable conversations and his assistance in clarifying the results of \cite{EMV}.
I thank Manjul Bhargava, Fabian Gundlach and Shou-Wu Zhang for helpful discussions.
 
I am grateful to Philippe Michel and Wei Zhang for very useful and illuminative comments on a previous version of this manuscript. I am deeply indebted to the referee for pointing out that Lemma 10.11 in the original manuscript was wrong and for bringing forth the description in \S \ref{intro:arithmetic-invariants} of the invariant ideals in a classical language. 

The author was supported by a Schmidt Fellowship at the Institute for Advanced Study during 2017-2018.

\section{Preliminaries}\label{sec:preliminaries}
\subsection{Notations and Conventions}
\begin{enumerate}[leftmargin=*]
\item We denote by the letter $v$ a place of $\mathbb{Q}$.
For a non-archimedean place $v$ define $q_v$ to be the size of the residue field of $\mathbb{Q}_v$.

\item For a linear algebraic group $\mathbf{M}$ defined over $\mathbb{Q}$ we denote
\begin{equation*}
[\mathbf{M}(\mathbb{A})]\coloneqq \lfaktor{\mathbf{M}(\mathbb{Q})}{\mathbf{M}(\mathbb{A})}
\end{equation*}
More generally for any subset $U\subseteq\mathbf{M}(\mathbb{A})$ we denote by $[U]$ its projection to $[\mathbf{M}(\mathbb{A})]$. We also use the notation $[g]$ for the coset of $g\in\mathbf{M}(\mathbb{A})$ in $[\mathbf{M}(\mathbb{A})]$.

\item
If $\mathbf{M}$ is anisotropic over $\mathbb{Q}$, i.e.\ there are no characters $\mathbf{M}\to\Gm$ defined over $\mathbb{Q}$, then  
the locally compact space $[\mathbf{M}(\mathbb{A})]$ carries a unique $\mathbf{M}(\mathbb{A})$-invariant probability measure which we call the Haar measure on $[\mathbf{M}(\mathbb{A})]$ and denote by $\meas_\mathbf{M}$. We use the notation $\meas_\mathbf{M}$ also for the covolume $1$ Haar measure on $\mathbf{M}(\mathbb{A})$.

\item For $S$, a finite set of places of $\mathbb{Q}$, we denote
\begin{align*}
\mathbf{M}(\mathbb{A}^S)&\coloneqq{\prod}'_{v\not\in S}\mathbf{M}(\mathbb{Q}_v)\\
\mathbf{M}(\mathbb{Q}_S)&\coloneqq\prod_{v\in S}\mathbf{M}(\mathbb{Q}_v)\\
\end{align*}

\item If $\mathbf{L}<\mathbf{M}$ is a closed algebraic subgroup denote the diagonal embedding of algebraic groups by $\mathbf{L}^\Delta<\mathbf{M}\times\mathbf{M}$. We use the similar notation $L^\Delta<M\times M$ for the diagonal embedding of a closed subgroup $L<M$ in a locally compact group $M$.

\item For any algebraic group $\mathbf{M}$ the morphism $\ctr\colon\mathbf{M}\times\mathbf{M}\to\mathbf{M}$ is defined by $(g_1,g_2)\mapsto g_1^{-1}g_2$.

\item For a reductive linear algebraic group $\mathbf{M}$ we denote by $\mathbf{M}^\sc$ its simply-connected cover. We fix an isogeny $\mathbf{M}^\sc\to\mathbf{M}$ and denote for any ring $R$ the image of $\mathbf{M}^\sc(R)$ in $\mathbf{M}(R)$ by $\mathbf{M}(R)^+$.

\item If $L<M$ is a unimodular closed subgroup of a unimodular locally compact group $M$ with  fixed Haar measures $\meas_L$ and $\meas_M$ respectively then we always normalize the $M$-invariant Haar measure on $\lfaktor{L}{M}$ so that
\begin{equation*}
\int_M f \dif \meas_M =\int_{L\backslash M} 
{\left(\int_L f(lg) \dif \meas_L(l)\right)}
\dif \meas_{L\backslash M}(Lg)
\end{equation*}
for any $f\in\mathscr{L}^1(M)$.

\item For $F$ a global field or a finite product of non-archimedean local fields we denote by $\mathcal{O}_F$ the ring of integers --- the unique maximal order, $F^{(1)}$ the multiplicative subgroup of $F^\times$ of norm $1$ elements and $\mathcal{O}_F^{(1)}$ the multiplicative group of norm $1$ integral elements.

\item When $F$ as above is a quadratic extension of either $\mathbb{Q}$ or $\mathbb{Q}_v$ it is equipped with an action of the Galois group $\mathfrak{G}\simeq \cyclic{2}$. We define the coboundary map
\begin{align*}
\cbd\colon F^\times&\to F^{(1)}\\
x&\mapsto \frac{x}{\tensor[^\sigma]{x}{}}
\end{align*}
Hilbert's Satz 90 implies that this map surjective.
\end{enumerate}

\subsection{Forms of \texorpdfstring{$\mathbf{PGL}_2$}{PGL} and Locally Homogeneous Spaces}
Let $\mathbf{B}$ be a quaternion algebra defined over $\mathbb{Q}$. Denote $\mathbf{Z}\coloneqq \Cent\mathbf{B}^\times$ --- the center of $\mathbf{B}^\times$ --- and define $\mathbf{G}\coloneqq\lfaktor{\mathbf{Z}}{\mathbf{B}^\times}$ to be the projective group of units. The linear group $\mathbf{G}$ is a form of $\mathbf{PGL}_2$ defined over $\mathbb{Q}$ and all $\mathbb{Q}$-forms of $\mathbf{PGL}_2$ arise this way.
A central object in our discussion is the finite volume adelic locally homogeneous space
\begin{equation}\label{eq:adelic-quotients-iso}
[\mathbf{G}(\mathbb{A})]=\lfaktor{\mathbf{G}(\mathbb{Q})}{\mathbf{G}(\mathbb{A})}
\simeq
\lfaktor{\mathbf{Z}(\mathbb{A})\mathbf{B}^\times(\mathbb{Q})}{\mathbf{B}^\times(\mathbb{A})}
\end{equation}

\subsubsection{Maximal Order in \texorpdfstring{$\mathbf{B}(\mathbb{Q})$}{B(Q)}}\label{sec:maximal-order-B}
Fix a maximal $\mathbb{Z}$-order $\mathbb{O}\subset\mathbf{B}(\mathbb{Q})$. For any non-archimedean place $v$ denote the $v$-adic closure of $\mathbb{O}$ by $\mathbb{O}_v\subset\mathbf{B}(\mathbb{Q}_v)$. The $\mathbb{Z}_v$-order $\mathbb{O}_v$ is maximal in the quaternion algebra $\mathbf{B}(\mathbb{Q}_v)$, c.f.\ \cite[(11.2)]{Reiner}.
For non-archimedean $v$ define the compact subgroup $\mathbb{O}_v^\times<\mathbf{B}^\times(\mathbb{Q}_v)$ and let the compact-open subgroup $K_v<\mathbf{G}(\mathbb{Q}_v)$ be its image under the quotient map $\mathbf{B}^\times\to\mathbf{G}$.

We define the adelic points of $\mathbf{B}^\times$ and $\mathbf{G}$ as a restricted product with respect to the compact subgroups $\mathbb{O}_v$ and $K_v$ respectively.
Moreover, for any finite set $S$ of places containing $\infty$ we denote
\begin{align*}
\mathbb{O}^{\times,S}&\coloneqq \prod_{v\not\in S} \mathbb{O}^\times_v\\
K^S&\coloneqq \prod_{v\not\in S} K_v
\end{align*}
and
\begin{align*}
\mathbb{O}^\times_f&\coloneqq\mathbb{O}^{\times,\{\infty\}}= \prod_{v\neq \infty} \mathbb{O}^\times_v\\
K_f&\coloneqq K^{\{\infty\}}=\prod_{v\neq\infty} K_v
\end{align*}
We need to review some elementary properties of $K_v$ and $\mathbb{O}_v$ for different places $v$.

\paragraph{Split Non-archimedean Places}
If $\mathbf{B}$ is split over a non-archimedean $v$, i.e.\ $\mathbf{B}(\mathbb{Q}_v)\simeq \mathbf{M}_2(\mathbb{Q}_v)$, then the maximal orders of $\mathbf{B}(\mathbb{Q}_v)$ are in bijection with the vertices of the reduced Bruhat-Tits tree of $\mathbf{B}^\times(\mathbb{Q}_v)$. Explicitly, fix an isomorphism $\mathbf{B}^\times(\mathbb{Q}_v)\simeq \End_{\mathbb{Q}_v}(\mathbb{Q}_v^2)$ then vertices of the Bruhat-Tits tree correspond to homothety classes of full-rank $\mathbb{Z}_v$-lattices $L\subset \mathbb{Q}_v^2$ and all the maximal orders are of the form $\End_{\mathbb{Z}_v}(L)$. In particular, $\mathbb{O}_v^\times$ is a stabilizer of a vertex in the tree and $K_v$ is a special maximal compact-open subgroup.

\paragraph{Ramified Non-archimedean Places}
If $\mathbf{B}$ is ramified over a non-archimedean $v$ then $\mathbf{B}(\mathbb{Q}_v)$ is a division algebra and $\mathbb{O}_v$ is the unique maximal order --- the integral closure of $\mathbb{Z}_v$ in $\mathbf{B}(\mathbb{Q}_v)$, c.f.\ \cite[(12.8)]{Reiner}.  Because of the uniqueness property of $\mathbb{O}_v$ it is conjugation invariant and
$\mathbb{O}_v^\times$ is a \emph{normal} subgroup of $\mathbf{B}^\times(\mathbb{Q}_v)$.

A quaternion division algebra over $\mathbb{Q}_v$  has ramification index $2$, c.f.\ \cite[(14.3)]{Reiner}, hence $K_v$ is a normal subgroup of index $2$ in the compact group $\mathbf{G}(\mathbb{Q}_v)$.

\paragraph{The Archimedean Analogue of a Maximal Order}
We will also need an archimedean analogue of a local maximal order. Fix once and for all a maximal compact torus $K_\infty<\mathbf{G}(\mathbb{R})$.
We define an isomorphism between $(\mathbf{B}(\mathbb{R}),K_\infty)$ and $(\mathbf{M}_2(\mathbb{R}),\mathbf{PSO}_2(\mathbb{R}))$ to be an isomorphism of algebras $\mathbf{B}(\mathbb{R})\simeq\mathbf{M}_2(\mathbb{R})$ which induces an isomorphism $\mathbf{G}(\mathbb{R})\simeq \mathbf{PGL}_2(\mathbb{R})$ mapping $K_\infty$ to $\mathbf{PSO}_2(\mathbb{R})$. Due to the Skolem-Noether theorem and the fact the normalizer of $\mathbf{PSO}_2(\mathbb{R})$ is $\mathbf{PGO}_2(\mathbb{R})$ such an isomorphism is unique up to composition with $\Ad\mathbf{PGO}_2(\mathbb{R})$. Assume $\mathbf{B}$ is split over $\mathbb{R}$ and fix such an 
isomorphism.

Inner-product norms on $\mathbb{R}^2$ are often used analogously to full rank lattices in the non-archimedean settings. 
Two inner-product norms on $\mathbb{R}^2$ are said to be homothetic if they differ by a positive multiplicative constant.
The action of $\mathbf{GL}_2(\mathbb{R})$ on $\mathbb{R}^2$ induces a transitive action on the space of inner-product norms on $\mathbb{R}^2$. This action descends to a transitive action of $\mathbf{PGL}_2(\mathbb{R})$ on homothety classes of inner-product norms.
Any inner-product norm $|\bullet|\colon\mathbb{R}^2\to\mathbb{R}_{>0}$ induces a sub-multiplicative operator norm on $\mathbf{M}_2(\mathbb{R})$ in the standard way
\begin{equation*}
\|g\|_{|\bullet|}=\sup_{0\neq v\in\mathbb{R}^2} \frac{|gv|}{|v|}
\end{equation*}
This operator norm depends only on the homothety class $\mathbb{R}_{>0} \cdot |\bullet|$. Let $\Stab_{|\bullet|}<\mathbf{PGL}_2(\mathbb{R})$ be the stabilizer of the homothety class of $|\bullet|$.
The closed unit-ball in $\mathbf{M}_2(\mathbb{R})$ with respect to $\|\bullet\|_{|\bullet|}$ is an $\Ad \Stab_{|\bullet|}$-invariant compact identity neighborhood. Unlike the endomorphism ring of a full-rank lattice this closed unit-ball is not a ring but only a multiplicative monoid.

Let $|\bullet|_\infty\colon\mathbb{R}^2\to\mathbb{R}_{>0}$ be the standard Euclidean norm. This is the unique inner-product norm on $\mathbb{R}^2$ stabilized by $\mathbf{O}_2(\mathbb{R})$ and its homothety class $\mathbb{R}^\times |\bullet|_\infty$ is the unique homothety class of inner-product norms stabilized by $\mathbf{PO}_2(\mathbb{R})$. Denote $\|\bullet\|_\infty\coloneqq \|\bullet\|_{|\bullet|_\infty}$ --- the operator norm on $\mathbf{B}(\mathbb{R})$ induced by $|\bullet|_\infty$ and the isomorphism above. This norm does not depend on the choice of isomorphism as it is $\Ad\mathbf{PGO}_2(\mathbb{R})$-invariant.

If $\mathbf{B}(\mathbb{R})$ is ramified we fix an isomorphism of $\mathbf{B}(\mathbb{R})$ and the Hamilton quaternions and define $\|\bullet\|_\infty$ to be the the quaternion norm. This definition does not depend on the choice of isomorphism as the quaternion norm is multiplicative and conjugation invariant. Equivalently, in this case $\|\bullet\|_\infty=\sqrt{\Nrd}$.

In both the ramified and unramified cases the norm $\|\bullet\|_\infty$ satisfies the following useful identity
\begin{equation}\label{eq:inverse-operator-norm}
\forall g\in\mathbf{B}^\times(\mathbb{R})\colon  \|g^{-1}\|_\infty=\frac{\|g\|_\infty}{|\Nrd g|}
\end{equation}

We need the following definitions
\begin{align*}\mathbb{O}_\infty&\coloneqq 
\left\{g\in\mathbf{B}^\times(\mathbb{R}) \mid \|g\|_\infty\leq 1 \right\}\\
\mathbb{O}_\infty^\times&\coloneqq 
\left\{g\in\mathbf{B}^\times(\mathbb{R}) \mid \|g\|_\infty=1, \Nrd g>0 \right\}\simeq \mathbf{SO}_2(\mathbb{R})\\
\widetilde{\Omega}_\infty&\coloneqq
\left\{g\in\mathbf{B}^\times(\mathbb{R}) \mid \|g^{\pm 1}\|_\infty\leq 2, \Nrd g >0 \right\}
\end{align*}
The set $\mathbb{O}_\infty$ is the closed unit-ball of $\|\bullet\|_\infty$, 
$\mathbb{O}_\infty^\times$ is the orientation-preserving isotropy group of the Euclidean norm $|\bullet|_\infty$ and $\widetilde{\Omega}_\infty$ is a connected, symmetric and compact  identity neighborhood. Moreover, $\mathbb{O}_\infty^\times\widetilde{\Omega}_\infty=\widetilde{\Omega}_\infty\mathbb{O}_\infty^\times=\widetilde{\Omega}_\infty$. Elements of $\widetilde{\Omega}_\infty$ satisfy the following inequalities which follow from \eqref{eq:inverse-operator-norm} and submultiplicativity of the operator norm
\begin{align}\label{eq:Omega-tilde-Nrd}
\Nrd g& = \frac{\|g\|_\infty \|g^{-1}\|_\infty}{\|g^{-1}\|_\infty^2}\geq \frac{1}{\|g^{-1}\|_\infty^2}\geq 1/4\\
\Nrd g&=\left(\Nrd g^{-1}\right)^{-1}\leq 4 \nonumber
\end{align}

\subsection{Simply Connected Cover}\label{sec:simply-connected-cover}
Let $\mathbf{G}^\sc\coloneqq \mathbf{B}^{(1)}$ be the group of unit quaternions in $\mathbf{B}$. The group $\mathbf{G}^\sc$ is the simply connected cover of $\mathbf{G}$. For an algebra $R/\mathbb{Q}$ we denote by $\mathbf{G}(R)^+$ the image of $\mathbf{G}^\sc(R)$ in $\mathbf{G}(R)$ under the isogeny map. The subgroup $\mathbf{G}(\mathbb{A})^+<\mathbf{G}(\mathbb{A})$ is normal and the reduced norm map $\Nrd\colon \mathbf{B}^\times\to\ \Gm$ induces a monomorphism of compact abelian groups
\begin{equation*}
\Nrd\colon\faktor{\mathbf{G}(\mathbb{A})}{\mathbf{G}(\mathbb{A})^+}\to
\faktor{\mathbb{A}^\times}{{\mathbb{A}^\times}^2}
\end{equation*}

To determine the image of $\Nrd(\mathbf{B}^\times(F))$ for a field $F/\mathbb{Q}$ notice that all the elements with a fixed reduced norm form a torsor of $\mathbf{G}^\sc$ defined over $F$. As such it has an $F$-point only if it is the trivial torsor. This can be checked using the Galois cohomology of $\mathbf{G}^\sc$. The cohomology group is trivial for each $p$-adic field, cf.\ \cite{Kneser-p-adic}, hence each element in $\mathbb{Q}_p^\times$ is a reduced norm; this can also be simply deduced from checking the two possible quaternion algebras over $\mathbb{Q}_p$. For the archimedean field $\mathbb{R}$ there are two possible quaternion algebras. In the split case every element of $\mathbb{R}^\times$ is a reduced norm and for the Hamilton quaternions only the positive elements $\mathbb{R}_{>0}$ are reduced norms.

For the global field $\mathbb{Q}$ this question is answered by the Hasse-Schilling-Maass theorem, cf.\ \cite[Theorem 33.15]{Reiner}. The following global-to-local map is injective as $\mathbf{G}^\sc$ is simple and simply connected
\begin{equation*}
H^1(\mathbb{Q},\mathbf{G}^\sc)\hookrightarrow H^1(\mathbb{R},\mathbf{G}^\sc)
\end{equation*}
Hence if $\mathbf{B}$ is split at $\infty$ then every element of $\mathbb{Q}^\times$ is a reduced norm, otherwise only elements of $\mathbb{Q}_{>0}$ are reduced norms.

The reduced norm defines 
a monomorphism of double coset spaces
\begin{equation*}\dfaktor{\mathbf{G}(\mathbb{Q})}{\mathbf{G}(\mathbb{A})}{\mathbf{G}(\mathbb{A})^+}\xrightarrow{\Nrd}\dfaktor{\mathbb{Q}^\times}{\mathbb{A}^\times}{{\mathbb{A}^\times}^2}
\end{equation*}
Following the discussion above we know that this morphism has full image if $\mathbf{B}$ is split at $\infty$ and the image is the index-$2$ subgroup $\dfaktor{\mathbb{Q}_{>0}}{\mathbb{R}_{>0}\times\mathbb{A}_f^\times}{{\mathbb{A}^\times}^2}$ otherwise.

\subsection{Toral Periods}\label{sec:toral-periods}
Periodic orbits of tori on $Y$ can be collected into natural arithmetic packets \cite{ELMVPeriodic,ELMVCubic} which generalize the packets of CM points and closed geodesics on the modular curve. 

These are easiest to define in adelic terms. Let $\mathbf{T}<\mathbf{G}$ be a maximal torus defined and anisotropic over $\mathbb{Q}$. We require the torus to be anisotropic so that the space $\lfaktor{\mathbf{T}(\mathbb{Q})}{\mathbf{T}(\mathbb{A})}$ has finite volume.

\subsubsection{Homogeneous Sets and Periodic Measures.}\label{sec:homogeneous-sets}
Einsiedler, Lindenstrauss, Michel and Venkatesh  have defined  in \cite{ELMVCubic}  the notion of a homogeneous toral set. For any $g=(g_v)_v\in\mathbf{G}(\mathbb{A})$ the set
\begin{equation*}
[\mathbf{T}(\mathbb{A})g]\subset[{\mathbf{G}(\mathbb{A})}]
\end{equation*}
is a homogeneous toral set. This set is a right  translate of $[\mathbf{T}(\mathbb{A})]\simeq\lfaktor{\mathbf{T}(\mathbb{Q})}{\mathbf{T}(\mathbb{A})}$ and hence carries a unique probability measure invariant under the locally compact abelian group $H_{\mathbb{A}}\coloneqq g^{-1}\mathbf{T}(\mathbb{A})g$. Denote this measure by $\mu$ and call it the \emph{periodic toral measure}.

\paragraph{Special Places}
Because the measure rigidity arguments we use require an action by a split torus at two different places we fix once and for all two finite rational primes $p_1,p_2$ such that $\mathbf{G}$ is split at $p_1$ and $p_2$. We fix two maximal split tori $A_{p_1}<\mathbf{G}(\mathbb{Q}_{p_1})$ and $A_{p_2}<\mathbf{G}(\mathbb{Q}_{p_2})$ and require that the intersection of $A_{p_1}$ and $K_{p_1}$ is maximal compact in $A_{p_1}$, equivalently, the apartment of $A_{p_1}$ in the Bruhat-Tits buildings contains the vertex stabilized by $K_{p_1}$.
In \S\ref{sec:maximal-order-B} we have already fixed a maximal compact torus $K_\infty<\mathbf{G}(\mathbb{R})$.

We restrict to the case when $\mathbf{T}$ is split at $p_1$ and $p_2$ and anisotropic at $\infty$. Unless stated otherwise we shall always assume that
\begin{equation}\label{eq:g_adel_restrictions}\tag{$\spadesuit$}
g_\infty^{-1}\mathbf{T}(\mathbb{R})g_\infty=K_\infty,\,g_{p_1}^{-1}\mathbf{T}(\mathbb{Q}_{p_1})g_{p_1}=A_{p_1},\,g_{p_2}^{-1}\mathbf{T}(\mathbb{Q}_{p_2})g_{p_2}=A_{p_2}
\end{equation}

\subsubsection{Packets}\label{sec:packets}
Let $S$ be a finite set of rational places containing at least $\infty,p_1,p_2$ and such that the following class number $1$ assumption holds
\begin{equation}\label{eq:S-class-number-1}
\#\dfaktor{\mathbf{G}(\mathbb{Q})}{\mathbf{G}(\mathbb{A})}{\mathbf{G}(\mathbb{Q}_S)\cdot K^S}=1
\end{equation}
The $\mathbf{G}(\mathbb{Q}_S)$-equivariant open embedding 
\begin{align*}
Y&\coloneqq\lfaktor{\Gamma}{\mathbf{G}(\mathbb{Q}_S)} \hookrightarrow \dfaktor{\mathbf{G}(\mathbb{Q})}{\mathbf{G}(\mathbb{A})}{K^S}
\\
\Gamma&\coloneqq \mathbf{G}(\mathbb{Q})\cap K^S
\end{align*}
is an isomorphism due to \eqref{eq:S-class-number-1}.

Denote the projection of $[\mathbf{T}(\mathbb{A})g]$ to $Y$ by $\mathcal{P}$. The set $\mathcal{P}$ is called a packet of periodic torus orbits. It is a union of periodic orbits\footnote{Following \cite{ELMVPeriodic} we say that an orbit of a locally compact group $H$ is periodic if it supports a \emph{finite} $H$-invariant Borel measure.} for the torus $H=\prod_{v\in S} H_v$ where $H_v=g_v^{-1}\mathbf{T}(\mathbb{Q}_v)g_v$ and our choices \eqref{eq:g_adel_restrictions} imply $H_\infty=K_\infty$, $H_{p_1}=A_{p_1}$ and $H_{p_2}=A_{p_2}$. 

\paragraph{Action on Torus Orbits}
Denote 
\begin{align*}
K_\mathbf{T}^S&\coloneqq gK^Sg^{-1}\cap\mathbf{T}(\mathbb{A}^S)<\mathbf{T}(\mathbb{A}^S)\\
K_{\mathbf{T},f}&\coloneqq K_\mathbf{T}^{\{\infty\}}=g K_f g^{-1}\cap \mathbf{T}(\mathbb{A}_f)
<\mathbf{T}(\mathbb{A}_f)
\end{align*}
These are compact-open subgroups of the ambient torus groups.
The following finite abelian group acts simply transitively on the set of $H$-orbits in $\mathscr{P}$ 
\begin{equation*}
C_S\coloneqq\dfaktor{\mathbf{T}(\mathbb{Q})}{\mathbf{T}(\mathbb{A})}{\mathbf{T}(\mathbb{Q}_S)\cdot K_\mathbf{T}^S}
\end{equation*}
The finiteness of $C_S$ implies that $\mathscr{P}$ is a \emph{finite} collection of periodic $H$-orbits.

We can actually incorporate the pointwise action of $H\simeq \mathbf{T}(\mathbb{Q}_S)$ on $\mathcal{P}$ and the action of $C_S$ on the set of $H$-orbits into a \emph{pointwise} action of the single group $\faktor{\mathbf{T}(\mathbb{A})}{K_{\mathbf{T}}^S}$

\paragraph{Periodic Measure on the Packet}
The measure $\mu$ defines a push-forward measure $\overline{\mu}$ on $Y$ supported on $\mathcal{P}$ and invariant under the action of $H$. The measure $\overline{\mu}$ is a finite average of periodic $H$-measures. All the periodic $H$-measures contribute to $\overline{\mu}$ with the same weight as can be seen using the action of $C_S$.

\subsubsection{Homogeneous Toral Sets in $\mathbf{B}^\times$}
Any maximal torus $\mathbf{T}<\mathbf{G}$ defined over $\mathbb{Q}$ is the image of a unique maximal torus $\widetilde{\mathbf{T}}<\mathbf{B}^\times$ defined over $\mathbb{Q}$. 

All maximal rational tori $\widetilde{\mathbf{T}}<\mathbf{B}^\times$ are of the form
\begin{equation*}
\widetilde{\mathbf{T}}\simeq \WR\Gm
\end{equation*}
where $E/\mathbb{Q}$ is a quadratic \'{e}tale-algebra embeddable into $\mathbf{B}(\mathbb{Q})$. More specifically, let $\iota\colon E\hookrightarrow\mathbf{B}(\mathbb{Q})$ be a ring embedding then the image of $\iota$ is the $\mathbb{Q}$-points of a maximal commutative algebra subvariety $\mathbf{E}$ with $\mathbf{E}(\mathbb{Q})=\iota(E)$. The corresponding torus $\widetilde{\mathbf{T}}$ is equal to $\mathbf{E}^\times$. Notice that an \'{e}tale-algebra $E$ does not define the subalgebra $\mathbf{E}<\mathbf{B}$ uniquely as there are many inequivalent ways to embed  $E$ in $\mathbf{B}(\mathbb{Q})$. The subalgebra $\mathbf{E}$ is defined by a specific embedding $\iota$, up to an automorphism.

Our requirement that $\mathbf{T}=\lfaktor{\mathbf{Z}}{\widetilde{\mathbf{T}}}\simeq\lfaktor{\Gm}{\WR\Gm}$ is anisotropic over $\mathbb{Q}$ is equivalent to $E$ being a quadratic field. The conditions \eqref{eq:g_adel_restrictions} imply that $E$ is imaginary and split at $p_1$ and $p_2$.

Choose any representative of $g$ in $\mathbf{B}^\times(\mathbb{A})$ and by abuse of notations denote it by $g$ as well. The isomorphism of adelic quotients \eqref{eq:adelic-quotients-iso} induces an identification of homogeneous toral sets
\begin{equation*}
[\mathbf{T}(\mathbb{A})g]=[\widetilde{\mathbf{T}}(\mathbb{A})g]
\subset
\lfaktor{\mathbf{Z}(\mathbb{A})\mathbf{B}^\times(\mathbb{Q})}{\mathbf{B}^\times(\mathbb{A})}
\end{equation*}

\paragraph{Class Group Action}
Let $S$ be a finite set of rational places as in \S\ref{sec:packets}. Define as before
\begin{align*}
K_{\widetilde{\mathbf{T}}}^S&\coloneqq g\mathbb{O}^{\times,S}g^{-1}\cap\widetilde{\mathbf{T}}(\mathbb{A}^S)<\widetilde{\mathbf{T}}(\mathbb{A}^S)\\
K_{\widetilde{\mathbf{T}},f}&\coloneqq K_{\widetilde{\mathbf{T}}}^{\{\infty\}}=g \mathbb{O}^\times_f g^{-1}\cap \mathbf{T}(\mathbb{A}_f)
<\mathbf{T}(\mathbb{A}_f)
\end{align*}

Because of our choice of $K_v$ to be the projection of $\mathbb{O}_v^\times$ 
there is also an surjective homomorphism of finite abelian groups
\begin{align}\label{eq:Cl_S_S-isomorphism}
\dfaktor{E^\times}{\mathbb{A}_E}{E_S\cdot K_{\widetilde{\mathbf{T}}}^S}
&=
\dfaktor{\widetilde{\mathbf{T}}(\mathbb{Q})}{\widetilde{\mathbf{T}}(\mathbb{A})}{\widetilde{\mathbf{T}}(\mathbb{Q}_S)\cdot K_{\widetilde{\mathbf{T}}}^S}\\
&\twoheadrightarrow \nonumber
\dfaktor{\mathbf{T}(\mathbb{Q})}{\mathbf{T}(\mathbb{A})}{\mathbf{T}(\mathbb{Q}_S)\cdot K_\mathbf{T}^S}=C_S
\end{align}
where $K_{\widetilde{\mathbf{T}}}^S\coloneqq g\left(\prod_{v\not\in S} \mathbb{O}_v^{\times}\right)g^{-1}\cap\widetilde{\mathbf{T}}(\mathbb{A}^S)$ is a compact-open subgroup in $\widetilde{\mathbf{T}}(\mathbb{A}^S)$.

The kernel of this map is the following quotient 
\begin{equation*}
\dfaktor{\Gm(\mathbb{Q})}{\Gm(\mathbb{A})}{\Gm(\mathbb{Q}_S)\cdot \prod_{v\not\in S}\mathbb{Z}_v^\times}
\end{equation*}
which is trivial because $\mathbb{Q}$ has a trivial class group. We see that \eqref{eq:Cl_S_S-isomorphism} is actually an isomorphism. We have thus expressed $C_S$ in a natural way as a quotient of the id\'{e}le class group of $E$. It is natural to consider $C_S$ as a generalized $S$-class group of the field $E$.

\subsubsection{Quadratic Orders and Discriminants}
\paragraph{The Local Order and Local Discriminant}\label{sec:discriminant-local}
\begin{defi} \hfill
\begin{enumerate}
\item Recall that $\widetilde{\mathbf{T}}=\mathbf{E}^\times$ where $\mathbf{E}<\mathbf{B}$ is a maximal commutative algebra. For each place $v$ we define 
\begin{equation*}
\Lambda_v\coloneqq \mathbf{E}(\mathbb{Q}_v)\cap g_v \mathbb{O}_v g_v^{-1}
\end{equation*}
For $v$ non-archimedean $\Lambda_v$ is a commutative ring and an order in the \'{e}tale-algebra $\mathbf{E}(\mathbb{Q}_v)\simeq E_v$. 

\item For $v$ non-archimedean denote the maximal order of the  \'{e}tale-algebra $E_v$ by $\mathcal{O}_{E_v}$, i.e.\ $\mathcal{O}_{E_v}=\prod_{w|v} \mathcal{O}_{E_w}$.
\end{enumerate}
\end{defi}

\begin{prop}\label{prop:Lambda-ae-maximal}
For almost all $v$ non-archimedean $\Lambda_v=\mathcal{O}_{E_v}$.
\end{prop}
\begin{proof}
As $\mathbb{Q}\cdot\mathbb{O}=\mathbf{B}(\mathbb{Q})$ we see that $\Lambda^\mathrm{naive}\coloneqq\mathbb{O}\cap\mathbf{E}(\mathbb{Q})$ is a $\mathbb{Z}$-lattice of full rank in the $2$-dimensional $\mathbb{Q}$-vector space $\mathbf{E}(\mathbb{Q})$. We can extend any $\mathbb{Z}$-basis $b$ of $\Lambda^\mathrm{naive}$ to a $\mathbb{Z}$-basis $b\cup c$ of\footnote{This can be seen from the fact that $\faktor{\mathbb{O}}{\Lambda^\mathrm{naive}}$ is a finitely generated torsion-free $\mathbb{Z}$-module and each such module is free.} $\mathbb{O}$.

Fix $v$ non-archimedean and denote $\Lambda^\mathrm{naive}_v\subset \mathbf{E}(\mathbb{Q}_v)$ the $v$-adic closure of $\Lambda^\mathrm{naive}$. We can use the basis above and weak approximation to write explicitly $\mathbb{O}_v=\Span_{\mathbb{Z}_v}b\cup c $, $\Lambda^\mathrm{naive}_v=\Span_{\mathbb{Z}_v} b$  and $\mathbf{E}(\mathbb{Q}_v)=\Span_{\mathbb{Q}_v} b$.  In particular $\mathbf{E}(\mathbb{Q}_v)\cap \mathbb{O}_v=\Lambda^\mathrm{naive}_v$.

For any $v$ such that $g_v\in K_v$ we see that
\begin{equation*}
\Lambda_v=\mathbf{E}(\mathbb{Q}_v)\cap \mathbb{O}_v=\Lambda^\mathrm{naive}_v
\end{equation*}

The lattice $\Lambda^\mathrm{naive}$ is an order in the number field $\mathbf{E}(\mathbb{Q})\simeq E$. The $p$-adic completion of $\Lambda^\mathrm{naive}$ is equal to the maximal order for any $p$ relatively prime to the conductor of $\Lambda^\mathrm{naive}$. Hence $\Lambda_v$ is maximal if $g_v\in K_v$ and $q_v$ is relatively prime to the conductor --- which happens for almost all $v$.
\end{proof}

\begin{lem}\label{lem:Lambda-Galois}
For any $v$ non-archmiedean there exists $\mathcal{f}_v\in\mathbb{Z}_v$ such that $\Lambda_v=\mathbb{Z}_v+\mathcal{f}_v\mathcal{O}_{E_v}$. The conductor of $\Lambda_v$ is $\mathcal{f}_v\mathcal{O}_{E_v}$ and $\Lambda_v$ is stable under the Galois action of $\Gal(E_v/\mathbb{Q}_v)$.
\end{lem}
\begin{proof}
The argument is the same as for orders in quadratic number fields.
\end{proof}

\begin{defi}
We define the local discriminant $D_v$ of the homogeneous toral set $[\mathbf{T}(\mathbb{A})g]$ in an equivalent way to \cite[\S 6.1]{ELMVCubic}. 
\begin{enumerate}
\item 
For $v$ non-archimedean $D_v$ is the discriminant of the order $\Lambda_v$. In particular, for all places $v$ where $E$ is unramified and $\Lambda_v$ is maximal we have $D_v$=1.

\item
For $v$ archimedean there is a natural topological ring isomorphism of $\mathbf{E}(\mathbb{R})$ either to $\mathbb{R}\times\mathbb{R}$ or to $\mathbb{C}$ unique up to an automorphism. Consider the standard volume form on $\mathbb{R}\times\mathbb{R}$ or $\mathbb{C}$ induced by the inner-product norm $\alpha\mapsto |\alpha|$ or $(\alpha,\beta)\mapsto \sqrt{|\alpha|^2+|\beta|^2}$ and pull it back to $\mathbf{E}(\mathbb{R})$.

Let $\Lambda_\infty\subset\mathbf{E}(\mathbb{R})$ be the intersection of the closed unit ball of $\|\bullet\|_\infty$ with $\mathbf{E}(\mathbb{R})$.
Define $D_\infty$ to be the square of the volume of $\Lambda_\infty$ with respect to the latter volume-form.

\item
Finally, the global discriminant is defined to be $D\coloneqq\prod_v D_v$.
\end{enumerate}
\end{defi}
\begin{remark}
Conjugating by $g_v$ we have  $\Lambda_v\simeq g_v^{-1} \mathbf{E}(\mathbb{Q}_v) g_v \cap \mathbb{O}_v$. Thus the local discriminant $D_v$ for $v=\infty, p_1, p_2$ is the same for all homogeneous toral sets for which \eqref{eq:g_adel_restrictions} holds.

Moreover, our choice of  $\|\bullet\|_\infty$ to be $K_\infty$-invariant in \S\ref{sec:maximal-order-B} and the requirement that $g_\infty^{-1}\mathbf{T}(\mathbb{R}) g_\infty=K_\infty$ in \eqref{eq:g_adel_restrictions} imply $D_\infty=1$.
\end{remark}

\paragraph{The Global Order}

\begin{defi} 
We define a global order $\Lambda<\mathbf{E}(\mathbb{Q})\simeq E$ by 
\begin{equation*}
\Lambda\coloneqq \bigcap_{v\neq\infty} \Lambda_v
\end{equation*}
where the intersection is taken in the $2$-dimensional $\mathbb{Q}$-vector space $\mathbf{E}(\mathbb{Q})$.
\end{defi}
Recall that by Proposition \ref{prop:Lambda-ae-maximal} $\Lambda_v$ is equal to the $v$-adic closure of $\mathcal{O}_E$ for almost all $v$, hence the intersection $\Lambda$ is a finite index $\mathbb{Z}$-sublattice in $\mathcal{O}_E$. Moreover, it is closed under multiplication, so it is an order in $\mathbf{E}(\mathbb{Q})$. The discriminant of $\Lambda$ is exactly $\prod_{v\neq \infty} D_v$.
Notice that in general $\Lambda\neq \mathbf{E}(\mathbb{Q})\cap\mathbb{O}$. 

\begin{remark}
A consequence of the discussion above is that for all $v\neq\infty$ the compact-open subgroup $\mathbf{K}_{\widetilde{\mathbf{T}},v}\coloneqq g_v \mathbb{O}_v g_v^{-1}\cap \widetilde{\mathbf{T}}(\mathbb{Q}_v)<\widetilde{\mathbf{T}}(\mathbb{Q}_v)$ is the unit group of the order $\Lambda_v$. 

In particular, if $K_{\widetilde{\mathbf{T}},f}\coloneqq \prod_{v\neq\infty} K_{\widetilde{\mathbf{T}},v}<\widetilde{\mathbf{T}}(\mathbb{A}_f)$ then
\begin{equation*}
C_{\{\infty\}}\simeq
\dfaktor{\widetilde{\mathbf{T}}(\mathbb{Q})}{\widetilde{\mathbf{T}}(\mathbb{A})}{\widetilde{\mathbf{T}}(\mathbb{Q}_S)\cdot K_{\widetilde{\mathbf{T}}}^S}
\simeq \Pic(\Lambda)
\end{equation*}
\end{remark}

\paragraph{Id\'eles and Ideals}
\begin{defi}\label{def:ideles-to-ideals}
Let $\left[\mathbf{T}(\mathbb{A})g\right]$ be a homogeneous toral set with splitting field $E/\mathbb{Q}$ and global order $\Lambda\coloneqq \cap_{v\neq\infty}\Lambda_v \subseteq \mathcal{O}_E$. 

\begin{enumerate}
\item Denote by $\Ideals(\Lambda)$ the abelian group of invertible proper $\Lambda$-fractional ideals. These are exactly the locally principle fractional ideals and there is a canonical group isomorphism
\begin{equation*}
\widetilde{\idl}\colon \faktor{\widetilde{\mathbf{T}}(\mathbb{A}_f)}{K_{\widetilde{\mathbf{T}},f}}=\faktor{\prod_{v\neq\infty} E_v^\times}{\prod_{v\neq\infty} \Lambda_v^\times}\to \Ideals(\Lambda)
\end{equation*}
defined by $(\alpha_v\Lambda_v^\times)_{v\neq\infty}\mapsto \bigcap_{v\neq\infty} \alpha_v \Lambda_v \subset E$.

\item
Define $\Ideals(\Lambda)_0\coloneqq \Ideals(\Lambda)\cup \{0\cdot\Lambda\}$. This set of ideals does not carry a group structure any more but there is a natural action of $\Ideals(\Lambda)$ on it, and hence also an action of the finite $E$-id\`eles. The map above extends naturally to a surjective equivariant map
\begin{equation*}
\widetilde{\idl}\colon \faktor{\mathbf{E}(\mathbb{A}_f)}{K_{\widetilde{\mathbf{T}},f}}=\faktor{\prod_{v\neq\infty} E_v}{\prod_{v\neq\infty} \Lambda_v^\times}\to \Ideals(\Lambda)_0
\end{equation*}
which is no longer a bijection. The preimage of the zero ideal contains any non-invertible ad\`ele. The preimage of any invertible fractional ideal still contains only one element.

\item The map $\idl$ above descend to the following function
\begin{align*}
\idl\colon \faktor{\mathbf{T}(\mathbb{A}_f)}{K_{\mathbf{T},f}}&=\dfaktor{\mathbb{A}_f^\times}{\prod_{v\neq\infty} E_v^\times}{\prod_{v\neq\infty} \Lambda_v^\times}\\
&=\dfaktor{\mathbb{Q}^\times}{\prod_{v\neq\infty} E_v^\times}{\prod_{v\neq\infty} \Lambda_v^\times}
\xrightarrow{\widetilde{\idl}} \lfaktor{\mathbb{Q}^\times}{\Ideals(\Lambda)}
\end{align*}
The second equality above holds because $\mathbb{Q}$ has trivial class group.
\end{enumerate}
\end{defi}

\subsubsection{Volume}\label{sec:vol}
The volume of a homogeneous toral set has been defined in \cite{ELMVCubic}. To motivate the definition consider a normalization in which the measure of the group under which the homogeneous set is invariant -- $H_{\mathbb{A}}$  -- is kept fixed while the homogeneous toral set varies in a family. In the adelic setting it is impossible to keep $H_{\mathbb{A}}$ independent of the homogeneous set in the family, yet we can normalize the measures in a uniform way.  

To do that fix a compact identity neighborhood $\Omega=\prod_v \Omega_v \subset \mathbf{G}(\mathbb{A})$. 
Normalize the Haar measure $\meas_{H_{\mathbb{A}}}$ on $H_{\mathbb{A}}$ so that $\meas_{H_{\mathbb{A}}}(\Omega)=1$. The measure $\meas_{H_{\mathbb{A}}}$ also induces an $H_{\mathbb{A}}$-invariant measure on $[\mathbf{T}(\mathbb{A})g]$ which differs from $\mu$ by a constant. The volume of the homogeneous set is defined as the volume of  $[\mathbf{T}(\mathbb{A})g]$ with respect to the measure induced by $\meas_{H_{\mathbb{A}}}$.

A formula for the volume can be written in terms of the covolume $1$ Haar measure  $\meas_\mathbf{T}$ on $\mathbf{T}(\mathbb{A})$
\begin{equation*}
\vol\left(\left[\mathbf{T}(\mathbb{A})g\right]\right)\coloneqq \meas_\mathbf{T}\left( g\Omega g^{-1}  \right)^{-1}
\end{equation*}

The definition of the volume depends on the choice of a compact identity neighborhood $\Omega$ but in an inessential way. Specifically for any compact identity neighborhoods $\Omega$ and $\Omega'$
\begin{equation}\label{eq:vol-Omega-Omega'}
{\vol}_\Omega \left(\left[\mathbf{T}(\mathbb{A})g\right]\right)\ll_{\Omega,\Omega'} {\vol}_{\Omega' }\left(\left[\mathbf{T}(\mathbb{A})g\right]\right)
\ll_{\Omega,\Omega'} {\vol}_{\Omega}\left(\left[\mathbf{T}(\mathbb{A})g\right]\right)
\end{equation}
Most importantly, the constants do not depend on the homogeneous toral set.

We fix once and for all $\Omega_v=K_v$ for all non-archimedean $v$ and $\Omega_\infty=\mathbf{Z}(\mathbb{R})\widetilde{\Omega}_\infty$, where $\widetilde{\Omega}_\infty$ is as in \S\ref{sec:maximal-order-B}.
The set $\Omega_\infty$ is a connected, compact, symmetric and $\Ad K_\infty$-invariant identity neighborhood in $\mathbf{G}(\mathbb{R})$. In the ramified case this neighborhood coincides with $\mathbf{G}(\mathbb{R})$. These choices simplify computations later.

\subsection{Joinings of Periodic Toral Measures}
Let $[\mathbf{T}(\mathbb{A})g]\subset [\mathbf{G}(\mathbb{A})]$ be a homogeneous toral set with periodic measure $\mu$ as in the previous section. Denote by $\mathbf{T}^\Delta<\mathbf{G}\times \mathbf{G}$ the diagonal embedding.

Fix $s\in\mathbf{T}(\mathbb{A})$ and consider  the following subset of the cartesian square of $[\mathbf{G}(\mathbb{A})]$
\begin{equation*}
[\mathbf{T}^\Delta(\mathbb{A})(g,sg)]\subset  \left[\left(\mathbf{G}\times \mathbf{G}\right)(\mathbb{A})\right]
\end{equation*}
This is a homogeneous set for the \emph{non-maximal} rank $1$ anisotropic torus $\mathbf{T}^\Delta$ in the rank 2 group $\mathbf{G}\times \mathbf{G}$. 

By the same arguments as in the previous sections this set carries a probability measure $\mu^{\mathrm{joint}}$ invariant under the action of\footnote{Notice that $s$ commutes with $\mathbf{T}(\mathbb{A})$.} $H_{\mathbb{A}}^\Delta$.

The measure $\mu^\mathrm{joint}$ projects in each coordinate to the regular periodic toral measure $\mu$ supported on $[\mathbf{T}(\mathbb{A})g]$. It is a self-joining of $\mu$ which is non-trivial because of the shift by $s\in\mathbf{T}(\mathbb{A})$. 

We call $s$ the \emph{twist} of the self-joining. Notice the that whole class of $s$ in $\lfaktor{\mathbf{T}(\mathbb{Q})}{\mathbf{T}(\mathbb{A})}$ defines exactly the same self-joining.

\subsubsection{Joining of Packets}
Let $S$ and $Y$ be as in \S\ref{sec:packets}.
Denote by $H^\Delta$ the diagonal embedding of $H$ into $\mathbf{G}(\mathbb{Q}_S)\times\mathbf{G}(\mathbb{Q}_S)$.
The set $[\mathbf{T}^\Delta(\mathbb{A})(g,sg)]$ projects to a finite collection of $H^\Delta$ orbits on $Y\times Y$ denoted by $\mathcal{P}^\mathrm{joint}$. The measure $\mu^\mathrm{joint}$ can be pushed forward to an $H^\Delta$-invariant probability measure on $\mathcal{P}^\mathrm{joint}$ which we  denote by $\overline{\mu^\mathrm{joint}}$. The measure $\overline{\mu^\mathrm{joint}}$ is a self-joining of the $H$-invariant measure $\overline{\mu}$ on $Y$.

\subsubsection{Volume and Discriminant}
The definitions of volume and discriminant extend trivially from homogeneous set of $\mathbb{Q}$-anisotropic rank $1$ tori in $\mathbf{G}$ to anisotropic rank $1$ tori in $\mathbf{G}\times\mathbf{G}$.
By choosing $\Omega\times \Omega$ as the reference identity neighborhood on $\left(\mathbf{G}\times \mathbf{G}\right)(\mathbb{A})$ and setting $\mathbb{O}\times\mathbb{O}$ as the reference maximal order in $\left(\mathbf{B}\times\mathbf{B}\right)(\mathbb{Q})$ we have
\begin{align*}
\vol\left(\left[\mathbf{T}^\Delta(\mathbb{A})(g,sg)\right]\right)&=\vol\left(\left[\mathbf{T}(\mathbb{A})g\right]\right)\\
\disc\left(\left[\mathbf{T}^\Delta(\mathbb{A})(g,sg)\right]\right)&=\disc\left(\left[\mathbf{T}(\mathbb{A})g\right]\right)
\end{align*}

\section{Principal Results}\label{sec:results}
In this section we present our main theorem and prove key corollaries, a few reduction steps and complementary propositions. The proof of the main theorem is presented in \S\ref{sec:proof} and builds upon the tools developed in the rest of the manuscript.

We will use the following shorthand to simplify our notation.
\begin{defi}
Denote $G_\mathrm{res}\coloneqq \dfaktor{\mathbf{G}(\mathbb{Q})}{\mathbf{G}(\mathbb{A})}{\mathbf{G}(\mathbb{A})^+}$ and let $\pi^+\colon\left[\mathbf{G}(\mathbb{A})\right]\to G_\mathrm{res}$ be the quotient map.

The topological space $G_\mathrm{res}$ is a compact abelian group such that the composition of quotient maps $\mathbf{G}(\mathbb{A})\to[\mathbf{G}(\mathbb{A})]\xrightarrow{\pi^+}G_\mathrm{res}$ is a continuous surjective group homomorphism, cf. \S\ref{sec:simply-connected-cover}. This implies that the push-forward of the probability Haar measure on $[\mathbf{G}(\mathbb{A})]$ to $G_\mathrm{res}$ is the probability Haar measure of $G_\mathrm{res}$.
\end{defi}

\subsection{Equidistribution of Toral Orbits}
The following is the key theorem of this work.
\begin{thm}\label{thm:joinings-adelic}
Let $\mathbf{G}$ be a form of $\mathbf{PGL}_2$ over $\mathbb{Q}$. Fix a maximal compact torus $K_\infty<\mathbf{G}(\mathbb{R})$ and two finite primes $p_1,p_2$. Let $\left\{\mathcal{H}_i\right\}_i$ be a sequence of joint homogeneous toral sets. For each $i$ write $\mathcal{H}_i=\left[\mathbf{T}^\Delta(\mathbb{A})(g,sg)\right]$ where $\mathbf{T},s,g$ depend on $i$. Recall that $\mathbf{T}<\mathbf{G}$ is a maximal torus defined and anisotropic over $\mathbb{Q}$, $g\in\mathbf{G}(\mathbb{A})$ and $s\in\mathbf{T}(\mathbb{A})$. 

Let $E_i/\mathbb{Q}$ be the quadratic field splitting $\mathbf{T}$ and let $D_i$ be the discriminant of $\mathcal{H}_i$. Denote by $\mathcal{f}_i$ the conductor of $D_i$, i.e.\ $\mathcal{f}_i^2 \mid D_i$ is the largest square divisor of $D_i$.

 Denote by $\mu_i$ the algebraic probability measure on $[\left(\mathbf{G}\times\mathbf{G}\right)(\mathbb{A})]$ supported on $\mathcal{H}_i$.

Assume the following for all $i\in\mathbb{N}$
\begin{enumerate}
\item $g_\infty^{-1} \mathbf{T}(\mathbb{R}) g_\infty=K_\infty$,
\item $p_1,p_2$ split in $E_i$,
\item The Dedekind $\zeta$ function of $E_i$ has no exceptional Landau-Siegel zero,
\item $\mathcal{f_i}\ll 1$.
\end{enumerate}

If $|D_i|\to\infty$ and the following holds for any compact subset $B\subset\mathbf{G}(\mathbb{A})$ 
\begin{equation*}
\forall i\gg_{B} 1\colon g^{-1}\mathbf{T}(\mathbb{Q})s g\cap B=\emptyset
\end{equation*} 
then any weak-$*$ limit point of $\left\{\mu_i\right\}_i$ is a $\left(\mathbf{G}\times\mathbf{G}\right)(\mathbb{A})^+$-invariant probability measure.
\end{thm}

\begin{cor}
Denote by $$L^2_{00}\left(\left[\left(\mathbf{G}\times\mathbf{G}\right)(\mathbb{A})\right],\meas_{\mathbf{G}\times \mathbf{G}}\}\right)<L^2\left(\left[\left(\mathbf{G}\times\mathbf{G}\right)(\mathbb{A})\right],\meas_{\mathbf{G}\times \mathbf{G}}\right)$$ the subspace orthogonal to the residual spectrum. Then in the setting of Theorem \ref{thm:joinings-adelic} above for any \emph{continuous compactly supported} function $f\in L^2_{00}\left(\left[\left(\mathbf{G}\times\mathbf{G}\right)(\mathbb{A})\right],\meas_{\mathbf{G}\times \mathbf{G}}\right)$
\begin{equation*}
\int f \dif\mu_i\to_{i\to\infty} 0
\end{equation*}
\end{cor}
\begin{proof}
Each fiber of $\pi^+$ admits a transitive $\mathbf{G}(\mathbb{A})^+$ action inducing an isomorphism of the fiber with $[\mathbf{G}(\mathbb{A})^+]$. This isomorphism depends on the choice of a base point. The probability Haar measure on $[\mathbf{G}(\mathbb{A})^+]$ defines a probability measure on the fiber which is independent of the choice of base point due to the invariance property of the Haar measure. The conditional measures of $\meas_\mathbf{G}$ on the fibers of $[\mathbf{G}(\mathbb{A})]\to G_\mathrm{res}$ are $\mathbf{G}(\mathbb{A})^+$-invariant probability measures hence they can be taken to coincide with the previously described measures on the fibers.
	
The residual spectrum is by definition the space of function factoring through $\pi^+\times \pi^+$ and a function is orthogonal to the residual spectrum if its conditional expectation with respect to the pull-back of the Borel $\sigma$-algebra under $\pi^+\times\pi^+$ vanishes. In terms of conditional measures this is equivalent to the function having integral $0$ over the conditional measure of $\meas_\mathbf{G}\times\meas_\mathbf{G}$ for almost each fiber. For a compactly supported continuous function $f$ orthogonal to the residual spectrum we deduce that it has integral $0$ over each fiber with respect to the $\left(\mathbf{G}\times \mathbf{G}\right)(\mathbb{A})^+$-invariant measure.

Because each $\left(\mathbf{G}\times \mathbf{G}\right)(\mathbb{A})^+$-invariant probability measure on $\left[\left(\mathbf{G}\times\mathbf{G}\right)(\mathbb{A})\right]$ is a convex combination of the measures on the fibers of $\pi^+\times\pi^+$ we deduce that all the limit points of $\int f\dif\mu_i$ are $0$.
\end{proof}

\subsection{Reduction to a Fixed Invariance Group at \texorpdfstring{$p_1,p_2$}{p1,p2}}
In the rest of the manuscript we work with homogeneous toral sets satisfying the conditions of \eqref{eq:g_adel_restrictions} which are more restrictive then the conditions in Theorem \ref{thm:joinings-adelic}. In particular, we require for all homogeneous toral sets $\left[\mathbf{T}(\mathbb{A})g\right]$ that $g_{p_j}^{-1} \mathbf{T}(\mathbb{Q}_{p_j}) g_{p_j}=A_{p_j}$ for $j\in\{1,2\}$ and some fixed split tori $A_{p_j}<\mathbf{G}(\mathbb{Q}_{p_j})$. In this section we show that Theorem \ref{thm:joinings-adelic} can be reduced to the case of joint homogeneous toral sets satisfying these additional conditions.

\begin{prop}\label{prop:shift_p1_p_2}
Let $\left\{\mathcal{H}_i\right\}_i$ and $\{\mu_i\}_i$ be as in Theorem \ref{thm:joinings-adelic}. Then there is a \emph{bounded} sequence $h_i\in\mathbf{G}(\mathbb{A})$ such that $\mathcal{H}_{i}(h_i,h_i)\subseteq\left[\left(\mathbf{G}\times\mathbf{G}\right)(\mathbb{A})\right]$ satisfies \eqref{eq:g_adel_restrictions} for all $i\in\mathbb{N}$. 
\end{prop}
\begin{proof}
The main observation is that the local discriminant is a proper continuous map on the variety of tori. 

Let $p\in\{p_1,p_2\}$. Because all $\mathbb{Q}_p$-split tori in $\mathbf{G}(\mathbb{Q}_p)$ are conjugate we identify the space of $\mathbb{Q}_p$-split tori with $\faktor{\mathbf{G}(\mathbb{Q}_p)}{\Nrml_{\mathbf{G}(\mathbb{Q}_p)}A_p}$. To each split torus we can associate a discriminant in the manner of \S\ref{sec:discriminant-local}. Specifically, let $\overline{A}<\mathbf{B}(\mathbb{Q}_p)$ be the split quadratic \'etale-algebra associated to $A_p$. If $T=hAh^{-1}$ for some $h\in\mathbf{G}(\mathbb{Q}_p)$ then $\disc(T)$ is the discriminant of the order $h \overline{A}h^{-1}\cap \mathbb{O}$. This function is continuous and proper as follows from \cite[\S4.2 and \S6.1]{ELMVCubic}.

If $\mathcal{H}_i=[\mathbf{T}(\mathbb{A})(g_i,s_ig_i)]$ then assumption (4) in Theorem \ref{thm:joinings-adelic} and properness of the local discriminant map implies that $g_{i,p}^{-1}\mathbf{T}_i(\mathbb{Q}_p) g_{i,p}$ is a bounded sequence in the space of tori $\faktor{\mathbf{G}(\mathbb{Q}_p)}{\Nrml_{\mathbf{G}(\mathbb{Q}_p)}A_p}$ for $p\in\{p_1,p_2\}$. Thus we can choose a \emph{bounded} sequence $h_{i,p}\in\mathbf{G}(\mathbb{Q}_p)$ such that $g_{i,p}^{-1}\mathbf{T}_i(\mathbb{Q}_p) g_{i,p} =h_{i,p}A_p h_{i,p}^{-1}$ for all $i\in\mathbb{N}$.

Define $h_i\in\mathbf{G}(\mathbb{A})$ to have coordinate $h_{i,p}$ for $p\in\{p_1,p_2\}$ and have trivial coordinates at all other places. This sequence obviously satisfies the claimed properties.
\end{proof}

\begin{cor}
Theorem \ref{thm:joinings-adelic} for joint homogeneous toral sets satisfying \eqref{eq:g_adel_restrictions} implies the general case of
Theorem \ref{thm:joinings-adelic}.
\end{cor}
\begin{proof}
Let $\{\mathcal{H}_i\}_i$ and $\{\mu_i\}_i$ be as in Theorem \ref{thm:joinings-adelic}. Because this sequence of measures is tight by Duke's theorem, we can pass without loss of generality to a convergent subsequence with limit $\mu$. Let $h_i\in\mathbf{G}(\mathbb{A})$ be the bounded sequence from Proposition \ref{prop:shift_p1_p_2} above. Without loss of generality we pass to a further subsequence such that $h_i\to_{i\to\infty} h\in\mathbf{G}(\mathbb{A})$.

For any $g\in\mathbf{G}(\mathbb{A})$ denote by $R_g\colon[\mathbf{G}(\mathbb{A})]\to[\mathbf{G}(\mathbb{A})]$ the transformation of multiplying by $g^{-1}$ on the right. For each $i$ the measure $\left(R_{h_i}\times R_{h_i}\right)_*.\mu_i$ is the algebraic measure supported on $\mathcal{H}_i(h_i,h_i)$ and we have 
\begin{equation*}
\left(R_{h_i}\times R_{h_i}\right)_*.\mu_i\to_{i\to\infty} 
\left(R_{h}\times R_{h}\right)_*.\mu
\end{equation*}
Our assumption implies that the measure on the right hand side is a $\left(\mathbf{G}\times\mathbf{G}\right)(\mathbb{A})^+$ invariant measure. The same statement then holds for $\mu$ because $\left(\mathbf{G}\times\mathbf{G}\right)(\mathbb{A})^+$ is a normal subgroup.
\end{proof}

\subsection{Limit Behavior of Residual Spectrum}
The following, significantly easier, proposition supplements the main theorem as it can be used to understand the asymptotic behavior for the residual spectrum.
\begin{prop}\label{prop:residual}
Let $\left\{\mathcal{\mu}_i\right\}_i$ and $E_i/\mathbb{Q}$ be as in Theorem \ref{thm:joinings-adelic}, although we do not require that conditions (1)-(4) from the theorem are satisfied.  

Assume one of the following two options holds: either all the fields $E_i$ are distinct or they are all equal to a fixed quadratic field $E_0/\mathbb{Q}$. In the former case define $H\coloneqq G_\mathrm{res}$ and in the latter case set $H\coloneqq \ker\left(\ch_{E_0} \circ \Nrd \right)$ where $\ch_{E_0}\colon \dfaktor{\mathbb{Q}^\times}{\mathbb{A}^\times}{{\mathbb{A}^\times}^2} \to \{\pm 1\}$ is the real adelic character attached to $E_0/\mathbb{Q}$ by global class field theory.

Then any limit point of $\left(\pi^+\times\pi^+\right)_*.\mu_i$ is an $H^\Delta$-invariant probability measure supported on a single coset of $H^\Delta$.
\end{prop}
\begin{remark}
It will be evident from the proof that in general $\left\{\left(\pi^+\times\pi^+\right)_*.\mu_i\right\}_i$ need not converge, even under the assumptions of the proposition above.
\end{remark}
\begin{proof}
Recall from \S\ref{sec:simply-connected-cover} that the reduced norm map induces a monomorphism
\begin{equation*}
\Nrd\colon G_\mathrm{res}\to \dfaktor{\mathbb{Q}^\times}{\mathbb{A}^\times}{{\mathbb{A}^\times}^2}
\end{equation*}
This map is onto if $\mathbf{B}$ is split at $\infty$ and it is the index $2$ subgroup $\dfaktor{\mathbb{Q}_{>0}}{\mathbb{R}_{>0}\times\mathbb{A}_f^\times}{{\mathbb{A}^\times}^2}$ otherwise.

Assume $\{\mu_i\}_i$ converges weak-$*$ and 
let $\left[\mathbf{T}_i(\mathbb{A})^\Delta (g_i,s_ig_i)\right]$ be the homogeneous toral set of $\mu_i$. By restricting to a subsequence we can assume without loss of generality that $\pi^+(g_i)$ and $\pi^+(s_i)$ converge in $G_\mathrm{res}$ to some $\gamma,\sigma\in G_\mathrm{res}$.

Fix an index $i\in\mathbb{N}$ and let $\mathbf{T}\coloneqq \mathbf{T}_i$ and $E\coloneqq E_i$. Because $\mathbf{T}(\mathbb{A})$ is abelian and $\mathbf{T}$ is isotropic over $\mathbb{Q}$ the homogeneous set $\left[\mathbf{T}(\mathbb{A})\right]$ is a compact abelian group. In particular, $\pi^+\left(\left[\mathbf{T}(\mathbb{A})\right]\right)$ is a closed subgroup of $G_\mathrm{res}$. To describe this subgroup explicitly recall that the isomorphism $\mathbf{T}\simeq \lfaktor{\Gm}{\WR\Gm}$ intertwines the reduced norm map with the regular field norm map. Thus $$\Nrd\circ \pi^+\left(\left[\mathbf{T}(\mathbb{A})\right]\right)=\ker \ch_E$$  where $\ch_E\colon \dfaktor{\mathbb{Q}^\times}{\mathbb{A}^\times}{{\mathbb{A}^\times}^2}\to\{\pm 1\}$ is the real adelic character attached to $E/\mathbb{Q}$ by global class field theory. We shall denote henceforth this character by $\ch_i$. 

If $\ch_i=\ch_{E_0}$ for all $i$ where $E_0/\mathbb{Q}$ is a fixed imaginary quadratic field then define $H \coloneqq \ker\left(\ch_{E_0} \circ \Nrd\right)< G_\mathrm{res}$. Otherwise, our assumption implies that all the characters $\ch_i$ are mutually distinct.
Because $\dfaktor{\mathbb{Q}^\times}{\mathbb{A}^\times}{{\mathbb{A}^\times}^2}$ is compact its Pontryagin dual is discrete. Hence if the character $\ch_i$ are distinct the sequence $\left\{\ch_i\right\}_i$ diverges. If $\left\{\ch_i\right\}_i$ diverges then the sequence of subgroups $\{\langle \ch_i\rangle \}_i$ converge in the Chabauty topology to the trivial group $1<G_\mathrm{res}$. Pontryagin-Chabauty duality \cite{Cornulier} then implies that $\ker \ch_i$ converges in the Chabauty topology to the full subgroup $\dfaktor{\mathbb{Q}^\times}{\mathbb{A}^\times}{{\mathbb{A}^\times}^2}$. In this case set $H\coloneqq G_\mathrm{res}$.

For all $i$ denote $\nu_i\coloneqq \left(\pi^+\times\pi^+\right)_*.\mu_i$ and let $\nu$ be the limit measure. From the discussion above it follows that $\Nrd_*.\nu_i$ is the $\ker \ch_i^\Delta$-invariant probability measure on $\ker \ch_i^\Delta(\pi^+(g),\pi^+(gs_i))$. The limit measure $\Nrd_*.\nu$ is invariant under the action of the Chabauty limit of the invariance subgroups $\ker \ch_i^\Delta$ which is $\Nrd(H)^\Delta$. We deduce that $\nu$ is invariant under $H^\Delta$.

We are left only with proving that $\nu$ is supported on a single coset of $H$. Using the continuous contraction map $\ctr\colon G_\mathrm{res}\times G_\mathrm{res}\to G_\mathrm{res}$ define the push-forward probability measures $\ctr_*.\nu_i$ on $G_\mathrm{res}$. The characterization of $\Nrd_*.\nu_i$ above implies that $$\Nrd_*.\ctr_*.\nu_i=\delta_{\Nrd(\pi^+(s_i))}\to_{i\to\infty} \delta_{\Nrd(\sigma)}$$ hence $\ctr_*.\nu=\delta_{\sigma}$. This implies that $\nu$ is supported on $G_\mathrm{res}^\Delta(e,\sigma)$ and the proof is concluded in the case that $H=G_\mathrm{res}$.

If $H=\ker \ch_E$ then $\nu_i(\pi^+(g_i),\pi^+(s_i g_i))^{-1}$ is independent of $i$ and is equal to the Haar measure on $H^\Delta$. The claim follows because
\begin{equation*}
\nu_i(\pi^+(g_i),\pi^+(s_i g_i))^{-1}\to_{i\to\infty} \nu (\gamma, \sigma \gamma)^{-1}
\end{equation*}
\end{proof}

\subsection{Many-Fold Toral Joinings}
A pleasant consequence of the joining theorem of Einsiedler and Lindenstrauss is that we can understand $n$-joinings of periodic toral measures using the theorem for $2$-joinings. The main observation is that if a reductive subgroup $\mathbf{L}<\underbrace{\mathbf{G}\times\cdots\times\mathbf{G}}_n$
projects onto $\mathbf{G}\times\mathbf{G}$ in any of the $\begin{pmatrix}
n \\ 2
\end{pmatrix}$ pairs of coordinates then it must be equal to the full $n$-product.

\begin{defi}
Fix $n\in\mathbb{N}$.
Let $\mathbf{T}<\mathbf{G}$ be a maximal torus defined and anisotropic $/\mathbb{Q}$. Denote by $\mathbf{T}^\Delta<\mathbf{G}^{\times n}$ the diagonal embedding.

Fix $s_1,\ldots,s_{n-1}\in\mathbf{T}(\mathbb{A})$ and $g\in\mathbf{G}(\mathbb{A})$. The set
\begin{equation*}
\left[\mathbf{T}^\Delta(\mathbb{A})(g,s_1 g,\ldots,s_{n-1}g)\right]
\subset
\left[\mathbf{G}^{\times n}(\mathbb{A})\right]
\end{equation*}
is an $n$-joint homogeneous toral set. This set supports a unique $\left(g^{-1}\mathbf{T}(\mathbb{A})g\right)^\Delta$-invariant probability measure.
\end{defi}

\begin{thm}\label{thm:joinings-adelic-n}
Let $\mathbf{G}$ be a form of $\mathbf{PGL}_2$ over $\mathbb{Q}$. Fix a maximal compact torus $K_\infty<\mathbf{G}(\mathbb{R})$ and two finite primes $p_1,p_2$. 

Let $\left\{\mathcal{H}_i\right\}_i$ be a sequence of $n$-joint homogeneous toral sets. For each $i$ write $\mathcal{H}_i=\left[\mathbf{T}^\Delta(\mathbb{A})(g,s_1 g,\ldots,s_{n-1}g)\right]$ where $\mathbf{T},\{s_j\}_{1\leq j <n},g$ depend on $i$. Recall that $\mathbf{T}<\mathbf{G}$ is a maximal torus defined and anisotropic over $\mathbb{Q}$, $g\in\mathbf{G}(\mathbb{A})$ and $s_1,\ldots,s_{n-1}\in\mathbf{T}(\mathbb{A})$. 

Let $E_i/\mathbb{Q}$ be the quadratic field splitting $\mathbf{T}$ and let $D_i$ be the discriminant of $\mathcal{H}_i$. Denote by $\mathcal{f}_i$ the conductor of $D_i$, i.e.\ $\mathcal{f}_i^2 \mid D_i$ is the largest square divisor of $D_i$.

Denote by $\mu_i$ the algebraic probability measure on $\left[\mathbf{G}^{\times n}(\mathbb{A})\right]$ supported on $\mathcal{H}_i$.

Assume the following for all $i\in\mathbb{N}$
\begin{enumerate}
\item $g_\infty^{-1} \mathbf{T}(\mathbb{R}) g_\infty=K_\infty$,
\item $p_1,p_2$ split in $E_i$,
\item The Dedekind $\zeta$ function of $E_i$ has no exceptional Landau-Siegel zero,
\item $\mathcal{f_i}\ll 1$.
\end{enumerate}

If $|D_i|\to\infty$ and the following holds for any compact subset $B\subset\mathbf{G}(\mathbb{A})$ and for any pair of distinct elements $s,s'\in\{1,s_1,\ldots,s_{n-1}\}$
\begin{equation}\label{eq:n-fold-condition}
\forall i\gg_{B} 1\colon g^{-1}\mathbf{T}(\mathbb{Q})s^{-1}s' g\cap B=\emptyset
\end{equation} 
then any weak-$*$ limit point of $\left\{\mu_i\right\}_i$ is a $\mathbf{G}^{\times n}(\mathbb{A})^+$-invariant probability measure.
\end{thm}

\begin{proof}
This follows from Theorem \ref{thm:2d-sieve} and \cite[Corollary 1.5]{ELJoinings}.
\end{proof}

\subsection{Galois Orbits of Special Points}
\begin{thm}
Let $\mathbf{G}$ be a form of $\mathbf{PGL}_2$ defined over $\mathbb{Q}$ and split over $\mathbb{R}$. Let $X$ be a product of $n$ quaternionic Shimura varieties relative to $\mathbf{G}$.

Let $\{x_i\}_i$ be a sequence of special points in $X$ all whose coordinates have CM by the same quadratic order $\Lambda_i<E_i$ of discriminant $D_i<0$ and conductor $\mathcal{f}_i$. Fix two primes $p_1$, $p_2$ and assume the following for all $i\in\mathbb{N}$
\begin{enumerate}
\item $p_1$, $p_2$ split in $E_i$,
\item The Dedekind $\zeta$-function of $E_i$ has no exceptional Landau-Siegel zero,
\item $\mathcal{f}_i\ll 1$.
\end{enumerate} 
Denote by $\nu_i$ the normalized counting measure on the finite Galois orbit of $x_i$. If the sequence $\{x_i\}_i$ has finite intersection with any proper special subvariety then any weak-$*$ limit of $\{\nu_i\}_i$ is a convex combination of the uniform probability measures on the connected components of $X$.  
\end{thm}
\begin{proof}
We will show how this theorem follows from Theorem \ref{thm:joinings-adelic-n} above.
The definition of a Shimura variety relative to $\mathbf{G}$, cf. \cite[\S 5]{Milne}, implies that there is a surjective projection map
\begin{equation*}
\Pi\colon\left[\mathbf{G}^{\times n}(\mathbb{A})\right]\to X
\end{equation*}
defined by dividing the adelic quotient by the compact group $\prod_{j=1}^n (K_\infty\times U_j)$ where $K_\infty<\mathbf{G}(\mathbb{R})$ is a compact torus and $U_j<\mathbf{G}(\mathbb{A}_f)$ is a compact-open subgroup for all $1\leq j\leq n$.

The reciprocity map of class field theory supplies in this case, cf.\ \cite[\S12]{Milne}, an identification between the Galois orbit of $x_i$ and the image under $\Pi$ of a homogeneous toral set
\begin{equation*}
\mathcal{H}_i=\left[\mathbf{T}^\Delta(g_1,\ldots,g_n)\right]\subset \left[\mathbf{G}^{\times n}(\mathbb{A})\right]
\end{equation*}
where $\mathbf{T}<\mathbf{G}$ satisfies the conditions of Theorem \ref{thm:joinings-adelic-n}.
Moreover, the counting measure on the Galois orbit is the push-forward of the period measure $\mu_i$ on $\mathcal{H}_i$.  

The homogeneous toral set $\mathcal{H}_i$ is of the form treated in Theorem \ref{thm:joinings-adelic-n} if all the $n$ coordinates of $x_i$ are Galois conjugate. In general, there can be more then one Galois orbit with the same CM order $\Lambda_i$, yet they all differ by an element of a maximal compact subgroup in $\mathbf{G}^{\times n}(\mathbb{A})$, i.e.\ by Atkin-Lehner involutions. Specifically, let $K_{f,j}<\mathbf{G}(\mathbb{A}_f)$ be a maximal compact subgroup containing $U_j$ then the homogeneous toral set $\mathcal{H}_i$ can be taken to be
\begin{equation*}
\mathcal{H}_i=\left[\mathbf{T}^\Delta(g,s_1 g k^i_1,\ldots,s_{n-1} g_n k^i_{n-1})\right]\subset \left[\mathbf{G}^{\times n}(\mathbb{A})\right]
\end{equation*}
where $g\in\mathbf{G}(\mathbb{A})$, $s_1,\ldots,s_{n-1}\in\mathbf{T}(\mathbb{A})$ and $k^i\coloneqq (1,k^i_1,\ldots,k^i_{n-1})\in \times K_{1,f}\times\cdots\times K_{n,f}$. Denote by $\mu_i$ the period measure supported on $\mathcal{H}_i$ and whose push-forward to $X$ is $\nu_i$.

If the sequence $\{x_i\}_i$ has a finite intersection with any proper special subvariety then the same property holds for any fixed pair of coordinates of $\{x_i\}_i$ when considered as a sequence of special points on a product of two varieties. This implies the genericity condition \eqref{eq:n-fold-condition} in Theorem \ref{thm:joinings-adelic-n} for the homogeneous toral sets $\mathcal{H}_i {k^i}^{-1}$. In particular, all the condition of this theorem hold for the sequence $\{\mathcal{H}_i {k^i}^{-1}\}$ and we deduce the any weak-$*$ limit of $\{\mu_i {k^i}^{-1} \}_i$ is a $\mathbf{G}^{\times n}(\mathbb{A})^+$-invariant probability measures.

Assume without loss of generality that $\mu_i {k^i}^{-1}\to_{i\to\infty} \mu$. By passing to a subsequence we can also assume $k^i\to_{i\to\infty} k^0\in K_{1,f}\times\cdots\times K_{n,f}$. Then we have that $\mu_i\to_{i\to\infty} \mu k_0$. Because $\mathbf{G}^{\times n}(\mathbb{A})^+$ is a normal subgroup we deduce that $\mu k_0$ is also $\mathbf{G}^{\times n}(\mathbb{A})^+$-invariant. The claim follows by pushing-forward $\mu k_0$ to $X$ using $\Pi$.
\end{proof}

\section{Measure Rigidity} \label{sec:rigidity}
We present a definition of an algebraic probability measure in the $S$-arithmetic setting and the adelic one. The $S$-arithmetic definition we use is from \cite{ELJoinings}.
\begin{defi}\label{def:alg-measure}
Let $\mathbf{M}$ be a linear algebraic group defined over $\mathbb{Q}$.
\begin{enumerate}
\item Fix a finite set of rational places $S$ containing $\infty$ and let $M<\mathbf{M}(\mathbb{Q}_S)$ be a closed finite index subgroup. Let
$\Gamma<M$ be a lattice. A probability measure $\nu$ on $\lfaktor{\Gamma}{M}$ is \emph{algebraic} if there is a closed unimodular algebraic subgroup $\mathbf{L}<\mathbf{M}$ defined and anisotropic over $\mathbb{Q}$, a finite index subgroup $L<\mathbf{L}(\mathbb{Q}_S)$ and some $g_S\in M$ such that $\nu$ is the probability $L$-Haar measure supported on $[Lg_S]\subseteq\lfaktor{\Gamma}{M}$.

\item A probability measure $\nu$ on $[\mathbf{M}(\mathbb{A})]$ is an \emph{algebraic measure} if there is a closed unimodular algebraic subgroup $\mathbf{L}<\mathbf{M}$ defined and anisotropic over $\mathbb{Q}$, an isogeny $\mathbf{L}'\to\mathbf{L}$ over $\mathbb{Q}$ and a closed subgroup of finite index $L<\Img\left[\mathbf{L}'(\mathbb{A})\to\mathbf{L}(\mathbb{A})\right]$ such that 
$\nu$ is the probability $L$-Haar measure on an orbit $[Lg]\subset [\mathbf{M}(\mathbb{A})]$ for some $g\in\mathbf{G}(\mathbb{A})$.

\end{enumerate}
\end{defi}

\begin{remark}
The datum defining a fixed adelic algebraic measure is the $\mathbf{G}(\mathbb{Q})$-orbit of a tuple $(\mathbf{L},\mathbf{L}'\to\mathbf{L},L,Lg)$ where $\gamma\in\mathbf{G}(\mathbb{Q})$ acts by
\begin{equation*}
\gamma.(\mathbf{L},\mathbf{L}'\to\mathbf{L},L,Lg)=
(\Ad_\gamma\mathbf{L},\mathbf{L}'\to\mathbf{L}\xrightarrow{\Ad_\gamma}\Ad_\gamma\mathbf{L},\Ad_\gamma L,(\Ad_\gamma L) (\gamma g))
\end{equation*}
\end{remark}

\begin{defi}
Write $$A_{p_i}^+\coloneqq A_{p_i}\cap\mathbf{G}(\mathbb{Q}_{p_i})^+$$ for $i\in\{1,2\}$. The subgroup $A_{p_i}^+$ is the image in $A_{p_i}$ of a maximal torus in $\mathbf{G}^\sc(\mathbb{Q}_{p_i})$ isogenic to $A_{p_i}$, hence it has finite index in $A_{p_i}$. 
\end{defi}

The essential ingredient in the proof of the following theorem is the joinings theorem of Einsiedler and Lindenstrauss \cite[Theorem 1.4]{ELJoinings}  and Duke's theorem for equidistribution in the absolute rank $1$ case. Notice that because we assume a fixed split prime the equidistribution in the absolute rank $1$ case that we use is already covered by Linnik's method \cite{LinnikBook}.

\begin{thm}\label{thm:regiditiy-adelic}
Let $\mu^\mathrm{joint}_i$ be a sequence of self-joinings of periodic toral measures on $\left[\left(\mathbf{G}\times \mathbf{G}\right)(\mathbb{A})\right]$ with discriminants $|D_i|\to_{i\to\infty}\infty$ and satisfying \eqref{eq:g_adel_restrictions}. Let the probability measure $\mu$ be any limit point of $\mu^\mathrm{joint}_i$ then $\mu$ is a convex combination of of $(A_{p_1}^+\times A_{p_2}^+)^\Delta$-invariant algebraic measures. Specifically, there is a Borel probability measure $\mathcal{P}$ on the space of probability measure $\mathcal{M}_1\left(\left[\left(\mathbf{G}\times \mathbf{G}\right)(\mathbb{A})\right]\right)$ supported on the subset of algebraic measures so that
\begin{equation*}
\mu=\int_{\mathcal{M}_1\left(\left[\left(\mathbf{G}\times \mathbf{G}\right)(\mathbb{A})\right]\right)} \lambda \dif\mathcal{P}(\lambda)
\end{equation*}

Moreover, for  almost all the algebraic measures $\lambda$ in the support of $\mathcal{P}$ the associated $\mathbb{Q}$-group $\mathbf{L}<\mathbf{G}\times\mathbf{G}$ can be taken either to be $\mathbf{G}^\Delta$ or $\mathbf{G}\times\mathbf{G}$ and $\lambda$ is the algebraic measure supported on $[\mathbf{L}(\mathbb{A})^+ \xi]$ for some $\xi\in\left(\mathbf{G}\times\mathbf{G}\right)(\mathbb{A})$.
\end{thm}

\begin{cor}\label{cor:hecke-limit-split-restrictions}
Let $\lambda$ be an algerbraic measure in the support of $\mathcal{P}$ in Theorem \ref{thm:regiditiy-adelic} above. If $\lambda$ is supported on $[\mathbf{G}^\Delta(\mathbb{A})^+\xi]$ then $\ctr(\xi)_{p_i}\in A_{p_i}$ for $i\in\{1,2\}$.
\end{cor}
\begin{proof}
Fix $i\in\{1,2\}$.
The measure $\lambda$ is $\left(A_{p_i}^+\right)^\Delta$-invariant and its stabilizer subgroup in $ \left(\mathbf{G}\times\mathbf{G}\right)(\mathbb{Q}_{p_i})$ is contained in $ (e,\ctr(\xi)_{p_i})\mathbf{G}^\Delta(\mathbb{Q}_{p_i})(e,\ctr(\xi)_{p_i})^{-1}$. Thus $\ctr(\xi)_{p_i}$ centralizes $A_{p_i}^+$ in $\mathbf{G}(\mathbb{Q}_{p_i})$. This centralizer is $A_{p_i}$.
\end{proof}

We will use the following standard result.
\begin{lem}\label{lem:G+-mixing}
For every rational place $v$ that splits $\mathbf{B}$ the action of $\mathbf{G}(\mathbb{Q}_v)^+$ is mixing for the Haar measure on $[\mathbf{G}(\mathbb{A})^+ \omega_0]$ for any $\omega_0\in\mathbf{G}(\mathbb{A})$.
\end{lem}
\begin{proof}
The Haar measure on $[\mathbf{G}(\mathbb{A})^+ \omega_0]$ is invariant under $\omega_0^{-1} \mathbf{G}(\mathbb{A})^+ \omega_0=\mathbf{G}(\mathbb{A})^+$.
Considering the $\mathbf{G}(\mathbb{A})^+$-equivariant isomorphism of measure spaces 
\begin{equation*}
[\mathbf{G}(\mathbb{A})^+\omega_0]=\lfaktor{\mathbf{G}(\mathbb{Q})}{\mathbf{G}(\mathbb{A})^+\omega_0}\simeq \lfaktor{\Cent\mathbf{G}^\sc(\mathbb{A})\cdot\mathbf{G}^\sc(\mathbb{Q})}{\mathbf{G}^\sc(\mathbb{A})}
\end{equation*}
it is enough to show that the action of $\mathbf{G}^\sc(\mathbb{Q}_v)$ on $[\mathbf{G}^\sc(\mathbb{A})]\coloneqq\lfaktor{\mathbf{G}^\sc(\mathbb{Q})}{\mathbf{G}^\sc(\mathbb{A})}$ is mixing. This result will follow from Howe-Moore \cite[Theorem 5.2]{HoweMoore} if we show that the only finite dimensional  $\mathbf{G}^\sc(\mathbb{Q}_v)$-sub-represenation in $L^2\left([\mathbf{G}^\sc(\mathbb{A})], \meas_{\mathbf{G}^\sc}\right)$ is the space of constant functions.

By strong approximation for simply-connected absolutely almost simple groups the group  $\mathbf{G}^\sc(\mathbb{Q}_v)$ acts minimally on $[\mathbf{G}^\sc(\mathbb{A})]$, i.e.\ all the $\mathbf{G}^\sc(\mathbb{Q}_v)$-orbits are topologically dense.
Let $$V<L^2\left([\mathbf{G}^\sc(\mathbb{A})], \meas_{\mathbf{G}^\sc}\right)$$ be a (closed) finite-dimensional sub-$\mathbf{G}^\sc(\mathbb{Q}_v)$-representation. The minimality of the $\mathbf{G}^\sc(\mathbb{Q}_v)$ action implies that the whole $\mathbf{G}^\sc(\mathbb{A})$-orbit of any $\mathbf{G}^\sc(\mathbb{Q}_v)$-smooth vector in $V$ is contained in $V$. The smooth vectors are dense in any closed sub-representation $V<L^2\left([\mathbf{G}^\sc(\mathbb{A})], \meas_{\mathbf{G}^\sc}\right)$, hence $V$ must be $\mathbf{G}^\sc(\mathbb{A})$-invariant. Because $\mathbf{G}^\sc$ is simply-connected it has no non-trivial residual spectrum and the only finite dimensional $\mathbf{G}^\sc(\mathbb{A})$-sub-representation is $\mathbb{C}\cdot 1$.
\end{proof}

To apply \cite[Theorem 1.4]{ELJoinings} to $\mu$ we need first to decompose it to ergodic measures on locally homogeneous spaces saturated by unipotents in the sense of \cite[Definition 1.1]{ELJoinings}.

The measure $\mu$ is 
$\left(A_{p_1}^+\times A_{p_2}^+\right)^\Delta$-invariant and we write
\begin{equation}\label{eq:mu-erg-decomp}
\mu=\int_{\mathcal{M}_1\left(\left[\left(\mathbf{G}\times \mathbf{G}\right)(\mathbb{A})\right]\right)} \lambda \dif\mathcal{P}(\lambda)
\end{equation}
for the ergodic decomposition of $\mu$ with respect to $\left(A_{p_1}^+\times A_{p_2}^+\right)^\Delta$.
\begin{lem}\label{lem:A-ergodic-decomp}
For $\mathcal{P}$-almost every $\lambda$ there is $\omega=(\omega_1,\omega_2)\in\mathbf{G}(\mathbb{A})\times\mathbf{G}(\mathbb{A})$ such that $\lambda$ is an  $\left(A_{p_1}^+\times A_{p_2}^+\right)^\Delta$-invariant measure  supported on $[\left(\mathbf{G}\times \mathbf{G}\right)(\mathbb{A})^+ \omega]$. Moreover, its projection to each coordinate is the $\mathbf{G}(\mathbb{A})^+$-Haar measure on $[\mathbf{G}(\mathbb{A})^+ \omega_i]$.
\end{lem}

\begin{proof} 
The $\left(A_{p_1}^+\times A_{p_2}^+\right)^\Delta$-invariance of $\lambda$ is built into the definition of an ergodic decomposition.
The measures $\lambda$ in the support of $\mathcal{P}$  are conditional measures of $\mu$ on the $\sigma$-algebra of $\left(A_{p_1}^+\times A_{p_2}^+\right)^\Delta$-invariant Borel sets. Denote by $\mathcal{B}^+$ the $\sigma$-algebra of Borel $\mathbf{G}(\mathbb{A})^+$-invariant sets in $[\mathbf{G}(\mathbb{A})]$. The $\sigma$-algebra of $\left(A_{p_1}^+\times A_{p_2}^+\right)^\Delta$-invariant sets in $[\left(\mathbf{G}\times\mathbf{G}\right)(\mathbb{A})]$ contains $\mathcal{B}^+\times\mathcal{B}^+$ --- the $\sigma$-algebra of $\mathbf{G}(\mathbb{A})^+\times\mathbf{G}(\mathbb{A})^+$-invariant Borel sets. Hence $\mathcal{P}$-almost every $\lambda$ is supported on an atom of $\mathcal{B}^+\times\mathcal{B}^+$ .

The $\sigma$-algebra $\mathcal{B}^+$ corresponds to the factor map
\begin{equation*}
\lfaktor{\mathbf{G}(\mathbb{Q})}{\mathbf{G}(\mathbb{A})}\to\dfaktor{\mathbf{G}(\mathbb{Q})}{\mathbf{G}(\mathbb{A})}{\mathbf{G}(\mathbb{A})^+}
\simeq\dfaktor{\Nrd(\mathbf{B}^\times(\mathbb{Q}))}{\Nrd(\mathbf{B}^\times(\mathbb{A}))}{{\mathbb{A}^\times}^2}
\end{equation*}
and its atoms are the fibers of this map, which are of the form $[\omega_0\mathbf{G}(\mathbb{A})^+]=[\mathbf{G}(\mathbb{A})^+\omega_0]$ for some $\omega_0\in\mathbf{G}(\mathbb{A})$. The atoms of $\mathcal{B}^+\times\mathcal{B}^+$ are then of the form $[\left(\mathbf{G}\times \mathbf{G}\right)(\mathbb{A})^+ \omega]$ for $\omega=(\omega_1,\omega_2)\in \mathbf{G}(\mathbb{A})\times\mathbf{G}(\mathbb{A})$. This proves the first part of the lemma.

Because of Duke's theorem, proved by Linnik under the assumption of a fixed split prime, $\mu$ projects in each coordinate to a $\mathbf{G}(\mathbb{A})^+$-invariant measure on $[\mathbf{G}(\mathbb{A})]$. We deduce that
\begin{equation*}
{\pi_i}_*\mu=\int {\pi_i}_*\lambda \dif\mathcal{P}(\lambda)
\end{equation*}
is $\mathbf{G}(\mathbb{A})^+$-invariant for $i=1,2$. All the  $\mathbf{G}(\mathbb{A})^+$-invariant and ergodic measures on $[\mathbf{G}(\mathbb{A})]$ are $\mathbf{G}(\mathbb{A})^+$-Haar measures on $\mathcal{B}^+$ atoms of the form $[\mathbf{G}(\mathbb{A})^+\omega_0]$. By Lemma \ref{lem:G+-mixing} these $\mathbf{G}(\mathbb{A})^+$-Haar measures are $A_{p_1}^+\times A_{p_2}^+$-ergodic. Hence a $\mathbf{G}(\mathbb{A})^+$-ergodic decomposition of ${\pi_i}_*\mu$ is also an $A_{p_1}^+\times A_{p_2}^+$-ergodic decomposition. By uniqueness of the ergodic decomposition it follows that for $\mathcal{P}$-almost every $\lambda$ the projections ${\pi_i}_*\lambda$ are $\mathbf{G}(\mathbb{A})^+$-Haar measures on a $\mathcal{B}^+$-atom.
\end{proof}

We now fix a measure $\lambda$ satisfying the conclusions of Lemma \ref{lem:A-ergodic-decomp} and apply \cite[Theorem 1.4]{ELJoinings} to it. To do that we need first to pass to an $S$-arithmetic setting.

\begin{lem}\label{lem:S-arithm-joinings}
Let $\lambda$ satisfy the conclusions of Lemma \ref{lem:A-ergodic-decomp}. In particular, $\lambda$ is supported on $[\mathbf{G}(\mathbb{A})^+\omega_1\times \mathbf{G}(\mathbb{A})^+ \omega_2]$.  Fix $S$ a finite set of rational places such that
\begin{enumerate}
\item $\infty,p_1,p_2\in S$,
\item $\mathbf{G}$ has class number $1$ with respect to $K^S$.
\item $\omega_1,\omega_2\in \mathbf{G}(\mathbb{Q}_S)\times K^S$.
\end{enumerate}
Denote by $\jmath^S$ the canonical projection
\begin{equation*}
\jmath^S\colon \left[\left(\mathbf{G}\times\mathbf{G}\right)(\mathbb{A})\right] \to Y_S\times Y_S
\coloneqq \lfaktor{\Gamma_S}{\mathbf{G}(\mathbb{Q}_S)}\times  \lfaktor{\Gamma_S}{\mathbf{G}(\mathbb{Q}_S)}
\end{equation*}
where $\Gamma_S\coloneqq \mathbf{G}(\mathbb{Q}) \cap K^S$.

Then the measure $\jmath^S_*\lambda$ is an algebraic measure on $Y_S\times Y_S$ supported on $[L_S g_S]$ for some $g_S\in\left(\mathbf{G}\times\mathbf{G} \right)(\mathbb{Q}_S)$ where $L_S<\mathbf{L}(\mathbb{Q}_S)\cap\left(\mathbf{G}\times\mathbf{G} \right)(\mathbb{Q}_S)^+$ is a finite index subgroup and $\mathbf{L}<\mathbf{G}$ a closed algebraic subgroup. The group $\mathbf{L}$ is isogenous either to $\mathbf{G}$ or to $\mathbf{G}\times\mathbf{G}$ and projects onto $\mathbf{G}$ in both coordinates.
\end{lem}
\begin{proof}
Write $\omega_{i,S}\in\mathbf{G}(\mathbb{Q}_S)$ for the $S$-coordinates of $\omega_i$, $i=1,2$. Denote $\omega_S=(\omega_{1,S},\omega_{2,S})$. Set $\Gamma_S^+\coloneqq \Gamma_S \cap \mathbf{G}(\mathbb{Q}_S)^+$; this is a lattice in $\mathbf{G}(\mathbb{Q}_S)^+$ and denote $Y_S^+\coloneqq\lfaktor{\Gamma_S^+}{\mathbf{G}(\mathbb{Q}_S)^+}$.

Strong approximation for $\mathbf{G}^\sc$ implies that $\jmath^S_*\lambda$ is supported on a \emph{single} orbit $[\mathbf{G}(\mathbb{Q}_S)^+\omega_1\times \mathbf{G}(\mathbb{Q}_S)^+\omega_2]$ and its projection to each coordinate is a $\mathbf{G}(\mathbb{Q}_S)^+$-Haar measure.
By applying a right translation by $\omega_S^{-1}$ we consider the measure $\jmath^S_*\lambda$ as an $\omega_S^{-1}\left(A_{p_1}^+\times A_{p_2}^+\right)^\Delta \omega_S$-invariant and ergodic measure on $Y_S^+\times Y_S^+$ whose projection to each coordinate is the Haar measure on $Y_S^+$.

The space $Y_S^+$ is saturated by unipotents because the group $\mathbf{G}(\mathbb{Q}_{p_1})^+\simeq \mathbf{PSL}_2(\mathbb{Q}_{p_1})$ is generated by unipotents and acts ergodically on $Y_S^+$ by Lemma \ref{lem:G+-mixing}. The group $A_{p_1}^+\times A_{p_2}^+$ is a compact extension of a class-$\mathcal{A}'$ group in the sense of \cite[Definition 1.3]{ELJoinings}, so $\jmath^S_*\lambda$ is an ergodic invariant measure for a class-$\mathcal{A}'$ group of rank $2$. Theorem 1.4 of \cite{ELJoinings} now applies and $\jmath^S_*\lambda$ is an algebraic measure on $[L_S g_S]$ for some $g_S\in\mathbf{G}(\mathbb{Q}_S)$ and $L_S$ of finite index in $\mathbf{L}(\mathbb{Q}_S)\cap\mathbf{G}(\mathbb{Q}_S)^+$ for some reductive group $\mathbf{L}<\mathbf{G}$ projecting onto $\mathbf{G}$ in both coordinates. By \cite[Lemma 7.4]{ELJoinings} $\mathbf{L}$ is either isogenous to $\mathbf{G}$ or to $\mathbf{G}\times\mathbf{G}$.
\end{proof}

\begin{lem}\label{lem:L-isomorphism-type}
In the setting of Lemma \ref{lem:S-arithm-joinings} the group $\mathbf{L}$ is either isomorphic to $\mathbf{G}$ or to $\mathbf{G}\times\mathbf{G}$.
\end{lem}
\begin{proof}
Consider the center $\Cent\mathbf{L}$. It projects in both coordinates to the center of $\mathbf{G}$ which is trivial as $\mathbf{G}$ is of adjoint type. Hence $\Cent \mathbf{L}$ projects to the trivial group in each coordinate so it is trivial. The group $\mathbf{L}$ is adjoint and the claim follows as both $\mathbf{G}$ and $\mathbf{G}\times\mathbf{G}$ are adjoint.
\end{proof}
If $\mathbf{L}\simeq\mathbf{G}\times\mathbf{G}$ then the inclusion $\mathbf{L}< \mathbf{G}\times\mathbf{G}$ is an equality. The following treats the case that $\mathbf{L}\simeq\mathbf{G}$.

\begin{lem}\label{lem:L-exact-subgroup}
If  $\mathbf{L}<\mathbf{G}\times\mathbf{G}$ is isomorphic to $\mathbf{G}$ and projects onto $\mathbf{G}$ in both coordinates then $\mathbf{L}$ is conjugate to $\mathbf{G}^\Delta$ over $\mathbb{Q}$.
\end{lem}

\begin{proof}
Consider the projections $\pi_1,\pi_2\colon \mathbf{G}\times\mathbf{G}\to\mathbf{G}$ restricted to $\mathbf{L}$.
Because $\mathbf{L}$ projects onto $\mathbf{G}$ in both coordinates and $\mathbf{L}$ is simple with trivial center the kernel of these projections is trivial. In particular, both projections are isomorphism of algebraic groups and $\pi_2\restriction_{\mathbf{L}}\circ {\pi_1\restriction_{\mathbf{L}}}^{-1}$ is an automorphism of $\mathbf{G}$. As all automorphisms of $\mathbf{G}$ are inner we see that $\pi_2\restriction_{\mathbf{L}}\circ {\pi_1\restriction_{\mathbf{L}}}^{-1}=\Ad_g$ for some $g\in\mathbf{G}(\mathbb{Q})$. Thus $(e,g^{-1})\mathbf{L}(e,g)<\mathbf{G}^\Delta$ and because $\mathbf{L}$ and $\mathbf{G}^\Delta$ have the same dimension and $\mathbf{G}^\Delta$ is connected we conclude $(e,g^{-1})\mathbf{L}(e,g)=\mathbf{G}^\Delta$
\end{proof}

\begin{lem}
In Lemma \ref{lem:S-arithm-joinings} we can take $L_S=\mathbf{L}(\mathbb{Q}_S)^+$.
\end{lem}
\begin{proof}
Lemmas \ref{lem:L-isomorphism-type} and \ref{lem:L-exact-subgroup} imply that the reduced norm map
\begin{equation*}
\Nrd\colon\left(\mathbf{G}\times\mathbf{G}\right)(\mathbb{Q}_S)\to\faktor{\mathbb{Q}_S^\times}{{\mathbb{Q}_S^\times}^2}
\end{equation*} 
restricts to the corresponding reduced norm on $\mathbf{L}(\mathbb{Q}_S)$. In particular, $\mathbf{L}(\mathbb{Q}_S)\cap\left(\mathbf{G}\times\mathbf{G}\right)(\mathbb{Q}_S)^+=\mathbf{L}(\mathbb{Q}_S)^+$. The group $\mathbf{L}(\mathbb{Q}_{p_1})^+$ is a product of at most $2$ copies of the abstractly simple \cite{DicksonPSLn} group $\mathbf{PSL}_2(\mathbb{Q}_{p_1})$ . In particular, $\mathbf{L}(\mathbb{Q}_{p_1})^+$ has no proper subgroups of finite index, hence $L_S\cap\mathbf{L}(\mathbb{Q}_{p_1})=\mathbf{L}(\mathbb{Q}_{p_1})^+$. 

By strong approximation $\mathbf{L}(\mathbb{Q}_{p_1})^+$ acts minimally on the closed set $[\mathbf{L}(\mathbb{Q}_S)^+]$. Because $[L_S]$ is contained in $[\mathbf{L}(\mathbb{Q}_S)^+]$ and it is $\mathbf{L}(\mathbb{Q}_{p_1})^+$-invariant we see that $[L_S]=[\mathbf{L}(\mathbb{Q}_S)^+]$. The $\mathbf{L}(\mathbb{Q}_S)^+$-Haar measure on $[\mathbf{L}(\mathbb{Q}_S)^+]$ is $L_S$-invariant. Uniqueness of the Haar measure on a homogeneous space implies that the $\mathbf{L}(\mathbb{Q}_S)^+$ and $L_S$ Haar measures on  $[L_S]=[\mathbf{L}(\mathbb{Q}_S)^+]$ coincide.

The conclusion of the Lemma follows by translating by $g_S$.
\end{proof}

\begin{proof}[Proof of Theorem \ref{thm:regiditiy-adelic}]
We patch the result of the previous lemmata into an adelic statement.

Fix a countable well-ordered direct system of finite sets of rational places $\left\{S\right\}$ exhausting all the places of $\mathbb{Q}$ and such that all $S$ satisfy the conditions in Lemma \ref{lem:S-arithm-joinings}. By excluding a countable union of $\mathcal{P}$-measure zero sets we see that $\mathcal{P}$-almost every $\lambda$ in \eqref{eq:mu-erg-decomp} projects onto an algebraic measure satisfying the conclusions of Lemma \ref{lem:S-arithm-joinings} for each $S$ in the direct system.

Let $S\subset S'$ be a pair of sets places in the direct system. The algebraic measure $\jmath^{S'}_*\lambda$ supported on $[\mathbf{L}_{S'}(\mathbb{Q}_{S'})g_{S'}]$ projects to the algebraic measure $\jmath^{S}_*\lambda$ supported on $[\mathbf{L}_{S}(\mathbb{Q}_{S})g_{S}]$. The factor map from $Y_{S'}\times Y_{S'}$ to $Y_S\times Y_S$ is the division by the compact subgroup $\prod_{v\in S'\setminus S} K_v$, thus $$\Gamma_{S'}\mathbf{L}_{S}(\mathbb{Q}_{S})^+g_{S}\prod_{v\in S'\setminus S} K_v=\Gamma_{S'}\mathbf{L}_{S'}(\mathbb{Q}_{S'})^+g_{S'}\prod_{v\in S'\setminus S} K_v$$
and $\gamma g_{S'}=l g_S k_S$ for some $\gamma\in \Gamma_{S'}=\left(\mathbf{G}\times\mathbf{G}\right)(\mathbb{Q})\cap K^{S'}\times K^{S'}$, $l\in\mathbf{L}_S(\mathbb{Q}_S)^+$ and $k_S\in\prod_{v\in S}1 \times \prod_{v\in S'\setminus S} K_v$.
Write $g_{S'}=(g_{S'}^0,g_{S'}^1)$ where $g_{S'}^0$ are the $\mathbb{Q}_S$ coordinates of $g_{S'}$ and  $g_{S'}^1$ are the coordinates in $S'\setminus S$ then $\gamma g_{S'}^0=l g_S$. 

The $g_{S'}^{-1} \mathbf{L}_{S'}(\mathbb{Q}_{S'})^+ g_{S'}$-periodic measure supported on $[\mathbf{L}_{S'}(\mathbb{Q}_{S'})^+ g_{S'}]$ projects to a finite collection of ${g_{S'}^0}^{-1} \mathbf{L}_{S'}(\mathbb{Q}_S)^+ g_{S'}^0$ periodic measures. The $\left(A_{p_1}^+\times A_{p_2}^+\right)^\Delta$-ergodicity of $\jmath^S_*\lambda$ implies that this collection is a single periodic measure.

The measure $\jmath^S_*\lambda$ is also the $g_S^{-1} \mathbf{L}_S(\mathbb{Q}_S)^+ g_S$-periodic measure with support $[\mathbf{L}_S(\mathbb{Q}_S)^+ g_S]$. The groups stabilizing the measure are equal and so are their normal subgroups of trivial reduced norm. Hence
${g_{S'}^0}^{-1} \mathbf{L}_{S'}(\mathbb{Q}_{S})^+ g_{S'}^0=g_S^{-1} \mathbf{L}_S(\mathbb{Q}_S)^+ g_S$. Because $\gamma g_{S'}^0=l g_S$ this implies that $\Ad_\gamma\mathbf{L}_{S'}(\mathbb{Q}_{S'})^+ =\mathbf{L}_S(\mathbb{Q}_S)^+$.

Because the image of the simply connected cover is Zariski dense over an infinite field we see that $\Ad_\gamma\mathbf{L}_{S'} =\mathbf{L}_S$. We are free to replace the datum $\left(\mathbf{L}_{S'}, \mathbf{L}_{S'}(\mathbb{Q}_{S'})^+, g_{S'}\right)$ by  the datum $\left(\Ad_\gamma\mathbf{L}_{S'}, \Ad_\gamma \mathbf{L}_{S'}(\mathbb{Q}_{S'})^+\gamma  g_{S'}\right)$ without changing the corresponding algebraic measure on $Y_{S'}\times Y_{S'}$. Using the new datum the algebraic measure $\jmath^{S'}_*\lambda$ is supported on $[\mathbf{L}_S(\mathbb{Q}_{S'})^+ \gamma g_{S'}]=[\mathbf{L}_S(\mathbb{Q}_{S'})^+ g_Sk_S]$ with $k_S\in\prod_{v\in S}1 \times \prod_{v\in S'\setminus S} K_v$.

Let $S_0$ be the minimal set of places in the well-ordered direct system. We make the choices of datum for the measures $\jmath^S_*\lambda$ consistently across the ordered system, i.e.\ for all $S$ the measure $\jmath^S_*\lambda$ it the algebraic measure supported on $[\mathbf{L}_{S_0}(\mathbb{Q}_S)g_{S_0}k_S]$  and $k_S$ has non-trivial entries only in coordinates not contained in sets of places preceding $S$. We can then extend $k_S$ trivially to an element of $K<\mathbf{G}(\mathbb{A}_f)$ and define $k=\prod_S k_S\in K$.

The adelic algebraic measure supported on $[\mathbf{L}(\mathbb{A})^+ g_{S_0}k]$ projects under $\jmath^S$ to the measure $\jmath^S_*\lambda$ for all $S$ in the direct system. As the set of compactly supported functions on $[\left(\mathbf{G}\times\mathbf{G}\right)(\mathbb{A})]$ which are $K^S\times K^S$-smooth for some $S$ is dense in the space of compactly supported continuous functions we conclude that $\lambda$ coincides with the algebraic measure supported on $[\mathbf{L}(\mathbb{A})^+ g_{S_0} k]$.

Lemma \ref{lem:L-exact-subgroup} now implies that we can take $\mathbf{L}$ either to be $\mathbf{G}\times\mathbf{G}$ or $\mathbf{G}^\Delta$.
\end{proof}

\section{Coordinates for Quaternion Algebras}\label{sec:B-coordinates} 
The usual representation in coordinates of a split quaternion algebra $\mathbf{B}(\mathbb{Q}_v)$ over a local field $\mathbb{Q}_v$ is the matrix algebra $\mathbf{M}_{2\times 2}(\mathbb{Q}_v)$. When $v\neq\infty$ and we have a fixed maximal order we can choose our coordinates so that this order is $\mathbf{M}_{2\times 2}(\mathbb{Z}_v)$. The downside of this "fixed coordinates" representation is that it is difficult in the general case to write down the intersection of a varying torus $\widetilde{\mathbf{T}}(\mathbb{Q}_v)$ with the maximal order or to describe coordinatewise the action of the torus by conjugation.

Another commonly used coordinate representation of a quaternion algebra, split or not, over $\mathbb{Q}_v$ is a coordinate system adjusted to the varying torus $\widetilde{\mathbf{T}}(\mathbb{Q}_v)$. In this description $\mathbf{B}(\mathbb{Q}_v)$ is identified with the subset of fixed point of a twisted Galois action on $\mathbf{M}_{2\times 2}(E_v)$; where $E_v/\mathbb{Q}_v$ is a quadratic \'etale-algebra splitting $\widetilde{\mathbf{T}}$. In this description $\widetilde{\mathbf{T}}(\mathbb{Q}_v)$ corresponds simply to the diagonal torus. The price we pay is that the coordinatewise expression for a fixed maximal order is varying. 

In this section, we present the expression for the maximal order in a coordinate system varying with the torus. The results of this section are well-known yet because they are of utilitarian nature it is difficult to point to an exhaustive reference.  
\begin{defi}
Define $\mathbf{M}_{2\times 2}=\Spec \mathbb{Q}\left[x_{i,j} \mid 1\leq i,j \leq 2\right]$ to  be the 4-dimensional affine space of $2\times 2$ matrices.
We define $\mathbf{GL}_2$ as a space of invertible $2\times 2$ matrices using to the closed immersion $\mathbf{GL}_2\hookrightarrow \mathbf{M}_{2\times 2}\times \mathbb{A}_1$ 
\begin{equation*}
\mathbb{Q}[\mathbf{GL}_2]
\left.=\mathbb{Q}\left[x_{i,j}, {\det}^{-1} \mid 1\leq i,j \leq 2\right]
\middle\slash \left<\left(x_{1,1}x_{2,2}-x_{1,2}x_{2,1}\right) {\det}^{-1}=1\right> \right.
\end{equation*}
\end{defi}

The torus $\widetilde{\mathbf{T}}\simeq \WR \Gm$ is split over $E$, hence $\mathbf{B}_E\simeq \mathbf{M}_{2\times 2,E}$. We now describe the Galois action on $\mathbf{M}_{2\times 2,E}$ corresponding to the $\mathbb{Q}$-form $\mathbf{B}$. 
\begin{defi}\label{def:BE-ME-iso}\addpenalty{-1000}
\hfill
\begin{enumerate}
\item Let $\widetilde{\mathbf{A}}$ be the torus of diagonal matrices in $\mathbf{GL}_2$.
We fix an isomorphism of algebras defined over $E$
\begin{equation*}
\mathbf{B}_E\to \mathbf{M}_{2\times 2,E}
\end{equation*}
which sends $\widetilde{\mathbf{T}}_E$ to $\widetilde{\mathbf{A}}_E$. Using this isomorphism we identity henceforth 
\begin{align*}
\mathbf{B}_E&=\mathbf{M}_{2\times 2, E}\\
\mathbf{B}^\times_E&=\mathbf{GL}_{2,E},\; \widetilde{\mathbf{T}}_E=\widetilde{\mathbf{A}}_E\\
\mathbf{G}_E&=\mathbf{PGL}_{2,E},\; \mathbf{T}_E=\mathbf{A}_E
\end{align*}

\item Denote $\mathfrak{G}\coloneq \Gal(E/\mathbb{Q})$ and let $\sigma$ be the non-trivial element of $\mathfrak{G}$. We consider two actions of $\mathfrak{G}$ on $\mathbf{M}_{2\times 2,E}$ which restrict to actions on $\mathbf{GL}_{2,E}$. The naive action is the one induced by considering $\mathbf{M}_{2\times 2,E}$ as base change of $\mathbf{M}_{2\times 2}$. This action acts on the coordinates by
\begin{align*}
x_{i,j}&\mapsto \tensor[^\sigma]{x}{_{i,j}}  & 1\leq i,j\leq 2\\
\end{align*}

\item
The twisted action corresponds to viewing $\mathbf{M}_{2\times 2,E}$ as base change of $\mathbf{B}$. As $\mathbf{B}$ is an inner-form of $\mathbf{M}_{2\times2}$ this actions differs from the naive one by conjugation by some $\theta\in\mathbf{PGL}_2(E)$, i.e.\
\begin{align*}
x_{i,j}&\mapsto \theta \tensor[^\sigma]{x}{_{i,j}} \theta^{-1}  & 1\leq i,j\leq 2\\
\end{align*}
\end{enumerate}
\end{defi}

The following is very well known.
\begin{prop}
The element $\theta$ has a representative of the form
\begin{equation*}
\theta=\begin{pmatrix}
0 & \epsilon \\
1 & 0
\end{pmatrix}
\end{equation*}
where $\epsilon\in\mathbb{Q}^\times$. 

Moreover, any $\epsilon\in\mathbb{Q}^\times$ defines in this way an inner-form of $\mathbf{M}_{2\times 2}$ which is split over $E$. The inner-forms corresponding to $\epsilon_1, \epsilon_2$ are isomorphic over $\mathbb{Q}$ if and only if $\epsilon_2\in \epsilon_1 \Nrd E^\times$.
This establishes a bijection between (inner-)forms of $\mathbf{GL}_2$ split over $E$ and $\faktor{\mathbb{Q}^\times}{\Nr E^\times}$.
\end{prop}
\begin{proof}
The torus $\widetilde{\mathbf{A}}_E\simeq\widetilde{\mathbf{T}}_E$ is stable under the twisted Galois action because $\widetilde{\mathbf{T}}$ is defined over $\mathbb{Q}$; thus $\theta\in \Nrml_{\mathbf{PGL}_2}(\widetilde{\mathbf{A}})(E)$. Because $\widetilde{\mathbf{T}}$ is not split, the twisted Galois action is non-trivial on $\widetilde{\mathbf{A}}_E$ and we can write a representative for $\theta$ of the required form with $\epsilon\in E^\times$.

Because $\sigma$ is an involution we see that $\tensor[^\sigma]{\theta}{}=\theta^{-1}$ which is equivalent to $\epsilon\in\mathbb{Q}$. 
Isomorphic forms correspond to
coboundarous Galois actions. A coboundary which stabilizes $\widetilde{\mathbf{A}}_E\simeq\widetilde{\mathbf{T}}_E$ is of the form $\theta\mapsto \tensor[^\sigma]{\upsilon}{^{-1}} \theta \upsilon$ where $\upsilon\in \Nrml_{\mathbf{PGL}_2}(\widetilde{\mathbf{A}})(E)$. This amounts to multiplying $\epsilon$ by a norm.
\end{proof}

\begin{remark}
Even for the case $\mathbf{B}=\mathbf{M}_{2\times 2}$ the twisted Galois action differs from the naive one. In this case $\sigma$ acts by conjugating the matrix elements, interchanging the two diagonal entries and interchanging the two anti-diagonal ones. This differs from the naive one also because it identifies the diagonal torus with $\widetilde{\mathbf{T}}_E$. In particular, the Galois fixed points in $\widetilde{\mathbf{A}}(E)$ are $\widetilde{\mathbf{T}}(\mathbb{Q})$ and \emph{not} $\widetilde{\mathbf{A}}(\mathbb{Q})$.
\end{remark}

\begin{prop}\label{prop:rational-points-global}
The subset $\mathbf{B}(\mathbb{Q})\subset \mathbf{M}_{2\times 2}(E)$ can be written as
\begin{equation*}
\mathbf{B}(\mathbb{Q})=\left\{\begin{pmatrix}
a & \epsilon b  \\
\tensor[^\sigma]{b}{} & \tensor[^\sigma]{a}{}
\end{pmatrix}  \relmiddle| a,b\in E \right\} 
\end{equation*}
\end{prop}
\begin{proof}
By Galois descent for quasi-projective varieties over perfect fields the fixed points of the Galois actions are exactly the points defined over the base field.

The proposition now follows directly by examining the fixed points of the twisted Galois action.
\end{proof}

\subsection{Coordinates over Local Fields}
For any place $v$ of $\mathbb{Q}$ let $E_v=\prod_{w|v} E_w$. The group $\mathfrak{G}$ acts on the \'{e}tale-algebra $E_v$ either as a Galois group of a field extension if $v$ is not split in $E$ or by switching the coordinates if $v$ splits. In both cases the fixed points are $\mathbb{Q}_v$ where in the split case $\mathbb{Q}_v$ is embedded diagonally in $E_v$. The base change of the isomorphism $\mathbf{B} \to \mathbf{M}_{2\times 2,E}$ to $E_v$ is an isomorphism
\begin{equation}\label{eq:local-B-M-iso}
\mathbf{B}_{E_v} \to \mathbf{M}_{2\times 2,{E_v}}
\end{equation}

The twisted action of $\mathfrak{G}$ extends by the base-change construction to an action on $\mathbf{M}_{2\times 2,{E_v}}$. This action coincides with the action of $\mathfrak{G}$ on $\mathbf{M}_{2\times 2,{E_v}}$ induced by the action of the Galois group $\Gal(E_v/\mathbb{Q}_v)$.
\begin{prop}\label{prop:rational-points-local}
The subset $\mathbf{B}(\mathbb{Q}_v)$ can be written as the following set in $\mathbf{M}_{2\times 2}(E_v)$.
\begin{equation*}
\mathbf{B}(\mathbb{Q}_v)=\left\{\begin{pmatrix}
\alpha & \epsilon \beta \\
\tensor[^\sigma]{\beta}{} & \tensor[^\sigma]{\alpha}{}
\end{pmatrix} \relmiddle{|} \alpha,\beta\in E_v \right\} 
\end{equation*}

Moreover, the elements of $\mathbf{B}(\mathbb{Q})\subset \mathbf{B}(\mathbb{Q}_v)$ are exactly the matrices for which $\alpha,\beta\in E$.
\end{prop} 
\begin{proof}
The matrix $\theta$
is a $\mathbb{Q}$-point of $\mathbf{PGL}_2$ and hence also $\mathbb{Q}_v$-point and a $E_v$-point. In case $v$ splits in $E$ the matrix $\theta$ sits diagonally in $\mathbf{M}_{2\times 2,E_v}\simeq \mathbf{M}_{2\times 2,\mathbb{Q}_v}\times \mathbf{M}_{2\times 2,\mathbb{Q}_v}$. Because the Galois action of $\mathfrak{G}$ on $\mathbf{M}_{2\times 2,E_v}$ coincides with the base-change action it is also given by the naive action composed with conjugation by $\theta$.

The first part of the proposition follows once more by computing the fixed points of a Galois action on a quasi-projective varieties.

The statement about points in $\mathbf{B}(\mathbb{Q})$ follows from Proposition \ref{prop:rational-points-global} and the universal property of base change.
\end{proof}

\subsection{The Different Ideal}

We review some basic properties about the different ideal of a quadratic order.
\begin{defi}
For $v\neq\infty$
define the inverse different ideal of $\Lambda_v\subset \mathbf{E}(\mathbb{Q}_v)= E_v$ by 
\begin{equation*}
\widehat{\Lambda_v}\coloneqq \left\{a\in E_v \mid \Tr (a \Lambda_v)\subseteq \mathbb{Z}_v \right\}
\end{equation*}

Define the different ideal by 
\begin{equation*}
\mathfrak{D}(\Lambda_v)\coloneqq\left(\Lambda_v \colon \widehat{\Lambda_v}\right)
=\left\{a\in\Lambda_v \mid a \widehat{\Lambda_v} \subseteq \Lambda_v  \right\}
\end{equation*}
\end{defi}

\begin{lem}\label{lem:inverse-different}
Let $v\neq\infty$.
The different ideal $\mathfrak{D}(\Lambda_v)$ is principal invertible and its generator $\mathscr{D}_v\in\Lambda_v$ satisfies
\begin{equation*}
\Nr\mathscr{D}_v=D_v
\end{equation*}
\end{lem}
\begin{remark}
The generator $\mathscr{D}_v$ is well-defined only up to multiplication by a unit of $\Lambda_v$.
\end{remark}

\begin{proof}
Notice that the maximal order $\mathcal{O}_{E_v}$ is a product of DVR's hence a principal ideal ring.
The proof proceeds in the same manner as for an order in a quadratic number field.
\end{proof}

\subsection{Local Maximal Orders in Coordinates}
Fix $v$ a place of $\mathbb{Q}$.
In this section we describe in terms of matrices the elements of the maximal order $g_v \mathbb{O}_v g_v^{-1}<\mathbf{B}(\mathbb{Q}_v)$. The description depends upon whether $v$ splits $\mathbf{B}$ or not.

\subsubsection{Split Case}
If $\mathbf{B}(\mathbb{Q}_v)$ is split then $\mathbf{B}(\mathbb{Q}_v)$ is a matrix algebra and $\epsilon=\tensor[^\sigma]{f}{}f$ for some $f\in E_v^\times$. 

Because $\mathbf{B}(\mathbb{Q}_v)$ is split it is isomorphic to a rank-$2$ matrix algebra. This statement can be strengthened so that the action of the \'{e}tale-algebra $\mathbf{E}(\mathbb{Q}_v)\subset\mathbf{B}(\mathbb{Q}_v)$ on the vector space coincides with multiplication in the \'{e}tale-algebra.
\begin{lem}\label{lem:B(Qv)-acts-on-Ev}
Consider $E_v$ as a $2$-dimensional $\mathbb{Q}_v$-vector space.
If $\mathbf{B}(\mathbb{Q}_v)$ is split then 
there is an isomorphism of $\mathbb{Q}_v$-algebras $\mathbf{B}(\mathbb{Q}_v)\simeq \End_{\mathbb{Q}_v}(E_v)$ such that elements of  $\mathbf{E}(\mathbb{Q}_v)\simeq E_v$ act by multiplication on the \'{e}tale-algebra $E_v$. Moreover, there is an isomorphism of $\mathbb{Q}_v$ vector space $\mathbf{B}(\mathbb{Q}_v)\simeq E_v\oplus E_v$ so that the action of $\mathbf{B}(\mathbb{Q}_v)$ on $E_v$ satisfies  
\begin{equation*}
\forall a\in E_v\colon (\alpha,\beta).a=\alpha\cdot a +\beta \cdot \tensor[^\sigma]{a}{}
\end{equation*}
\end{lem}

\begin{proof}
Using Proposition \ref{prop:rational-points-local} we can write $\mathbf{B}(\mathbb{Q}_v)\simeq E_v\oplus E_v$ in the following way
\begin{equation}\label{eq:B(Qv)=Ev+Ev}
(\alpha,\beta)\mapsto 
\begin{pmatrix}\alpha & 0 \\ 0&\tensor[^\sigma]{\alpha}{}\end{pmatrix}
+\begin{pmatrix}0 & \epsilon \cdot \beta/\tensor[^\sigma]{f}{} \\ \tensor[^\sigma]{\beta}{}/f&0\end{pmatrix}
=
\begin{pmatrix}\alpha & 0 \\ 0&\tensor[^\sigma]{\alpha}{}\end{pmatrix}
+\begin{pmatrix}0 & \beta \cdot f \\ \tensor[^\sigma]{\beta}{}/f&0\end{pmatrix}
\end{equation}

Let $\mathbf{B}(E_v)=\mathbf{M}_{2\times 2}(E_v)$ act on $E_v\times E_v$ in the usual way on the left. We embed $E_v\hookrightarrow E_v\times E_v$ by
\begin{equation*}
a\mapsto \begin{pmatrix}
f \cdot a \\ \tensor[^\sigma]{a}{}
\end{pmatrix}
\end{equation*}
and consider the action of $\mathbf{B}(\mathbb{Q}_v)\subset \mathbf{M}_{2\times 2}(E_v)$ on $\Img\left(E_v \hookrightarrow E_v\times E_v \right)$.

The subspace $E_v$ in $\mathbf{B}(\mathbb{Q}_v)$ corresponding to the first coordinate in \eqref{eq:B(Qv)=Ev+Ev} acts by multiplication $\alpha. a=\alpha a$ and the subspace corresponding to the second coordinate in $\eqref{eq:B(Qv)=Ev+Ev}$ acts by $\beta.a= \beta\cdot \tensor[^\sigma]{a}{}$. 

Thus $\mathbf{B}(\mathbb{Q}_v)$ stabilizes $\Img\left(E_v \hookrightarrow E_v\times E_v \right)$ and acts faithfully on it. By comparing dimensions over $\mathbb{Q}_v$ we see that this actions is an isomorphism of algebras $\mathbf{B}(\mathbb{Q}_v)\simeq \End_{\mathbb{Q}_v}(E_v)$. Because the subalgebra $\mathbf{E}(\mathbb{Q}_v)$ is equal to the first coordinate in \eqref{eq:B(Qv)=Ev+Ev} it acts by ring multiplication as required.
\end{proof}

\begin{lem}\label{lem:local-order-split-0}
Let $v\neq\infty$.  If $\mathbf{B}(\mathbb{Q}_v)$ is split then in terms of the representation in Proposition \ref{prop:rational-points-local}

\begin{equation}\label{eq:End(Lambda_v)}
{\End}_{\mathbb{Z}_v}(\Lambda_v)\simeq \left\{\begin{pmatrix}
\alpha & \beta f \\ 
\tensor[^\sigma]{\beta}{}/f & \tensor[^\sigma]{\alpha}{}
\end{pmatrix} \relmiddle{|} \alpha\in\widehat{\Lambda_v}, \beta-\tensor[^\sigma]{\alpha}{}\in\Lambda_v \right\}
\end{equation}
\end{lem}
\begin{proof}
Because $\mathbf{E}(\mathbb{Q}_v)$ acts on $E_v$ by ring multiplication any $l\in\Lambda_v\subset\mathbf{E}(\mathbb{Q}_v)$ belongs to $\End_{\mathbb{Z}_v}(\Lambda_v)$. Thus $x\cdot l\in\End_{\mathbb{Z}_v}(\Lambda_v)$ for  any $x\in\End_{\mathbb{Z}_v}(\Lambda_v)$ and $l\in\Lambda_v\subset \mathbf{E}(\mathbb{Q}_v)$.

Because the ring $\End_{\mathbb{Z}_v}(\Lambda_v)$ is a maximal order in $\mathbf{B}(\mathbb{Q}_v)$ each element in it is integral. Thus for any $x\in\End_{\mathbb{Z}_v}(\Lambda_v)$
\begin{equation}\label{eq:Trd-integral}
\forall l\in\Lambda_v\subset \mathbf{E}(\mathbb{Q}_v)\colon 
\Trd(x\cdot l)\in\mathbb{Z}_v
\end{equation}
Writing $x$ above as $x=(\alpha,\beta)$ using \eqref{eq:B(Qv)=Ev+Ev} equation \eqref{eq:Trd-integral} amounts to the statement  that $\alpha\in\widehat{\Lambda_v}$.

An element $x=(\alpha,\beta)$ belongs to $\End_{\mathbb{Z}_v}(\Lambda_v)$ if and only if
\begin{equation*}
\forall l\in\Lambda_v\colon \Lambda_v \ni \alpha l+ \beta \tensor[^\sigma]{l}{}=
\Tr(\alpha l)+(\beta-\tensor[^\sigma]{\alpha}{})\tensor[^\sigma]{l}{}
\end{equation*}
which can be seen by Lemma \ref{lem:Lambda-Galois} to be equivalent to $\beta-\tensor[^\sigma]{\alpha}{}\in\Lambda_v$. This proves that the endomorphism ring is contained in the right hand side of \eqref{eq:End(Lambda_v)}. The reverse inclusion follows by checking directly that each matrix in the right hand side of \eqref{eq:End(Lambda_v)} preserves $\Lambda_v$.
\end{proof}

\begin{prop}\label{prop:local-order-split}
If $v\neq\infty$ then there is some $\tau_v\in E_v^\times$ such that
\begin{equation*}
g_v \mathbb{O}_v g_v^{-1} = \left\{\begin{pmatrix} 
\alpha & \beta\tau_v \\ 
\tensor[^\sigma]{\beta}{}/\tau_v & \tensor[^\sigma]{\alpha}{}
\end{pmatrix} \relmiddle{|} \alpha\in \widehat{\Lambda_v}, \beta-\tensor[^\sigma]{\alpha}{}\in\Lambda_v \right\}
\end{equation*}
\end{prop}
\begin{remark}
The condition $\beta-\tensor[^\sigma]{\alpha}{}\in\Lambda_v$ can be rewritten in the equivalent more symmetric form $\alpha+\beta\in \Lambda_v$.
\end{remark}
\begin{proof}
Maximal $\mathbb{Z}_v$-orders in matrix algebras are endomorphism rings of $\mathbb{Z}_v$-lattices, cf.\ \cite{Reiner}. Because of the isomorphism from Lemma \ref{lem:B(Qv)-acts-on-Ev} we know that there is a $\mathbb{Z}_v$-lattice $\mathfrak{L}\subset E_v$ of full rank such that $g_v \mathbb{O}_v g_v^{-1}=\End_{\mathbb{Z}_v}(\mathfrak{L})$ and 
\begin{equation*}
\Lambda_v=\left\{a\in E_v \mid a\mathfrak{L}\subset\mathfrak{L} \right\}
\end{equation*}
In other words, $\mathfrak{L}$ is a proper fractional ideal of $\Lambda_v$. 

The ring $\Lambda_v$ is monogenic by the same argument as for orders in quadratic number rings so \cite[Proof of Proposition 2.1]{ELMVPeriodic} applies and $\mathfrak{L}=l\cdot\Lambda_v$ is an invertible principle fractional ideal with some $l\in E_v^\times$. The element $l\in\mathbf{E}(\mathbb{Q}_v)\subset\mathbf{B}(\mathbb{Q}_v)$ sends $\Lambda_v$ to $\mathfrak{L}$ hence 
\begin{equation*}
g_v \mathbb{O}_v g_v^{-1}={\End}_{\mathbb{Z}_v}(\mathfrak{L})
=l\cdot {\End}_{\mathbb{Z}_v}(\Lambda_v)\cdot l^{-1}
\end{equation*}

The proposition follows from Lemma \ref{lem:local-order-split-0} by setting $\tau_v=\frac{l}{\tensor[^\sigma]{l}{}}f$.
\end{proof}

\begin{prop}\label{prop:tau-a.e.-trivial}
The element $\tau_v\in E_v^\times$ from Proposition \ref{prop:local-order-split} above belongs to $\Lambda_v^\times$ for almost all $v$.
\end{prop}
\begin{proof}
This follows from the fact that any two $\mathbb{Z}$-lattices of full rank in a $\mathbb{Q}$-vector space are equivalent at almost all $v$. 
The order $\mathbb{O}$ is a full rank $\mathbb{Z}$-lattice in the vector space $\mathbf{B}(\mathbb{Q})$. The following subset of $\mathbf{B}(\mathbb{Q})$
\begin{equation*}
\left\{\begin{pmatrix}
a & \epsilon b  \\
\tensor[^\sigma]{b}{} & \tensor[^\sigma]{a}{}
\end{pmatrix}  \relmiddle| a,b\in \mathcal{O}_E \right\} 
\end{equation*}
is also a $\mathbb{Z}$-lattice of full rank by Proposition \ref{prop:rational-points-global} and so it is locally equivalent to $\mathbb{O}$ for almost all $v$. The claim follows by observing that for almost all $v$ we have $g_v\in\mathbb{O}_v^\times$, $\widehat{\Lambda_v}=\Lambda_v=\mathcal{O}_{E_v}$ and $\epsilon\in\mathbb{Z}_v^\times$.
\end{proof}

\begin{prop}\label{prop:local-order-split-archimedean}
If $v=\infty$ then
\begin{equation*}
\left\| g_\infty^{-1}
\begin{pmatrix} 
\alpha & \beta f \\ 
\tensor[^\sigma]{\beta}{}/f & \tensor[^\sigma]{\alpha}{}
\end{pmatrix}
g_\infty \right\|_\infty = |\alpha|+|\beta|
\end{equation*}
In particular,
\begin{align*}
g_\infty \widetilde{\Omega}_\infty g_\infty^{-1}
&=  \left\{\begin{pmatrix} 
\alpha & \beta f \\ 
\tensor[^\sigma]{\beta}{}/f & \tensor[^\sigma]{\alpha}{}
\end{pmatrix} \relmiddle{|} \alpha,\beta\in\mathbb{C}, |\alpha|+|\beta|\leq 2, |\alpha|-|\beta|\geq 1/2 \right\}\\
g_\infty\mathbb{O}_\infty g_\infty^{-1}
&=  \left\{\begin{pmatrix} 
\alpha & \beta f \\ 
\tensor[^\sigma]{\beta}{}/f & \tensor[^\sigma]{\alpha}{}
\end{pmatrix} \relmiddle{|} \alpha,\beta\in\mathbb{C}, |\alpha|+|\beta|\leq 1\right\}
\end{align*}
\end{prop}
\begin{proof}
From the definition of $\|\bullet\|_\infty$ in \S\ref{sec:maximal-order-B} we know that $\|\Ad g_\infty^{-1} \bullet \|_\infty$ is an operator norm on $\mathbf{B}(\mathbb{R})$ induced from some inner-product norm on $E_\infty\simeq\mathbb{C}\simeq \mathbb{R}^2$ when we let  $\mathbf{B}(\mathbb{R})$ act on $E_\infty$ by $\mathbb{R}$-linear endomorphism. This action is explicitly described in Lemma \ref{lem:B(Qv)-acts-on-Ev}.

Fix one of the two possible isomorphism $E_\infty\simeq \mathbb{C}$ and identify the two fields.
Let $|\bullet|_\infty$ be an inner-product norm on $\mathbb{C}$ corresponding to $\|\bullet\|_\infty$. The inner-product norm corresponding to $\|\Ad g_\infty^{-1} \bullet \|_\infty$ is $g_\infty.|\bullet|_\infty\coloneqq v\mapsto |g_\infty^{-1}v|_\infty$. Because of the choices made in \S\ref{sec:maximal-order-B} and \eqref{eq:g_adel_restrictions} the homothety class $\mathbb{R}_{>0}|\bullet|_\infty$ is invariant under the action of $K_\infty= g_\infty^{-1} \mathbf{T}(\mathbb{R}) g_\infty$. Hence the homothety class of $g_\infty.|\bullet|_\infty$ is invariant under $\mathbf{T}(\mathbb{R})$.

We deduce that in the representation of Proposition \ref{prop:rational-points-local} and Lemma \ref{lem:B(Qv)-acts-on-Ev} the homothety class of  $g_\infty.|\bullet|_\infty$ is invariant under the action $\mathbf{E}^\times(\mathbb{R})<\mathbf{B}^\times(\mathbb{R})$ which acts on $\mathbb{C}$ by multiplication. This implies that $g_\infty.|\bullet|_\infty$ is in the homothety class of the standard norm on $\mathbb{C}$ defined by $|x|^2=x\cdot\tensor[^\sigma]{x}{}$.

Using this explicit description of $g_\infty.|\bullet|_\infty$ it is simple to compute the associated operator norm in the coordinates of Lemma \ref{lem:B(Qv)-acts-on-Ev} which turns out to be the norm
\begin{equation*}
\|(\alpha,\beta)\|=|\alpha|+|\beta|=\sqrt{(\Re\alpha)^2 +(\Im\alpha)^2}+\sqrt{(\Re\beta)^2+(\Im\beta)^2}
\end{equation*}
The description of $g_\infty \Omega_\infty g_\infty^{-1}$ is now a simple calculation using the definition in \S\ref{sec:vol}.
\end{proof}

\subsubsection{Ramified Case}\label{sec:ramified-local-order}
Assume now that $\mathbf{B}(\mathbb{Q}_v)$ is ramified. There is a unique maximal order which includes all integral elements. In particular, we have $\mathbb{O}_v=g_v^{-1}\mathbb{O}_v g_v$ and $\Lambda_v=\mathcal{O}_{E_v}$. Moreover, there is an easy criterion to check whether an element is integral using its norm, cf. \cite[Chapter 3]{Reiner}
\begin{align*}
\mathbb{O}_v&=\left\{x\in\mathbf{B}(\mathbb{Q}_v) \mid \Nrd(x)\in\mathbb{Z}_v \right\}\\
&=\left\{\begin{pmatrix}
\alpha & \epsilon \beta \\
\tensor[^\sigma]{\beta}{} & \tensor[^\sigma]{\alpha}{}
\end{pmatrix},\;  \alpha,\beta\in E_v \relmiddle{|} \left|\Nr(\alpha)-\epsilon\Nr(\beta)\right|_v\leq 1 \right\} 
\end{align*}
where the second equality uses Proposition \ref{prop:rational-points-local}. 
The following is a simple statement about $p$-adic numbers
\begin{lem}\label{lem:p-adic-sum}
Let $\pi$ be a uniformizer of the maximal ideal in $\mathbb{Z}_v$.
Two numbers $a,b\in\mathbb{Q}_v$ satisfy $|a-b|_v\leq 1$ if and only if one of the following two options happens
\begin{enumerate}
\item $a,b\in\mathbb{Z}_v$
\item $|a|_v=|b|_v=|\pi|_v^{-n}$ for some $n\in\mathbb{Z}$ and $a/b\equiv 1 \mod \pi^{n}\mathbb{Z}_v$.
\end{enumerate}
\end{lem}
\begin{proof}
Follows from elementary properties of $p$-adic fields.
\end{proof}

\begin{prop}\label{prop:local-order-ramified-inert}
Assume that $v$ is inert in $E$ and $\mathbf{B}(\mathbb{Q}_v)$ is ramified. 
Let $\pi$ be a uniformizer of the maximal ideal in $\mathbb{Z}_v$ and write $\ord_v \epsilon=2k+1$ for $k\in\mathbb{Z}$ then
\begin{equation*}
g_v \mathbb{O}_v g_v^{-1} = 
\left\{\begin{pmatrix}
\alpha & \pi^{k+1} \beta \\
\pi^{-k}\,\tensor[^\sigma]{\beta}{} & \tensor[^\sigma]{\alpha}{}
\end{pmatrix} \relmiddle{|} \alpha,\beta\in\Lambda_v,  \right\} 
\end{equation*}
\end{prop}

Notice that in this case $\widehat{\Lambda_v}=\Lambda_v$ because $\Lambda_v$ is the maximal order and $E_v/\mathbb{Q}_v$ is unramified.
\begin{proof}
If $v$ is inert in $E$ then $E_v/\mathbb{Q}_v$ is an unramified quadratic extension. Hence $\ord_v \Nr(\alpha)$ is even for any $\alpha\in E_v^\times$. Moreover, as the norm map of an unramified extension of local fields is surjective when restricted to the unit groups we deduce that $\Nr(E_v^\times)=\pi^{2\mathbb{Z}}\cdot \mathbb{Z}_v^\times$. Because $B$ is ramified $\epsilon$ is not an $E_v^\times$-norm in $\mathbb{Q}_v^\times$ and $\ord_v \epsilon=2k+1$ for $k\in\mathbb{Z}$.

Let $\alpha,\beta'\in E_v$ such that $|\Nr(\alpha)-\epsilon\Nr(\beta')|_v\leq 1$. The second option in Lemma \ref{lem:p-adic-sum} can never happen for $a=\Nr(\alpha), b=\epsilon\Nr(\beta')$ because $\ord_v \Nr(\alpha)$ is even and $\ord_v\left(\epsilon\Nr(\beta')\right)$ is odd. 

We conclude that necessarily $\Nr(\alpha')\in\mathbb{Z}_v$ and $\epsilon \Nr(\beta')\in\mathbb{Z}_v$. This implies that $\alpha\in \mathcal{O}_{E_v}=\Lambda_v$ and $\beta'\in\pi^{-k}\Lambda_v$.
\end{proof}

\begin{prop}\label{prop:local-order-ramified-ramified}
Assume both $E_v/\mathbb{Q}_v$ and $\mathbf{B}(\mathbb{Q}_v)$ are ramified. Let $\Pi\in\mathcal{O}_{E_v}$ be a uniformizer then there exists $u\in\mathbb{Z}_v^{\times}$ and $k\in\mathbb{Z}$ such that 
\begin{equation*}
g_v \mathbb{O}_v g_v^{-1} \subseteq 
\left\{\begin{pmatrix}
\alpha & \Pi^k u\beta \\
\Pi^{-k} \tensor[^\sigma]{\beta}{} & \tensor[^\sigma]{\alpha}{}
\end{pmatrix} \relmiddle{|} \alpha,\beta \in \widehat{\Lambda_v} \right\} 
\end{equation*}
\end{prop}
\begin{proof}
Let $\pi=\Nr\Pi$ a uniformizer for $\mathbb{Z}_v$.
Because $E_v/\mathbb{Q}_v$ is totally ramified by local class field theory there exists an index $2$ subgroup $U_{E_v}<\mathbb{Z}_v^\times$ such that $\Nr(E_v^\times)=\pi^{\mathbb{Z}} U_{E_v}$. Hence $\epsilon =\pi^k u$ for some $k\in\mathbb{Z}$ and $u\in\mathbb{Z}_v^\times \setminus U_{E_v}$.

Let $x\in g_v \mathbb{O}_v g_v^{-1}$ and write $x=\begin{pmatrix}
\alpha & \epsilon \beta' \\
\tensor[^\sigma]{\beta}{}' & \tensor[^\sigma]{\alpha}{}
\end{pmatrix}$ for some $\alpha,\beta'\in E_v$. Set also $\beta'=\tensor[^\sigma]{\Pi}{^{-k}} \beta$ where $\beta\in E_v$ then $\epsilon\beta'=\Pi^k u \beta$.

Any element $l\in\Lambda_v$ belongs to $g_v \mathbb{O}_v g_v^{-1}$ so $x\cdot l\in g_v \mathbb{O}_v g_v^{-1}$ and $x\cdot l$ is integral. This implies
\begin{equation*}
\forall l\in \Lambda_v\colon \Tr(\alpha\cdot l)=\Trd(x\cdot l)\in \mathbb{Z}_v
\end{equation*}
thus $\alpha\in\widehat{\Lambda_v}$.

If the first option in Lemma \ref{lem:p-adic-sum} holds then $\epsilon\Nr(\beta')=u\Nr(\beta)\in\mathbb{Z}_v$ and necessarily $\beta\in\mathcal{O}_{E_v}=\Lambda_v\subseteq \widehat{\Lambda_v}$. The second case is relevant only when $\alpha\in\widehat{\Lambda_v}\setminus\Lambda_v$ and then $|\epsilon\Nr(\beta')|_v=|\Nr(\beta)|_v=|\Nr(\alpha)|_v$. As $\widehat{\Lambda_v}$ is a principal ideal we deduce that $\beta$ must also belong to $\widehat{\Lambda_v}$.
\end{proof}

\subsubsection{General Case}
We summarize the results of this section in a form useful to us.
\begin{prop}\label{prop:local-order-general}
For any finite rational place $v$  there is some $\tau_v\in E_v^\times$ such that
\begin{equation*}
g_v \mathbb{O}_v g_v^{-1} \subseteq \left\{\begin{pmatrix} 
\alpha & \beta \upsilon_v \tau_v \\ 
\tensor[^\sigma]{\beta}{}/\tau_v & \tensor[^\sigma]{\alpha}{}
\end{pmatrix} \relmiddle{|} \alpha,\beta\in \widehat{\Lambda_v} \right\}
\end{equation*}

If $\mathbf{B}$ is split at $v$ then $\upsilon_v=1$, if $\mathbf{B}$ is ramified and $E$ is inert at $v$ then $\upsilon_v$ is a uniformizer in $\mathbb{Z}_v$ and if both $\mathbf{B}$ and $E$ are ramified at $v$ then $\upsilon_v$ is a unit which is not an $E_v^\times$ norm.

Moreover, $\tau_v\in\Lambda_v^\times$ for almost all $v$ and $\tau_v=1$ if $\mathrm{B}$ is ramified at $v$.
\end{prop}
\begin{proof}
Immediate corollary of Propositions \ref{prop:local-order-split}, \ref{prop:local-order-ramified-inert}, \ref{prop:local-order-ramified-ramified} and \ref{prop:tau-a.e.-trivial}.
\end{proof}

\subsection{Good Integral Representatives}
In this section we will discuss how to find good representatives in $\mathbb{O}_v\subset\mathbf{B}(\mathbb{Q}_v)$ of elements in $\mathbf{G}(\mathbb{Q}_v)$ using the Cartan decomposition. The notion of a good representative generalizes the idea of writing a rational number as an integer fraction in lowest terms.

\begin{defi}\addpenalty{-1000}
\hfill
\begin{enumerate}
\item For a finite rational place $v$ where $\mathbf{G}$ splits let $\mathscr{B}_v$ be the the Bruhat-Tits building of $\mathbf{G}(\mathbb{Q}_v)$. If $\mathbf{G}$ is ramified at $v\neq\infty$ let $\mathscr{B}_v$ the connected graph with two vertices corresponding to $\faktor{\mathbf{G}(\mathbb{Q}_v)}{K_v}$.

Denote by $d$ the geodesic distance function on the graph $\mathscr{B}_v$ normalized so that the length of each edge is $1$.

\item If $\mathbf{B}(\mathbb{R})$ is split set $\mathscr{B}_\infty=\faktor{\mathbf{G}(\mathbb{R})}{\Nrml_{\mathbf{G}(\mathbb{R})}(K_\infty)}$. This manifold is the upper half-plane which we equip with the standard hyperbolic distance function $d$. If $\mathbf{B(\mathbb{R})}$ is ramified let $\mathscr{B}_\infty$ be a single point with the trivial metric.

\item For each place $v$ let $x_0$ be the point in $\mathscr{B}_v$ stabilized by $K_v$. Let $q$ be the residue characteristic for $v\neq\infty$ and $q=e$ for $v=\infty$.
Define a continuous function $\Denom_v\colon\mathbf{G}(\mathbb{Q}_v)\to \mathbb{R}_{>0}$ by
\begin{equation*}
\Denom_v(x_v)\coloneqq q^{d(x_0,x_v.x_0)}
\end{equation*}

\item Define the continuous function $\Denom_f\colon\mathbf{G}(\mathbb{A}_f)\to\mathbb{N}$ by
\begin{equation*}
\Denom_f\big((x_v)_{v\neq\infty}\big)=\prod_{v\neq \infty} \Denom_v(x_v)
\end{equation*}
\end{enumerate}
\end{defi}

\begin{prop}\label{prop:Omega_xi_Omega-reduced-norm}
Let $v$ be a rational place and $x_v\in\mathbf{G}(\mathbb{Q}_v)$. For any $h\in\Omega_v x_v \Omega_v\subset \mathbf{G}(\mathbb{Q}_v)$ there is $r\in\mathbb{O}_v\subset \mathbf{B}(\mathbb{Q}_v)$ such that $h=\mathbf{Z}(\mathbb{Q}_v) r$ and
\begin{align*}
|\Nrd(r)|_v\Denom_v(x_v)=1 & \textrm{   if } v\neq\infty\\
2^{-8}\leq |\Nrd(r)|_\infty\Denom_v(x_v)\leq 1 & \textrm{   if } v=\infty
\end{align*}

Moreover, for $v\neq\infty$  this representative is optimal in the following way. If $h\in \Omega_v x_v\Omega_v$ then it has no representative in $\mathbb{O}_v$ whose reduced norm has smaller valuation then $r$ above.
\end{prop}

\begin{proof}
In the split case this follows from the Cartan decomposition, the equality $K_v \Omega_v=\Omega_v K_v=\Omega_v$ and \eqref{eq:Omega-tilde-Nrd} for $v=\infty$. In the finite ramified case this is a consequence of the fact that the ramification index of $\mathbf{B}(\mathbb{Q}_v)$ is $2$, i.e.\ the value group of $\mathbb{Q}_v$ is an index $2$ subgroup of the value group of the division algebra $\mathbf{B}(\mathbb{Q}_v)$. In the infinite ramified case $\Omega_v=\mathbf{G}(\mathbb{Q}_v)$ and the statement is trivial.

The last statement about the optimality of the representative for $v\neq\infty$ follows from the mutual disjointness of the $K_v$ double cosets in the Cartan decomposition.
\end{proof}

\begin{defi}\label{def:Bowen-ball}
Let $B\subseteq G(\mathbb{Q}_{p_1})$ be an identity neighborhood and $n\in\mathbb{N}\cup\{0\}$. Let $\lambda\colon \mathbb{Q}_{p_1}^\times\to A_{p_1}$ be a cocharacter spanning $X_\bullet(A_{p_1})$. Set $a=\lambda(p_1)\in A_{p_1}$.
We use the following notation for the symmetric homogeneous $B$-Bowen ball for the $a$-action.
\begin{equation*}
B^{(-n,n)}\coloneqq \bigcap_{k=-n}^{n} a^k B a^{-k}
\end{equation*}
Notice that the definition of $B^{(-n,n)}$ does not depends on the choice of $\lambda$. 

We define similarly
\begin{equation*}
\mathbb{O}_{p_1}^{(-n,n)}\coloneqq \bigcap_{k=-n}^{n} a^k \mathbb{O}_{p_1} a^{-k} \subset \mathbf{B}(\mathbb{Q}_{p_1})
\end{equation*}
\end{defi}

\begin{prop}\label{prop:Omega_xi_Omega-reduced-norm-Bowen}
Fix $\xi_{p_1}\in A_{p_1}$ and $n\geq 0$. For any $h\in K_{p_1}^{(-n,n)}\xi_{p_1}K_{p_1}^{(-n,n)}$ there is $r\in\mathbb{O}_{p_1}^{(-n,n)}$ with $h=\mathbf{Z}(\mathbb{Q}_{p_1})r$ and $|\Nrd(r)|_{p_1}\Denom_{p_1}(\xi_{p_1})=1$.
\end{prop}
\begin{proof}
Because $\xi_{p_1}$ centralizes $A_{p_1}$
\begin{equation*}
h\in K_{p_1}^{(-n,n)}\xi_{p_1}K_{p_1}^{(-n,n)} \subseteq \bigcap_{k=-n}^n a^k K_{p_1}\xi_{p_1}K_{p_1} a^{-k}
\end{equation*}
Applying Lemma \ref{prop:Omega_xi_Omega-reduced-norm} for each set in the intersection above we conclude that for every $-n\leq k \leq n$ there is a representative $r_k\in a^k\mathbb{O}_{p_1} a^{-k}$ of $h$ such that $|\Nrd r_k|_v \Denom_{p_1}(\xi_{p_1})=1$.\

All the representatives $r_k$ for different values of $k$ are in the same $\mathbf{Z}(\mathbb{Q}_{p_1})$-orbit and their reduced norms have the same absolute value. Thus they are all in the same orbit of $\mathbb{Z}_{p_1}^\times<\mathbb{Q}_{p_1}^\times=\mathbf{Z}(\mathbb{Q}_{p_1})$. Multiplying each $r_k$ by an appropriate element 
of $\mathbb{Z}_{p_1}^\times<a^k\mathbb{O}_{p_1} a^{-k}$ for all $k$ we can make all the representatives $r_k$ equal to each other without affecting the valuation of their reduced norm and so that they still satisfy $r_k\in a^k\mathbb{O}_{p_1} a^{-k}$. This common representative satisfies the conditions of the claim.
\end{proof}

\section{The Double Quotient \texorpdfstring{${\mathbf{G}^\Delta}\backslash{\mathbf{G}\times\mathbf{G}}\slash{\mathbf{T}^\Delta}$}{G^Delta \textbackslash GxG / T^Delta}}\label{sec:GIT}
When studying the relative position of a homogeneous Hecke set and a joint homogeneous toral set we need to understand the double quotient $\dfaktor{\mathbf{G}^\Delta}{\mathbf{G}\times\mathbf{G}}{\mathbf{T}^\Delta}$. This can be achieved using GIT and Galois descent.

\subsection{The GIT Double Quotient}
\begin{defi}\hfill\\
\begin{enumerate}
\item 
Denote $\mathbf{M}\coloneqq\mathbf{G}\times\mathbf{T}$. The linear algebraic $\mathbf{M}$ group is defined over $\mathbb{Q}$. We consider an action of the group $\mathbf{M}$ 
on the affine variety $\mathbf{G}\times\mathbf{G}$ by letting the $\mathbf{G}$ coordinate act by diagonal multiplication on the left and by letting the $\mathbf{T}$ coordinate to act by diagonal multiplication by the inverse on the right. For geometric points the action is
\begin{equation*}
(l,t).(g_1,g_2)=(lg_1t^{-1},lg_2t^{-1})
\end{equation*}

\item
The universal categorical quotient for the action of the linear reductive group $\mathbf{M}$ on the affine variety $\mathbf{G}\times\mathbf{G}$ is representable by the following affine variety defined over $\mathbb{Q}$ \cite[Theorem 1.1]{GIT}
\begin{equation*}
\mathbf{W}\coloneqq\dfaktor{\mathbf{G}^\Delta}{\mathbf{G}\times\mathbf{G}}{\mathbf{T}^\Delta}
\coloneqq \Spec{\tensor[^{\mathbf{G}^\Delta}]{\mathbb{Q} \left[\mathbf{G}\times\mathbf{G}\right]}{^{\mathbf{T}^\Delta}}}
\end{equation*}
where $\tensor[^{\mathbf{G}^\Delta}]{\mathbb{Q} [\mathbf{G}\times\mathbf{G}]}{^{\mathbf{T}^\Delta}}$ is the ring of regular functions of $\mathbf{G}\times\mathbf{G}$ invariant under the $\mathbf{M}$ action.

\item
For any $\gamma\in\left(\mathbf{G}\times\mathbf{G}\right)(\mathbb{Q})$ define $\mathbf{M}_\gamma$ to be the stabilizer of $\gamma$. It is a linear algebraic group defined over $\mathbb{Q}$.

\item Denote by $\pi_\mathbf{W}\colon \mathbf{G}\times\mathbf{G}\to\mathbf{W}$ the $\mathbf{M}$-equivariant projection map.
\end{enumerate}
\end{defi}

\begin{defi}
Let $w_\mathbf{T}\in\Nrml_\mathbf{G}\mathbf{T}(\mathbb{Q})$ be a rational representative of the non-trivial class of the Weyl group of $\mathbf{T}$, i.e.\ $w_\mathbf{T}\not\in\mathbf{T}(\mathbb{Q})$. Such a representative always exists, for example because of Proposition \ref{prop:rational-points-global}. Although, the element $w_\mathbf{T}$ is not uniquely defined the variety $w_\mathbf{T}\mathbf{T}$ is a well-defined closed sub-variety of $\mathbf{G}$ defined over $\mathbb{Q}$.
\end{defi}

\begin{prop}\label{prop:M-orbits}
Let $\gamma=(\gamma_1,\gamma_2)\in\left(\mathbf{G}\times\mathbf{G}\right)(\mathbb{Q})$ be a rational point.
\begin{enumerate}
\item The $\mathbf{M}$-orbit of $\gamma$ is Zariski closed.
\item Recall that $\ctr(\gamma)=\gamma_1^{-1}\gamma_2$.
The stabilizer $\mathbf{M}_\gamma$ is trivial if $\ctr(\gamma)\not\in \Nrml_{\mathbf{G}}\mathbf{T}(\mathbb{Q})$.
If  $\ctr(\gamma)\in \Nrml_{\mathbf{G}}\mathbf{T}(\mathbb{Q})$ then
\begin{equation*}
\mathbf{M}_\gamma\simeq \Cent_{\mathbf{T}}(\ctr(\gamma))=\begin{cases}
\mathbf{T} & \ctr(\gamma)\equiv 1 \mod \mathbf{T}(\mathbb{Q})\\
\mathbf{T}[2] & \ctr(\gamma)\not\equiv 1 \mod \mathbf{T}(\mathbb{Q})
\end{cases}
\end{equation*}
The isomorphism above is $t\mapsto(\gamma_1t\gamma_1^{-1},t)$.
Moreover, the diagonalizable abelian affine group $\mathbf{T}[2]$ is isomorphic to $\mu_2$ over $\mathbb{Q}$.
\item If $\ctr(\gamma)\not\in \ w_\mathbf{T}\mathbf{T}(\mathbb{Q})$ then the following set is a singleton
\begin{equation*}
\ker\left[H^1\left(\mathbb{Q},\mathbf{M}_\gamma\right)
\to H^1\left(\mathbb{Q},\mathbf{M}\right) \right]
\end{equation*}
\item If $\ctr(\gamma)\not\in w_\mathbf{T}\mathbf{T}(\mathbb{Q})$ then $\pi_\mathbf{W}^{-1}(\gamma)(\mathbb{Q})$ is a single $\mathbf{M}(\mathbb{Q})$-orbit.
\end{enumerate}
\end{prop}
\begin{proof}
\textit{Part (1).}
Assume the orbit of $\gamma$ is not Zariski closed.
By \cite[Corollary 1.2]{GIT} the map $\pi_{\mathbf{W}}$ separates $\mathbf{M}$-invariant closed subsets. Because $\mathbf{G}\times\mathbf{G}$ is Noetherian we deduce that the fiber $\pi_{\mathbf{W}}^{-1}\left(\pi_{\mathbf{W}}(\gamma)\right)$ contains a  unique minimal non-empty Zariski closed $\mathbf{M}$-invariant subset, let $\mathbf{S}$ be the closed subvariety supported on this set.  Because the map $\pi_{\mathbf{W}}$ separates invariant closed subsets the support of $\mathbf{S}$ is contained in any non-empty invariant closed subset in the fiber.

The  orbit $\mathbf{M}.\gamma$ is open in its closure so $\overline{\mathbf{M}.\gamma}^\mathrm{Zar}\setminus\mathbf{M}.\gamma$ is an invariant  Zariski closed subset which by our assumption is non-empty, hence $\mathbf{S}$ is contained in it. In particular, $\gamma\not\in\mathbf{S}(\mathbb{Q})$.

The Zariski closure of the orbit of $\gamma$ contains $\mathbf{S}$ and hence by \cite[Corollary 4.3]{instability} there is a one-parameter subgroup $\lambda\colon\Gm\to\mathbf{M}=\mathbf{G}\times\mathbf{T}$ defined over $\mathbb{Q}$ such that $\delta=\lim_{s\to0}\lambda(s).\gamma\in\mathbf{S}(\mathbb{Q})$.

The torus $\mathbf{T}$ is anisotropic over $\mathbb{Q}$ so the image of $\lambda$ lies in $\mathbf{G}$ and $\delta\in\overline{\mathbf{G}.\gamma}^\mathrm{Zar}(\mathbb{Q})$.
But  $\mathbf{G}.\gamma\simeq\mathbf{G}.e=\mathbf{G}^\Delta$ which is Zariski closed, so $\delta\in\left(\mathbf{G}.\gamma\cap \mathbf{S}\right)(\mathbb{Q})$. As $\mathbf{S}$ is $\mathbf{M}$-invariant  we deduce a contradiction that $\gamma\in S(\mathbb{Q})$.

\textit{Part (2)}
Let $e\neq(g,t)\in\mathbf{M}_\gamma(\bar{\mathbb{Q}})<\mathbf{G}(\bar{\mathbb{Q}})\times \mathbf{T}(\bar{\mathbb{Q}})$ then
\begin{align*}
(g\gamma_1 t^{-1}, g\gamma_2 t^{-1})=(\gamma_1,\gamma_2) &\Longrightarrow t \ctr(\gamma) t^{-1}=\ctr(\gamma) \\
&\Longrightarrow \ctr(\gamma)\in \Nrml_{\mathbf{G}(\bar{\mathbb{Q}})}(t)=\mathbf{T}(\bar{\mathbb{Q}})
\end{align*}
Moreover, in this case $g=\gamma_1 t \gamma_1^{-1}$.
We deduce that the stabilizer is trivial unless $\ctr(\gamma)\in\Nrml_{\mathbf{G}}\mathbf{T}(\mathbb{Q})$ and it is isomorphic to $\Cent_{\mathbf{T}}(\ctr(\gamma))$ otherwise with the isomorphism exactly as stated in claim (2). 

If $\ctr(\gamma)\in\mathbf{T}(\mathbb{Q})$ then the entire torus $\mathbf{T}$ centralizers it. If $\ctr(\gamma)\in \Nrml_{\mathbf{G}}\mathbf{T}(\mathbb{Q})$ then only elements of order $2$ centralize it. This finishes the computation of the stabilizers.

To see that $\mathbf{T}[2]\simeq \mu_2$ we consider the dual group $\widehat{\mathbf{T}[2]}\simeq \faktor{\widehat{\mathbf{T}}}{\widehat{\mathbf{T}}^2}$ which has two geometric points. As the Galois group $\Gal(E/\mathbb{Q})$ acts by inversion on $\mathbf{T}$ its action on $\faktor{\widehat{\mathbf{T}}}{\widehat{\mathbf{T}}^2}$ is trivial. Hence this dual group is the constant $\cyclic{2}$-group scheme and its dual is $\mu_2$.

\textit{Part (3)}
If the stabilizer is trivial then the statement is obvious. Otherwise, we use the projection on the second coordinate $\mathbf{M}=\mathbf{G}\times\mathbf{T}\to \mathbf{T}$ to construct a sequence of maps
\begin{equation}\label{eq:H1(Mgamma)-to-H1(T)}
H^1(\mathbb{Q},\mathbf{M}_\gamma)\to H^1(\mathbb{Q},\mathbf{M})\to H^1(\mathbb{Q},\mathbf{T})
\end{equation}
In order to prove that $\ker\left[H^1\left(\mathbb{Q},\mathbf{M}_\gamma\right)
\to H^1\left(\mathbb{Q},\mathbf{M}\right) \right]=1$ it is enough to show that the kernel of the composite map of \eqref{eq:H1(Mgamma)-to-H1(T)} is trivial.

The isomorphism $\mathbf{M_\gamma}\simeq \Cent_{\mathbf{T}}(\ctr(\gamma))$ is given by the inclusion map in the second coordinate, hence the composite map of \eqref{eq:H1(Mgamma)-to-H1(T)} is exactly the map of cohomology sets $H^1(\mathbb{Q},\Cent_{\mathbf{T}}(\ctr(\gamma)))\to H^1(\mathbb{Q},\mathbf{T})$
induced by the inclusion $\Cent_{\mathbf{T}}(\ctr(\gamma))\hookrightarrow\mathbf{T}$. This map is the identity if $\Cent_{\mathbf{T}}(\ctr(\gamma))=\mathbf{T}$ and obviously has trivial kernel.

\textit{Part (4).}
Because the $\pi_{\mathbf{W}}$-fiber of any rational point contains a unique Zariski closed orbit of a rational point we conclude that $\pi_{\mathbf{W}}$ separates $\mathbf{M}(\bar{\mathbb{Q}})$-orbits of rational points. We are left with proving that for any $\gamma\not\in w_\mathbf{T}\mathbf{T}(\mathbb{Q})$ the orbit $\mathbf{M}(\bar{\mathbb{Q}})$ contains a unique $\mathbf{M}(\mathbb{Q})$-orbit. 

The collection of $\mathbf{M}(\mathbb{Q})$-orbits in $\mathbf{M}(\bar{\mathbb{Q}}).\gamma$ is in bijection with
\begin{equation}\label{eq:GxT-cohomology-kernel}
\ker\left[H^1\left(\mathbb{Q},\mathbf{M}_\gamma\right)
\to H^1\left(\mathbb{Q},\mathbf{M}\right) \right]
\end{equation}
which is trivial by part (3) if $\ctr(\gamma)\not\in w_\mathbf{T}\mathbf{T}(\mathbb{Q})$.
\end{proof}

The proposition above implies that the set theoretic double cosets $\dfaktor{\mathbf{G}^\Delta(\mathbb{Q})}{\left(\mathbf{G} \times \mathbf{G}\right)(\mathbb{Q})}{\mathbf{T}^\Delta(\mathbb{Q})}$ are almost parametrized by the associated points in $\mathbf{W}(\mathbb{Q})$. Not all the points in $\mathbf{W}(\mathbb{Q})$ actually correspond to set theoretic double cosets of rational points of $\mathbf{G}$.

\subsection{The Quotient of \texorpdfstring{$\mathbf{G}$}{G} by the Adjoint Action of \texorpdfstring{$\mathbf{T}$}{T}}
Because the left action of $\mathbf{G}^\Delta$ and right action of $\mathbf{T}^\Delta$ on $\mathbf{G}\times\mathbf{G}$ commute we have the following commutative diagram
\begin{center}
\begin{tikzcd}
 \mathbf{G}\times \mathbf{G}  \arrow[d,twoheadrightarrow] \arrow[r,twoheadrightarrow] & \mathbf{W} \arrow[dd]\\
\lfaktor{\mathbf{G}^\Delta}{\mathbf{G}\times\mathbf{G}}  \arrow[d,,"\scalebox{2}{\rotatebox{90}{\(\sim\)}}"]  & \\
\mathbf{G} \arrow[r,twoheadrightarrow]& \lfaktor{\Ad \mathbf{T}}{\mathbf{G}}
\end{tikzcd}
\end{center}
where $\lfaktor{\Ad \mathbf{T}}{\mathbf{G}}$ is the GIT quotient of the affine variety $\mathbf{G}$ by the adjoint action of the reductive group $\mathbf{T}$. The existence of the morphism $\mathbf{W}\to \lfaktor{\Ad \mathbf{T}}{\mathbf{G}}$ follows from the universal property of the categorical quotient $\mathbf{W}$. The composite $\mathbf{G}^\Delta$-invariant map $\mathbf{G}\times\mathbf{G}\to\mathbf{G}$ in the left column of the diagram is exactly the contraction map $\ctr$.

\begin{prop}\label{prop:AdTG-W-iso}
The morphism $\mathbf{W}\to \lfaktor{\Ad \mathbf{T}}{\mathbf{G}}$ is an isomorphism.
\end{prop}
\begin{proof}
Because the morphism $\lfaktor{\mathbf{G}^\Delta}{\mathbf{G}\times\mathbf{G}} \to \mathbf{G}$ is an isomorphism the ring of regular functions on $\lfaktor{\Ad \mathbf{T}}{\mathbf{G}}$ is identified with the regular functions on $\lfaktor{\mathbf{G}^\Delta}{\mathbf{G}\times\mathbf{G}}$ which are invariant under the right action of $\mathbf{T}^\Delta$. This is the same as the ring of $\mathbf{M}$-invariant regular functions on $\mathbf{G}\times\mathbf{G}$ because the left action of $\mathbf{G}^\Delta$ commutes with the right action of $\mathbf{T}^\Delta$.
\end{proof}

\begin{cor}\label{cor:AdTG(Q)-separation}
Let $\pi_{\mathbf{W}}^0\colon \mathbf{G}\to \lfaktor{\Ad \mathbf{T}}{\mathbf{G}}\simeq \mathbf{W}$ be the $\Ad\mathbf{T}$-equivariant projection map.
For any $\gamma_0\in\mathbf{G}(\mathbb{Q})$ if $\gamma_0\not\in w_\mathbf{T}\mathbf{T}(\mathbb{Q})$ 
then ${\pi_{\mathbf{W}}^0}^{-1}(\gamma_0)(\mathbb{Q})$ is a single
 $\Ad\mathbf{T}(\mathbb{Q})$-orbit.
\end{cor}
\begin{proof}
Follows immediately from Proposition \ref{prop:AdTG-W-iso} and part (4) of Proposition \ref{prop:M-orbits}. 
\end{proof}

\begin{defi}
Recall that $\mathbf{G}^\sc$ is identified with the group of unit quaternions in $\mathbf{B}$. Define $\mathbf{T}^{(1)}$ to be the maximal torus defined over $\mathbb{Q}$ in $\mathbf{G}^\sc$ which maps under the isogeny $\mathbf{G}^\sc\to\mathbf{G}$ to $\mathbf{T}$. 

The identity $\mathbf{G}^\sc=\mathbf{B}^{(1)}$ implies that the torus $\mathbf{T}^{(1)}$ is the subgroup of unit quaternions in $\widetilde{\mathbf{T}}$.
\end{defi}

Using the isogeny $\mathbf{T}^{(1)}\to\mathbf{T}$ we let $\mathbf{T}^{(1)}$ act on $\mathbf{G}$ through the adjoint action of $\mathbf{T}$. As $\mathbf{T}^{(1)}\to\mathbf{T}$ is surjective we have
\begin{equation*}
\lfaktor{\Ad \mathbf{T}}{\mathbf{G}}=\lfaktor{\Ad \mathbf{T}^{(1)}}{\mathbf{G}}
\end{equation*}

We let $\mathbf{T}^{(1)}$ act on $\mathbf{B}^\times$ by the adjoint action. Because the actions of $\Ad \mathbf{T}^{(1)}$ and $\mathbf{Z}$ on $\mathbf{B}^\times$ commute the reductive group $\mathbf{Z}$ acts on the affine variety $\lfaktor{\Ad\mathbf{T}^{(1)}}{\mathbf{B}^\times}$ and the GIT quotient for this action is canonically isomorphic to $\lfaktor{\Ad \mathbf{T}}{\mathbf{G}}$. In particular, the morphism of Noetherian schemes
\begin{equation*}
\lfaktor{\Ad\mathbf{T}^{(1)}}{\mathbf{B}^\times}\to\lfaktor{\Ad \mathbf{T}}{\mathbf{G}}\simeq \mathbf{W}
\end{equation*}
is universally submersive.

\subsection{The Quotient of \texorpdfstring{$\mathbf{GL}_2$}{GL2} by the Adjoint Action of the Diagonal Torus}
To describe the ring of regular function of $\lfaktor{\Ad\mathbf{T}^{(1)}}{\mathbf{B}^\times}$ we begin with a simpler case when $\mathbf{B}=\mathbf{M}_{2\times 2}$ is split over $\mathbb{Q}$ and $\mathbf{T}^{(1)}$ is replaced by the torus of diagonal matrices with determinant $1$. 

\begin{defi}
Let $\mathbf{A}^{(1)}<\mathbf{GL}_2$ be the rank-$1$ torus of diagonal matrices with determinant 1. Denote by $\mathbf{A}$ the maximal torus of split diagonal matrices in $\mathbf{PGL}_2$. The map $\mathbf{A}^{(1)}\to\mathbf{A}$ is surjective in terms of schemes.
\end{defi}

\begin{defi}\label{def:inv-functions}\hfil
\begin{enumerate}
\item We let $\mathbf{A}^{(1)}$ act on $\mathbf{M}_{2\times 2}\times \mathbb{A}_1$ by conjugating the $2\times 2$ matrix and leaving the $\det^{-1}$ coordinate invariant. We denote this action by $\Ad \mathbf{A}^{(1)}$. 
This action is clearly equivariant with respect to the map $\mathbf{GL}_2\hookrightarrow \mathbf{M}_{2\times 2}\times \mathbb{A}_1$.

\item Define $\vartheta_1,\vartheta_2,\psi \in \mathbb{Q}[\mathbf{M}_{2\times 2}]$ by 
\begin{align*}
\vartheta_1&\coloneqq x_{1,1} &
\vartheta_2&\coloneqq x_{2,2} &
\psi&\coloneqq x_{1,2}x_{2,1} 
\end{align*}
\end{enumerate}
\end{defi}

\begin{lem}\label{lem:M2x2A1-invariants}
There is an equality of $\mathbb{Q}$-algebras
\begin{equation*} 
\mathbb{Q}[\mathbf{M}_{2\times 2}\times \mathbb{A}_1]^{\Ad\mathbf{A}^{(1)}}=\mathbb{Q}[\vartheta_1,\vartheta_2,\psi,{\det}^{-1}]
\end{equation*}
The left-hand side is the ring of regular functions of $\mathbf{M}_{2\times 2}\times \mathbb{A}_1$ invariant under $\Ad \mathbf{A}^{(1)}$ and the right-hand side is a polynomial algebra over $\mathbb{Q}$.
\end{lem}
\begin{proof}
It is easy the check that $\vartheta_1,\vartheta_2,\psi,{\det}^{-1}$ are $\Ad\mathbf{A}^{(1)}$-invariant. We need to show that these function generate the ring of invariants and that there are no non-trivial syzygies.

Because $\mathbb{Q}[\mathbf{M}_{2\times 2}\times \mathbb{A}_1]$ is a polynomial ring and the action of $\Ad\mathbf{A}^{(1)}$ preserves monomials, the invariant ring is generated by monomials.
Let $f\in \mathbb{Q}[\mathbf{M}_{2\times 2}\times \mathbb{A}_1]^{\Ad\mathbf{A}^{(1)}}$ be a monomial and write
\begin{equation*}
f=\left({\det}^{-1}\right)^d \prod_{1\leq i,j \leq 2} x_{i,j}^{a_{i,j}}
=\left({\det}^{-1}\right)^d \vartheta_1^{a_{1,1}} \vartheta_2^{a_{2,2}} x_{1,2}^{a_{1,2}} x_{2,1}^{a_{2,1}}
\end{equation*}
For $f$ to be $\Ad\mathbf{A}^{(1)}$-invariant we must have $a_{1,2}=a_{2,1}$ which implies that $f\in \mathbb{Q}[\vartheta_1,\vartheta_2,\psi,{\det}^{-1}]$.

A syzygy $Q$ is a formal polynomial in the variables $\vartheta_1,\vartheta_2,\psi,{\det}^{-1}$ with coefficients in $\mathbb{Q}$ which vanishes as an element of $\mathbb{Q}[\mathbf{M}_{2\times 2}\times \mathbb{A}_1]^{\Ad\mathbf{A}^{(1)}}$. Because each variable $x_{i,j},\det^{-1}$, $1\leq i,j \leq 2$ appears in only one of the monomials $\vartheta_1,\vartheta_2,\psi,{\det}^{-1}$ we conclude that if $Q$ vanishes as an element of $\mathbb{Q}[\vartheta_1,\vartheta_2,\psi,{\det}^{-1}]$ it also vanishes as an element of the free polynomial algebra $\mathbb{Q}[\mathbf{M}_{2\times 2}\times \mathbb{A}_1]$. Thus all the relations between $\vartheta_1,\vartheta_2,\psi,{\det}^{-1}$ are trivial.
\end{proof}

\begin{prop}
The ring of $\Ad \mathbf{A}^{(1)}$-invariant regular functions on $\mathbf{GL}_2$ is
\begin{equation*}
\mathbb{Q}[\mathbf{GL}_2]^{\Ad\mathbf{A}^{(1)}}=
\left.\mathbb{Q}\left[\vartheta_1,\vartheta_2,\psi,{\det}^{-1}\right] 
\middle\slash \left<\left(\vartheta_1\vartheta_2-\psi\right){\det}^{-1}=1 \right> \right.
\end{equation*}
\end{prop}
\begin{proof}
If a \emph{linearly} reductive group $\mathbf{H}$ acts on two affine schemes $\mathbf{X},\mathbf{Y}$ of finite type over a field $k$ then a lemma of 
Nagata \cite[Lemma 5.1.A]{Nagata} implies that an $\mathbf{H}$-equivariant closed immersion $\mathbf{X}\hookrightarrow\mathbf{Y}$ over $k$ descends to a \emph{closed immersion} of GIT quotients $\lfaktor{\mathbf{H}}{\mathbf{X}}\hookrightarrow\lfaktor{\mathbf{H}}{\mathbf{Y}}$.

The proposition follows by applying this result to the closed immersion $\mathbf{GL}_2\hookrightarrow \mathbf{M}_{2\times 2} \times \mathbb{A}_1$ and using Lemma \ref{lem:M2x2A1-invariants}.
\end{proof}

\begin{defi}\hfill
\begin{enumerate}
\item 
Define the the degree of $\vartheta_1$ and $\vartheta_2$ to be $1$, the degree of $\psi$ to be $2$ and the degree of $\det^{-1}$ to be $-2$. Define the degree of a monomial in $\vartheta_1,\vartheta_2,\psi,{\det}^{-1}$ as the sum of the degrees of the individual variables appearing in the product. The degree of a constant is $0$.

\item 
A polynomial in $\mathbb{Q}\left[\vartheta_1,\vartheta_2,\psi,{\det}^{-1}\right]$ is of zero degree if it is the sum of zero degree monomials. Denote the $\mathbb{Q}$-algebra of zero-degree elements by $\mathbb{Q}\left[\vartheta_1,\vartheta_2,\psi,{\det}^{-1}\right]^0$
\end{enumerate}
\end{defi}

\begin{cor}\label{cor:PGL2-AdA-inv}
The ring  of $\Ad \mathbf{A}$-invariant regular functions on $\mathbf{PGL}_2$ is the ring 
\begin{equation*}
\mathbb{Q}[\mathbf{PGL}_2]^{\Ad\mathbf{A}}
\left.=\mathbb{Q}\left[\vartheta_1,\vartheta_2,\psi,{\det}^{-1}\right]^0 
\middle\slash \left<\left(\vartheta_1\vartheta_2-\psi\right){\det}^{-1}=1 \right> \right.
\end{equation*}

Moreover, this ring is generated by the functions
\begin{align*}
&\psi{\det}^{-1}, & &\vartheta_1^2 {\det}^{-1}, & &\vartheta_2^2 {\det}^{-1}
\end{align*}
\end{cor}
\begin{proof}
Because the actions of $\mathbf{Z}$ and $\Ad \mathbf{A}^{(1)}$ on $\mathbf{GL}_2$ commute there is an isomorphism
\begin{equation*}
\lfaktor{\mathbf{Z}}{\left(\lfaktor{\Ad\mathbf{A}^{(1)}}{\mathbf{GL}_2}\right)}\to\lfaktor{\Ad \mathbf{A}^{(1)}}{\mathbf{PGL}_2}=\lfaktor{\Ad \mathbf{A}}{\mathbf{PGL}_2}
\end{equation*}
The equality on the right follows from the surjectivity of $\mathbf{A}^{(1)}\to\mathbf{A}$.

This implies that the ring $\mathbb{Q}[\mathbf{PGL}_2]^{\Ad\mathbf{A}}$ is a the subring of elements in $\mathbb{Q}[\mathbf{GL}_2]^{\Ad\mathbf{A}^{(1)}}$ which are $\mathbf{Z}$-invariant. It is a direct computation to see that these are exactly the degree $0$ elements and that the given functions generate this ring. 
\end{proof}
\begin{remark}
A slightly more delicate analysis shows that 
\begin{equation*}
\mathbb{Q}[\mathbf{PGL}_2]^{\Ad\mathbf{A}}\simeq \mathbb{Q}[x,y,z]\slash \left<x^2=yz\right>
\end{equation*}
where $x=1+\psi \det^{-1}, y=\vartheta_1^2\det^{-1}, z=\vartheta_1^2\det^{-1}$. Geometrically, $\lfaktor{\Ad \mathbf{A}}{\mathbf{PGL}_2}$ is a circular conical surface. The singular point $x,y,z=0$ corresponds to the $\Ad \mathbf{A}$ orbit of $\begin{pmatrix*} 0 & 1 \\ 1 & 0\end{pmatrix*}$. 
\end{remark}

\subsection{Descent From \texorpdfstring{$\mathbf{GL}_{2,E}$}{GL2_E} to \texorpdfstring{$\mathbf{B}^\times$}{B*}}
\label{sec:PGL2-AdA-descend}
Recall from \S\ref{sec:B-coordinates} that we have fixed an isomorphism of algebraic groups over $E$
\begin{equation*}
\mathbf{B}^\times_E \simeq \mathbf{GL}_{2,E}
\end{equation*}
such that $\widetilde{\mathbf{T}}$ is identified with $\widetilde{\mathbf{A}}$. As this isomorphism identifies the reduced norm map with the determinant map, the torus $\mathbf{T}^{(1)}_E$ is identified with $\mathbf{A}^{(1)}_E$.

\begin{lem}\label{lem:generators-galois}
Let $g\in\mathbf{B}^\times(\mathbb{Q})\subset\mathbf{GL}_2(E)$ or $g\in\mathbf{B}^\times(\mathbb{Q}_v)\subset\mathbf{GL}_2(E_v)$ for some rational place $v$ then
\begin{align*}
\tensor[^\sigma]{\vartheta_1(g)}{}&=\vartheta_2(g) & \tensor[^\sigma]{\psi(g)}{}&=\psi(g) & \tensor[^\sigma]{{\det}^{-1}(g)}{}&={\det}^{-1}(g)
\end{align*}
\end{lem}
\begin{proof}
This follows from Definition \ref{def:inv-functions} and Propositions \ref{prop:rational-points-global} and \ref{prop:rational-points-local}.
\end{proof}
\begin{prop}\label{prop:GIT-descent}
The image of $g\in\mathbf{G}(\mathbb{Q})$ in $\mathbf{W}(\mathbb{Q})$ is determined by the values of
\begin{equation*}
\vartheta_1^2 \det^{-1}(g),\psi \det^{-1}(g)\in E
\end{equation*}
\end{prop}
\begin{proof}
The universality of the GIT quotients for affine schemes implies that 
\begin{equation}\label{eq:GIT-universal}
\left(\lfaktor{\Ad\mathbf{T}}{\mathbf{G}}\right)_E=\lfaktor{\Ad \mathbf{T}_E}{\mathbf{G}_E}
\simeq\lfaktor{\Ad\mathbf{A}_E}{\mathbf{PGL}_{2,E}}
=\left(\lfaktor{\Ad \mathbf{A}}{\mathbf{PGL}}\right)_E
\end{equation}
where we have the induced isomorphism $\mathbf{G}_E \simeq \mathbf{PGL}_{2,E}$
sending $\mathbf{T}_E$ to the diagonal torus $\mathbf{A}_E$.

The claim follows from \eqref{eq:GIT-universal}, Lemma \ref{lem:generators-galois} and Corollary \ref{cor:PGL2-AdA-inv}.  
\end{proof}

\section{Homogeneous Hecke Sets}\label{sec:Hecke}
In this section we study elementary properties of the possible counterexamples to equidistribution arising in \S\ref{sec:rigidity}. 

\begin{defi}
For any $\xi\in\left(\mathbf{G}\times\mathbf{G}\right)(\mathbb{A})$ we define
$\left[\mathbf{G}^\Delta(\mathbb{A})\xi\right]$
to be a homogeneous Hecke set.
This set carries a $\xi^{-1} \mathbf{G}^\Delta(\mathbb{A}) \xi$-invariant algebraic probability measure. 

We define also $\left[\mathbf{G}^\Delta(\mathbb{A})^+\xi\right]$ to be a simply connected homogeneous Hecke set.
This set carries a $\xi \mathbf{G}^\Delta(\mathbb{A})^+ \xi$-invariant algebraic probability measure.
\end{defi}

\begin{remark}
Fixing the subgroup $\mathbf{G}^\Delta<\mathbf{G}\times \mathbf{G}$ the datum defining a homogeneous Hecke set $\left[\mathbf{G}^\Delta(\mathbb{A})\xi\right]$  is $[\xi]\in \lfaktor{\mathbf{G}^\Delta(\mathbb{A})}{\left(\mathbf{G}\times\mathbf{G}\right)(\mathbb{A})}$. Using the contraction map this can be identified with $\ctr(\xi)\in\mathbf{G}(\mathbb{A})$.

The datum defining a simply connected homogeneous Hecke set $\left[\mathbf{G}^\Delta(\mathbb{A})^+\xi\right]$ for $\xi=(\xi_1,\xi_2)$ is $[\xi]\in \lfaktor{\mathbf{G}^\Delta(\mathbb{A})^+}{\left(\mathbf{G}\times\mathbf{G}\right)(\mathbb{A})}$. Using the contraction map this can be identified with
 $[\xi_1]\in\faktor{\mathbf{G}(\mathbb{A})}{\mathbf{G}(\mathbb{A})^+}$ and $\ctr(\xi)\in\mathbf{G}(\mathbb{A})$.
\end{remark}

Homogeneous Hecke sets generalize the notion of a classical Hecke correspondence. 
An obvious necessary condition for equidistribution is that the joint homogeneous toral sets are not trapped in a sequence of homogeneous Hecke sets with periodic measure converging to a periodic measure on some other fixed homogeneous Hecke set. The goal of this manuscript is to show that this condition is not only necessary but also sufficient, at least under the hypothesis described in the introduction. In this section we translate this condition to a condition on the twist $s\in\mathbf{T}(\mathbb{A})$.

Because these sets are somewhat more general then the classical Hecke correspondences we need to extend some well-known results about Hecke correspondences and present them in a language adapted to the applications dicussed in this manuscript.

In  this section we fix a joint homogeneous toral set $\left[\mathbf{T}^\Delta(g,sg)\right]$ satisfying \eqref{eq:g_adel_restrictions}.

\subsection{Homogeneous Hecke Sets Containing a Joint Homogeneous Toral Set}
\begin{lem}\label{lem:hecke-containing-toral}
All the homogeneous Hecke sets containing $\left[\mathbf{T}^\Delta(\mathbb{A})(g,sg)\right]$ are of the form $\left[\mathbf{G}^\Delta(\mathbb{A})(g,t_{\mathbb{Q}} s g)\right]$ for some $t_{\mathbb{Q}}\in\mathbf{T}(\mathbb{Q})$.
\end{lem}

\begin{remark}
If $\left[\mathbf{T}^\Delta(\mathbb{A})(g,sg)\right]$ satisfies \eqref{eq:g_adel_restrictions} and $\left[\mathbf{T}^\Delta(\mathbb{A})(g,sg)\right]\subset \left[\mathbf{G}^\Delta(\mathbb{A})\xi\right]$ then 
\begin{equation*}
\ctr(\xi_{\infty})\in K_\infty,\,
\ctr(\xi_{p_1})\in A_{p_1},\,
\ctr(\xi_{p_2})\in A_{p_2}
\end{equation*}
\end{remark}

\begin{proof}
Because $\Cent_{\left(\mathbf{G}\times\mathbf{G}\right)(\mathbb{A})}\left(\mathbf{T}^\Delta(\mathbb{A})\right)=\left(\mathbf{T}\times\mathbf{T}\right)(\mathbb{A})$ we have for any $t_{\mathbb{Q}}\in\mathbf{T}(\mathbb{Q})$
\begin{equation*}
\left[\mathbf{T}^\Delta(\mathbb{A})(g,sg)\right]=\left[\mathbf{T}^\Delta(\mathbb{A})(g,t_{\mathbb{Q}}sg)\right]
\subset \left[\mathbf{G}^\Delta(\mathbb{A})(g,t_{\mathbb{Q}} s g)\right]
\end{equation*}

On the other hand
if $\left[\mathbf{T}^\Delta(\mathbb{A})(g,sg)\right]\subset \left[\mathbf{G}^\Delta(\mathbb{A})(g,\xi_0 g)\right]$ for some $\xi_0\in\mathbf{G}(\mathbb{A})$ then a simple calculation shows that
\begin{equation}\label{eq:toral-in-hecke-formula}
\forall t\in\mathbf{T}(\mathbb{A})\colon 
t \xi_0 s^{-1} t^{-1} \in \mathbf{G}(\mathbb{Q})
\end{equation}
In particular,  $\xi_0s^{-1}\in\mathbf{G}(\mathbb{Q})$. If $\xi_0 s^{-1}\in w_\mathbf{T} \mathbf{T}(\mathbb{Q})$, that is it belongs to the non-trivial class of the normalizer of $\mathbf{T}$, then we deduce that $t^2 \xi_0s^{-1}\in\mathbf{G}(\mathbb{Q})$ for all $t\in\mathbf{T}(\mathbb{A})$ which is a contradiction.
Otherwise, Corollary \ref{cor:AdTG(Q)-separation}  implies
\begin{equation*}
\mathbf{T}(\mathbb{A})= \mathbf{T}(\mathbb{Q}) \Stab_{\Ad\mathbf{T}(\mathbb{A})}(\xi_0 s^{-1}) 
\end{equation*}
Considering all the options for the stabilizer in Proposition \ref{prop:M-orbits} we deduce that $\Stab_{\Ad\mathbf{T}(\mathbb{A})}(\xi_0 s^{-1})=\mathbf{T}(\mathbb{A})$ and $\xi_0 s^{-1}\in\mathbf{T}(\mathbb{Q})$. 
\end{proof}

\subsection{Volume of a Homogeneous Hecke Set}
The volume of a homogeneous Hecke set is defined similarly to the volume of a homogeneous toral set
\begin{align*}
\vol\left(\left[\mathbf{G}^\Delta(\mathbb{A})\xi\right]\right)&\coloneqq
\meas_{\mathbf{G}^\Delta}\left(\xi\Omega\times\Omega \xi^{-1}\right)^{-1}\\
&=\meas_\mathbf{G}\left(\Omega\cap  \ctr(\xi)\Omega
\ctr(\xi)^{-1}\right)^{-1}
\end{align*}
where $\meas_\mathbf{G}=\meas_{\mathbf{G}^\Delta}$ is a covolume $1$ Haar measure. The volume of a simply connected Hecke correspondence is defined analogously.

The map $\xi\mapsto \vol\left(\left[\mathbf{G}^\Delta(\mathbb{A})\xi\right]\right)$ is a continuous map from $\left(\mathbf{G}\times\mathbf{G}\right)(\mathbb{A})$ to $\mathbb{R}_{>0}$ which factors through the map $\ctr\colon\left(\mathbf{G}\times\mathbf{G}\right)(\mathbb{A})\to\mathbf{G}(\mathbb{A})$.

\subsubsection{Volume Computation Using the Bruhat-Tits tree}
\begin{defi}\label{def:Denom_sf}
Define the proper continuous function $\Denom_{sf}\colon\mathbf{G}(\mathbb{A}_f)\to\mathbb{N}$ by
\begin{equation*}
\Denom_{\mathit{sf}}(h_f)=\prod_{\infty\neq v \text{ splits }\mathbf{B}} \Denom_v(h_v)
\end{equation*}

Because $1\leq\Denom_v(\xi_v)\leq q_v$ for any finite $v$ where $\mathbf{B}$ ramifies we see that for all $h_f\in\mathbf{G}(\mathbb{A}_f)$
\begin{equation*}
\Denom_{\mathit{sf}}(h_f) \asymp_{\mathbf{G}} \Denom_f(h_f)
\end{equation*}
\end{defi}

\begin{lem}\label{lem:volume-of-hecke}
Let $\xi\in\left(\mathbf{G}\times\mathbf{G}\right)(\mathbb{A})$ with $\ctr(\xi)_\infty\in K_\infty$ then
\begin{equation*}
\vol\left(\left[\mathbf{G}^\Delta(\mathbb{A})\xi\right]\right)\meas_\mathbf{G}\left(\Omega\right)
=\Denom_{\mathit{sf}}(\ctr(\xi)_f) \prod_{p \mid \Denom_{\mathit{sf}}(\ctr(\xi)_f)} \left(1+\frac{1}{p}\right)
\end{equation*}
\end{lem}
\begin{proof}
By definition
\begin{equation*}
\vol\left(\left[\mathbf{G}^\Delta(\mathbb{A})\xi\right]\right)\meas_\mathbf{G}\left(\Omega\right)
=\frac{\meas_\mathbf{G}\left(\Omega\right)}{\meas_\mathbf{G}\left(\Omega\cap \ctr(\xi)\Omega\ctr(\xi)^{-1}\right)}
\end{equation*}
Because $\Omega_\infty$ is $\Ad K_\infty$-invariant and $\ctr(\xi)_\infty\in K_\infty$
we can rewrite the quotient of measures as
\begin{equation}\label{eq:G-Omega-vol-quotient}
\frac{\meas_\mathbf{G}\left(\Omega\right)}{\meas_\mathbf{G}\left(\Omega\cap \ctr(\xi)\Omega\ctr(\xi)^{-1}\right)}
=\prod_{v\neq\infty}\left[K_v\colon K_v\cap \xi_v K_v {\xi_v}^{-1}\right]
\end{equation}

The group $K_v$ is for almost all $v$ the maximal compact subgroup in the restricted product definition of $\mathbf{G}(\mathbb{A})$, hence $\xi_v\in K_v$ for almost all $v$.  If $\mathbf{B}$ is ramified over $\mathbb{Q}_v$ then $K_v<\mathbf{G}(\mathbb{Q}_v)$ is a normal subgroup \S\ref{sec:maximal-order-B}. Hence $\left[K_v\colon K_v\cap \xi_v K_v {\xi_v}^{-1}\right]\neq1$ only if $\xi_v\not\in K_v$ and $\mathbf{B}(\mathbb{Q}_v)$ is split. In particular, the product \eqref{eq:G-Omega-vol-quotient} is finite.

When $\mathbf{G}(\mathbb{Q}_v)$ is split, i.e.\ $\mathbf{G}(\mathbb{Q}_v)\simeq \mathbf{PGL}_2(\mathbb{Q}_v)$, the index $\left[K_v\colon K_v\cap \xi_v K_v {\xi_v}^{-1}\right]$ can be calculated using the Bruhat-Tits building $\mathscr{B}_v$ of $\mathbf{G}(\mathbb{Q}_v)$.

Our conditions in \S\ref{sec:maximal-order-B} imply that $K_v$ is actually the whole stabilizer of $x_0$, in particular it preserves types of vertices.
The subgroup $\xi_v K_v {\xi_v}^{-1}$ is the stabilizer of the vertex $\xi_v.x_0$, thus $K_v \cap \xi_v K_v {\xi_v}^{-1}$ stabilizes the whole geodesic segment connecting $x_0$ to $\xi_v.x_0$. The cosets of $K_v \cap \xi_v K_v {\xi_v}^{-1}$ in $K_v$ are in bijection with the vertices in the $K_v$-orbit of $\xi_v.x_0$. We claim that this orbit is exactly the vertices $y$ such that $d(x_0,y)=d(x_0,\xi_v.x_0)$. It is clear that the orbit is contained in this set as the action of the group on the building is by isometries. 

Fix $y$ such that $d(x_0,y)=d(x_0,\xi_v.x_0)$.
Let $z_1$ be the vertex adjacent to $\xi_v.x_0$ on the geodesic segment connecting $x_0$ and $\xi_v.x_0$ and set $z_2$ to be the vertex adjacent to $y$ on the geodesic segment connecting $x_0$ and $y$.
The edges $(z_1,\xi_v.x_0)$ and $(z_2,y)$ define alcoves in the tree. Let $\mathcal{A}_1$, $\mathcal{A}_2$ be two apartments containing these alcoves and $x_0$.

Because the action of the type-preserving subgroup of $\mathbf{G}(\mathbb{Q}_v)$ on the building is strongly transitive\footnote{The action is transitive on pairs $(\mathscr{C},\mathscr{A})$ of an apartment $\mathscr{A}$ and an alcove $\mathscr{C}\subset\mathscr{A}$.} there is an element of $\mathbf{G}(\mathbb{Q}_v)$ sending   $\mathcal{A}_1$ to $\mathcal{A}_2$ and $(z_1,\xi_v.x_0)$ to $(z_2,y)$. Such an element must stabilize $x_0$ and send $\xi_v.x_0$ to $y$, hence $y$ is in the $K_v$-orbit of $\xi_v.x_0$ as required.

By counting vertices of distance $d(x_0,\xi_v.x_0)$ from $x_0$ in a $q_v+1$ regular tree we see that if $d(x_0,\xi_v.x_0)>0$
\begin{equation*}
\left[K_v\colon K_v\cap \xi_v K_v {\xi_v}^{-1}\right]=(q_v+1)q_v^{d(x_0,\xi_v.x_0)-1}=q_v^{d(x_0,\xi_v.x_0)}(1+\frac{1}{q_v})
\end{equation*} 
\end{proof}

\subsection{Equivalence of Necessary Conditions for Equidistribution}
\begin{lem}\label{lem:Denom-minimal-ideal-norm}
Let $\tau\in\mathbf{T}(\mathbb{A})$ then $\Denom_f(g_f^{-1}\tau_f g_f)$ is the minimal norm of an \emph{integral} fraction ideal in the homothety class $\idl(\tau)\in\lfaktor{\mathbb{Q}^\times}{\Ideals(\Lambda)}$.  
\end{lem}
\begin{proof}
To show the equality between these two positive integers we show that their $p$-parts are equal for all primes $p$. The representatives of $\tau$ in $\mathbf{B}(\mathbb{A})$ can be written in coordinates as
\begin{equation*}
\tau=\left(\mathbb{Q}_v^\times
\begin{pmatrix}
\alpha_v & 0 \\
0 & \tensor[^\sigma]{\alpha}{_v}
\end{pmatrix}
\right)_v
\end{equation*}

Fix $v\neq\infty$.
A representative $r_v$ of $g_v^{-1}\tau_v g_v$ is contained in $\Omega_v$ if and only if $g_v r_v g_v^{-1}$ is contained in $g_v\mathbb{O}_v g_v^{-1}$. Hence by Proposition \ref{prop:Omega_xi_Omega-reduced-norm} $\Denom_v(g_v^{-1}\tau_v g_v)$ has the same valuation as the minimal reduced norm of a representative of $\tau_v$ contained in $g_v\mathbb{O}_v g_v^{-1}\cap \mathbf{E}(\mathbb{Q}_v)$. The latter set is by the definition of the local order equal to
\begin{equation*}
g_v\mathbb{O}_v g_v^{-1}\cap \mathbf{E}(\mathbb{Q}_v)=\left\{
\begin{pmatrix}
\lambda_v & 0 \\
0 & \tensor[^\sigma]{\lambda}{_v}
\end{pmatrix} \mid \lambda_v\in \Lambda_v
\right\}
\end{equation*}
We deduce that $\ord_v\Denom_v(g_v^{-1}\tau_v g_v)=\ord_v \Nr(q_v\alpha_v)$ where $q_v\in\mathbb{Q}_v^\times$ is an element of minimal valuation satisfying $q_v\alpha_v\in\Lambda_v$. 

Set $q\in\mathbb{Q}^\times$ so that $q \mathbb{Q}=\bigcap_{v\neq\infty} q_v \mathbb{Q}_v$ then by definition $q\widetilde{\idl}\left((\alpha_v)_{v\neq\infty}\right)$ is the minimal integral element in the homothety class $\idl(\tau)$ and its norm has the same valuation for all primes $p$ as $\Denom_f(g_f^{-1}\tau_f g_f)$.
\end{proof}

\begin{prop}\label{prop:equidist-conditions-equivalence}
Let $\left\{\left[\mathbf{T}_i^\Delta(\mathbb{A})(g_i,s_ig_i)\right]\right\}_i$ be a sequence of joint homogeneous toral set with associated global orders $\Lambda_i$. Denote $\mathfrak{s}_i=\idl(s_{i,f})\in \lfaktor{\mathbb{Q}^\times}{\Ideals(\Lambda_i)}$ and let $[\mathfrak{s}_i]$ be the class of $\mathfrak{s}_i$ in $\Pic(\Lambda_i)$.

The following are equivalent
\begin{enumerate}
\item 
\begin{equation*}
\min_{\left[\mathbf{T}_i^\Delta(\mathbb{A})(g_i,s_ig_i)\right]\subset \left[\mathbf{G}^\Delta(\mathbb{A})\xi\right]} \vol\left( \left[\mathbf{G}^\Delta(\mathbb{A})\xi\right]\right) \to_{i\to\infty} \infty
\end{equation*}

\item 
\begin{equation*}
\min_{\substack{\mathfrak{a} \subseteq \Lambda \\ [\mathfrak{a}]=[\mathfrak{s}_i]}} \Nr\mathfrak{a}\to_{i\to\infty} \infty
\end{equation*}

\item For every compact set $B\subset\mathbf{G}(\mathbb{A})$ there is $N\in\mathbb{N}$ such that for all $i> N$
\begin{equation*}
g_i^{-1} \mathbf{T}_i(\mathbb{Q}) s_i g_i \cap B =\emptyset
\end{equation*}
\end{enumerate}
\end{prop}
\begin{proof}
The equivalence of (2) and (3) is a consequence of Lemma \ref{lem:Denom-minimal-ideal-norm} above and the fact that the function $h\mapsto \Denom_f(h_f)$ is a continuous \emph{proper} function from $K_\infty\times \mathbf{G}(\mathbb{A}_f)$ to $\mathbb{N}$.

The equivalence of (1) and (2) follows from Lemmata \ref{lem:hecke-containing-toral}, \ref{lem:volume-of-hecke}, \ref{lem:Denom-minimal-ideal-norm}, the remark in Definition \ref{def:Denom_sf} and the fact that for all $N\in\mathbb{N}$
\begin{equation*}
1\leq \prod_{p\mid N} \left(1+\frac{1}{p}\right)\ll \log\log N
\end{equation*}
which follows from the PNT.
\end{proof}

\section{Geometric Expansion of the Pair Cross-Correlation}\label{sec:cross-correlation}
Throughout this section we fix a joint homogeneous toral set $[\mathbf{T}^\Delta(\mathbb{A})(g,sg)]$ with periodic measure $\mu$ and a simply connected homogeneous Hecke set $[\mathbf{G}^\Delta(\mathbb{A})^+\xi]$ with periodic measure $\nu$. We also write $\xi=(\xi_1,\xi_2)$.

\subsection{Pair Cross-Correlation}
We define the pair cross-correlation between the periodic measure $\mu$ and the periodic measure $\nu$. 
\begin{defi}\label{def:automorphic-kernel}
Let $V\subseteq\left[\left(\mathbf{G}\times \mathbf{G}\right)(\mathbb{A})\right]$ be a compact identity neighborhood. Define the automorphic kernel $K_V\colon\left[\left(\mathbf{G}\times \mathbf{G}\right)(\mathbb{A})\right]^{\times 2}\to\mathbb{R}$
\begin{equation*}
K_V(x,y)=\sum_{\gamma\in\left(\mathbf{G}\times\mathbf{G}\right)(\mathbb{Q})} \mathbb{1}_V(x^{-1}\gamma y)
\end{equation*}
\end{defi}

As $V$ is compact the sum on the right is finite for every $x$ and $y$. Moreover, the number of non-trivial summands is uniformly bounded when $x$ and $y$ are restricted to fixed compact subsets. In particular the convergence is uniform on compact sets.

\begin{defi}\label{def:cross-correlation}
Let $B=\prod_v B_v\subseteq\mathbf{G}(\mathbb{A})$  be a compact identity neighborhood  with $B_v=K_v$ for almost all $v$ and $B_v$ a compact-open subgroup for all $v$ non-archimedean. 
Let $\lambda_1$, $\lambda_2$ be probability measures on $\left[\left(\mathbf{G}\times \mathbf{G}\right)(\mathbb{A})\right]$ and $K_{B\times B}$ as in Definition \ref{def:automorphic-kernel}.

For a fixed closed subset $C\subseteq[\mathbf{G}(\mathbb{A})]$ we define
\begin{equation*}
\Cor_{C}[\lambda_1,\lambda_2](B)\coloneqq\int_{C\times C} \int_{C\times C} K_{B\times B}(x,y)\dif\lambda_1(x)\dif\lambda_2(y)
\end{equation*}
We also write $\Cor[\lambda_1,\lambda_2](B)=\Cor_{Y_{\mathbb{A}}}[\lambda_1,\lambda_2](B)$. 

We say that $B$ is injective on $C$ if the quotient map $\mathbf{G}(\mathbb{A})\to [\mathbf{G}(\mathbb{A})]$ is injective when restricted to $gB$ for any $g\in\mathbf{G}(\mathbb{A})$ such that $[g]\in C$. When $C$ is compact there is alway an identity neighborhood injective on $C$.
\end{defi}

\begin{lem}
In the setting of Definition \ref{def:cross-correlation}
We always have
\begin{equation*}
\lambda_1\times\lambda_2\left(x,y\in C\times C \mid y\in xB \right) \leq 
\Cor_{C}[\lambda_1\times\lambda_2](B)
\end{equation*}
with equality if $B$ is injective on $C$.
\end{lem}
\begin{proof}
Follows directly from the definitions.
\end{proof}

\subsection{Main Theorem about Cross-Correlation}
The main result in this section and the main structural result in this manuscript is Theorem \ref{thm:cross-correlation-shifted-convolution} below to be proved in \S\ref{sec:proof-of-geometric-expansion}. First we need a few definitions.
\begin{defi}\label{defi:cross-correlation-constants}
We introduce the following notations for a fixed joint homogeneous toral set $\left[\mathbf{T}^\Delta(\mathbb{A})(g,sg)\right]$.

\begin{enumerate}
\item The twist $s\in\mathbf{T}(\mathbb{A})$ defines a homothety class of invertible fractional $\Lambda$-ideals 
\begin{equation*}
\mathbb{Q}^\times\mathfrak{s}\coloneqq \idl(s)\in \lfaktor{\mathbb{Q}^\times}{\Ideals(\Lambda)}
\end{equation*}

\item Define the following invertible fractional $\Lambda$-ideal which encapsulates the splitting behaviour of $\mathbf{B}$ outside of $\infty$. 
\begin{equation*}
\mathfrak{e}\coloneqq \widetilde{\idl}\left((\upsilon_v \tau_v)_v\right)
\end{equation*}
where $\upsilon_v,\tau_v$ are as in Proposition \ref{prop:local-order-general}.

\item We also need the following integer which is also closely related to the splitting of $\mathbf{B}$
\begin{equation*}
\upsilon\coloneqq \sign(\epsilon) \prod_{\substack{\mathbf{B}\textrm{ is ramified and }\\ E \textrm{ is inert at } p}} p
\end{equation*}
\end{enumerate}
\end{defi}

\begin{remark}
In the simplest case when $\mathbf{G}\simeq\mathbf{PGL}_2$ is split we can choose $\epsilon=1$ and then $\upsilon=1$ and $\mathfrak{e}=\Lambda$. 
\end{remark}

\begin{defi}
For any $[\mathfrak{g}]\in \Pic(\Lambda)$
define the arithmetic functions $f_{[\mathfrak{g}]},g_{[\mathfrak{g}]}\colon\mathbb{Z} \to \mathbb{Z}$ as follows
\begin{align*}
g_{[\mathfrak{g}]}(x)&=\#\left\{\mathfrak{a}\in \Ideals(\Lambda)_0 \mid \Nr(\mathfrak{a})=x,\,
\mathfrak{a}\subseteq\Lambda,\,
[\mathfrak{a}]= [\mathfrak{g}] \textrm{ or } \mathfrak{a}=0
\right\}\\
f_{[\mathfrak{g}]}(x)&=\#\left\{\mathfrak{b}\in \Ideals(\Lambda) \mid \Nr(\mathfrak{b})=x,\,
\mathfrak{b}\subseteq\Lambda,\,
[\mathfrak{b}]\in [\mathfrak{g}]\Pic(\Lambda)^2
\right\}
\end{align*}

Define also the multiplicative function $r\colon\mathbb{N}\to\mathbb{N}$ by requiring that for any \emph{odd} prime $p\mid D$ if $\mathbf{B}$ splits at $p$ then
\begin{equation*}
r(p^k)=\begin{cases}
1 & k <\ord_p D\\
2 & k \geq \ord_p D
\end{cases}
\end{equation*}
If $2\mid D$ we set $r(2^k)=2^{\mu_{\mathrm{wild}}}$, where $\mu_\mathrm{wild}\in\{0,1,2,3\}$ is defined in Corollary \ref{cor:Pic(Lambda)[2]-size}. If $p\mid D$ and $\mathbf{B}$ ramifies at $p$ we define $r(p^k)=2$. For all primes $p\not\mid D$ we set $r(p^k)=1$.
\end{defi}

\begin{thm}\label{thm:cross-correlation-shifted-convolution}
Fix a joint homogeneous toral set $\left[\mathbf{T}^\Delta(\mathbb{A})(g,sg)\right]$ with splitting field $E/\mathbb{Q}$ and quadratic order $\Lambda\leq\mathcal{O}_E$ of discriminant $D$. Assume \eqref{eq:g_adel_restrictions} is satisfied.

Let $B=\prod_v B_v\subset \mathbf{G}(\mathbb{A})$ with $B_v=\Omega_v$ for all $v\neq p_1$ and $B_{p_1}=K_{p_1}^{(-n,n)}$ for some $n\in\mathbb{N}$.
Fix also a simply connected homogeneous Hecke set $\left[\mathbf{G}^\Delta(\mathbb{A})^+\xi\right]$ with $\ctr(\xi)_{p_1}\in A_{p_1}$ and assume 
\begin{equation*}
g^{-1}\mathbf{T}(\mathbb{Q})s g \cap B\ctr(\xi)B=\emptyset
\end{equation*}

Let $\mu$ be the algebraic probability measure supported on $\left[\mathbf{T}^\Delta(\mathbb{A})(g,sg)\right]$ and let $\nu$ be the algebraic probability measure supported on $\left[\mathbf{G}^\Delta(\mathbb{A})^+\xi\right]$. Denote $\kappa=2^8 \Denom_\infty(\ctr(\xi)_\infty)\Denom_f(\ctr(\xi)_f)$ and $\omega=-\sign(\Nrd(\ctr(\xi)_\infty))\Denom_f(\ctr(\xi)_f)$ then
\begin{align*}
\Cor[\mu,\nu](B)&\ll
\vol\left(\left[\mathbf{T}(\mathbb{A})g\right]\right)^{-1} 
\vol\left(\left[\mathbf{G}^\Delta(\mathbb{A})^+\xi\right]\right)^{-1} p_1^{-2n}\\
&\sum_{\substack{0\leq x\leq \kappa |D| \\ x\equiv \omega D \mod \upsilon p_1^{2n}}} g_{[\mathfrak{s}]}(x)
f_{[p_1^n\mathfrak{se}]^{-1}}\left(\frac{x-\omega D}{\upsilon p_1^{2n}}\right) r\left(\frac{x-\omega D}{\upsilon p_1^{2n}}\right)
\end{align*}
\end{thm}
\begin{remark}
Notice that $\upsilon$ is supported on primes that are inert in $E/\mathbb{Q}$ while $p_1$ splits; thus $\gcd(\upsilon,p_1^{2n})=1$.
\end{remark}

\subsection{Geometric Expansion}
\begin{defi}
Set
\begin{equation*}
W_{\mathbb{Q}}=\dfaktor{\mathbf{G}^\Delta(\mathbb{Q})}{\left(\mathbf{G} \times \mathbf{G}\right)(\mathbb{Q})}{\mathbf{T}^\Delta(\mathbb{Q})}
\end{equation*}
We denote by $[\gamma]\in W_{\mathbb{Q}})$ the double coset corresponding to $\gamma\in\left(\mathbf{G} \times \mathbf{G}\right)(\mathbb{Q})$. 

We have a natural map $W_{\mathbb{Q}}\to\mathbf{W}(\mathbb{Q})$ where $\mathbf{W}$ is the GIT quotient defined in \S \ref{sec:GIT}. Recall from Proposition \ref{prop:M-orbits} that this map is injective outside of $\left\{[(\gamma_0,\gamma_0 w_\mathbf{T} t_{\mathbb{Q}})] \mid \gamma_0\in \mathbf{G}(\mathbb{Q}),\; t_{\mathbb{Q}}\in \mathbf{T}(\mathbb{Q}) \right\}$.
\end{defi} 

\begin{defi}
For any closed subgroup $N<\mathbf{M}(\mathbb{A})$ denote 
\begin{equation*}
N^\dagger\coloneqq N\cap\left(\mathbf{G}(\mathbb{A})^+\times\mathbf{T}(\mathbb{A})\right)
\end{equation*} 
The subgroup $N^\dagger$ is always normal in $N$.
\end{defi}

The following proposition is the geometric expansion of the relative trace corresponding to the subgroups $\mathbf{G}^\Delta$ and $\mathbf{T}^\Delta$ of $\mathbf{G}\times\mathbf{G}$. The situation is relatively simple as the stabilizers have finite volume adelic quotients.
\begin{prop}\label{prop:geometric-expansion}
Let $\mu$ be the periodic measure on a joint homogeneous toral set $[\mathbf{T}^\Delta(\mathbb{A})(g,sg)]$ and $\nu$ the periodic measure on a simply-connected Hecke correspondence $[\mathbf{G}^\Delta(\mathbb{A})^+\xi]$, $\xi=(\xi_1, \xi_2)$. Set $B'=\xi_1 B g^{-1}\times \xi_2 B g^{-1} s^{-1}$  
 then
\begin{align*}
\Cor[\mu,\nu](B)&=\int_{[\mathbf{G}(\mathbb{A})^+]} \int_{[\mathbf{T}(\mathbb{A})]} K_{B'}(l,t)\dif l \dif t\\
&=\sum_{[\gamma]\in W_{\mathbb{Q}}} \sum_{\varkappa\in
{\pi_{\mathbf{G}}\left(\mathbf{M}_\gamma(\mathbb{Q})\right)} \backslash {\mathbf{G}(\mathbb{Q})} \slash {\mathbf{G}(\mathbb{Q})^+}} 
\vol(\mathbf{M}_\gamma)
\cdot
\RO_{\gamma,\varkappa}(B)\\
\RO_{\gamma,\varkappa}(B)&\coloneqq\int_{\mathbf{M}_\gamma(\mathbb{A})^\dagger \backslash \mathbf{M}(\mathbb{A})^\dagger}
\mathbb{1}_{B'}\left((\varkappa l)^{-1} \gamma t\right) 
\dif(l,t)\\
\vol(\mathbf{M}_\gamma)&\coloneqq
\meas_{\mathbf{M}_\gamma(\mathbb{A})^\dagger}\left(\lfaktor{\mathbf{M}_\gamma(\mathbb{Q})^\dagger}{\mathbf{M}_\gamma(\mathbb{A})^\dagger}\right)
\end{align*}
where the Haar measures on $\lfaktor{\mathbf{M}_\gamma(\mathbb{A})^\dagger}{\mathbf{M}(\mathbb{A})^\dagger}$ and $\mathbf{M}_\gamma(\mathbb{A})^\dagger$ are mutually normalized.

Following the relative trace formula terminology we call $\RO_{\gamma,\varkappa}(B)$ a relative orbital integral.
\end{prop}

We will use the following lemma in the proof of the proposition.
\begin{lem}\label{lem:Mdag-Haar-invariant}
For any $a\in\mathbf{M}(\mathbb{A})$  let $\Ad_a\colon\mathbf{M}(\mathbb{A})^\dagger\to\mathbf{M}(\mathbb{A})^\dagger$ be the conjugation automorphism of the normal subgroup $\mathbf{M}(\mathbb{A})^\dagger$. Then the map $\Ad_a$ fixes any Haar measure on $\mathbf{M}(\mathbb{A})^\dagger$.
\end{lem}
\begin{proof}
Let $\meas_{\mathbf{M}(\mathbb{A})^\dagger}$ be a Haar measure on $\mathbf{M}(\mathbb{A})^\dagger$, then $\left(\Ad_a\right)_* \meas_{\mathbf{M}(\mathbb{A})^\dagger}$ is a Haar measure as well and proportional to the original one $\left(\Ad_a\right)_* \meas_{\mathbf{M}(\mathbb{A})^\dagger}=\alpha(a) \meas_{\mathbf{M}(\mathbb{A})^\dagger}$. The map $\alpha\colon \mathbf{M}(\mathbb{A})\to \mathbb{R}_{>0}$ is a character which is trivial on $\mathbf{M}(\mathbb{A})^\dagger$, hence it factors through the $2$-torsion group $\lfaktor{\mathbf{M}(\mathbb{A})^\dagger}{\mathbf{M}(\mathbb{A})}$. Because $\mathbb{R}_{>0}$ has no non-trivial torsion elements $\alpha$ is trivial.
\end{proof}

\begin{proof}[Proof of Proposition \ref{prop:geometric-expansion}]
Let $[\gamma]\in W_{\mathbb{Q}}$ be a double coset with representative $\gamma\in\left(\mathbf{G}\times\mathbf{G}\right)(\mathbb{Q})$. Denote $f\coloneqq\mathbb{1}_{B'}$.
We unfold the definition of the cross-correlation and exchange summation and integration using the uniform convergence of the kernel on compact subsets
\begin{align}
\Cor[\mu,\nu](B)&=
\int_{[\mathbf{M}(\mathbb{A})^\dagger]} \sum_{\gamma \in \left(\mathbf{G} \times \mathbf{G}\right)(\mathbb{Q})} f(l^{-1}\gamma t) \dif(l,t) 
\nonumber\\
&=\sum_{[\gamma]\in W_{\mathbb{Q}}} \sum_{\gamma' \in [\gamma]} \int_{[\mathbf{M}(\mathbb{A})^\dagger]} f(l^{-1}\gamma' t) \dif(l,t)
\nonumber\\
&=\sum_{[\gamma]\in W_{\mathbb{Q}}} \sum_{\gamma' \in [\gamma]} \int_{[\mathbf{M}(\mathbb{A})^\dagger]} f(m^{-1}.\gamma' ) \dif m 
\label{eq:geometric-expansion-1}
\end{align}

We now deal individually with each internal sum for $[\gamma]$ fixed. Let $\mathcal{F}\subset\mathbf{M}(\mathbb{A})^\dagger$ be a fundamental domain for the left action of $\mathbf{M}(\mathbb{Q})^\dagger$ on $\mathbf{M}(\mathbb{A})^\dagger$. 
We write the internal sum in \eqref{eq:geometric-expansion-1} as a sum of integrals on $\mathcal{F}$. To do this we choose some fixed representatives for each set of cosets appearing in the following.
\begin{align}
&\sum_{\gamma' \in [\gamma]} \int_{[\mathbf{M}(\mathbb{A})^\dagger]} 
 f(m^{-1}.\gamma') \dif m
= \sum_{m_{\mathbb{Q}}\in\mathbf{M}_\gamma(\mathbb{Q}) \backslash \mathbf{M}(\mathbb{Q})}
 \int_{\mathcal{F}}  f\left(m^{-1}m_{\mathbb{Q}}^{-1}.\gamma \right) \dif m \nonumber\\
&=\sum_{\bar{\varkappa}\in\mathbf{M}_\gamma(\mathbb{Q}) \backslash \mathbf{M}(\mathbb{Q})  \slash \mathbf{M}(\mathbb{Q})^\dagger}
\sum_{m_{\mathbb{Q}} \in \bar{\varkappa}^{-1} \mathbf{M}_\gamma(\mathbb{Q})^\dagger \bar{\varkappa} \backslash  \mathbf{M}(\mathbb{Q})^\dagger}
\int_{\mathcal{F}}  f\left((m^{-1}m_{\mathbb{Q}}^{-1} \bar{\varkappa}^{-1})  .\gamma\right) \dif m
\nonumber \\
&=
\sum_{\bar{\varkappa}\in\mathbf{M}_\gamma(\mathbb{Q}) \backslash \mathbf{M}(\mathbb{Q})  \slash \mathbf{M}(\mathbb{Q})^\dagger}
\sum_{m_{\mathbb{Q}} \in \bar{\varkappa}^{-1} \mathbf{M}_\gamma(\mathbb{Q})^\dagger \bar{\varkappa} \backslash  \mathbf{M}(\mathbb{Q})^\dagger}
\int_{m_{\mathbb{Q}}\mathcal{F}}  
 f\left((m^{-1} \bar{\varkappa}^{-1}).\gamma\right) \dif m
\label{eq:geometric-expansion-2}
\end{align}

Fix now a representative $\bar{\varkappa}\in \dfaktor{\mathbf{M}_\gamma(\mathbb{Q})} {\mathbf{M}(\mathbb{Q})} {\mathbf{M}(\mathbb{Q})^\dagger}$.
The function $f\left((m^{-1} \bar{\varkappa}^{-1}).\gamma\right)$ is a well-defined compactly supported integrable function on  $\lfaktor{\bar{\varkappa}^{-1} \mathbf{M}_\gamma(\mathbb{A})^\dagger\bar{\varkappa}}{\mathbf{M}(\mathbb{A})^\dagger}$. 

Using the mutual normalization of Haar measures we can rewrite the inner sum in \eqref{eq:geometric-expansion-2} as
\begin{equation}
\label{eq:geometric-expansion-3}
\int_{\bar{\varkappa}^{-1} \mathbf{M}_\gamma(\mathbb{A})^\dagger\bar{\varkappa} \backslash \mathbf{M}(\mathbb{A})^\dagger}\  f\left((m^{-1} \bar{\varkappa}^{-1}).\gamma\right) \dif m 
\cdot 
\sum_{m_{\mathbb{Q}} \in \bar{\varkappa}^{-1} \mathbf{M}_\gamma(\mathbb{Q})^\dagger \bar{\varkappa} \backslash  \mathbf{M}(\mathbb{Q})^\dagger}
\meas_{\bar{\varkappa}^{-1} \mathbf{M}_\gamma(\mathbb{A})^\dagger\bar{\varkappa}}\left(m_{\mathbb{Q}}\mathcal{F}\right)
\end{equation}
For a fixed Haar measure $\meas_{\mathbf{M}_\gamma(\mathbb{A})^\dagger}$ on $\mathbf{M}_\gamma(\mathbb{A})^\dagger$ define a Haar measure on $\bar{\varkappa}^{-1} \mathbf{M}_\gamma(\mathbb{A})^\dagger\bar{\varkappa}$ by $\left(\Ad_{\bar{\varkappa}^{-1}}\right)_*\meas_{\mathbf{M}_\gamma(\mathbb{A})^\dagger}$. Using this normalization we have
\begin{equation}\label{eq:fund-domain-sum}
\sum_{m_{\mathbb{Q}} \in \bar{\varkappa}^{-1} \mathbf{M}_\gamma(\mathbb{Q})^\dagger \bar{\varkappa} \backslash  \mathbf{M}(\mathbb{Q})^\dagger}
\meas_{\bar{\varkappa}^{-1} \mathbf{M}_\gamma(\mathbb{A})^\dagger\bar{\varkappa}}\left(m_{\mathbb{Q}}\mathcal{F}\right)
=
\sum_{m_{\mathbb{Q}} \in \mathbf{M}_\gamma(\mathbb{Q})^\dagger \backslash  \mathbf{M}(\mathbb{Q})^\dagger}
\meas_{\mathbf{M}_\gamma(\mathbb{A})^\dagger}\left(m_{\mathbb{Q}}\mathcal{F}\right)
\end{equation}
The set $\bigsqcup_{m_{\mathbb{Q}} \in \mathbf{M}_\gamma(\mathbb{Q})^\dagger \backslash  \mathbf{M}(\mathbb{Q})^\dagger} m_{\mathbb{Q}}\mathcal{F}$ is a fundamental domain  for the left action of $\mathbf{M}_\gamma(\mathbb{Q})^\dagger$ on $\mathbf{M}(\mathbb{A})^\dagger$; hence the sum \eqref{eq:fund-domain-sum} is equal to
$\meas_{\mathbf{M}_\gamma(\mathbb{A})^\dagger} \left(\lfaktor{ \mathbf{M}_\gamma(\mathbb{Q})^\dagger}{\mathbf{M}_\gamma(\mathbb{A})^\dagger}\right)$.

Under the normalization of Haar measures as above there is an isomorphism of the following measure spaces equipped with their respective Haar measures
\begin{align*}
\lfaktor{\bar{\varkappa}^{-1} \mathbf{M}_\gamma(\mathbb{A})^\dagger\bar{\varkappa}}{\mathbf{M}(\mathbb{A})^\dagger} 
&\simeq\lfaktor{ \mathbf{M}_\gamma(\mathbb{A})^\dagger}{\mathbf{M}(\mathbb{A})^\dagger}\\
\left(\bar{\varkappa}^{-1} \mathbf{M}_\gamma(\mathbb{A})^\dagger\bar{\varkappa}\right) m
&\mapsto \left(\mathbf{M}_\gamma(\mathbb{A})^\dagger\right) \bar{\varkappa} m \bar{\varkappa} ^{-1}
\end{align*} 
This implies 
\begin{align*}
\int_{\bar{\varkappa}^{-1} \mathbf{M}_\gamma(\mathbb{A})^\dagger\bar{\varkappa} \backslash \mathbf{M}(\mathbb{A})^\dagger}\  f\left((m^{-1} \bar{\varkappa}^{-1}).\gamma\right) \dif m 
&=
\int_{\mathbf{M}_\gamma(\mathbb{A})^\dagger \backslash \mathbf{M}(\mathbb{A})^\dagger}\  f\left((\bar{\varkappa}^{-1}m^{-1}).\gamma\right) \dif m \\
&=
\int_{\mathbf{M}_\gamma(\mathbb{A})^\dagger \backslash \mathbf{M}(\mathbb{A})^\dagger}\  f\left((m^{-1} \bar{\varkappa}^{-1}).\gamma\right) \dif m 
\end{align*}
where the second and third lines are equal by Lemma \ref{lem:Mdag-Haar-invariant}.

Combining all of the above and using the following bijection induced by the projection map $\pi_{\mathbf{G}}\colon\mathbf{M}=\mathbf{G}\times\mathbf{T}\to\mathbf{G}$
\begin{equation*}
\dfaktor{\mathbf{M}_\gamma(\mathbb{Q})}{\mathbf{M}(\mathbb{Q})}{\mathbf{M}(\mathbb{Q})^\dagger}\simeq \dfaktor{\pi_{\mathbf{G}}\left(\mathbf{M}_\gamma(\mathbb{Q})\right)}{\mathbf{G}(\mathbb{Q})}{\mathbf{G}(\mathbb{Q})^+}
\end{equation*}  
we arrive to the required final form.
\end{proof}

\begin{lem}\label{lem:RO-compact}
Fix $\gamma\in\left(\mathbf{G}\times\mathbf{G}\right)(\mathbb{Q})$ with $\mathbf{M}_{\gamma}(\mathbb{A})$ compact. Then under a suitable normalization of measures
\begin{equation*}
\RO_{\gamma,\varkappa}(B)\coloneqq\int_{\mathbf{M}(\mathbb{A})^\dagger}
\mathbb{1}_{B'}\left((\varkappa l)^{-1} \gamma t\right) 
\dif(l,t)
\end{equation*}
and
\begin{equation*}
\sum_{\varkappa\in
{\pi_{\mathbf{G}}\left(\mathbf{M}_\gamma(\mathbb{Q})\right)} \backslash {\mathbf{G}(\mathbb{Q})} \slash {\mathbf{G}(\mathbb{Q})^+}} 
\vol(\mathbf{M}_\gamma)
\cdot
\RO_{\gamma,\varkappa}(B)=
\frac{1}{\#\mathbf{M_\gamma}(\mathbb{Q})}
\sum_{\varkappa\in {\mathbf{G}(\mathbb{Q})} \slash {\mathbf{G}(\mathbb{Q})^+}} 
\RO_{\gamma,\varkappa}(B)
\end{equation*}
\end{lem}
Notice that the group $\mathbf{M_\gamma}(\mathbb{Q})$ is a discrete subgroup of a compact group hence it is finite.

This proposition shows that the case of a compact stabilizer is very similar to that of a trivial one, the only difference being the easy to compute factor.
\begin{proof}
The group $\mathbf{M}_\gamma(\mathbb{A})^\dagger$ is a closed subgroup of a compact group, hence it is compact. We normalize the Haar measure on $\mathbf{M}_\gamma(\mathbb{A})^\dagger$ so that it is equal to $1$. This normalization results in $\RO_{\gamma,\varkappa}(B)$ being equal to the integral above over $\mathbf{M}(\mathbb{A})^\dagger$. In this normalization we also have $\vol(\mathbf{M}_\gamma)=\left(\#\mathbf{M}_\gamma(\mathbb{Q})^\dagger\right)^{-1}$. 

When summing over $\varkappa\in {\mathbf{G}(\mathbb{Q})} \slash {\mathbf{G}(\mathbb{Q})^+}$ instead of $\varkappa\in {\pi_{\mathbf{G}}\left(\mathbf{M}_\gamma(\mathbb{Q})\right)} \backslash {\mathbf{G}(\mathbb{Q})} \slash {\mathbf{G}(\mathbb{Q})^+}$ 
the same summand appear multiple times and needs to be accounted for. The multiplicity of a summand is the size of the corresponding fiber in ${\mathbf{G}(\mathbb{Q})} \slash {\mathbf{G}(\mathbb{Q})^+} \to {\pi_{\mathbf{G}}\left(\mathbf{M}_\gamma(\mathbb{Q})\right)} \backslash {\mathbf{G}(\mathbb{Q})} \slash {\mathbf{G}(\mathbb{Q})^+}$. As $\mathbf{M}(\mathbb{A})^\dagger<\mathbf{M}(\mathbb{A})$ is normal all the fibers have the same size which is
\begin{equation*}
\left[\mathbf{M}_\gamma(\mathbb{Q}) : \mathbf{M}_\gamma(\mathbb{Q})^\dagger\right]
\end{equation*}
Finally, the correct proportionality factor between the two sums in the claim is 
\begin{equation*}
\vol(\mathbf{M}_\gamma) 
\left[\mathbf{M}_\gamma(\mathbb{Q}) : \mathbf{M}_\gamma(\mathbb{Q})^\dagger\right]^{-1}
=\left(\#\mathbf{M_\gamma}(\mathbb{Q})\right)^{-1}
\end{equation*}
\end{proof}

\subsection{Reduction to Compact Stabilizers}
Recall from Proposition \ref{prop:equidist-conditions-equivalence} that the minimal volume of a homogeneous Hecke set containing a joint homogeneous toral set depends on the distance of the discrete orbit $g^{-1} \mathbf{T}(\mathbb{Q})s g$ from the identity. In particular, for a sequence of joint homogeneous toral sets we need to assume that for every compact subset $B_0\subset \mathbf{G}(\mathbb{A})$ the orbit $g^{-1} \mathbf{T}(\mathbb{Q})s g$ does not intersect $B_0$ for all joint homogeneous toral sets with discriminant large enough.

In this section we show that for a fixed simply connected homogeneous Hecke set the assumption above implies that the contribution to the cross correlation from terms with a non-compact stabilizer vanishes. This is the fundamental application of this assumption.

We will use once more the fact that the shift $g^{-1} \mathbf{T}(\mathbb{Q})s g$ is large when bounding the pertinent shifted convolution sum.

Whenever the shift has a small representative the cross-correlation with some Hecke correspondence will have terms with non-compact stabilizers and these will be the dominant contribution to the cross-correlation. The simplest bad case is the cross-correlation between a periodic joint toral measure and a Hecke correspondence containing its support.

\begin{lem}\label{lem:reduction-to-compact-stab}
Assume that
\begin{equation}\label{eq:big-shift-vs-hecke}
g^{-1}\mathbf{T}(\mathbb{Q})s g\cap B^{-1}\ctr(\xi) B=\emptyset
\end{equation}

Then for all $\gamma\in\left(\mathbf{G}\times\mathbf{G}\right)(\mathbb{Q})$  if $\mathbf{M}_\gamma(\mathbb{A})^\dagger$ is not compact then $\RO_{\gamma,\varkappa}(B)=0$ for all $$\varkappa\in\dfaktor{\pi_\mathbf{G}(\mathbf{M}_\gamma(\mathbb{Q}))}{\mathbf{G}(\mathbb{Q})}{\mathbf{G}(\mathbb{Q})^+}$$
\end{lem}
\begin{proof}
Write $\xi=(\xi_1,\xi_2)$.
Assume $\mathbf{M}_\gamma(\mathbb{A})^\dagger$ is not compact then according to Proposition \ref{prop:M-orbits} $\gamma=(\gamma_1,\gamma_1 t_{\mathbb{Q}})$ for some $t_{\mathbb{Q}}\in\mathbf{T}(\mathbb{Q})$. If $\RO_{\gamma,\varkappa}(B)\neq0$ for some $\varkappa$ then
\begin{align*}
\exists l\in\mathbf{G}(\mathbb{A})^+&, t\in\mathbf{T}(\mathbb{A})\colon\;
((\varkappa l)^{-1}\gamma_1 t, l^{-1} \varkappa^{-1}\gamma_1 t_{\mathbb{Q}} t)\in B'=\xi_1 Bg^{-1}\times \xi_2 B g^{-1} s^{-1}\\
&\Longrightarrow t_{\mathbb{Q}}\in g B^{-1} \xi_1^{-1}\xi_2 B g^{-1}s^{-1}
\Longrightarrow g^{-1}t_{\mathbb{Q}}s g\in B^{-1}\xi_1^{-1}\xi_2 B
\end{align*}
In contradiction to the assumption \eqref{eq:big-shift-vs-hecke}.
\end{proof}

\begin{remark}
Notice that if condition \eqref{eq:big-shift-vs-hecke} holds then $\Cor[\mu,\nu](B_0)$ will have no contribution from terms with a non-compact stabilizer for any 
$B_0\subset B$. This will be useful as in the endgame we would like to bound the cross-correlation between a limit measure and a simply connected Hecke correpondence for an arbitrarily small identity neighborhood. 
\end{remark}

\begin{cor}\label{cor:geometric-expansion-final-form}
Assume 
\begin{equation*}
g^{-1}\mathbf{T}(\mathbb{Q})s g\cap B^{-1}\ctr(\xi) B=\emptyset
\end{equation*}
Then 
\begin{equation*}
\Cor[\mu,\nu](B)
=\sum_{\substack{[\gamma]\in W_{\mathbb{Q}} \\ \psi\det^{-1}(\gamma)\neq 0 }} \frac{1}{\#\mathbf{M}_\gamma(\mathbb{Q})} \sum_{\varkappa\in
{\mathbf{G}(\mathbb{Q})} \slash {\mathbf{G}(\mathbb{Q})^+}} 
\RO_{\gamma,\varkappa}(B)
\end{equation*}
and
\begin{equation*}
\#\mathbf{M}_\gamma(\mathbb{Q})=\begin{cases}
1 & \ctr(\gamma)\not\in\Nrml_{\mathbf{G}}\mathbf{T}(\mathbb{Q})\\
2 & \textrm{otherwise}
\end{cases}
\end{equation*}
\end{cor}
\begin{proof}
Lemma \ref{lem:reduction-to-compact-stab} implies that the geometric expansion of $\Cor[\mu,\nu](B)$ has no contributions from $[\gamma]$ such that $\ctr(\gamma)\in \mathbf{T}(\mathbb{Q})$.
The condition $\psi\det^{-1}(\gamma)\neq0$ is exactly equivalent to $\ctr(\gamma)\not\in \mathbf{T}(\mathbb{Q})$.

The claimed expression for $\Cor[\mu,\nu](B)$ now holds due to Proposition \ref{prop:geometric-expansion} and Lemma \ref{lem:RO-compact}. To calculate $\mathbf{M}_\gamma(\mathbb{Q})$ in the relevant cases we use Proposition \ref{prop:M-orbits}. This Proposition implies that $\mathbf{M}_\gamma$ is trivial if $\ctr(\gamma)\not\in\Nrml_{\mathbf{G}}(\mathbf{T})(\mathbb{Q})$ and $\mathbf{M}_\gamma\simeq \mu_2$ otherwise. The final part of the claim holds because $\mu_2(\mathbb{Q})\simeq \cyclic{2}$.
\end{proof}

\subsection{Decomposition of the Relative Orbital Integral}
\begin{defi}
Fix $\gamma\in\left(\mathbf{G}\times\mathbf{G}\right)(\mathbb{Q})$ with $\mathbf{M}_\gamma(\mathbb{A})$ compact and let $\varkappa\in\mathbf{G}(\mathbb{Q})$. We  split the relative orbital integral into an archimedean and non-archimedean parts 
\begin{align*}
\RO_{\gamma,\varkappa}(B)&=\RO^\infty_\gamma(B)\cdot\RO^f_\gamma(B)\\
\RO^\infty_{\gamma,\varkappa}(B)&\coloneqq \int_{\mathbf{M}(\mathbb{R})^\dagger}\mathbb{1}_{B'_\infty}\left((\varkappa l)^{-1}\gamma t\right) \dif(l,t)\\
\RO^f_{\gamma,\varkappa}(B)&\coloneqq \int_{\mathbf{M}(\mathbb{A}_f)^\dagger}\mathbb{1}_{B'_f}\left((\varkappa l)^{-1}\gamma t\right) \dif(l,t)
\end{align*}
\end{defi}

The complicated expression to handle is the non-archimedean part. We will see that the archimedean part is rather simple due to the fact that we have restricted to the case $H_\infty=K_\infty$ in \eqref{eq:g_adel_restrictions}. 

In the next section we interpret the non-archimedean relative orbital integral as counting the number of intersections between the $\mathbf{M}(\mathbb{A}_f)$-orbit of $\gamma$ and $B'$ modulo a compact-open subgroup of $\mathbf{M}(\mathbb{A}_f)$. The main result is a finite-to-one map between intersections and pairs of integral $\Lambda$-ideals satisfying a list of arithmetic conditions. 

Unlike Linnik's argument for the equidistribution of CM points on a modular curve we do not calculate the relative orbital integrals at each place separately. Instead, we match them globally with a different global object. 
\subsection{Archimedean Relative Orbital Integral}
\begin{lem}\label{lem:archimedean-RO}
Let $\gamma=(\gamma_1,\gamma_2)\in\left(\mathbf{G}\times\mathbf{G}\right)(\mathbb{Q})$ and $\varkappa\in\mathbf{G}(\mathbb{Q})$. 
Assume $B_\infty=\Omega_\infty$ then $\RO^\infty_{\gamma,\varkappa}(B)=0$ if $\ctr(\gamma)\not\in g_\infty \Omega_\infty \ctr(\xi)_\infty \Omega_\infty g_\infty^{-1}$ and $$\RO^\infty_{\gamma,\varkappa}(B)\leq \meas_{\mathbf{T}(\mathbb{R})}\left(\mathbf{T}(\mathbb{R})\right)
\meas_{\mathbf{G}(\mathbb{R})^+}\left(\xi_{1,\infty}\Omega_\infty^2\xi_{1,\infty}^{-1} \cap \xi_{2,\infty}\Omega_\infty^2\xi_{2,\infty}^{-1}\right)$$ otherwise.
\end{lem}
\begin{proof}
From $\Omega_\infty K_\infty=\Omega_\infty$ and $g_\infty^{-1} \mathbf{T}(\mathbb{R}) g_\infty=K_\infty$ we deduce $B'_\infty \cdot \mathbf{T}(\mathbb{R})^\Delta=\xi_1\Omega_\infty g_\infty^{-1}\times \xi_2\Omega_\infty g_\infty^{-1}$. Hence $\mathbb{1}_{B'_\infty}\left((\varkappa l)^{-1}\gamma t\right)=1$ if and only if 
\begin{equation*}
(\varkappa l)^{-1}\in \xi_{1,\infty}\Omega_\infty g_\infty^{-1}\gamma_1^{-1} \cap \xi_{2,\infty}\Omega_\infty g_\infty^{-1}\gamma_2^{-1}
\end{equation*}
If the intersection on the right hand side is non-empty then there are some $\omega_1,\omega_2\in\Omega_\infty$ such that
\begin{equation}\label{eq:gama_infty-Omega-xi-Omega}
\ctr(\gamma)=g_\infty \omega_1 \ctr(\xi)_\infty \omega_2 g_\infty^{-1}\in g_\infty \Omega_\infty \ctr(\xi)_\infty \Omega_\infty g_\infty^{-1}
\end{equation}
This proves the first claim.

Moreover, the infinite part of the relative orbital integral is
\begin{align*}
\RO^\infty_{\gamma,\varkappa}(B)
&\coloneqq \int_{\mathbf{M}(\mathbb{R})^\dagger}\mathbb{1}_{B'_\infty}\left((\varkappa l)^{-1}\gamma t\right) \dif(l,t)\\
&= \meas_{\mathbf{T}(\mathbb{R})}\left(\mathbf{T}(\mathbb{R})\right)
\int_{\mathbf{G}(\mathbb{R})^+} \mathbb{1}_{\xi_{1,\infty}\Omega_\infty g_\infty^{-1}\gamma_1^{-1} \cap \xi_{2,\infty}\Omega_\infty g_\infty^{-1}\gamma_2^{-1}}\left((\varkappa l)^{-1}\right) \dif l\\
&=\meas_{\mathbf{T}(\mathbb{R})}\left(\mathbf{T}(\mathbb{R})\right)
\meas_{\mathbf{G}(\mathbb{R})^+} \left(\xi_{1,\infty}\Omega_\infty g_\infty^{-1}\gamma_1^{-1}\varkappa \cap \xi_{2,\infty}\Omega_\infty g_\infty^{-1}\gamma_2^{-1}\varkappa\right)
\end{align*}
The right hand side above is trivially zero unless $\Nrd\varkappa^{-1}\gamma_1 g_\infty \xi_{1,\infty}^{-1}>0$ so we may assume it is the case. Using the right invariance of a Haar measure on $\mathbf{G}(\mathbb{R})^+$ and \eqref{eq:gama_infty-Omega-xi-Omega} we have
\begin{align*}
&\meas_{\mathbf{G}(\mathbb{R})^+} \left(\xi_{1,\infty}\Omega_\infty g_\infty^{-1}\gamma_1^{-1}\varkappa \cap \xi_{2,\infty}\Omega_\infty g_\infty^{-1}\gamma_2^{-1}\varkappa\right)=\\
&\meas_{\mathbf{G}(\mathbb{R})^+} \left(\xi_{1,\infty}\Omega_\infty\xi_{1,\infty}^{-1} \cap \xi_{2,\infty}\Omega_\infty g_\infty^{-1}\gamma_2^{-1}\gamma_1 g_\infty \xi_{1,\infty}^{-1}\right)=\\
&\meas_{\mathbf{G}(\mathbb{R})^+} \left(\xi_{1,\infty}\Omega_\infty \omega_1^{-1} \xi_{1,\infty}^{-1} \cap \xi_{2,\infty}\Omega_\infty \omega_2^{-1} \xi_{2,\infty}^{-1}\right)\leq \meas_{\mathbf{G}(\mathbb{R})^+}\left(\xi_{1,\infty}\Omega_\infty^2\xi_{1,\infty}^{-1} \cap \xi_{2,\infty}\Omega_\infty^2\xi_{2,\infty}^{-1}\right)
\end{align*}
as claimed.
\end{proof}

\subsection{Non-Archimedean Relative Orbital Integrals}
\begin{defi}
Let $B_f<\mathbf{G}(\mathbb{A}_f)$ be a compact-open subgroup and fix a homogeneous Hecke set $\left[\mathbf{G}^{\Delta}(\mathbb{A})^+(\xi_1,\xi_2)\right]$ and a homogeneous toral set $\left[\mathbf{T}(\mathbb{A})g\right]$. We fix the following notations
\begin{align*}
B_{\mathbf{G},f}&\coloneqq \xi_{1,f} B_f \xi_{1,f}^{-1}\cap \xi_{2,f} B_f \xi_{2,f}^{-1}
<\mathbf{G}(\mathbb{A}_f)\\
B_{\mathbf{G},f}^+&\coloneqq B_{\mathbf{G},f} \cap \mathbf{G}(\mathbb{A}_f)^+
<\mathbf{G}(\mathbb{A}_f)^+\\
B_{\mathbf{T},f}&\coloneqq g_f B_f g_f^{-1}\cap \mathbf{T}(\mathbb{A}_f)<\mathbf{T}(\mathbb{A}_f)\\
B_{\mathbf{M},f}&\coloneqq B_{\mathbf{G},f}
\times
B_{\mathbf{T},f}
<\mathbf{M}(\mathbb{A}_f)\\
B_{\mathbf{M},f}^\dagger&\coloneqq B_{\mathbf{G},f}^+
\times
B_{\mathbf{T},f}
<\mathbf{M}(\mathbb{A}_f)^\dagger\\
\end{align*}
Each of these is a compact-open subgroup of the appropriate group.
\end{defi}

\begin{defi}
Let $\gamma\in\left(\mathbf{G}\times\mathbf{G}\right)(\mathbb{Q})$ and $\varkappa\in\mathbf{G}(\mathbb{Q})$.
Define the following functions
\begin{align*}
f_{\gamma,\varkappa}(l,t)&\coloneqq\mathbb{1}_{B'_f}\left((\varkappa l)^{-1} \gamma t\right) \\
f_{\gamma}(l,t)&\coloneqq\mathbb{1}_{B'_f}\left(l^{-1}\gamma t\right)
\end{align*}
The former is $B_{\mathbf{M},f}^\dagger$-invariant function on $\mathbf{M}(\mathbb{A}_f)^\dagger$ and the latter is a  $B_{\mathbf{M},f}$-invariant function on $\mathbf{M}(\mathbb{A}_f)$.
\end{defi}

\subsubsection{Intersection Numbers}
\begin{lem}\label{lem:cosets+-coset}
Let $B=B_\infty\times B_f \subset\mathbf{G}(\mathbb{A})$ be an identity neighborhood such that $B_f\subseteq \mathbf{G}(\mathbb{A}_f)$ is contained in the union of all compact-open subgroups. Then
each coset from $\faktor{\mathbf{M}(\mathbb{A}_f)}{B_{\mathbf{M},f}}$ contains at most two cosets from $\faktor{\mathbf{G}(\mathbb{Q}) \mathbf{M}(\mathbb{A}_f)^\dagger}{B_{\mathbf{M},f}^\dagger}$; where we consider $\mathbf{G}$ as a subgroup of $\mathbf{M}=\mathbf{G}\times\mathbf{T}$  in the usual way. 
\end{lem}
\begin{proof}

Let $\varkappa_{-1}\in\mathbf{B}^\times(\mathbb{Q})$ be an element with reduced norm $-1$ if it exists, i.e.\ if $\mathbf{B}$ is split at $\infty$, see \S\ref{sec:simply-connected-cover}. By abuse of notation we use the notation $\varkappa_{-1}$ also for the corresponding element in $\mathbf{G}(\mathbb{Q})$.

Fix $m\in\mathbf{M}(\mathbb{A}_f)$. Assume $\varkappa_i m_i^\dagger B_{\mathbf{M},f}^\dagger\subseteq m B_{\mathbf{M},f}$ where $\varkappa_i\in\mathbf{G}(\mathbb{Q})$, $m_i^\dagger\in\mathbf{M}(\mathbb{A}_f)^\dagger$ for $i\in\left\{1,2\right\}$. We show that either $\varkappa_1\in\varkappa_2 \mathbf{G}(\mathbb{Q})^+$ or $\varkappa_1\in\varkappa_2 \varkappa_{-1} \mathbf{G}(\mathbb{Q})^+$ if $\varkappa_{-1}$ exists.

Our assumption implies that there is some $b\in B_{\mathbf{M},f}$ such that $\varkappa_1 m_1^\dagger=\varkappa_2 m_2^\dagger b$. We apply the injective map
\begin{equation*}
\Nrd\colon\mathbf{G}(\mathbb{A}_f)\to \faktor{\mathbb{A}_f^\times}{{\mathbb{A}_f^\times}^2}
\end{equation*}
And deduce $\Nrd (\varkappa_2^{-1}\varkappa_1)=\Nrd(b)$. The condition satisfied by $B_f$ implies that
$\Nrd b \in \widehat{\mathbb{Z}}^\times \mod \left(\mathbb{A}_f^\times\right)^2$.
Thus the valuation of $\Nrd (\varkappa_2^{-1}\varkappa_1)$ is even at each finite place. As $\Nrd (\varkappa_2^{-1}\varkappa_1)$ is rational it must belong to $\pm {\mathbb{Q}^\times}^2$, i.e.\ it is either trivial in $\faktor{\mathbb{A}_f^\times}{{\mathbb{A}_f^\times}^2}$ or has the same class as $\varkappa_{-1}$. Because $\Nrd$ has kernel $\mathbf{G}(\mathbb{A}_f)^+$ we conclude that $\varkappa_2^{-1}\varkappa_1$ either belongs to $\mathbf{G}(\mathbb{Q})^+$ or to $\varkappa_{-1}\mathbf{G}(\mathbb{Q})^+$. The claim follows immediately.
\end{proof}
\begin{remark}
It is not difficult to analyze for a specific $B$ exactly how many cosets from $\faktor{\mathbf{G}(\mathbb{Q}) \mathbf{M}(\mathbb{A}_f)^\dagger}{B_{\mathbf{M},f}^\dagger}$ are contained in a fixed coset from $\faktor{\mathbf{M}(\mathbb{A}_f)}{B_{\mathbf{M},f}}$. This would allow converting several of the inequalities in what follows to equalities. As this of no practical use to us we do not pursue it here.
\end{remark}

\begin{prop}\label{prop:nonarchimidean-RO-Ngamma}
Let $[\gamma]\in W_{\mathbb{Q}}$ then
\begin{equation*}
\sum_{\varkappa\in \mathbf{G}(\mathbb{Q})\slash\mathbf{G}(\mathbb{Q})^+}
\RO^f_{\gamma,\varkappa}(B)
\leq
2\meas_{\mathbf{G}(\mathbb{A}_f)}(B_{\mathbf{G},f}) \meas_{\mathbf{T}(\mathbb{A}_f)}(B_{\mathbf{T},f})N_{[\gamma]}
\end{equation*}
where $N_{[\gamma]}$ is the number of times the $\mathbf{M}(\mathbb{A}_f)$ orbit of $\gamma$ intersects $B_f'$ modulo $B_{\mathbf{M},f}$.
\end{prop}
\begin{proof}
For any $\varkappa\in \faktor{\mathbf{G}(\mathbb{Q})}{\mathbf{G}(\mathbb{Q})^+}$ we can write
\begin{equation*}
\RO^f_{\gamma,\varkappa}(B)
=\meas_{\mathbf{G}(\mathbb{A}_f)}(B_{\mathbf{G},f}) \meas_{\mathbf{T}(\mathbb{A}_f)}(B_{\mathbf{T},f}) N_{\gamma,\varkappa}
\end{equation*}
where $N_{\gamma,\varkappa}$ is the number of times the $\mathbf{M}(\mathbb{A}_f)^\dagger$ orbit of $\varkappa^{-1}.\gamma$ intersect $B_f'$ modulo $B_{\mathbf{M},f}^\dagger$. This follows from the 
 $B_{\mathbf{M},f}^\dagger$-invariance of $f_{\gamma,\varkappa}$. The proof is finished by applying Lemma \ref{lem:cosets+-coset}.
\end{proof}

\begin{lem}\label{lem:contraction-of-Ngamma}
Consider the action of the algebraic group $\mathbf{M}=\mathbf{G}\times\mathbf{T}$ on the affine variety $\mathbf{G}$ where the $\mathbf{G}$-coordinate acts trivially and the $\mathbf{T}$-coordinate acts by conjugation.
The contraction morphism $\ctr\colon\mathbf{G}\times \mathbf{G}\to\mathbf{G}$ defined by $(g_1,g_2)\mapsto g_1^{-1}g_2$ is an $\mathbf{M}$-equivariant morphism of affine varieties.

The set $B_f'=\xi_{1,f}B_fg_f^{-1}\times\xi_{2,f}B_fg_f^{-1} s_f^{-1}\subset\left(\mathbf{G}\times\mathbf{G}\right)(\mathbb{A}_f)$ is $B_{\mathbf{M},f}$-invariant and the contraction map is a bijection between $\lfaktor{B_{\mathbf{M},f}}{B_f'}$ and its image $\lfaktor{\Ad B_{\mathbf{T},f}}{\ctr(B_f')}$. In particular for any $\gamma=(\gamma_1,\gamma_2)\in\left(\mathbf{G}\times\mathbf{G}\right)(\mathbb{Q})$
\begin{equation*}
N_{[\gamma]}=\# \left(\lfaktor{\Ad B_{\mathbf{T},f}}{\Ad\mathbf{T}(\mathbb{A}_f)\ctr(\gamma) \cap \ctr(B_f')} \right)
\end{equation*}
\end{lem}
\begin{proof}
The map $\ctr\colon \lfaktor{B_{\mathbf{M},f}}{B_f'}\to \lfaktor{\Ad B_{\mathbf{T},f}}{\ctr(B_f')}$ is obviously surjective and we need only prove injectivity. Assume 
\begin{equation}\label{eq:h-Tconj-h'}
t h_1^{-1} h_2 t^{-1}=h_1'^{-1} h_2'
\end{equation}
for some $(h_1,h_2),(h_1',h_2')\in\left(\mathbf{G}\times\mathbf{G}\right)(\mathbb{A}_f)$ and $t\in\Ad B_{\mathbf{T},f}$. We need to prove that there is some $l\in B_{\mathbf{G},f}$ so that $l h_i t^{-1}=h_i'$ for $i\in\{1,2\}$.

Set $l=h_1'th_1^{-1}\in\mathbf{G}(\mathbb{A}_f)$ then \eqref{eq:h-Tconj-h'} implies that $lh_it^{-1}=h_i'$ for $i\in\{1,2\}$. To see that $l\in B_{\mathbf{G},f}$ notice that for $i\in\{1,2\}$
\begin{equation*}
l=h_i't h_i^{-1}\in \xi_{i,f} B_f B_{\mathbf{T},f} B_f^{-1} \xi_{i,f}^{-1}=\xi_{i,f} B_f \xi_{i,f}^{-1}
\end{equation*}
and $B_{\mathbf{G},f}=\bigcap_{i\in\{1,2\}}\xi_{i,f} B_f \xi_{i,f}^{-1}$.
\end{proof}

\subsubsection{Matching Intersections to Pairs of Integral Ideals}
The last lemma indicates that $N_{[\gamma]}$ can be computed by understanding intersections of $\Ad B_{\mathbf{T},f}$-orbits on $\Ad\mathbf{T}(\mathbb{A}_f)\mathbf{G}(\mathbb{Q})\subset\mathbf{G}(\mathbb{A}_f)$ with $\ctr(B_f')$. 
We restrict the the case $B_{\mathbf{T},f}=K_{\mathbf{T},f}$ and develop arithmetic invariants, refining results of GIT, to detect these intersections.

\begin{defi}\label{def:inv-first}
Recall that $K_{\mathbf{T},f}\coloneqq \mathbf{T}(\mathbb{A}_f)\cap g_f K_f g_f^{-1}$.
We construct a function 
$$\inv_f\colon \lfaktor{\Ad K_{\mathbf{T},f}}{\mathbf{G}(\mathbb{A}_f)}\to \lfaktor{{\mathbb{Q}^\times}^\Delta}{\Ideals(\Lambda)_0\times \Ideals(\Lambda)_0}$$
in the following manner.

Let $[h_f]=[(h_v)_{v\neq \infty}]\in \lfaktor{\Ad K_{\mathbf{T},f}}{\mathbf{G}(\mathbb{A}_f)}$. For all $v\neq\infty$ choose a representative of $h_v$ in $\mathbf{B}^\times(\mathbb{Q}_v)$ and write it in coordinates using Proposition \ref{prop:local-order-general}.
\begin{equation*}
h_v= \mathbb{Q}_v^\times
\begin{pmatrix} 
\alpha_v & \beta_v \upsilon_v \tau_v \\ 
\tensor[^\sigma]{\beta}{_v}/\tau_v & \tensor[^\sigma]{\alpha}{_v}
\end{pmatrix}
\end{equation*}
where $\alpha_v,\beta_v\in E_v$ and $\upsilon_v,\tau_v$ are as in the proposition. Now define $\inv_f([h_f])$ as
\begin{equation*}
\inv_f\big([h_f]\big)=\left(\widetilde{\idl}(\alpha_v), \widetilde{\idl}(\beta_v)\right)=
\left(\bigcap_{v\neq\infty}\ \alpha_v \Lambda_v, \bigcap_{v\neq\infty}\ \beta_v \Lambda_v \right)
\end{equation*}
This pair of ideals is obviously well defined up to multiplication by a common ideal of the form $\bigcap_{v\neq\infty} q_v \Lambda_v$ where $q_v\in\mathbb{Q}_v^\times$ for all $v\neq\infty$. Because $\mathbb{Q}$ has class number one this is equivalent to multiplying the ideals by the same element of $\mathbb{Q}^\times$.

The map $\inv_f$ is also invariant under conjugation by $K_{\mathbf{T},f}$. Recall that $\cbd(x)=x/\tensor[^\sigma]{x}{}$ for all $x\in E_v^\times$.
Conjugating by an element of $K_{\mathbf{T},f}$ is equivalent to multiplying $\beta_v$ by an element of $\cbd(\Lambda_v^\times)\subset \Lambda_v^\times$ hence defines the same fractional ideals. 
\end{defi}

\begin{defi}\label{def:inv-second}
\hfill
\begin{enumerate}
\item 
Set $\mathbf{G}(\mathbb{A})_{\mathrm{accessible}}\coloneqq \Ad\mathbf{T}(\mathbb{A})\mathbf{G}(\mathbb{Q})$. This is an $\Ad\mathbf{T}(\mathbb{A})$-invariant subset of $\mathbf{G}(\mathbb{A})$.

\item
Define the map
\begin{equation*}
\inv\colon 
\lfaktor{\Ad \mathbf{T}(\mathbb{R})\Ad K_{\mathbf{T},f}}{\mathbf{G}(\mathbb{A})_{\mathrm{accessible}}}
\to
\lfaktor{\mathbb{Q}^\times}{\Ideals(\Lambda)_0\times \Ideals(\Lambda)_0}
\end{equation*}
by evaluating $\inv_f$ from Definition \ref{def:inv-first} on the finite part.

\item For each $v$ we have defined a map $\pi_{\mathbf{W}}\colon\lfaktor{\Ad\mathbf{T}(\mathbb{Q}_v)}{\mathbf{G}(\mathbb{Q}_v)}\to\mathbf{W}(\mathbb{Q}_v)$. For an element of $\mathbf{G}(\mathbb{A})_{\mathrm{accessible}}$ the map $\pi_{\mathbf{W}}$ maps each place to the \emph{same} value in $\mathbf{W}(\mathbb{Q})$. We thus have a well-defined map
\begin{equation*}
\pi_{\mathbf{W}}\colon \lfaktor{\Ad\mathbf{T}(\mathbb{A})}{\mathbf{G}(\mathbb{A})_{\mathrm{accessible}}} \to \mathbf{W}(\mathbb{Q})
\end{equation*}
which send an element to the common value of $\pi_{\mathbf{W}}$ evaluated at any place.
\end{enumerate}
\end{defi}

The next proposition is a central observation relating cross-correlation to shifted convolution sums of ideal counting functions.
\begin{prop}\label{prop:fiber-of-invariants}
Let $\left[\mathbf{T}(\mathbb{A})g\right]$ be a homogeneous toral set. Let $$[h]\in\lfaktor{\Ad \mathbf{T}(\mathbb{R})\Ad K_{\mathbf{T},f}}{\mathbf{G}(\mathbb{A})_{\mathrm{accessible}}}$$ such that $\inv(h)\in \lfaktor{\mathbb{Q}^\times}{\Ideals(\Lambda)_0\times\Ideals(\Lambda)}$, i.e.\ $h\not\in \mathbf{T}(\mathbb{Q})$. Then there is a faithful action of $\prod_{v\neq\infty} H^1(\mathfrak{G},\Lambda_v^\times)$ on $\inv^{-1}([h])$ and the number of orbits is $\leq 8$.
\end{prop}
\begin{proof}
Our strategy is to show that $\left(\inv\times\,\pi_{\mathbf{W}}\right)^{-1}([h])$ is a principal homogeneous space for $\prod_{v\neq\infty} H^1(\mathfrak{G},\Lambda_v^\times)$
and to prove that $\pi_{\mathbf{W}}\left(\inv^{-1}(\inv([h]))\right)$ contains at most $8$ elements.

Denote $\mathbb{Q}^\times(\mathfrak{a},\mathfrak{b})\coloneqq \inv([h])$. We first show there are at most $8$-possible elements in $\pi_{\mathbf{W}}\left(\inv^{-1}\left(\mathbb{Q}^\times(\mathfrak{a},\mathfrak{b})\right)\right)$. Proposition \ref{prop:GIT-descent} implies that an element of $\mathbf{W}(\mathbb{Q})$ is uniquely determined by the values in $E$ of the regular functions $\vartheta_1^2 \det^{-1},\vartheta_1 \vartheta_2 \det^{-1},\psi \det^{-1}$. 
We define a map $\mathbf{W(\mathbb{Q})}\to \mathbb{P}^2(E)$ by
\begin{equation*}
w\mapsto [\vartheta_1^2 \det^{-1}(w): \vartheta_1 \vartheta_2 \det^{-1}(w):\psi \det^{-1}(w)]
\end{equation*}
We claim that this map is injective on $\mathbf{W}(\mathbb{Q})$. A priori it defines the values of the necessary functions only up to a common multiplicative constant. But the value of this constant is uniquely determined by the syzygy $\vartheta_1 \vartheta_2 \det^{-1}-\psi \det^{-1}=1$.

Let $t \gamma t^{-1}\in\mathbf{G}(\mathbb{A})_{\mathrm{accessible}}$, $\gamma\in\mathbf{G}(\mathbb{Q})$, $t\in\mathbf{T}(\mathbb{A})$  represent an element in $\inv^{-1}\left(\mathbb{Q}^\times(\mathfrak{a},\mathfrak{b})\right)$.
For any $v\neq\infty$ the pair of fractional ideal $(\mathfrak{a},\mathfrak{b})$ corresponds to some local ideals $(\alpha_v \Lambda_v, \beta_v\Lambda_v)$  such that 
\begin{equation*}
\begin{pmatrix} 
\alpha_v  & \beta_v \upsilon_v \tau_v \\ 
\tensor[^\sigma]{\beta}{_v} /\tau_v & \tensor[^\sigma]{\alpha}{_v} 
\end{pmatrix}
\end{equation*}
is in the $\Ad\mathbf{T}(\mathbb{Q}_v)$-orbit of $\gamma$. We can calculated the coordinate functions of $[\gamma]\in\mathbf{W}(\mathbb{Q})$ using this matrix but their values depend on the specific representative in $\alpha_v\Lambda_v^\times$, $\beta_v \Lambda_v^\times$. Nevertheless
 the local principle ideals $(\alpha_v \Lambda_v, \beta_v \Lambda_v)$ uniquely determine for any $w\mid v$ the $w$-part of 
\begin{equation*}
[\vartheta_1^2 \det^{-1}(\gamma): \vartheta_1 \vartheta_2 \det^{-1}(\gamma):\psi \det^{-1}(\gamma)]
\end{equation*}
and this hold for all finite $E$-places $w$.  
As all the entries are in $E$ they are uniquely defined up to an element of $\mathcal{O}_E^\times$. Moreover, the last two entries are in $\mathbb{Q}$ so they are defined up to an element in $\mathbb{Z}^\times$. In total we are left with $4\cdot 2=8$ possibilities at most for the homogeneous vector above. Hence there are at most $8$ possible corresponding points in $\mathbf{W}(\mathbb{Q})$.

Next we need to study the fiber of $\inv\times\,\pi_{\mathbf{W}}$. Let $\left(\mathbb{Q}^\times(\mathfrak{a},\mathfrak{b}),[\gamma]\right)\in \Img \inv\times\pi_{\mathbf{W}}$ where $\mathfrak{a}\in\Ideals_0(\Lambda)$, $\mathfrak{b}\in\Ideals(\Lambda)$ and $\gamma\in\mathbf{G}(\mathbb{Q})$. 

Assume first $\gamma\not\in w_\mathbf{T}\mathbf{T}(\mathbb{Q})$. Due to Proposition \ref{prop:M-orbits} the fiber in $\mathbf{G}(\mathbb{A})_{\mathrm{accessible}}$ of $\pi_{\mathbf{W}}$ over this point is $\Ad \mathbf{T}(\mathbb{A}) \gamma$. 
In coordinates we can write
\begin{equation}\label{eq:piW-fiber-accessible}
\Ad\mathbf{T}(\mathbb{A})\gamma=\left\{
\mathbb{Q}_v^\times
\begin{pmatrix}
a & \epsilon b \cbd(x_v)\\
\tensor[^\sigma]{b}{} \tensor[^\sigma]{\cbd(x_v)}{}& \tensor[^\sigma]{a}{}
\end{pmatrix}_v \relmiddle{|} x=(x_v)_v\in \prod_v E_v^\times
\right\}
\end{equation}
where $a\in E$ and $b\in E^\times$. The pertinent fiber of $\inv\times\,\pi_{\mathbf{W}}$ in $\mathbf{G}(\mathbb{A})_\mathrm{accessible}$ is $\Ad\mathbf{T}(\mathbb{A})\gamma \cap \inv^{-1}\left(\mathbb{Q}^\times(\mathfrak{a},\mathfrak{b})\right)$.

Otherwise, if 
$\gamma\in w_\mathbf{T}\mathbf{T}(\mathbb{Q})$ then $\mathfrak{a}=0$. For any $h'=t' \gamma' t'^{-1}\in\inv^{-1}\left(\mathbb{Q}^\times(0,\mathfrak{b})\right)$ we have
 $\gamma'=\gamma t_{\mathbb{Q}}$ with $t_{\mathbb{Q}}\in \mathbf{T}(\mathbb{Q})$. Moreover, because $\gamma$ and $\gamma'$ have the same invariants we see that $t_{\mathbb{Q}}\in \mathbf{T}(\mathbb{Q}) \cap K_{\mathbf{T},f}=\mathbf{T}(\mathbb{Z})\simeq \lfaktor{\mathbb{Z}^\times}{\Lambda^\times}$. In particular, using Hilbert's Satz 90 for $E$ we see that $\gamma'\in \Ad \mathbf{T}(\mathbb{Q}) \gamma$.

In both case we conclude that the fiber is equal to $\Ad\mathbf{T}(\mathbb{A})\gamma \cap \inv^{-1}\left(\mathbb{Q}^\times(\mathfrak{a},\mathfrak{b})\right)$. 
These are all the elements of \eqref{eq:piW-fiber-accessible} satisfying
\begin{equation}\label{eq:invXpiW-fiber-accessible}
\left(\widetilde{\idl}(a),\widetilde{\idl}\left(\frac{\epsilon}{\upsilon_v \tau_v} b \cbd(x)\right)\right)
\in \mathbb{Q}^\times (\mathfrak{a},\mathfrak{b})
\end{equation}
Denote $\mathfrak{b}'\coloneqq\widetilde{\idl}\left(\frac{\epsilon}{\upsilon_v \tau_v} b\right)\in\Ideals(\Lambda)$ then there is some $q\in\mathbb{Q}^\times$ such that
\begin{equation*}
\widetilde{\idl}(\cbd(x))=q\frac{\mathfrak{b}}{\mathfrak{b}'}
\end{equation*}
The left hand side is a fractional $\Lambda$-ideal of norm $1$, hence $q=\pm q_0$ where $q_0=\sqrt{\Nr\mathfrak{b}'/\Nr\mathfrak{b}}$. In particular, 
$\widetilde{\idl}(\cbd(x))=q_0\frac{\mathfrak{b}}{\mathfrak{b}'}$. 

Fix $c_v\in E_v^\times$ for all $v\neq\infty$ such that $\widetilde{\idl}((c_v)_v)=q_0\frac{\mathfrak{b}}{\mathfrak{b}'}$. For each $v\neq\infty$ the element $c_v$ has norm $1$ so by Hilbert's Satz 90 there is some $y_v\in E_v^\times$ satisfying $c_v=\cbd(y_v)$. The condition \eqref{eq:invXpiW-fiber-accessible} is equivalent to
\begin{equation*}
\forall v\neq\infty\colon \cbd(x_v)\in\cbd(y_v)\Lambda_v^\times \Leftrightarrow \cbd(x_v y_v^{-1})\in \Lambda_v^{(1)}
\end{equation*}
Thus the fiber in $\mathbf{G}(\mathbb{A})_\mathrm{accessible}$ is a principal homogeneous space for $\Ad\mathbf{T}(\mathbb{R})\times \prod_{v\neq\infty} \Lambda_v^{(1)}$. 
Writing these in coordinates we see that the fiber in $\mathbf{G}(\mathbb{A})_\mathrm{accessible}$ is a principal homogeneous space for $\Ad\mathbf{T}(\mathbb{R})\times \prod_{v\neq\infty} \Lambda_v^{(1)}$.

The quotient of the fiber by the action of $\Ad\mathbf{T}(\mathbb{R})\Ad K_{\mathbf{T},f}$ is a principle homogeneous space for $\prod_{v\neq\infty} \lfaktor{\cbd(\Lambda_v^\times)}{\Lambda_v^{(1)}}\simeq \prod_{v\neq\infty} H^1(\mathfrak{G},\Lambda_v^\times)$ as claimed.
\end{proof}
\begin{cor}
Let $[h]\in\lfaktor{\Ad \mathbf{T}(\mathbb{R})\Ad K_{\mathbf{T},f}}{\mathbf{G}(\mathbb{A})_{\mathrm{accessible}}}$ such that $\inv(h)\in \lfaktor{\mathbb{Q}^\times}{\Ideals(\Lambda)_0\times\Ideals(\Lambda)}$, i.e.\ $h\not\in \mathbf{T}(\mathbb{Q})$. Then
\begin{equation*}
\#\inv^{-1}([h])\ll \Pic(\Lambda)[2]
\end{equation*} 
\end{cor}
\begin{proof}
Follows from Proposition \ref{prop:fiber-of-invariants} above and Corollary \ref{cor:Pic(Lambda)[2]-H1(Lambda_v)-size}.
\end{proof}

\begin{lem}\label{lem:Picard-classes-of-invariants}
Let $h=t\gamma t^{-1}\in\mathbf{G}(\mathbb{A})_\mathrm{accessible}$, $t\in\mathbf{T}(\mathbb{A})$, $\gamma\in\mathbf{G}(\mathbb{Q})$. If $\inv(h)=\mathbb{Q}^\times (\mathfrak{a},\mathfrak{b})$  then
\begin{enumerate}
\item $[\mathfrak{a}]=1$ in $\Pic(\Lambda)$ if $\mathfrak{a}\neq 0$,
\item $[\mathfrak{b}\mathfrak{e}]\in\Pic(\Lambda)^2$  if $\mathfrak{b}\neq 0$.
\end{enumerate}
\end{lem}
\begin{proof}
Fixing representatives for $\gamma$ and $h_v$ for all places $v\neq\infty$ we have
\begin{equation*}
\begin{pmatrix}
\alpha_v & \beta_v \upsilon_v \tau_v \\
\tensor[^\sigma]{\beta}{_v}/\tau_v & \tensor[^\sigma]{\alpha}{_v}
\end{pmatrix}
=
q_v
\begin{pmatrix}
a & \epsilon b \frac{\lambda_v}{\tensor[^\sigma]{\lambda}{_v}} \\
\tensor[^\sigma]{b}{}\frac{\tensor[^\sigma]{\lambda}{_v}}{\lambda_v}& \tensor[^\sigma]{a}{}
\end{pmatrix}
\end{equation*}
where $\alpha_v,\beta_v\in E_v$, $\lambda_v\in E_v^\times$, $q_v\in\mathbb{Q}_v^\times$ and $a,b\in E^\times$. 
The ideals $\mathfrak{a}$ and $\mathfrak{eb}$ either vanish or their classes in $\Pic(\Lambda)$ satisfy 
\begin{align*}
[\mathfrak{a}]&=(\alpha_v)_{v\neq\infty} \mod \mathbb{Q}^\times\prod_{v\neq \infty} \Lambda_v^\times
=(a)_{v\neq\infty} \mod \mathbb{Q}^\times\prod_{v\neq \infty} \Lambda_v^\times=1\\
[\mathfrak{be}]&=(\beta_v \upsilon_v \tau_v)_{v\neq\infty} \mod \mathbb{Q}^\times\prod_{v\neq \infty} \Lambda_v^\times
=(\epsilon b\frac{\lambda_v}{\tensor[^\sigma]{\lambda}{_v}})_{v\neq\infty} \mod \mathbb{Q}^\times\prod_{v\neq \infty}\Lambda_v^\times
\in\Pic(\Lambda)^2\\
\end{align*}
\end{proof}

\begin{prop}\label{prop:intersection-to-invariants}
Fix $(x_v)_v\in\mathbf{G}(\mathbb{A}_f)$.
Let $B_v=\Omega_v$ for $v\neq p_1$ and $B_{p_1}=K_{p_1}^{(-n,n)}\subset \Omega_{p_1}$ for some  $n\in\mathbb{N}$.
If $h\in\mathbf{G}(\mathbb{A})_{\mathrm{accessible}}$ is contained in $\prod_{v} g_v B_v^{-1} x_v B_v g_v^{-1} s_v^{-1}$  then $$\inv(h)=\mathbb{Q}^\times \widehat{\Lambda} (\mathfrak{a}\cdot \mathfrak{s}^{-1},\mathfrak{b}\cdot \tensor[^\sigma]{\mathfrak{s}}{^{-1}})$$ for some $\mathfrak{a},\mathfrak{b} \in \Ideals_0(\Lambda)$ satisfying
\begin{enumerate}
\item $\mathfrak{a},\mathfrak{b}\subseteq \Lambda$,
\item $p_1^n \mid \mathfrak{b}$,
\item $\Nr(\mathfrak{a})-\upsilon\Nr(\mathfrak{b})=\sign\left(\Nrd(x_\infty)\right) \Denom_f (x_f)|D|$,
\item $\Nr(\mathfrak{a})\leq 2^8\Denom_\infty(x_\infty)\Denom_f(x_f)|D|$,
\item $[\mathfrak{a}]=[\mathfrak{s}]$ in $\Pic(\Lambda)$ if $\mathfrak{a}\neq0$,
\item $[\mathfrak{b}]\in[\mathfrak{se}]^{-1}\Pic(\Lambda)^2$  if $\mathfrak{b}\neq0$.
\end{enumerate}
\end{prop}
\begin{proof}
For each place $v$ choose a representative $r_v\in\mathbf{B}^\times(\mathbb{Q}_v)\cap\mathbb{O}_v$ of $g_v^{-1}h_v s_v g_v\in\Omega_v x_v \Omega_v$ satisfying the conclusion of Proposition \ref{prop:Omega_xi_Omega-reduced-norm}. For $v=p_1$ let $r_v$ satisfy the stronger conclusions of Proposition \ref{prop:Omega_xi_Omega-reduced-norm-Bowen}.
The element $g_v r_v g_v^{-1}$ belongs to $g_v\mathbb{O}_v g_v^{-1}$ and has same reduced norm as $r_v$, i.e.\ $|\Nrd g_v r_v g_v^{-1}|_v= \mathfrak{d}_v(x_v)^{-1}$ for $v\neq\infty$ and $|\Nrd g_\infty r_\infty g_\infty^{-1}|\geq 2^{-8} \mathfrak{d}_\infty(x_\infty)^{-1}$. We can use this to represent $h_v$ in $\mathbf{B}^\times(\mathbb{Q}_v)$ as
\begin{equation}\label{eq:h-good-representative}
h_v= \mathbf{Z}(\mathbb{Q}_v)
\begin{pmatrix}
\alpha_v s_v^{-1} & \beta_v\tensor[^\sigma]{s}{_v}^{-1} \upsilon_v \tau_v  \\
\tensor[^\sigma]{\beta}{_v}s_v^{-1}/\tau_v  & \tensor[^\sigma]{\alpha}{_v}\tensor[^\sigma]{s}{_v}^{-1}
\end{pmatrix}
\end{equation}
where $\alpha_v,\beta_v\in\widehat{\Lambda}_v$ due to Proposition \ref{prop:local-order-general}.
The definition of $\inv$ implies that $\inv(h)=\mathbb{Q}^\times(\mathfrak{a}_0,\mathfrak{b}_0)$ where $\mathfrak{a}_0=\bigcap_{v\neq\infty} \alpha_v s_v^{-1}\Lambda_v$ and $\mathfrak{b}_0=\bigcap_{v\neq\infty} \beta_v\tensor[^\sigma]{s}{_v}^{-1}\Lambda_v$. Define $\widehat{\mathfrak{a}}\coloneqq\bigcap_{v\neq\infty} \alpha_v\Lambda_v$ and $\widehat{\mathfrak{b}}\coloneqq\bigcap_{v\neq\infty} \beta_v\Lambda_v$ then $\widehat{\mathfrak{a}},\widehat{\mathfrak{b}}\subseteq \widehat{\Lambda}$. Finally, set $\mathfrak{a}\coloneqq \widehat{\Lambda}^{-1} \widehat{\mathfrak{a}}$ and $\mathfrak{b}\coloneqq \widehat{\Lambda}^{-1} \widehat{\mathfrak{b}}$ then $\mathfrak{a},\mathfrak{b}\subseteq\Lambda$ and $\mathfrak{a}_0=\widehat{\Lambda}\mathfrak{a}\mathfrak{s}^{-1}$, $\mathfrak{b}_0=\widehat{\Lambda}\mathfrak{b}\tensor[^\sigma]{\mathfrak{s}}{^{-1}}$.

We conclude that $\inv(h)=\mathbb{Q}^\times\widehat{\Lambda}(\mathfrak{a}\mathfrak{s}^{-1}, \mathfrak{b}\tensor[^\sigma]{\mathfrak{s}}{^{-1}})$. Obviously (1) is satisfied and (5), (6) are simply a restatement of Lemma \ref{lem:Picard-classes-of-invariants}.

We claim that
\begin{align}\label{eq:reg-functions<->norms}
\Nr(\widehat{a})&=\sign\left(\Nrd(x_\infty)\right)\Denom_f(x_f)\vartheta_1\vartheta_2\det^{-1}(h)\\
\upsilon \Nr(\widehat{b})&=\sign\left(\Nrd(x_\infty)\right)\Denom_f(x_f)\psi\det^{-1}(h)\nonumber
\end{align}
For any prime $p$ the $p$-part of \eqref{eq:reg-functions<->norms} follows by calculating $\vartheta_1\vartheta_2\det^{-1}$ and $\psi\det^{-1}$ using the representative in \eqref{eq:h-good-representative} for the corresponding non-archimedean $v$. To establish \eqref{eq:reg-functions<->norms} with the correct signs calculate $\vartheta_1\vartheta_2\det^{-1}$ and $\psi\det^{-1}$ using  the representative in \eqref{eq:h-good-representative} for $v=\infty$.

Recall that $\Nrd(\mathfrak{a})=\Nrd(\widehat{\Lambda}^{-1}\widehat{\mathfrak{a}})=|D|\Nrd(\widehat{\mathfrak{a}})$ and similarly for $\mathfrak{b}$.
Claim (3) now follows from \eqref{eq:reg-functions<->norms} and the syzygy $\vartheta_1\vartheta_2\det^{-1}-\psi\det^{-1}=1$. The archimedean bound $|\vartheta_1\vartheta_2\det^{-1}(h)|\leq 2^8 \Denom_\infty(x_\infty)$ follows from using \eqref{eq:h-good-representative} for $v=\infty$, Proposition \ref{prop:local-order-split-archimedean} and the inequality $|\Nrd g_v r_v g_v^{-1}|\geq 2^{-8}  \Denom_\infty(x_\infty)^{-1}$. Claim (4) follows from this bound and \eqref{eq:reg-functions<->norms}.

To prove (2) we use the conclusions of Proposition \ref{prop:Omega_xi_Omega-reduced-norm-Bowen}. Rewrite
\begin{equation*}
g_{p_1} \mathbb{O}_{p_1}^{(-n,n)} g_{p_1}^{-1}=\bigcap_{k=-n}^n t^k g_{p_1} \mathbb{O}_{p_1}^{(-n,n)} g_{p_1}^{-1} t^{-k}
\end{equation*}
where $t=g_{p_1}\lambda(p_1)g_{p_1}^{-1}\in\mathbf{T}(\mathbb{Q}_{p_1})$. Using the freedom in the choice of $\lambda$ we can assume
\begin{equation*}
t\in \mathbb{Q}_{p_1}^\times \begin{pmatrix}
\pi & 0 \\ 0 & \tensor[^\sigma]{\pi}{}
\end{pmatrix}
\end{equation*}
where $\pi\in E_{p_1}$ is a uniformizer for one of the two maximal ideals of $\mathcal{O}_{p_1}$ and $\tensor[^\sigma]{\pi}{}$ is a uniformizer for the second one. Because $r_{p_1}\in t^k g_{p_1} \mathbb{O}_{p_1}^{(-n,n)} g_{p_1}^{-1} t^{-k}$ for all $-n\leq k\leq n$ we conclude using Proposition \ref{prop:local-order-general} that 
\begin{equation}\label{eq:p_1^n-local}
\mathcal{D}_{p_1}\beta_{p_1}\in\bigcap_{k=-n}^n \frac{\tensor[^\sigma]{\pi}{^k}}{\pi^k}\Lambda_{p_1}=p_1^n \Lambda_{p_1}
\end{equation}
Because $\mathfrak{b}=\widehat{\Lambda}^{-1}\widehat{\mathfrak{b}}=\bigcap_v \mathcal{D}_{p_1} \beta_v \Lambda_v$ claim (1) follows from \eqref{eq:p_1^n-local}.      
\end{proof}

\begin{remark}
Elements $h$ with $\inv(h)=\mathbb{Q}^\times\widehat{\Lambda}(\mathfrak{a}\cdot \mathfrak{s}^{-1},0)$ in the Proposition above corresponds to $\mathbf{M}(\mathbb{A}_f)$-orbits of elements $\gamma\in\left(\mathbf{G}\times\mathbf{G}\right)(\mathbb{Q})$ with stabilizer $\mathbf{M}_\gamma\simeq\mathbf{T}$. 

We know from Lemma \ref{lem:reduction-to-compact-stab} and Proposition \ref{prop:equidist-conditions-equivalence} that these elements' contribution to the cross-correlation should vanish if the minimal norm of an integral ideal in the Picard class $[\mathfrak{s}]$ is $\geq C$ for some constant $C$ depending only on $\xi$. If $\sign(\Nrd(\ctr(\xi)_{\infty}))=1$ this seems to contradict the proposition above which states that such $\gamma$ correspond to integral ideals $\mathfrak{a}$ with $\Nr(\mathfrak{a})=\sign(\Nrd(\ctr(\xi)_\infty))\Denom_f(\ctr(\xi)_f)|D|$ which for $|D|$ large enough is bigger then $C$.

The contradiction is resolved by observing that in this case $\ctr(\gamma)\in\mathbf{T}(\mathbb{Q})$ and in the language of the proof above we know that not only $\widehat{\mathfrak{a}}\subseteq \widehat{\Lambda}$ but also $\widehat{\mathfrak{a}}\subseteq \Lambda$. This implies that $\mathfrak{a}\subset \widehat{\Lambda}^{-1}\subsetneq \Lambda$ and if $\Nr(\mathfrak{a})=\sign(\Nrd(\ctr(\xi)_\infty))\Denom_f(\ctr(\xi)_f)|D|$ then $\widehat{\Lambda}\mathfrak{a}$ is an integral ideal in the class $[\mathfrak{s}]$ of norm $\sign(\Nrd(\ctr(\xi)_\infty))\Denom_f(\ctr(\xi)_f)$ which does not grow to infinity with $|D|$.

This reasoning can also be used to exclude the situation $\inv(h)=\mathbb{Q}^\times(0,\mathfrak{b}\cdot \tensor[^\sigma]{\mathfrak{s}}{^{-1}})$, i.e.\ $\mathbf{M}_\gamma\simeq\mu_2$. These unnecessary terms will have negligible contribution to the shifted convolution sum and we do not take the extra effort to write them off. The essential part is that the contribution of a non-compact stabilizer to the geometric expansion is eliminated in Lemma \ref{lem:reduction-to-compact-stab}. 
\end{remark}

\begin{prop}\label{prop:fiber-over-integral-invariants}
Let $h\in\mathbf{G}(\mathbb{A})_\mathrm{accessible}$ and $\inv(h)=\mathbb{Q}^\times\widehat{\Lambda} (\mathfrak{a}\mathfrak{s}^{-1},\mathfrak{b}\tensor[^\sigma]{\mathfrak{s}}{^{-1}})$ be as in Proposition \ref{prop:intersection-to-invariants} above. Let $p\mid D$ be an \emph{odd} prime where $\mathbf{B}$ splits. Denote by $v$ the place associated to $p$ and recall from Proposition \ref{prop:fiber-of-invariants} that $H^1(\mathfrak{G},\Lambda_v^\times)$ acts faithfully on $\inv^{-1}\left(\inv(h)\right)$.

Let $-1\in H^1(\mathfrak{G},\Lambda_v^\times)$ be the unique non-trivial element, cf.\ Lemmata \ref{lem:units-cohomology-maximal} and \ref{lem:units-cohomology-non-maximal}. If 
$\ord_p \Nr(\mathfrak{b})<\ord_p D$ then $-1.h_v\not\in g_v B_v^{-1} x_v B_v g_v^{-1}s_v^{-1}$.
\end{prop}
\begin{proof}
Assume in contradiction that both $h_v$ and $-1.h_v$ belong to $g_v B_v^{-1} x_v B_v g_v^{-1}s_v^{-1}$. We write $(\pm 1.h_v) s_v$ in coordinates as in \eqref{eq:h-good-representative} in the proof of Proposition \ref{prop:intersection-to-invariants} above (notice that $\upsilon_v=1$ in the split case)
\begin{equation*}
(\pm 1.h_v)s_v=
\mathbf{Z}(\mathbb{Q}_v)
\begin{pmatrix}
\alpha_v & \pm\beta_v \tau_v  \\
\pm\tensor[^\sigma]{\beta}{_v}/\tau_v  & \tensor[^\sigma]{\alpha}{_v}
\end{pmatrix}
\end{equation*}
We have used above the fact from Lemmata \ref{lem:units-cohomology-maximal} and \ref{lem:units-cohomology-non-maximal} that for odd residue characteristic the unique non-trivial cohomology class is represented by $-1\in \Lambda_v^{(1)}$.  The explicit form of the action of $H^1(\mathfrak{G},\Lambda_v^\times)$ is evident from the proof of Proposition \ref{prop:fiber-of-invariants}.

Because the matrix on the right hand side above belongs to $g_v\mathbb{O}_v g_v^{-1}$ for both $(\pm 1.h_v)s_v$ Proposition \ref{prop:local-order-split} implies that
\begin{equation*}
\pm \beta_v -\tensor[^\sigma]{\alpha}{_v}\in\Lambda_v \Longrightarrow 2\beta_v\in\Lambda_v \Longrightarrow \beta_v\in\Lambda_v
\end{equation*}
In the last implication we have used once more that the residue characteristic is odd.

Following the definition of $\mathfrak{b}$ in the proof of Proposition \ref{prop:intersection-to-invariants} above we see that the completion of $\mathfrak{b}$ at $v$ is $\mathscr{D}_v \beta_v \Lambda_v\subseteq \mathscr{D}_v \Lambda_v$. Hence $\ord_p \Nr\mathfrak{b}\geq \ord_p D$ in contradiction to the assumption.
\end{proof}

\subsection{Proof of Theorem \ref{thm:cross-correlation-shifted-convolution}}
\label{sec:proof-of-geometric-expansion}
We need one last lemma before we can proceed to the proof the theorem.

\begin{lem}\label{lem:p1-Bowen-measure}
For any $n\in\mathbb{N}$ and any $a\in A_{p_1}$
\begin{equation*}
\frac{\meas_{\mathbf{G}(\mathbb{Q}_{p_1})}\left(K_{p_1}\cap aK_{p_1}a^{-1}\right)}
{\meas_{\mathbf{G}(\mathbb{Q}_{p_1})}\left(K_{p_1}^{(-n,n)}\cap aK_{p_1}^{(-n,n)}a^{-1}\right)}
=p_1^{2n}
\end{equation*}
\end{lem}
\begin{proof}
Let $\mathscr{A}$ be the apartment in $\mathscr{B}_{p_1}$ stabilized by $A_{p_1}$. We fix an orientation on $\mathscr{A}$ and enumerate all vertices in $\mathcal{A}$ consecutively according to the adjacency: $\ldots,x_{-1},x_0,x_1,\ldots$ in such a way that $x_0$ is the vertex stabilized by $K_{p_1}$ and $a.x_0=x_k$ for some $k\geq 0$.

Because $\mathbf{G}(\mathbb{Q}_p)$ acts by simplicial automorphism we have that $K_{p_1}\cap aK_{p_1}a^{-1}$ is the stabilizer in $\mathbf{G}(\mathbb{Q}_p)$ of the finite path $[x_0,\ldots,x_k]$ in $\mathscr{A}$ and $K_{p_1}^{(-n,n)}\cap aK_{p_1}^{(-n,n)}a^{-1}$ is the stabilizer in $\mathbf{G}(\mathbb{Q}_p)$ of the finite path $[x_{-n},\ldots,x_0,\ldots x_k,\ldots,x_{n+k}]$. In particular,
\begin{align*}
&\frac{\meas_{\mathbf{G}(\mathbb{Q}_{p_1})}\left(K_{p_1}\cap aK_{p_1}a^{-1}\right)}
{\meas_{\mathbf{G}(\mathbb{Q}_{p_1})}\left(K_{p_1}^{(-n,n)}\cap aK_{p_1}^{(-n,n)}a^{-1}\right)}\\
&=
\left[
\Stab_{\mathbf{G}(\mathbb{Q}_p)}\left([x_0,\ldots,x_k]\right)
\colon
\Stab_{\mathbf{G}(\mathbb{Q}_p)}\left([x_{-n},\ldots,x_0,\ldots x_k,\ldots,x_{n+k}]\right)
\right]\\
&=
\#\left(\Stab_{\mathbf{G}(\mathbb{Q}_p)}\left([x_0,\ldots,x_k]\right)
.\left([x_{-n},\ldots,x_0]\cup[x_k,\ldots,x_{k+n}]\right)\right)
\end{align*}

Because the action is by simplicial automorphism and $\mathcal{B}_{v_1}$ is a tree for any $h\in \Stab_{\mathbf{G}(\mathbb{Q}_p)}\left([x_0,\ldots,x_k]\right)$ the position of $h.\left([x_{-n},\ldots,x_0]\cup[x_k,\ldots,x_{k+n}]\right)$ is completely determined by the position of $h.x_{-n}$ and $h.x_{k+n}$. 

Let $y_1$ be a vertex so that $d(x_0,y_1)=d(x_0,x_{-n})$ and the geodesic connecting $x_0$ and $y_1$ does not pass through $x_1$. Similarly, let $y_2$ be a vertex so that $d(x_k,y_2)=d(x_k,x_{k+n})$ and  the geodesic connecting $x_k$ and $y_2$ does not pass through $x_{k-1}$. We can use 
the strong transitivity of the action to show the existence of an element $h\in\mathbf{G}(\mathbb{Q}_p)$ so that $h.[x_0,\ldots,x_k]=[x_0,\ldots,x_k]$ and $h.x_{-n}=y_1$ and $h.x_{k+n}=y_2$. Counting pairs of vertices $y_1,y_2$ as above we deduce
\begin{equation*}
\#\left(\Stab_{\mathbf{G}(\mathbb{Q}_p)}\left([x_0,\ldots,x_k]\right)
.\left([x_{-n},\ldots,x_0]\cup[x_k,\ldots,x_{k+n}]\right)\right)= p_1^{2n}
\end{equation*}
\end{proof}

\begin{proof}[Proof of Theorem \ref{thm:cross-correlation-shifted-convolution}]
We begin with some necessary measure computations. For the torus $\mathbf{T}$ we have $g_\infty^{-1}\mathbf{T}(\mathbb{R})g_\infty=K_\infty$ hence $\mathbf{T}(\mathbb{R})\subset g_\infty \Omega_\infty g_\infty^{-1}$. Denote $a=\lambda(p_1)$ as in Definition \ref{def:Bowen-ball} then for every $m\in\mathbb{Z}$
\begin{equation*}
g_{p_1}^{-1}\mathbf{T}(\mathbb{Q}_{p_1})g_{p_1}\cap a^m K_{p_1} a^{-m}=
A_{p_1}\cap a^m K_{p_1} a^{-m} =A_{p_1}\cap K_{p_1} 
\end{equation*}
Hence $\mathbf{T}(\mathbb{Q}_{p_1})\cap g_{p_1}K_{p_1}^{(-n,n)}g_{p_1}^{-1}=\mathbf{T}(\mathbb{Q}_{p_1})\cap g_{p_1}K_{p_1}g_{p_1}^{-1}$
\end{proof}
We conclude that
\begin{equation*}
\vol\left(\left[\mathbf{T}(\mathbb{A})g\right]\right)^{-1}=
\meas_{\mathbf{T}(\mathbb{R})}\left(\mathbf{T}(\mathbb{R})\right)
\meas_{\mathbf{T}(\mathbb{Q}_{p_1})}\left(g_{p_1}K_{p_1}^{(-n,n)}g_{p_1}^{-1}\right)
\prod_{v\neq\infty,p_1}\meas_{\mathbf{T}(\mathbb{Q}_v)}\left(g_v K_v g_v^{-1}\right)
\end{equation*}

For $\mathbf{G}^\Delta$ we use Lemma \ref{lem:p1-Bowen-measure} above and the condition $\ctr(\xi)_{p_1}\in A_{p_1}$  to deduce that
\begin{align*}
\vol&\left(\left[\mathbf{G}^\Delta(\mathbb{A})^+\xi\right]\right)^{-1} p_1^{-2n}\geq\\
&\meas_{\mathbf{G}(\mathbb{R})^+}\left(\xi_{1,\infty}\Omega_\infty^2\xi_{1,\infty}^{-1} \cap \xi_{2,\infty}\Omega_\infty^2\xi_{2,\infty}^{-1}\right)
\meas_{\mathbf{G}(\mathbb{Q}_{p_1})}\left(K_{p_1}^{(-n,n)}\cap  \ctr(\xi)_{p_1} K_{p_1}^{(-n,n)} \ctr(\xi)_{p_1}^{-1}\right)\\
&\cdot
\prod_{v\neq\infty,p_1}\meas_{\mathbf{G}(\mathbb{Q}_v)^+}\left(g_v K_v g_v^{-1}\right)
\end{align*}
The inequality above can be replaced by $\asymp_{\Omega_\infty}$ because of \eqref{eq:vol-Omega-Omega'} from \S\ref{sec:vol}.

Using Corollary \ref{cor:geometric-expansion-final-form}, Lemma \ref{lem:archimedean-RO}, Proposition \ref{prop:nonarchimidean-RO-Ngamma} and the volume computations above we can write
\begin{align*}
\Cor[\mu,\nu](B)
&\ll \vol \left(\left[\mathbf{T}(\mathbb{A})g\right]\right)^{-1}
\vol\left(\left[\mathbf{G}^\Delta(\mathbb{A})^+\xi\right]\right)^{-1} p_1^{-2n}\\
&\sum_{\substack{[\gamma]\in W_{\mathbb{Q}} \\ \psi\det^{-1}(\gamma)\neq 0 }} N_{[\gamma]} \cdot 
\mathbb{1}_{g_\infty \Omega_\infty \ctr(\xi)_\infty \Omega_\infty g_\infty^{-1}}(\ctr(\gamma)_\infty)
\end{align*}

We now need to bound the bottom sum over $[\gamma]\in W_{\mathbb{Q}}$. Using Lemma \ref{lem:contraction-of-Ngamma} we know that the last sum is equal to the number of $\Ad \mathbf{T}(\mathbb{R})\Ad K_{\mathbf{T},f}$-orbits intersecting $\ctr(B')$ in 
the set $\mathbf{G}(\mathbb{A})_\mathrm{accessible}$. To each such intersection we can associate a pair of integral ideals satisfying the conclusions of Proposition \ref{prop:intersection-to-invariants}. 

Due to Propositions \ref{prop:fiber-of-invariants} and \ref{prop:fiber-over-integral-invariants} we know that the map $\textit{intersection}\mapsto (\mathfrak{a}, \mathfrak{b})$ where $\mathfrak{a},\mathfrak{b}$ are the integral ideals of Proposition \ref{prop:intersection-to-invariants} is at most $8 r(\Nr(\mathfrak{b}))$ to $1$.

For each pertinent pair of ideal $(\mathfrak{a},\mathfrak{b})$ set $\mathfrak{b}=p_1^n\mathfrak{b}'$ where $0\neq\mathfrak{b}'\subseteq \Lambda$ is an integral invertible $\Lambda$-ideal satisfying $[\mathfrak{b}']\in [p_1^n\mathfrak{s}\mathfrak{e}]^{-1}\Pic(\Lambda)^2$. The claim follows when we notice that $g_{[\mathfrak{s}]}(x)
f_{[p_1^n\mathfrak{se}]^{-1}}\left(\frac{x-\omega D}{\upsilon p_1^{2n}}\right)$ with $x\leq \kappa |D|$ is exactly the number of pairs of ideals $(\mathfrak{a},p_1^n\mathfrak{b}')$ satisfying the conclusions of Proposition $\ref{prop:intersection-to-invariants}$ and $\Nr(\mathfrak{a})=x$.

\section{Sums of Multiplicative Functions along Two Variables Polynomials}\label{sec:sieving}
In this section we generalize the results of Shiu and Nair \cite{Shiu, Nair} to sums of the form
\begin{equation*}
\sum_{(x,y)\in \mathscr{E}\cap\mathbb{Z}^2} f(Q(x,y))
\end{equation*}
where $f$ is  a slowly growing non-negative multiplicative function, $Q\in\mathbb{Z}[x,y]$ and $\mathscr{E}\subsetneq\mathbb{R}^2$ is a closed smooth convex domain. Similar sums for homogeneous two-variable polynomials over axis-aligned boxes have been studied in \cite{BretecheBrowning, BretecheTenenbaum}.

Most of the proof in \cite{Shiu, Nair} follows through in higher dimensions even for the case of more general domains as long as good estimates are available for the lattice counting problem.

Nevertheless, the following presentation contains two ideas which seem to be novel even in the $1$-variable case. They are essential when we need to apply the sieved upper bound to a family of polynomials $Q$ in a uniform manner. Both of them have to do with the behavior of $Q$ at primes of bad reduction.

The first one is a simple yet crucial observation that the function counting $\cyclic{p^k}$-points on $X_Q$ -- the plane curve cutout by $Q$ -- can be 
replaced by a function counting only the points that do not have the maximal possible amount of lifts to $X_Q\left(\cyclic{p^{k+1}}\right)$. 

The second one is related directly to the dependence of the upper bound on the singularities of the reduction of $X_Q$ modulo $p$. The structure of singularities for a $1$-variable polynomial modulo $p$, i.e.\ $0$-dimensional affine scheme of finite type, is simple and can be summarized by the discriminant of the polynomial. The possible singularities of a reduction of a curve, although they are rather well-understood through resolution of singularities, they are significantly more diverse. The most general expression replacing the dependence on the discriminant in $2$-variables seems to be a product of values of local Igusa zeta-functions. We chose not to pursue this path here as it does not lends itself easily to applications. Instead, we observe that as long as there is an \emph{a priori} bound $\#X_Q\left(\cyclic{p^k}\right)\leq Cp^{k(2-r)}$ with $C>0$ and $1 \geq r >0$ independent of $p^k$, our upper bound can be shown to depend only on $C$ and $r$. Such bounds seem to be easy to establish explicitly, at least for the application at hand. Moreover, this approach generalizes verbatim to polynomials with arbitrary many variables.
\begin{defi}\label{def:polynomial-rho}
\hfill
\begin{enumerate}
\item 
For any polynomial in two-variables $Q\in\mathbb{Z}[x,y]$ and $a\in\mathbb{N}$ denote by $\rho_Q(a)$ the number of solution in $\faktor{\mathbb{Z}^2}{a\mathbb{Z}^2}$ to the equation
\begin{equation*}
Q(x,y)\equiv 0 \mod a
\end{equation*}

\item
Let $X_Q$ be the affine plane curve cutout by $Q$, i.e.\
\begin{equation*}
X_Q\coloneqq \Spec \mathbb{Z}[x,y]/\left<Q(x,y)\right>
\end{equation*}
By definition $\rho_Q(a)=\left|X_Q\left(\cyclic{a}\right)\right|$. 

\item
Fix a prime power $p^k$. A lift of a point $x\in X_Q\left(\cyclic{p^k}\right)$ is a
 $\cyclic{p^{k+1}}$-point of $X_Q$ which reduces to $x$ mod $p^k$.

We split $X_Q\left(\cyclic{p^k}\right)$ into three types of points
\begin{itemize}
\item smooth points, by Hensel's lemma each such point has exactly $p$ lifts;
\item singular points with a lift, by the Taylor polynomial formula each such point has exactly $p^2$ lifts;
\item singular points without a lift.
\end{itemize}

\item
We denote by $\widetilde{\rho}_Q(p^k)$ the number of $\cyclic{p^k}$-points on $X_Q$ which are either smooth or have no lift. We extend $\widetilde{\rho}_Q$ to a multiplicative function on $\mathbb{N}$ in the regular fashion. Obviously $\widetilde{\rho}_Q(a)\leq \rho_Q(a)$ for all $a$.

\item  Denote by $\rho^\mathrm{sing}_Q(p^k)$ the number of $\cyclic{p^k}$-points on $X_Q$ which are not smooth, i.e.\ either having $0$ or $p^2$ lifts.
\end{enumerate}
\end{defi}

The following lemmata are elementary properties of points on curves over congruence classes of integers.

\begin{lem}[DeMillo-Lipton-Schwartz-Zippel Lemma] \label{lem:Schwartz-Zippel}
Let $0\neq Q\in\mathbb{Z}[x,y]$ then the inequality $\rho_Q(p)\leq \deg(Q) p$ holds for any prime $p\nmid Q$.
\end{lem}
\begin{proof}
This has been proven in \cite{Schwartz} and a slightly weaker bound has been shown in \cite{DeMilloLipton, Zippel}. See \cite[Lemma 1.2]{TaoPolynomial} for a streamlined proof.
\end{proof}

\begin{lem}\label{lem:rho-tilderho-formula}
Let $0\neq Q\in\mathbb{Z}[x,y]$ then \begin{equation*}
\widetilde{\rho}_Q(p^k)=\rho_Q(p^k)-\frac{\rho^\mathrm{sing}(p^{k+1})}{p^2}
\end{equation*}
\end{lem}
\begin{proof}
This follows immediately from the observation that the points in $X_Q\left(\cyclic{p^{k+1}}\right)$ that reduce to smooth points modulo $p^n$ are exactly the smooth points modulo $p^{n+1}$.
\end{proof}

In the following two definitions we describe the objects appearing in our main sieving theorem.

\begin{defi}\hfill
\begin{itemize}
\item We say that a convex domain $\mathscr{E}\subset \mathbb{R}^2$ is $C^2$ if its boundary is a twice continuously differentiable curve. We then denote by $R_{\max}(\mathscr{E})$ the maximum of the radius of curvature of the boundary of $\mathscr{E}$ and by $A(\mathscr{E})$ the area of $\mathscr{E}$. If no confusion arises we shall use the shorthand $R_{\max}$ for $R_{\max}(\mathscr{E})$. 

\item 
Let $C_l,\theta_l>0$.
We denote by $\mathcal{L}(C_l,\theta_l)$ the collection of $C^2$ convex planar domains $\mathscr{E}$ such that for any $a\in\mathbb{N}$ and $(x_0,y_0)\in\mathbb{Z}^2$ if $A(\mathscr{E})\geq a^2$ then
\begin{equation*}
\left|\#\left(a^{-1}(\mathscr{E}-(x_0,y_0))\cap \mathbb{Z}^2\right)
-a^{-2}A(\mathscr{E})\right|\leq 
C_l\left(R_{\max}(\mathscr{E})/a\right)^{\theta_l}
\end{equation*}
\end{itemize}
\end{defi}
\begin{remark}
The Van der Corput bound \cite{vanDerCorput} implies that for any $\varepsilon>0$ there is $C_l>0$ depending on $\varepsilon$ such that any $C^2$ convex domain belongs to $\mathcal{L}(C_l,2/3+\varepsilon)$.

The bound of Huxley \cite[Proposition 5 and Theorem 5]{Huxley} for lattice points in $C^3$ planar domains implies that for any $\varepsilon>0$ there is $C_l>0$ depending on $\varepsilon$ such that all ellipses belong to 
$\mathcal{L}(C_l,131/208+\varepsilon)$. A suitable generalization of the Gauss circle problem conjecture should imply that the constant $131/208$ can be replaced by $1/2$, at least for ellipses defined by integral binary quadratic forms.
\end{remark}

\begin{defi}
Let $A\geq 1$ and $B,\varepsilon>0$.
We say that a multiplicative function $f\colon\mathbb{N}\to\mathbb{R}$ is of class $\mathcal{M}(A,B,\varepsilon)$ if it is non-negative and for any integer $n>0$
\begin{equation*}
f(n)\leq \min\left(A^{\Omega(n)},Bn^{\varepsilon}\right)
\end{equation*} 
\end{defi}

\begin{thm}\label{thm:2d-sieve}
Let $\mathscr{E}\subset \mathbb{R}^2$ be planar domain of class $\mathcal{L}(C_l,\theta_l)$ for $C_l,\theta_l>0$. Denote by $A(\mathscr{E})$ the area of $\mathscr{E}$, let $R_{\max}$ be the maximal radius of curvature of the boundary and assume  $R_{\max}^{\theta_l}\leq A(\mathscr{E})^{1-3\eta}$ for some $1/2>\eta>0$.

Let $Q\in\mathbb{Z}[x,y]$ such that there are $C>0, 1\geq r> 0$ satisfying $\widetilde{\rho}_Q(p^k)\leq C p^{k(2-r)}$ for all prime powers $p^k$. Let $X\geq 1$ be a constant satisfying 
\begin{equation*}
\max \left\{|Q(x,y)| \mid (x,y)\in \mathscr{E} \cap \mathbb{Z}^2 \right\} \leq X \leq  A(\mathscr{E})^\delta
\end{equation*}
for some $\delta>0$.

Let $f$ be a non-negative multiplicative function of class $\mathcal{M}(A,B,\varepsilon)$ for some $A\geq 1$, $B>0$ and $0<\varepsilon< \min\left\{r,\eta r/(4\delta) \right\}$. 
Then
\begin{equation}\label{eq:sieve-thm-main}
\sum_{(x,y)\in \mathscr{E}\cap\mathbb{Z}^2} f\left(Q(x,y)\right)\ll 
A(\mathscr{E})
\prod_{\substack{\deg(Q)<p\leq X \\ p\nmid Q}} \left(1-\frac{\rho_Q(p)}{p^2}\right)
\sum_{a\leq X} \frac{f(a)\widetilde{\rho}_Q(a)}{a^2}
\end{equation}
where the implicit constant depends only on $C_l,A,B,\varepsilon,\deg(Q),C,r,\eta,\delta$.
\end{thm}

\subsection{Notations}
We introduce several notations to be used in the section.

For an integer $n>1$ denote by $\omega(n)$ the number of distinct prime factors of $n$ and let $\Omega(n)$ be the number of prime factors counted with multiplicity. Denote also by
$P^+(n), P^-(n)$ the largest and the smallest prime divisor of $n$ respectively. It shall also be useful to define $P^+(1)=1, P^-(1)=\infty$.

For any two integers $a,b$ we write $a\mid b^\infty$ if the prime support of $a$ is contained in the prime support of $b$. Lastly, we denote by $\gcd(a,b^\infty)$ the product of all primes powers dividing $a$ for primes appearing in the support of $b$.

\subsection{Sieving}
The following lemma is a straightforward generalization of the lower bound in \cite[Lemma 2.1]{GKM}. 
\begin{lem}\label{lem:log-mean-mertens}
Let $g\colon \mathbb{N} \to \mathbb{R}$ be a multiplicative function such that there is some $d>0$ so that $0\leq g(p)\leq d$ for all primes $p$. Then for any $z>1$
\begin{equation}\label{eq:log-mean-mertens-ineq}
\sum_{n\leq z} \mu(n)^2 \frac{g(n)}{n} \gg_d \prod_{d<p\leq z} \left(1-\frac{g(p)}{p}\right)^{-1}
\end{equation}
\end{lem}
\begin{proof}
As the left hand side of \eqref{eq:log-mean-mertens-ineq} is supported on the square-free numbers we can assume without loss of generality that $g$ is completely multiplicative. Define a new completely multiplicative function $h$ by $h(p)=d-g(p)$. The Dirichlet convolution $g*h$ is a multiplicative function which satisfies $g*h(p)=g(1)h(p)+g(p)h(1)=d$ for any prime $p$. Hence for any square-free integer $n$ we have $g*h(n)=d^{\omega(n)}$. This implies
\begin{equation}\label{eq:f-h-log-mean-product}
\sum_{n\leq z} \mu(n)^2 \frac{g(n)}{n} \cdot \sum_{n\leq z} \mu(n)^2 \frac{h(n)}{n}
\geq \sum_{n\leq z} \mu(n)^2 \frac{d^{\omega(n)}}{n} \gg_d (\log z)^d
\end{equation}

On the other hand
\begin{equation*}
 \sum_{n\leq z} \mu(n)^2 \frac{h(n)}{n} \leq \prod_{p\leq z} \left(1+\frac{h(p)}{p}\right) \ll_d \prod_{d<p\leq z} \left(1-\frac{h(p)}{p}\right)^{-1}
\end{equation*}
Hence
\begin{align}
\left(\sum_{n\leq z} \mu(n)^2 \frac{h(n)}{n}\right)^{-1} &\cdot \prod_{d<p\leq z}\left(1-\frac{g(p)}{p}\right)
\gg_d \prod_{d<p\leq z} \left(1-\frac{d-g(p)}{p}\right)\left(1-\frac{g(p)}{p}\right)
\nonumber\\
&= \prod_{d<p\leq z} \left(1-\frac{d}{p}+\frac{dg(p)-g(p)^2}{p^2}\right)
\geq \prod_{d<p\leq z} \left(1-\frac{d}{p}\right)
\gg_d (\log z)^{-d}
\label{eq:inv-log-mean-h}
\end{align}
The claim now follows by multiplying inequality \eqref{eq:f-h-log-mean-product} by \eqref{eq:inv-log-mean-h}.
\end{proof}

The following result is where we apply a sieve. As we require only upper bounds we use the large sieve due to its great generality.
\begin{lem}\label{lem:large-sieve-applied}
Let $Q\in\mathbb{Z}[x,y]$ be a polynomial. Let $\mathscr{E}\subset \mathbb{R}^2$ be a domain of class $\mathcal{L}(C_l, \theta_l)$. If 
\begin{equation*}
1\leq z\leq \min \left\{\left(\frac{A(\mathscr{E})}{R_{\max}^{\theta_l}}\right)^{1/5},A(\mathscr{E})^{1/2}\right\}
\end{equation*}
then
\begin{equation*}
S\coloneqq\left|\left\{(x,y)\in \mathscr{E}\cap\mathbb{Z}^2 \mid P^-(Q(x,y))\geq z \right\}\right|\ll_{\deg(Q),C_l} A(\mathscr{E})\prod_{\deg(Q)<p\leq z} \left(1-\frac{\rho_Q(p)}{p^2}\right)
\end{equation*}
\end{lem}
\begin{remark}
The exponent $1/5$ in the level of distribution is certainly far from optimal yet for our application any positive exponent suffices. 
\end{remark}

\begin{proof}
The inequality is trivially true if there is $p\mid Q$ such that $p\leq z$; hence we assume this is not the case.

We use the large sieve in the setup of Kowalski \cite{Kowalski}. Our sieve setting is 
$$\left(\mathbb{Z}^2,\mathrm{primes},\mathbb{Z}^2\to \faktor{\mathbb{Z}^2}{p \mathbb{Z}^2}\right)$$
and the siftable set is $\mathscr{E}_{\mathbb{Z}}\coloneqq \mathscr{E}\cap \mathbb{Z}^2 \subset \mathbb{Z}^2$ with the counting measure.

We choose our sieve support to be the set of square-free positive integers $\leq z$.
The large sieve inequality as presented in \cite[Proposition 2.3]{Kowalski} implies that
\begin{align}
S&\leq \Delta H^{-1} \label{eq:large-sieve}
\\
H&\coloneqq \sum_{n\leq z } \mu(n)^2 \prod_{p|n}\frac{\rho_Q(p)}{p^2-\rho_Q(p)} \notag
\end{align}
Where $\Delta$ is the large sieve constant which we bound from above using the equidistribution method as in \cite[\S2.13]{Kowalski}. Define for any integer $n\leq A(\mathscr{E})^{1/2}$ and any $y\in\faktor{\mathbb{Z}^2}{n\mathbb{Z}^2}$ the discrepancy 
\begin{equation*}
r_n(y)\coloneqq \left|\mathscr{E}\cap \left(n\mathbb{Z}^2+y\right)\right|
-n^{-2}|\mathscr{E}_{\mathbb{Z}}|
\end{equation*}
The assumption $\mathscr{E}\in\mathcal{L}(C_l,\theta_l)$ implies for $\mathscr{E}$ and $n^{-1}\left(\mathscr{E}-y\right)$, whose areas are $\geq 1$, that
\begin{equation*}
|r_n(y)| \leq C_l\left(\frac{R_{\max}}{n}\right)^{\theta_l}
\end{equation*}

We use the orthonormal base of characters for finite abelian groups and bound $\Delta$ using \cite[Corollary 2.13]{Kowalski}
\begin{align*}
\Delta-|\mathscr{E}_{\mathbb{Z}}|&\leq \max_{m\leq z}\sum_{n\leq z} \sum_{y\in \faktor{\mathbb{Z}^2}{[m,n]\mathbb{Z}^2}} n\left|r_{[m,n]}(y)\right|
\leq C_l \max_{m\leq z}\sum_{n\leq z} n [m,n]^2 
 \left(\frac{R_{\max}}{[m,n]}\right)^{\theta_l}\\
&\leq  C_l R_{\max}^{\theta_l}z^2\sum_{n\leq z} n^{2} 
\ll C_l R_{\max}^{\theta_l}z^5\leq C_l A(\mathscr{E})\\
\end{align*}
where in the last inequality we have used the upper bound assumption on $z$.
Applying the lattice count bound to $|\mathscr{E}_{\mathbb{Z}}|$ we deduce that $\Delta\ll_{C_l} A(\mathscr{E})$.

We bound $H$ below by
\begin{equation*}
H\geq
\sum_{n\leq z } \mu(n)^2 \prod_{p|n}\frac{\rho_Q(p)}{p^2}
\end{equation*}
Next we apply Lemma \ref{lem:log-mean-mertens} to the multiplicative function $\rho_Q(n)/n$ which is bounded by Lemma \ref{lem:Schwartz-Zippel} to deduce
\begin{equation*}
H \gg_{\deg(Q)} \prod_{\deg(Q)<p\leq z} \left(1-\frac{\rho_Q(p)}{p^2}\right)^{-1}
\end{equation*}

The claim follows by combining the bounds on $H$ and $\Delta$ with \eqref{eq:large-sieve}.
\end{proof}

\subsection{Extending the Level of Distribution}
The range of $z$ where the lemma above is applicable is very restricted. The following results show that we can actually take this range to be any power of $A(\mathscr{E})$ if we are willing to pay a price in the constant depending only on the exponent.

\begin{lem}\label{lem:extending-p-product}
Let $Q\in\mathbb{Z}[x,y]$ be a polynomial then for any $z\geq 1$, $s>0$
\begin{equation*}
\prod_{\deg(Q) <p\leq z^{1/s}} \left(1-\frac{\rho_Q(p)}{p^2}\right) \ll_{\deg(Q)} s^{\deg(Q)} \prod_{\substack{\deg(Q) <p\leq z \\ p\nmid Q}} \left(1-\frac{\rho_Q(p)}{p^2}\right)
\end{equation*}
\end{lem}
\begin{proof}
The proof of \cite[Lemma 2(i)]{Nair} applies when combined with Lemma \ref{lem:Schwartz-Zippel}.
\end{proof}

\begin{lem}\label{lem:large-sieve-extended}
Let $5\geq\eta>0$ and $\varsigma_0>0$.
In the setting of Lemma \ref{lem:large-sieve-applied} above if $R_{\max}^{\theta_l}\leq A(\mathscr{E})^{1-\eta}$  and $1\leq z \leq A(\mathscr{E})^{\varsigma_0}$ then 
\begin{equation*}
S\ll_{\deg(Q),C_l,\eta,\varsigma_0} A(\mathscr{E})\prod_{\deg(Q)<p\leq z} \left(1-\frac{\rho_Q(p)}{p^2}\right)
\end{equation*}
\end{lem}
\begin{proof}
The statement is trivial $\exists p \leq z$ such that $p\mid Q$, hence assume the contrary.

Assume $z>A(\mathscr{E})^{\eta/5}$ as otherwise Lemma \ref{lem:large-sieve-applied} applies directly.
Applying Lemma \ref{lem:large-sieve-applied} for $z_0=A(\mathscr{E})^{\eta/5}\leq\min\left\{\left(\frac{A(\mathscr{E})}{R_{\max}^{\theta_l}}\right)^{1/5},A(\mathscr{E})^{1/2}\right\}$ we deduce that
\begin{equation*}
S \ll_{\deg(Q),C_l} A(\mathscr{E})\prod_{\deg(Q)<p\leq A(\mathscr{E})^{\eta/5}} \left(1-\frac{\rho_Q(p)}{p^2}\right)
\end{equation*}
and the claim follows from Lemma \ref{lem:extending-p-product} with $s=5\varsigma_0/\eta$.
\end{proof}

\subsection{The Sieve Bound for Values in a Homogeneous Arithmetic Progressions}
We now generalize these results to subsets of points where the polynomial value is divisible by a fixed integer. 
\begin{lem}\label{lem:large-sieve-a}
Let $Q\in\mathbb{Z}[x,y]$ be a polynomial. Let $\mathscr{E}\subset \mathbb{R}^2$ be a domain of class $\mathcal{L}(C_l, \theta_l)$. Fix $1/2>\eta>0$, $\varsigma>0$ and assume $R_{\max}^{\theta_l}\leq A(\mathscr{E})^{1-3\eta}$. Then for any $a,z\in\mathbb{N}$ such that $a\leq A(\mathscr{E})^\eta$ and $1\leq z\leq A(\mathscr{E})^{\varsigma}$
\begin{align*}
S\coloneqq&
\left|\left\{(x,y)\in \mathscr{E}\cap\mathbb{Z}^2 \mid a|Q(x,y),\; \gcd(a, Q(x,y)/a)=1 \; \mathrm{and}\; P^-(Q(x,y)/a)\geq z  \right\}\right|\\
&\ll_{\deg(Q),C_l,\eta,\varsigma} \frac{A(\mathscr{E})\widetilde{\rho}_Q(a)}{a^2}\prod_{\substack{\deg(Q)<p\leq z \\ p \nmid a}} \left(1-\frac{\rho_Q(p)}{p^2}\right)
\end{align*}
\end{lem}
\begin{proof}
The statement is trivial if there is $p\leq z$ such that $p\nmid a$ and $p\mid Q$ so we assume the contrary.

Let $(x_0,y_0)\in\faktor{\mathbb{Z}^2}{a\mathbb{Z}^2}$ be one of the $\rho_Q(a)$ classes where $Q$ vanishes modulo $a$. Define $Q_0\in\mathbb{Z}[x,y]$ by 
\begin{equation*}
Q(ax+x_0,ay+y_0)=a\cdot Q_0(x,y)
\end{equation*}
Then for any $p\nmid a$ we have $\rho_{Q_0}(p)=\rho_{Q}(p)$. Notice that if $p^k\parallel a$ and the point $(x_0,y_0)\in X_Q\left(\cyclic{p^k}\right)$ has $p^2$ lifts then $p\mid Q_0$. By the assumption $\gcd(a, Q(x,y)/a)=1$ no point in the sieved set reduces to such $(x_0,y_0)$. Hence it is sufficient to consider only the $\widetilde{\rho}_Q(a)$ classes of points which at all primes are either smooth or have not lift.

We apply Lemma \ref{lem:large-sieve-extended} to $Q_0$ and the convex $C^2$ domain $a^{-1}(\mathscr{E}-(x_0,y_0))$ with area $A(\mathscr{E})/a^2$ and maximal curvature radius $R_{\max}/a$. We take $\varsigma_0=\frac{\varsigma}{1-2\eta}>0$.

 The first condition of Lemma \ref{lem:large-sieve-extended} reads
\begin{equation}\label{eq:a-sieve-condition1}
(R_{\max}/a)^{\theta_l}\leq A(\mathscr{E})^{1-\eta}/a^{2-2\eta}\Leftrightarrow
R_{\max}^{\theta_l}\leq A(\mathscr{E})^{1-\eta}/a^{2-2\eta-\theta_l}
\end{equation}
Using the assumption $a\leq A(\mathscr{E})^\eta$ we deduce that
\begin{equation*}
A(\mathscr{E})^{1-\eta}/a^{2-2\eta-\theta_l}\geq 
A(\mathscr{E})^{1-\eta-\eta(2-2\eta-\theta_l)}=
A(\mathscr{E})^{1-\eta(3-2\eta-\theta_l)}\geq A(\mathscr{E})^{1-3\eta}
\end{equation*}
Thus \eqref{eq:a-sieve-condition1} is satisfied because of the assumption $R_{\max}^{\theta_l}\leq A(\mathscr{E})^{1-3\eta}$.
The second condition of Lemma \ref{lem:large-sieve-extended} reads
\begin{equation}\label{eq:a-sieve-condition2}
1\leq z\leq \left(\frac{A(\mathscr{E})}{a^2}\right)^{\varsigma_0}
\end{equation}
Using the assumption $a\leq A(\mathscr{E})^\eta$ we see that
\begin{equation*}
\left(\frac{A(\mathscr{E})}{a^2}\right)^{\varsigma_0}\geq A(\mathscr{E})^{(1-2\eta)\varsigma_0}=A(\mathscr{E})^\varsigma
\end{equation*}
Hence condition \eqref{eq:a-sieve-condition2} is satisfied as well.

Summing the bounds we obtain from applying Lemma \ref{lem:large-sieve-extended} to each of the relevant $\widetilde{\rho}_Q(a)$ residue class and using the fact $\rho_{Q_0}(p)=\rho_{Q}(p)$ for each $p \nmid a$ we obtain the claimed inequality.
\end{proof}

\begin{corprf}\label{cor:large-sieve-a-nogcd}
In the setting of Lemma \ref{lem:large-sieve-a} above the following holds
\begin{align*}
\big|\big\{(x,y)\in \mathscr{E}\cap\mathbb{Z}^2 &\mid a|Q(x,y) \; \mathrm{and}\; P^-(Q(x,y)/a)\geq z  \big\}\big|\\
&\ll_{\deg(Q),C_l,\eta,\varsigma} 
\frac{A(\mathscr{E})\rho_Q(a)}{a^2}\prod_{\substack{\deg(Q)<p\leq z \\ p \nmid a}} \left(1-\frac{\rho_Q(p)}{p^2}\right)
\end{align*}
\end{corprf}
\begin{proof}
The only place where the condition $\gcd(a,Q(x,y)/a)=1$ was used is to dispose of the residue classes which do not lift. Hence the proof follows in the same manner except that $\widetilde{\rho}_Q(a)$ is replaced by $\rho_Q(a)$.
\end{proof}

\begin{defi}
For any $Q\in\mathbb{Z}[x,y]$ define the multiplicative function $\theta_Q$ by 
\begin{equation*}
\theta_Q(p^k)=\begin{cases}
1+2\frac{\rho_Q(p)}{p^2} & p\nmid Q\\
1 &p\mid Q
\end{cases}
\end{equation*}
for all primes $p$ and all integers $k\geq 1$.

Write $\theta_Q=1*\lambda_Q$, then by M\"{o}bius inversion 
\begin{equation*}
\lambda_Q(n)=\begin{cases}
\mu(n)^2\frac{2^{\omega(n)}\rho_Q(n)}{n^2} & \gcd(n,Q)=1\\
0 & \gcd(n,Q)\neq 1
\end{cases}
\end{equation*}
\end{defi}
\begin{remark}\label{rem:fthetaq-order}
If $f\in\mathcal{M}(A,B,\varepsilon)$ then an easy computation shows that for any $\varepsilon'>0$
 $$f\theta_Q\in\mathcal{M}\left(A',B',\varepsilon+\varepsilon' \right)$$ with $A'\ll_{\deg(Q)} A$ and $B'\ll_{\varepsilon',\deg(Q)} B$. 
\end{remark}

\begin{cor}\label{cor:large-sieve-a-theta}
The following inequality holds in the setting of Lemma \ref{lem:large-sieve-a} above
\begin{equation*}
S\ll_{\deg(Q),C_l,\eta,\varsigma} \frac{A(\mathscr{E}) \widetilde{\rho}_Q(a)}{a^2} \theta_Q(a)
\prod_{\substack{\deg(Q)<p\leq z\\ p\nmid \gcd(Q,a)}} \left(1-\frac{\rho_Q(p)}{p^2}\right)
\end{equation*}
\begin{proof}
Follows immediately from Lemma \ref{lem:large-sieve-a} and the fact that $0\leq \frac{\rho_Q(p)}{p^2}\leq \frac{1}{2}$ for any $p\geq 2\deg(Q)$, $p\nmid Q$.
\end{proof}
\end{cor}

\subsection{Decoupling Multiplicative Functions}
The following lemma is standard, if not in form then in function. Although it is not singled as such, this is a main technical tool in \cite{NairTenenbaum}.
\begin{lem}\label{lem:multiplicative-decoupling}
Let $g$ and $\psi$ be non-negative multiplicative functions and denote $h=1*\psi$. Then for any $z\geq1$
\begin{align*}
\sum_{a\leq z} g(a) h(a)
&\leq 
\mathfrak{M}_z(g,\psi) \cdot \sum_{a\leq z} g(a)\\
\mathfrak{M}_z(g,\psi)&\coloneqq \prod_{p\leq z}\left[1+ \sum_{v=1}^{\log z /\log p} \psi(p^v)\sum_{j=v}^{\infty} g(p^j)\right]
\end{align*}
\end{lem}
\begin{remark}\label{rem:multiplicative-decoupling}
We can write an upper bound for $\mathfrak{M}_z(g,\psi)$ using $h$ 
\begin{equation*}
\mathfrak{M}_z(g,\psi)\leq \prod_{p\leq z} \sum_{j=0}^{\infty} g(p^j) \sum_{v=0}^j\psi(p^v) =
\prod_{p\leq z} \sum_{j=0}^{\infty} g(p^j) h(p^j)
\end{equation*}
\end{remark}
\begin{proof}
First we expend $h$ in terms of $\psi$
\begin{equation*}
\sum_{a\leq z} g(a) h(a)
=\sum_{a\leq z} \sum_{k\mid a} \psi(k) g(a)
=\sum_{k\leq z} \sum_{t\leq z/k} \psi(k) g(kt)
\end{equation*}
We decompose each $t$ in the sum above as as $t=ln$ where $l=\gcd(t,k^\infty)$ then the expression above is equal to
\begin{equation*}
\sum_{k\leq z} \sum_{t\leq z/k} \psi(k) g(kl)g(n)
\leq \left(\sum_{k\leq z}\psi(k) \sum_{l\mid k^\infty} g(kl)\right) \sum_{n\leq z} g(n)
\end{equation*}

To complete the proof we bound the double sum in the scopes.
\begin{equation*}
\sum_{k\leq z}\psi(k) \sum_{l\mid k^\infty}  g(kl)
\leq
\prod_{p\leq z}\left[1+ \sum_{v=1}^{\log z/\log p} \psi(p^v)\sum_{j=v}^{\infty} g(p^j)\right]
\end{equation*}
\end{proof}

We use the decoupling lemma above to prove the two key results to be used in the proof of Theorem \ref{thm:2d-sieve}. 
The first one shows that on average the product over primes in Lemma \ref{lem:large-sieve-a} can be extended to include primes dividing $a$.
\begin{lem}\label{lem:large-sieve-a-averaged}
Let $Q\in\mathbb{Z}[x,y]$ such that there are $C>0, 1\geq r> 0$ satisfying $\widetilde{\rho}_Q(p^k)\leq C p^{k(2-r)}$ for all prime powers $p^k$.
Let $f$ be a non-negative multiplicative function such that $f(n)<B n^\varepsilon$ for some $B>0$, $r>\varepsilon>0$ and all $n$.  
Then for any $z \geq 1$ 
\begin{equation*}
\sum_{a\leq z} \frac{\widetilde{\rho}_Q(a)f(a)}{a^2} \theta_Q(a)
\ll_{\deg(Q),B,C,r,\varepsilon}
\sum_{a\leq z} \frac{\widetilde{\rho}_Q(a)f(a)}{a^2}
\end{equation*}
\end{lem}
\begin{proof}
Apply Lemma \ref{lem:multiplicative-decoupling} to  $g(a)\coloneqq f(a)\widetilde{\rho}_Q(a)/a^2\geq0$ and $\psi=\lambda_Q$.
To complete the proof we need bound $\mathfrak{M}_z(g,\lambda_Q)$. 

Using the assumptions we bound
$f(p^k)\widetilde{\rho}_Q(p^k)\leq BCp^{k(2-r+\varepsilon)}$ and $f(p^2)\rho_Q(p^2)\leq f(p^2)\rho_Q(p)p^2\leq A^2\deg(Q)p^3$.
Because $\lambda_Q$ is supported on the square-free integers we see that
\begin{align*}
\mathfrak{M}(g,\lambda_Q)
&\leq
\prod_{p\leq z} \left[1+\sum_{v\geq 1} \lambda_Q(p^v)\sum_{j=v}^{\infty} g(p^j) \right]
\leq
\prod_{\substack{p\leq z \\ p\nmid Q}} \left(1+2\frac{\rho_Q(p)}{p^2}\sum_{j=1}^{\infty}\frac{BC}{p^{j(r-\varepsilon)}} \right)\\
&\leq \prod_{p< \infty} \left(1+2\deg(Q)BC\frac{1}{p^{1+r-\varepsilon}}\frac{1}{1-p^{-(r-\varepsilon)}}\right)
\ll_{\deg(Q),B,C,r-\varepsilon}1
\end{align*}
\end{proof}

The second result to use the Lemma \ref{lem:multiplicative-decoupling} provides a saving for the pertinent sums over smooth integers satisfying $P^+(a)\leq z^{1/s}$ for a fixed $s>0$. The crux of the following lemma is that it saves an exponent in $s$, and it will be applied to control a term that grows geometrically in the parameter $s$.
\begin{lem}\label{lem:sieved-sum-exp-save}
Let $Q\in\mathbb{Z}[x,y]$ such that there are $C>0, 1\geq r> 0$ satisfying $\widetilde{\rho}_Q(p^k)\leq C p^{k(2-r)}$ for all prime powers $p^k$.

Let $f$ be a non-negative multiplicative function of class $\mathcal{M}(A,B,\varepsilon)$ with $0<\varepsilon<r$. Then for any $\alpha,s>0$, $z>1$ and $(r-\varepsilon)\log(z)/(2s) \geq \kappa >0$
\begin{equation}\label{eq:P+-exponential-save}
\sum_{\substack{z^\alpha \leq a \leq z \\ P^+(a)\leq z^{1/s}}} \frac{\widetilde{\rho}_Q(a)f(a)}{a^2} 
\ll_{\kappa, A, B,\varepsilon,C,r,\deg(Q)}
\exp(-s\alpha\kappa) \sum_{a \leq z } \frac{\widetilde{\rho}_Q(a)f(a)}{a^2}
\end{equation}
\end{lem}
\begin{remark}
If the curve cutout by $Q$ is smooth over $\mathbb{Q}$ then by Hensel's lemma we have $\rho_Q(p^k)\leq \rho_{Q}(p)p^{k-1}\leq \deg(Q)p^{k(2-1)}$ for all primes $p$ of good reduction. In this case the constants $C,r$ only depend on the number of points on the curve modulo powers of primes of bad reduction.
\end{remark}

\begin{proof}
Let $(r-\varepsilon)/2\geq\beta\coloneqq\frac{\kappa s}{\log z}>0$ then the left hand side of \eqref{eq:P+-exponential-save} is bounded above by
\begin{equation}\label{eq:rho-f-beta}
z^{-\alpha\beta}\sum_{\substack{z^\alpha \leq  a \leq z \\ P^+(a)\leq z^{1/s}}} \frac{\widetilde{\rho}_Q(a)f(a)}{a^2}a^\beta
\leq 
z^{-\alpha\beta}\sum_{\substack{a \leq z \\ P^+(a)\leq z^{1/s}}} \frac{\widetilde{\rho}_Q(a)f(a)}{a^2}a^\beta
\end{equation}
Define the non-negative multiplicative function $g$ by $g(a)=f(a)\widetilde{\rho}_Q(a)/a^2$ if $P^+(a)\leq z^{1/s}$ and $g(a)=0$ otherwise. Let $\psi$ be the M\"obius inversion of the multiplicative function $a\mapsto a^\beta$. An explicit formula for $\psi$ is 
\begin{equation*}
\psi(p^k)=p^{\beta k}-p^{\beta (k-1)}
\end{equation*}
for all primes $p$ and all $k\geq 0$.

Applying Lemma \ref{lem:multiplicative-decoupling} with $g$ and $\psi$ as above we can bound the right hand side of \eqref{eq:rho-f-beta} above by
\begin{equation*}
z^{-\alpha \beta} \cdot \mathfrak{M}_{z}(g,\psi) \sum_{\substack{ a \leq z \\ P^+(a)\leq z^{1/s}}} \frac{\widetilde{\rho}_Q(a)f(a)}{a^2}
\end{equation*}

We now wish to estimate $\mathfrak{M}_{z}(g,\psi)$.
Using the fact that $g$ is supported on integers without prime factors bigger then $z^{1/s}$ and Remark \ref{rem:multiplicative-decoupling} we deduce
\begin{align}
\mathfrak{M}_{z}(g,\psi)
&\leq
\prod_{p\leq z^{1/s}}\left(1+ \sum_{j=1}^{\infty}g(p^j) p^{j\beta}  \right)
\label{eq:M-exp-prod}
\end{align}

Let $K_0\coloneqq\lceil 4/(r-\varepsilon) \rceil>0$.
For any $k\leq K_0$ we estimate $\widetilde{\rho}_Q(p^k)\leq \rho_Q(p)p^{2k-2}\leq \deg(Q)p^{2k-1}$. For $k> K_0$ we use the assumption $\widetilde{\rho}_Q(p^k)\leq Cp^{k(2-r)}$. Combined with inequality \eqref{eq:M-exp-prod} this implies
\begin{equation}\label{eq:M-exp-prod-2}
\log \mathfrak{M}_{z}(g,\psi) \leq K_0 A^{K_0}\deg(Q)\sum_{p\leq z^{1/s}}\frac{p^{K_0\beta}}{p}
+BC\sum_{p\leq z^{1/s}} \sum_{j=K_0+1}^\infty p^{-j(r-\beta-\varepsilon)}
\end{equation}
We bound the second summand above using the inequality $r-\beta-\varepsilon\geq (r-\varepsilon)/2>0$
\begin{align*}
\sum_{p\leq z^{1/s}}  \sum_{j=K_0+1}^\infty p^{-j(r-\beta-\varepsilon)}
&\leq
\sum_{p\leq z^{1/s}}  p^{-(K_0+1)(r-\varepsilon)/2}\frac{1}{1-p^{-(r-\varepsilon)/2}}\\
&\ll_{K_0} 
\sum_{p\leq z^{1/s}}  p^{-K_0(r-\varepsilon)/2}
\leq
\sum_{p\leq \infty}  p^{-2}\ll 1
\\
\end{align*}

We bound the main term in \eqref{eq:M-exp-prod-2} above using the PNT
\begin{equation*}
\sum_{p\leq z^{1/s}}\frac{p^{K_0\beta}}{p}=\Li(z^{K_0\beta/s})+z^{K_0\beta/s}\cdot o_{\log(z)/s}(1) \ll_{K_0} 1
\end{equation*} 

We can now conclude that
\begin{equation*}
\mathfrak{M}_z(g,\psi) \ll_{\kappa, A, B,\varepsilon,C,r,\deg(Q)} 1
\end{equation*}
\end{proof}

Finally, the following lemma shows that the sums over extremely smooth integers are completely negligible.
\begin{lem}\label{lem:extremely-smooth-sums}
Let $Q\in\mathbb{Z}[x,y]$ such that there are $C>0$, $1\geq r> 0$ satisfying $\widetilde{\rho}_Q(p^k)\leq C p^{k(2-r)}$ for all prime powers $p^k$.
Then for any $\beta>0$ and $1\geq\alpha\geq 0$
\begin{equation}\label{eq:case-3-lemma}
\sum_{\substack{z^{\alpha}\leq a \leq z \\ P^+(a)\leq \log z \log \log z}} \frac{\widetilde{\rho}_Q(a)}{a^2}
\ll_{C,\alpha,\beta}
 z^{-r\alpha+\beta}
\end{equation}
\end{lem}

\begin{proof}
The assumed upper bound on $\widetilde{\rho}_Q$ implies $\widetilde{\rho_{Q}}(a)/a^2\leq C^{\omega(a)}/a^r$ for all $a\in\mathbb{N}$.
This can be used to bound the left hand side of \eqref{eq:case-3-lemma} by
\begin{equation}\label{eq:case-3-lemma-r}
\sum_{\substack{z^{\alpha}\leq a \leq z \\ P^+(a)\leq \log z \log \log z}} \frac{C^{\omega(a)}}{a^r}
\end{equation}

We apply the standard bound $\omega(a)\leq K \log a/\log \log a$ for some fixed $K>0$ to deduce for any $z^{\alpha}\leq a \leq z$
\begin{equation*}
C^{\omega(a)}\leq z^{K\log C/(\log\log z+\log \alpha)}
\end{equation*}
We split our calculation into two cases. 
\begin{enumerate}
\item 
If $\log\log z\geq K\log C\cdot2/\beta-\log \alpha$ then $C^{\omega(a)}\leq z^{\beta/2}$ 
and we bound \eqref{eq:case-3-lemma-r} by
\begin{equation*}
z^{-r\alpha+\beta/2} \sum_{\substack{a \leq z \\ P^+(a)\leq \log z \log \log z}} 1=
 z^{-r\alpha+\beta/2} \Psi(z,\log z\log \log z) 
\end{equation*}
This case is settled because of the inequality $\Psi(z,\log z\log \log z)\ll_\beta z^{\beta/2}$ which follows from \cite[Lemma 1]{Shiu}.

\item 
If on the other hand $\log\log z< K\log C \cdot 2/\beta-\log \alpha$ then using the trivial bound $\widetilde{\rho}_Q(a)\leq a^2\leq z^2\ll_{C,\alpha,\beta} 1$  we estimate \eqref{eq:case-3-lemma} by
\begin{equation*}
\sum_{\substack{z^{\alpha}\leq a \leq z \\ P^+(a)\leq \log z \log \log z}} \frac{\widetilde{\rho}_Q(a)}{a^2} \ll_{C,\alpha,\beta}
 \sum_{z^\alpha \leq a \leq z} \frac{1}{a^2}\ll z^{-\alpha}
\end{equation*}
\end{enumerate}

\end{proof}

\subsection{Proof of Theorem \ref{thm:2d-sieve}}
In this section only all the implicit constants in the $\ll$ notation are allowed to depend on $\eta,A,B,\varepsilon,C,r,\deg(Q),\delta,C_l$ without further notation.
Denote $\mathscr{E}_{\mathbb{Z}}\coloneqq \mathscr{E}\cap \mathbb{Z}^2$.
We introduce notation similar to the one in \cite{Nair}. Let $Z\coloneqq A(\mathscr{E})^\eta$. For any fixed $(x,y)\in \mathscr{E}_{\mathbb{Z}}$  we write a decomposition into prime powers
\begin{equation*}
Q(x,y)=p_1^{e_1}\cdot p_2^{e_2}\cdot\ldots \cdot p_l^{e_l}
\end{equation*}
where $p_1<p_2<\cdots<p_l$. Define $a\coloneqq p_1^{e_1}\cdot \ldots p_j^{e_j}$ so that $a\leq Z$ but $a\cdot p_{j+1}^{e_j+1}>Z$, in particular $a=1$ if all $p_1^{e_1}>Z$. Let $b\coloneqq Q(x,y)/a$ and set $q\coloneqq p_{j+1}$, $e\coloneqq e_{j+1}$. Because $f$ is multiplicative we always have $f\left(Q(x,y)\right)=f(a)f(b)$. 

Following \cite{Shiu} we split the sum on the left hand side of \eqref{eq:sieve-thm-main} into four ranges
\begin{enumerate}
\item $R_1$ is the set of all $(x,y)\in \mathscr{E}_{\mathbb{Z}}$ such that $q\geq Z^{1/2}$;
\item $R_2$ is the set of all $(x,y)\in \mathscr{E}_{\mathbb{Z}}$ such that $q<Z^{1/2}, a\leq Z^{1/2}$;
\item $R_3$ is the set of all $(x,y)\in \mathscr{E}_{\mathbb{Z}}$ such that $q<\log Z \log\log Z, a>Z^{1/2}$;
\item $R_4$ is the set of all $(x,y)\in \mathscr{E}_{\mathbb{Z}}$ such that $\log Z \log\log Z \leq q < Z^{1/2}, a>Z^{1/2}$.
\end{enumerate}
Moreover, for any fixed integers $a,z$ we denote 
\begin{equation*}
S(a,z)\coloneqq \left\{(x,y)\in \mathscr{E}_{\mathbb{Z}} \mid a|Q(x,y),\; \gcd(a, Q(x,y)/a)=1 \; \mathrm{and}\; P^-(Q(x,y)/a)\geq z  \right\}
\end{equation*}

Recall that $\Omega(b)$ is the number of prime factors of $b$ counted with multiplicity. For any $(x,y)\in R_1$ we have
\begin{equation*}
Z^{1/2 \cdot \Omega(b)}\leq b\leq X\leq A(\mathscr{E})^\delta
\end{equation*}
hence $\Omega(b) \ll 1$ and $f(b)\ll 1$. This implies that we have a bound
\begin{equation*}
\sum_{(x,y)\in R_1}f\left(Q(x,y)\right) \ll \sum_{a\leq Z} f(a) |S(a,Z^{1/2})|
\end{equation*}
We can apply Lemma \ref{lem:large-sieve-a} with $\varsigma=\delta$ to bound $|S(a,Z^{1/2})|$ from above. Combining this with Corollary \ref{cor:large-sieve-a-theta} and Lemma \ref{lem:large-sieve-a-averaged} we deduce
\begin{equation*}
\sum_{(x,y)\in R_1}f\left(Q(x,y)\right) \ll A(\mathscr{E}) \prod_{\deg(Q)<p\leq Z^{1/2}} \left(1-\frac{\rho_Q(p)}{p^2}\right) \sum_{a \leq Z} \frac{f(a)\widetilde{\rho}_{Q}(a)}{a^2}
\end{equation*}
which is consistent with the claimed bound due to Lemma \ref{lem:extending-p-product}.

Next we make an observation necessary to treat the sums over $R_2$ and $R_3$ and which also indicates the natural limit of the theorem. The following lower bound for the right hand side of \eqref{eq:sieve-thm-main} holds
\begin{align*}
A(\mathscr{E})\prod_{\substack{\deg(Q)<p\leq X \\ p\nmid Q}} \left(1-\frac{\rho_Q(p)}{p^2}\right) &\sum_{a \leq X } \frac{f(a)\widetilde{\rho}_{Q}(a)}{a^2}
\geq
A(\mathscr{E})\prod_{\substack{\deg(Q)<p\leq X \\ p\nmid Q}} \left(1-\frac{\deg(Q)}{p}\right)\\ 
&\gg
A(\mathscr{E})/(\log X)^{\deg(Q)} \gg A(\mathscr{E})/(\log A(\mathscr{E}))^{\deg(Q)}
\end{align*}
Thus any bound of the form $\ll A(\mathscr{E})^{1-\varepsilon_0}$ for $\varepsilon_0>0$ is consistent with the claim.

For any $(x,y)\in R_2$ 
\begin{equation*}
Z<aq^e\leq Z^{1/2} q^e \Longrightarrow Z^{1/2}<q^e
\end{equation*}
but $q<Z^{1/2}$ hence $e\geq 2$. For each prime $p\leq Z^{1/2}$ let $e_p\geq 2$ be the minimal integer satisfying $p^{e_p} > Z^{1/2}$.  Notice that $p^{e_p}=p^{e_p-1}p\leq Z^{1/2} p\leq Z= A(\mathscr{E})^\eta$ so we can apply Corollary \ref{cor:large-sieve-a-nogcd} with $a=p^{e_p}$. This implies the following bound
\begin{align}
\sum_{(x,y)\in R_2}f\left(Q(x,y)\right)
&\ll X^\varepsilon \sum_{p\leq Z^{1/2}} \left|\left\{(x,y)\in \mathscr{E}_{\mathbb{Z}} \mid p^{e_p}|Q(x,y) \right\}\right| 
\nonumber\\
&\ll X^\varepsilon A(\mathscr{E}) \sum_{p\leq Z^{1/2}} \frac{\rho_Q(p^{e_p})}{p^{2e_p}}
\ll X^\varepsilon A(\mathscr{E}) \sum_{p\leq Z^{1/2}} \frac{1}{p^{e_p}}
\label{eq:R2-sum}
\end{align}
The latter sum is bounded by
\begin{equation*}
\sum_{p\leq Z^{1/2}} \frac{1}{p^{e_p}} \leq \sum_{p\leq Z^{1/4}} \frac{1}{Z^{1/2}} +\sum_{Z^{1/4}<p\leq Z^{1/2}} \frac{1}{p^2} 
\ll Z^{1/4-1/2} +Z^{-1/4}\ll Z^{-1/4}
\end{equation*}
We conclude from \eqref{eq:R2-sum}
\begin{equation*}
\sum_{(x,y)\in R_2}f\left(Q(x,y)\right) \ll X^\varepsilon A(\mathscr{E}) Z^{-1/4} \ll A(\mathscr{E})^{1+\varepsilon\delta-\eta/4}
\end{equation*}
Because $\varepsilon\delta-\eta/4<0$ this bound saves a power of $A(\mathscr{E})$ and is compatible with the claim.

We proceed to estimate the sums over $R_3$ using Lemma \ref{lem:large-sieve-a} 
\begin{align*}
\sum_{(x,y)\in R_3}f\left(Q(x,y)\right)
&\ll X^\varepsilon \sum_{\substack{Z^{1/2}\leq a\leq Z \\ P^+(a)\leq \log Z \log \log Z}}  |S(a,1)|\\
&\ll X^\varepsilon A(\mathscr{E}) \sum_{\substack{Z^{1/2}\leq a\leq Z \\ P^+(a)\leq \log Z \log \log Z}} \frac{\widetilde{\rho}_Q(a)}{a^2}
\end{align*}
Apply next Lemma \ref{lem:extremely-smooth-sums} with $\beta=r/4$ to deduce
\begin{equation*}
\sum_{(x,y)\in R_3}f\left(Q(x,y)\right) \ll X^\varepsilon A(\mathscr{E}) Z^{-r/4}\ll A(\mathscr{E})^{1+\delta\varepsilon-\eta r/4}
\end{equation*}
which is consistent with the claim because we have assumed $\delta\varepsilon-\eta r/4<0$.

We split $R_4$ further according to the value of $q$. For any integer 
\begin{equation*}
s_0\coloneqq 2 \leq s\leq s_1\coloneqq\frac{\log Z}{\log\left(\log Z \log \log Z \right)}
\end{equation*}
let $R_4^s$ be the set of $(x,y)\in R_4$ such that $Z^{1/(s+1)}\leq q\leq Z^{1/s}$. Recalling that $q$ is the smallest prime divisor of $b$ we see that for $(x,y)\in R_4^s$ 
\begin{equation*}
Z^{\Omega(b)/(s+1)}\leq b\leq X\leq A(\mathscr{E})^\delta
\end{equation*}
hence $\Omega(b)\leq (s+1)\delta/\eta$ and $f(b)\ll A_0^s$ where $A_0\coloneqq A^{\delta/\eta}$. 
We can now write
\begin{equation}\label{eq:R4-sum}
\sum_{(x,y)\in R_4}f\left(Q(x,y)\right) \leq 
\sum_{s_0\leq s\leq s_1} A_0^s \sum_{\substack{Z^{1/2}\leq a \leq Z\\ P^+(a)\leq Z^{1/s}}}  f(a)|S(a,Z^{1/(s+1)})|
\end{equation}
Similarly to the case of $R_1$ we apply Lemma \ref{lem:large-sieve-a} with $\varsigma=\delta$ and Corollary \ref{cor:large-sieve-a-theta} to bound the right hand side of \eqref{eq:R4-sum} from above by
\begin{align}
A(\mathscr{E}) &\sum_{s_0\leq s\leq s_1} A_0^s \sum_{\substack{Z^{1/2}\leq a \leq Z\\ P^+(a)\leq Z^{1/s}}}   \prod_{\deg(Q)<p\leq Z^{1/(s+1)}} \left(1-\frac{\rho_Q(p)}{p^2}\right) \frac{f(a)\theta_Q(a)\widetilde{\rho}_{Q}(a)}{a^2}
\nonumber\\
&\ll
A(\mathscr{E}) \prod_{\substack{\deg(Q)<p\leq X \\ p\nmid Q}} \left(1-\frac{\rho_Q(p)}{p^2}\right)
\sum_{s_0\leq s\leq s_1} A_0^s (s+1)^{\deg(Q)} \sum_{\substack{Z^{1/2}\leq a \leq Z\\ P^+(a)\leq Z^{1/s}}} \frac{f(a)\theta_Q(a)\widetilde{\rho}_{Q}(a)}{a^2}
\label{eq:R4-sum-2}
\end{align}
where in the second line we have applied Lemma \ref{lem:extending-p-product}. Let $\kappa\coloneqq 4\ln(A_0)$. If $\kappa> \frac{3r+\varepsilon}{4}\log\left(\log Z \log \log Z\right)$ then $Z\ll 1$, hence $s_1\ll 1$ and $$\sum_{s_0\leq s\leq s_1} A_0^s (s+1)^{\deg(Q)}\ll 1$$

Otherwise $\kappa\leq \frac{3r+\varepsilon}{4}\log\left(\log Z \log \log Z\right)$
and we can estimate each of the innermost sums in \eqref{eq:R4-sum-2}  using Lemma \ref{lem:sieved-sum-exp-save} with $\kappa$ as above and $f$ replaced by $f\theta_Q$. The conditions of the lemma are satisfied due to Remark \ref{rem:fthetaq-order} with $\varepsilon'=(r-\varepsilon)/2$. Then
\begin{align*}
\sum_{s_0\leq s\leq s_1} A_0^s (s+1)^{\deg(Q)} &\sum_{\substack{Z^{1/2}\leq a \leq Z\\ P^+(a)\leq Z^{1/s}}} \frac{f(a)\theta_Q(a)\widetilde{\rho}_{Q}(a)}{a^2}\\
&\ll
\sum_{a \leq Z} \frac{\widetilde{\rho}_Q(a)f(a)\theta_Q(a)}{a^2}
\cdot \sum_{s_0\leq s\leq s_1} A_0^s (s+1)^{\deg(Q)} \exp(-s\kappa/2)\\
&\ll \sum_{a \leq Z} \frac{\widetilde{\rho}_Q(a)f(a)\theta_Q(a)}{a^2} 
\end{align*}

In both cases we deduce
\begin{align*}
\sum_{(x,y)\in R_4}&\ll 
A(\mathscr{E}) \prod_{\substack{\deg(Q)<p\leq X \\ p\nmid Q}} \left(1-\frac{\rho_Q(p)}{p^2}\right) \sum_{a \leq Z} \frac{\widetilde{\rho}_Q(a)f(a)\theta_Q(a)}{a^2}\\
&\ll
A(\mathscr{E}) \prod_{\substack{\deg(Q)<p\leq X \\ p\nmid Q}} \left(1-\frac{\rho_Q(p)}{p^2}\right) \sum_{a \leq Z} \frac{\widetilde{\rho}_Q(a)f(a)}{a^2}
\end{align*}
where in the last line we have applied Lemma \ref{lem:large-sieve-a-averaged}. This is again consistent with the claim.

\subsection{Sums Restricted by Congruence Conditions}
In this section we extend Theorem \ref{thm:2d-sieve} to sums with congruence restrictions.

\begin{defi}
Let $Q\in\mathbb{Z}[x,y]$. For any $k\in\mathbb{N}$ and $0\leq l<k$ define $\rho_Q(l;k)$ to be the number of solutions modulo $k$ to the equation $Q(x,y)\equiv l$.
In particular, $\rho_Q(0;k)=\rho_Q(k)$ and $\rho_Q(l;k)=\rho_{Q-l}(k)$;
\end{defi}

\begin{prop}\label{prop:2d-sieve-q0}
Fix $k_0\in\mathbb{N}$.
Consider the setting of Theorem \ref{thm:2d-sieve} but assume the stronger assumptions
$(R_{\max}/k_0)^{\theta_l}\leq A(\mathscr{E})/k_0^2$ and $X\leq A(\mathscr{E})^{\delta} k_0^{1-2\delta}$.
Then
\begin{equation*}
\sum_{\substack{(x,y)\in \mathscr{E}\cap\mathbb{Z}^2 \\ k_0 \mid Q(x,y)}} f\left(\frac{Q(x,y)}{k_0}\right)\ll 
A(\mathscr{E})
\prod_{\substack{\deg(Q) < p\leq X/k_0 \\ p\nmid k_0,\; p\nmid Q}} \left(1-\frac{\rho_Q(p)}{p^2}\right)
\sum_{a\leq X/k_0} \frac{f(a)\widetilde{\rho}_Q(k_0a)}{(k_0a)^2}
\end{equation*}
The implicit constant is the same as in Theorem \ref{thm:2d-sieve}.
\end{prop} 
\begin{proof}
Let $r=(r_1,r_2)$ be a representative of one of the congruence classes modulo $k_0$ solving the equation $Q(x,y)\equiv 0 \mod k_0$. Define $Q_1^r,Q_0^r\in\mathbb{Z}[x,y]$ by
\begin{equation*}
Q_1^r(x,y)\coloneqq Q(k_0x+r_1,k_0y+r_2)=k_0Q_0^r(x,y)
\end{equation*}
Notice that $\deg(Q_i^r)=\deg(Q)$ for $i=0,1$. Moreover $\rho_Q(p^k)=\rho_{Q_0^r}(p^k)$ for any $p\nmid k_0$ and $k\geq 0$ and the same holds for $\widetilde{\rho}$. We can apply now Theorem \ref{thm:2d-sieve} to the sum over a single congruence class as follows.
\begin{align*}
\sum_{\substack{(x,y)\in \mathscr{E} \cap\mathbb{Z}^2\\ (x,y)\equiv r \bmod k_0}} &f\left(\frac{Q(x,y)}{k_0}\right)
=\sum_{(x,y)\in k_0^{-1}(\mathscr{E}-r)\cap\mathbb{Z}^2}  f\left(Q_0^r(x,y)\right)
\\
&\ll 
\frac{A(\mathscr{E})}{k_0^2}
\prod_{\substack{\deg(Q)<p\leq X/k_0 \\ p\nmid k_0,\; p\nmid Q}} \left(1-\frac{\rho_Q(p)}{p^2}\right)
\sum_{a\leq X/k_0} \frac{f(a)\widetilde{\rho}_{Q_0^r}(a)}{a^2}
\\
&=
A(\mathscr{E})
\prod_{\substack{\deg(Q)<p\leq X/k_0 \\ p\nmid k_0,\; p\nmid Q}} \left(1-\frac{\rho_Q(p)}{p^2}\right)
\sum_{a\leq X/k_0} \frac{f(a)\widetilde{\rho}_{Q}\left(\frac{a}{\gcd(a,k_0^\infty)}\right) \widetilde{\rho}_{Q_0^r}\left(\gcd(a,k_0^\infty)\right)}{(k_0a)^2}
\end{align*}
A direct calculation shows that the conditions of Theorem \ref{thm:2d-sieve} are satisfied when applied to the sum above. Summing over all the pertinent conjugacy classes we deduce
\begin{align}
\sum_{\substack{(x,y)\in \mathscr{E}\cap\mathbb{Z}^2 \\ k_0 \mid Q(x,y)}} f\left(\frac{Q(x,y)}{k_0}\right)\ll
A(\mathscr{E})
&\prod_{\substack{\deg(Q)<p\leq X \\ p\nmid k_0,\; p\nmid Q}} \left(1-\frac{\rho_Q(p)}{p^2}\right)
\nonumber\\
&\cdot
\sum_{a\leq X} \frac{f(a)\widetilde{\rho}_{Q}\left(\frac{a}{\gcd(a,k_0^\infty)}\right) }{(k_0a)^2}
\left(\sum_r \widetilde{\rho}_{Q_0^r}\left(\gcd(a,k_0^\infty)\right) \right)
\label{eq:sieve-sum-over-congruences}
\end{align}
For any $b \mid k_0^\infty$ we can see from the definitions in \ref{def:polynomial-rho} that
\begin{equation*}
\sum_r \widetilde{\rho}_{Q_0^r}(b) =\widetilde{\rho}_Q(k_0b)
\end{equation*}
The claim follows from combining this observation with \eqref{eq:sieve-sum-over-congruences}.
\end{proof}

\begin{prop}\label{prop:2d-sieve-q0q1q2}
Consider the setting of Theorem \ref{thm:2d-sieve} but assume the stronger assumptions $R_{\max}^{\theta_l}\leq A(\mathscr{E})^{1-4\eta}$ and $X\leq A(\mathscr{E})^{\delta/2}$. Fix $k_0,k_1,k_2\in\mathbb{N}$ such that all primes dividing $k_1$ also divide $k_2$, $\gcd(k_0,k_2)=1$ and $k\coloneqq k_0 k_1k_2\leq A(\mathscr{E})^{\eta/2}$. Let  $l\in\left(\cyclic{k_2}\right)^\times$ then
\begin{align*}
\sum_{\substack{(x,y)\in \mathscr{E}\cap\mathbb{Z}^2 
\\ Q(x,y) \equiv k_0k_1 l \bmod k}} 
&f\left(\frac{Q(x,y)}{k_0}\right)
\ll 
A(\mathscr{E}) \frac{f(k_1)\rho_Q\left(k_1l;k_1k_2\right)}{(k_1k_2)^2}\\
&\cdot\prod_{\substack{\deg(Q) < p\leq X/(k_0k_1) \\ p\nmid k_0k_2,\; p\nmid Q}} \left(1-\frac{\rho_Q(p)}{p^2}\right)
\sum_{\substack{a\leq X/(k_0k_1)\\\gcd(a,k_2)=1}} \frac{f(a)\widetilde{\rho}_Q(k_0a)}{(k_0a)^2}
\end{align*}
The implicit constant is the same as in Theorem \ref{thm:2d-sieve}.
\end{prop}
\begin{proof}
Let $r=(r_1,r_2)$ be a representative of one of the $\rho_Q(k_1l; k_1k_2)$ congruence classes modulo $k_1k_2$ solving the equation $Q(x,y)\equiv k_1 l \mod k_1k_2$. Define $Q_2,Q_1,Q_0\in\mathbb{Z}[x,y]$ by
\begin{equation*}
Q_2(x,y)\coloneqq Q(kx+r_1,ky+r_2)=k_1 Q_1(x,y)=k_1(l+ k_2Q_0(x,y))
\end{equation*}
Notice that $\deg(Q_i)=\deg(Q)$ for $i=0,1,2$. 
Moreover $\rho_Q(p^m)=\rho_{Q_2}(p^m)=\rho_{Q_1}(p^m)$ for any $p\nmid k_1 k_2$ and $m\geq 0$. The same holds for $\widetilde{\rho}$. Because $l$ is a unit modulo $k_2$ we conclude that  $\rho_{Q_1}(p^m)=0$ for $p\mid k_2$, $m\geq 1$. Finally, notice that because the prime support of $k_1$ is contained in that of $k_2$ the condition $p\mid k_1 k_2$ is equivalent to $p\mid k_2$.

Because $\gcd(k_0,k_2)=1$ we now that $k_0\mid Q_2(x,y)$ if and only if $k_0\mid Q_1(x,y)$.
Write the pertinent sum over the fixed congruence class represented by $r$
\begin{equation*}
\sum_{\substack{(x,y)\in (k_1k_2)^{-1}(\mathscr{E}-r) \cap\mathbb{Z}^2 \\ k_0 \mid Q_2(x,y)}} 
f\left(\frac{Q_2(x,y)}{k_0}\right)=
f(k_1)\sum_{\substack{(x,y)\in (k_1k_2)^{-1}(\mathscr{E}-r) \cap\mathbb{Z}^2 \\ k_0 \mid Q_1(x,y)}}  f\left(\frac{Q_1(x,y)}{k_0}\right)
\end{equation*}
A direct calculation shows that the conditions of Proposition \ref{prop:2d-sieve-q0} are satisfied when applied to the sum on the right hand side (the restriction on $X$ holds because $k\leq A(\mathscr{E})^{1/4}$ as we have assumed $\eta<1/2$).
The claim follows by summing over all the relevant conjugacy classes modulo $k_1k_2$.
\end{proof} 

   \section{Proof of Main Theorem}\label{sec:proof}
In this section we use the following notation for all integers $n\geq 0$
\begin{equation*}
B^{(-n,n)}\coloneqq \prod_{v\neq p_1} \Omega_v \times K_{p_1}^{(-n,n)}\subset \mathbf{G}(\mathbb{A})
\end{equation*}
For the sake of brevity, we shall denote $B\coloneqq B^{(-0,0)}$.

Moreover, for any $\xi\in\left(\mathbf{G}\times\mathbf{G}\right)(\mathbb{A})$ we denote by $\nu_\xi$ be the algebraic measure supported on $\left[\mathbf{G}^\Delta(\mathbb{A})^+
\xi\right]$.

\subsection{Reduction to a Bound on Cross-Correlation}
We begin by showing that Theorem \ref{thm:joinings-adelic} would follow from an appropriate bound on the cross-correlation. 

\begin{lem}\label{lem:cross-correlation-to-thm}
Let $\mathcal{H}_i=\left[\mathbf{T}^\Delta(\mathbb{A})(g,sg)\right]$ be a sequence of homogeneous toral sets where $\mathbf{T}$, $g$ and $s$ depend on the index $i\in\mathbb{N}$. Assume that the splitting conditions \eqref{eq:g_adel_restrictions} are satisfied for all $i$. Denote by $\mu_i$ the algebraic measure supported on $\mathcal{H}_i$ and assume $\mu_i\to_{i\to\infty} \mu$.

Assume that there is some $F\colon \mathbf{G}(\mathbb{A}) \to \mathbb{R}_{>0}$ continuous such that for all $n\in\mathbb{N}$, for all $\xi\in\left(\mathbf{G}\times\mathbf{G}\right)(\mathbb{A})$ and for all $i\gg_{n,\xi} 1$ 
\begin{equation*}
\Cor[\mu_i, \nu_\xi]\left(B^{(-n,n)}\right) \ll F(\ctr(\xi)) p_1^{-2(1+\rho)n}
\end{equation*}
for some $\rho>0$ fixed. Then $\mu$ is a $\left(\mathbf{G}\times\mathbf{G}\right)(\mathbb{A})^+$-invariant probability measure.
\end{lem}
\begin{proof}
From Duke's theorem we know that $\mu$ is a probability measure. Theorem \ref{thm:regiditiy-adelic} and Corollary \ref{cor:hecke-limit-split-restrictions} imply that $\mu$ is a convex combination of a $\left(\mathbf{G}\times\mathbf{G}\right)(\mathbb{A})^+$-invariant probability measure and algebraic measures supported on homogeneous Hecke sets of the form $\left[\mathbf{G}^\Delta(\mathbb{A})^+\xi \right]$ such that 
\begin{equation*}
\ctr(\xi)_{p_j}\in A_{p_j}
\end{equation*}
for $j\in \{1,2\}$.

Assume in contradiction that $\mu$ is not $\left(\mathbf{G}\times\mathbf{G}\right)(\mathbb{A})^+$-invariant
then there is a finite non-vanishing measure $\lambda_0$ on $\lfaktor{\mathbf{G}^\Delta(\mathbb{A})^+}{\left(\mathbf{G}\times\mathbf{G}\right)(\mathbb{A})}$ so that
\begin{equation*}
\mu \geq
\int_{\mathbf{G}^\Delta(\mathbb{A})^+ \backslash
\left(\mathbf{G}\times\mathbf{G}\right)(\mathbb{A})} \nu_\xi \dif\lambda_0(\xi)
\end{equation*}
and the following set has full $\lambda_0$-measure
\begin{equation*}
\Xi_1\coloneqq \left\{\xi\in\lfaktor{\mathbf{G}^\Delta(\mathbb{A})^+}{\left(\mathbf{G}\times\mathbf{G}\right)(\mathbb{A})} \relmiddle{|} \ctr(\xi)_{p_1}\in A_{p_1}, \ctr(\xi)_{p_2}\in A_{p_2}  \right\}
\end{equation*}

Moreover, because $\lambda_0$ is a finite measure it is regular 
so there is a compact subset $\Xi_0\subset \Xi_1$ of positive measure. We now have
\begin{align*}
\mu &\geq  \lambda_0(\Xi_0) \cdot  \bar{\nu} \\
\bar{\nu} &\coloneqq \frac{1}{\lambda_0(\Xi_0)}
\int_{\Xi_0} \nu_\xi \dif\lambda_0(\xi)
\end{align*}
and $\bar{\nu}$ is a probability measure on $\left[\left(\mathbf{G}\times\mathbf{G}\right)(\mathbb{A})\right]$.

Let $a=\lambda(p_1)\in A_{p_1}$ where $\lambda\in X_\bullet(A_{p_1})$ generates the cocharacter group. The element $a^\Delta\in A_{p_1}^\Delta$ acts on $\left[\left(\mathbf{G}\times\mathbf{G}\right)(\mathbb{A})\right]$  on the right.
For all $\xi\in\Xi_0$
the action of $a^\Delta$ on the space $\left[\left(\mathbf{G}\times\mathbf{G}\right)(\mathbb{A})\right]$ keeps $\nu_\xi$ invariant because $\ctr(\xi)\in A_{p_1}$. Additivity of entropy implies
\begin{equation*}
\ent_{a^\Delta}(\bar{\nu})=\frac{1}{\lambda_0(\Xi_0)}
\int_{\Xi_0} h_{a^\Delta}(\nu_\xi) \dif\lambda_0(\xi)
\end{equation*}
The measurable dynamical system $\left(\left[\left(\mathbf{G}\times\mathbf{G}\right)(\mathbb{A})\right], \nu_\xi, a^\Delta\right)$ is measure theoretically isomorphic to $a$ acting on the space 
\begin{equation*}
\lfaktor{\Cent_{\mathbf{G}^\sc}(\mathbb{A})
\mathbf{G}^\sc(\mathbb{Q})}{\mathbf{G}^\sc(\mathbb{A})}
\end{equation*}
equipped with the probability Haar measure. This entropy can be computed using the leaf-wise measure \cite{Pisa,MargulisTomanov} on the horospherical subgroup of $a$. As the Haar measure is invariant under the full group action the leaf-wise measure will be the Haar measure on the horospherical subgroup and
\begin{equation*}
\ent_{a^\Delta}(\bar{\nu})=\log p_1
\end{equation*}

We will show next that the assumed cross-correlation estimate implies that the entropy of $\bar{\nu}$ must be at least $(1+\rho)\log p_1$ which contradicts the equality above.

Weak-$*$ convergence of measures implies that for any bounded open subset $C^\circ\subset \left[\mathbf{G}(\mathbb{A})\right]$
\begin{align*}
\Cor_{C^\circ}[\mu,\nu_\xi]\left({B^{(-n,n)}}^\circ\right)&\leq \liminf_{i\to\infty}
\Cor_{C^\circ}[\mu_i,\nu_\xi]\left({B^{(-n,n)}}^\circ\right)\\
&\leq \liminf_{i\to\infty}
\Cor[\mu_i,\nu_\xi]\left({B^{(-n,n)}}\right)
\end{align*}
Fix a closed identity neighborhood $\Omega_{\infty,0}\subset \Omega_{\infty}^\circ$ and set $B_0^{(-n,n)}=\Omega_{\infty,0}\times \prod_{v\neq p_1,\infty} \Omega_v\times K_{p_1}^{(-n,n)}$.
Taking a monotone sequence of bounded open subsets which exhausts $\left[\mathbf{G}(\mathbb{A})\right]$ we deduce that
\begin{equation*}
\Cor[\mu,\nu_\xi]\left({B_0^{(-n,n)}}\right)
\leq \liminf_{i\to\infty}
\Cor[\mu_i,\nu_\xi]\left({B^{(-n,n)}}\right)
\ll F(\ctr(\xi)) p_1^{-2(1+\rho)n}
\end{equation*}

Monotonicity of integration and Fubini imply that
\begin{align*}
\Cor[\bar{\nu},\bar{\nu}]\left(B_0^{(-n,n)}\right)&\leq \frac{1}{\lambda_0(\Xi_0)} \Cor[\mu,\bar{\nu}]\left(B_0^{(-n,n)}\right)\\
&= \frac{1}{\lambda_0(\Xi_0)^2}
\int_{\Xi_0} \Cor[\mu,\nu_\xi]\left(B_0^{(-n,n)}\right) \dif\lambda_0(\xi)\\
&\ll \frac{p_1^{-2(1+\rho)n}}{\lambda_0(\Xi_0)^2}
\int_{\Xi_0} F(\ctr(\xi)) \dif\lambda_0(\xi)
\end{align*}
Notice that $\int_{\Xi_0} F(\ctr(\xi)) \dif\lambda_0(\xi)$ is finite because $\Xi_0$ is compact and $F$ is continuous.

An upper bound on the self-correlation of a measure for Bowen balls implies a lower bound for the metric entropy. The self-correlation bound for the adelic quotient implies an identical bound for any $S$-arithmetic quotient, as long as we take the set of places $S$ to include $\infty, p_1$. On the other hand, a lower bound for the entropy for $S$-arithmetic quotients for arbitrary large $S$ implies the same bound for the adelic quotient.

Using \cite[Proposition 3.2]{ELMVPeriodic}, which generalizes, \textit{mutatis mutandis}, to the $S$-arithmetic setting, we deduce from 
the last inequality that $\ent_{a^\Delta}(\bar{\nu})\geq (1+\rho)\log p_1$ as required.
\end{proof}

\subsection{From a Shifted Convolution to Sums over a Polynomial}
The first step in producing an upper bound on the cross-correlation as required in Lemma \ref{lem:cross-correlation-to-thm} is translation of the shifted-convolution sum in Theorem \ref{thm:cross-correlation-shifted-convolution} to sums of a \emph{multiplicative} function over values of a $2$-variable polynomial. 

In this section we work in the setting of Theorem \ref{thm:cross-correlation-shifted-convolution} which we now review.
Fix a joint homogeneous toral set $\left[\mathbf{T}^\Delta(\mathbb{A})(g,sg)\right]$ satisfying \eqref{eq:g_adel_restrictions} with a splitting field $E/\mathbb{Q}$ and quadratic order $\Lambda\leq\mathcal{O}_E$ of discriminant $D$. 

Fix also a simply connected homogeneous Hecke set $\left[\mathbf{G}^\Delta(\mathbb{A})^+\xi\right]$ with $\ctr(\xi)_{p_1}\in A_{p_1}$ and assume 
\begin{equation*}
g^{-1}\mathbf{T}(\mathbb{Q})s g \cap B\ctr(\xi)B=\emptyset
\end{equation*}
Notice that this condition implies the same for $B^{(-n,n)}$ for all $n$.

Let $\mu$ be the algebraic probability measure supported on $\left[\mathbf{T}^\Delta(\mathbb{A})(g,sg)\right]$ and let $\nu_\xi$ be the algebraic probability measure supported on $\left[\mathbf{G}^\Delta(\mathbb{A})^+\xi\right]$. Denote $\kappa=2^8 \Denom_\infty(\ctr(\xi)_\infty)\Denom_f(\ctr(\xi)_f)$ and $\omega=\sign(\Nrd(\ctr(\xi)_\infty))\Denom_f(\ctr(\xi)_f)$.

Initially, we transform the shifted-convolution sum to a sum of a non-multiplicative function over polynomial values. Afterwards we shall use principal genus theory to split the sum in the following lemma into sums that can be effectively bounded by multiplicative functions.
\begin{lem}\label{lem:shifted-to-polynomial}
Fix an arbitrary $\mathbb{Z}$-basis $A,B\in E^\times$ for the fractional $\Lambda$-ideal $\mathfrak{s}^{-1}$ and let $q(x,y)\in\mathbb{Z}[x,y]$ be the associated norm form
\begin{equation*}
q(x,y)\coloneqq \frac{\Nr(Ax+By)}{\Nr(\mathfrak{s}^{-1})}
\end{equation*}
This is a primitive integral binary quadratic form of discriminant $D$.

The shifted convolution sum of Theorem \ref{thm:cross-correlation-shifted-convolution} satisfies
\begin{align*}
\sum_{\substack{0\leq x\leq \kappa |D| \\ x\equiv \omega|D| \mod \upsilon p_1^{2n}}} &g_{[\mathfrak{s}]}(x)
f_{[p_1^n\mathfrak{se}]^{-1}}\left(\frac{x-\omega D}{\upsilon p_1^{2n}}\right) r\left(\frac{x-\omega D}{\upsilon p_1^{2n}}\right)\\
&= \frac{1}{\# \Lambda^\times}\sum_{\substack{(x,y)\in\mathbb{Z}^2 \colon q(x,y)\leq \kappa|D| \\ \upsilon p_1^{2n}\mid Q(x,y)}} \left(f_{[p_1^n\mathfrak{se}]^{-1}}\cdot r\right)\left(\frac{Q(x,y)}{\upsilon p_1^{2n}}\right) 
\end{align*}
where
\begin{equation*}
Q(x,y)\coloneqq q(x,y)-\omega D
\end{equation*}
\end{lem}
\begin{proof}
This follows immediately from the correspondence between invertible integral ideals in the class $[\mathfrak{s}]\in\Pic(\Lambda)$ and points in $\mathfrak{s}^{-1}$. Explicitly, if $\mathfrak{a}\in[\mathfrak{s}]$ then there is some $a\in E^\times$ so that
$\mathfrak{a}=a\mathfrak{s}$ and $a\mathfrak{s}\subseteq \mathfrak{\Lambda}$, i.e.\ $a\in\mathfrak{s}^{-1}$. Moreover, two different values of $a$ corresponding to $\mathfrak{a}$ must differ by a unit of $\Lambda$.
\end{proof}

\begin{defi}
We now fix $q(x,y)\in\mathbb{Z}$ to be the unique \emph{reduced}\footnote{Reduced with respect to the usual fundamental domain for the $\mathbf{SL}_2(\mathbb{Z})$-action on the upper half plain.} norm form for $\mathfrak{s}^{-1}$ and denote
\begin{equation*}
\mathscr{E}\coloneqq \left\{(x,y)\in\mathbb{R}^2 \mid q(x,y)\leq \kappa |D| \right\}
\end{equation*}

In the current section we shall always denote by $R_{\max}$ and $A(\mathscr{E})$ the maximal radius of curvature and area of $\mathscr{E}$.
\end{defi}

\begin{lem}\label{lem:ellipse-area-rmax}
The set $\mathscr{E}$ is an ellipse centered at the origin. Its area is $A(\mathscr{E})=2\pi \kappa \sqrt{|D|}$ and the maximal radius of curvature satisfies
\begin{equation*}
R_{\max}\leq \sqrt{A(\mathscr{E})} \left(\frac{\sqrt{|D|}}{\mathfrak{N}}\right)^{3/2}
\end{equation*}
where $\mathfrak{N}\coloneqq \min_{\substack{\mathfrak{a} \subseteq \Lambda \\ [\mathfrak{a}]=[\mathfrak{s}]}} \Nr \mathfrak{a}$
\end{lem}
\begin{proof}
The domain $\mathscr{E}$ is an ellipse because $q$ is positive-definite.
The formula for the area follows from the fact that $\disc(q)=D$. To estimate $R_{\max}$ consider the ellipse $\mathscr{E}_0$ of area $\pi$ homothetic to $\mathscr{E}$ and let $a\geq a^{-1}>0$ be the lengths of its semi-major axes. The maximal radius of curvature satisfies
\begin{equation}\label{eq:Rmax-axis}
R_{\max}=\sqrt{A(\mathscr{E})}R_{\max}(\mathscr{E}_0)=\sqrt{A(\mathscr{E})} \frac{a^2}{a^{-1}}=\sqrt{A(\mathscr{E})} a^3
\end{equation}

The group $\mathbf{SL}_2(\mathbb{R})$ acts transitively on the space of ellipses of area $\pi$ and centered at the origin. The stabilizer of the unit circle $S^1$ is $\mathbf{SO}_2(\mathbb{R})$. We identify this space of ellipses with the upper half-plane $\mathbb{H}$ by sending $S^1$ to $i\in\mathbb{H}$. The point in $\mathbb{H}$ corresponding to $\mathscr{E}_0$ coincides with the point corresponding to $q$ in the fundamental domain. Denote this point by $x_0\in\mathbb{H}$. This point can be written down explicitly as\footnote{Notice that because an ideal class and its inverse are Galois conjugate
$\min_{\substack{\mathfrak{a} \subseteq \Lambda \\ [\mathfrak{a}]=[\mathfrak{s}]}} \Nr \mathfrak{a}=\min_{\substack{\mathfrak{a} \subseteq \Lambda \\ [\mathfrak{a}]=[\mathfrak{s}^{-1}]}} \Nr \mathfrak{a}$} 
\begin{equation*}
x_0=\frac{-b+i\sqrt{|D|}}{2\mathfrak{N}}
\end{equation*}
where $\left\langle\mathfrak{N}, \frac{-b+i\sqrt{|D|}}{2} \right\rangle \subset E$ is the primitive integral ideal in the class $[\mathfrak{s}^{-1}]$. In particular, 
\begin{equation*}
\Im(x_0)=\frac{\sqrt{|D|}}{2\mathfrak{N}}
\end{equation*}

If $\mathscr{E}_0=g.S^1$ then the lengths of the semi-major axes are exactly the element of the diagonal matrix in the Cartan decomposition of $g$, i.e. $g\in \mathbf{SO}_2(\mathbb{R}) \begin{pmatrix}
a & 0 \\ 0 & a^{-1}
\end{pmatrix} \mathbf{SO}_2(\mathbb{R})$. In particular, $a^2+a^{-2}=\Tr(g^tg)$. 

We would like to find the relation of between $a$ and $\Im(x_0)$. Using the Iwasawa decomposition of $\mathbf{SL}_2(\mathbb{R})$ we can write 
\begin{equation*}
g\in \begin{pmatrix}
1 & t \\ 0 & 1
\end{pmatrix}
\begin{pmatrix}
\sqrt{\Im(x_0)} & 0 \\ 0 & \sqrt{\Im(x_0)}^{-1}
\end{pmatrix}
\mathbf{SO}_2(\mathbb{R})
\end{equation*}
for some $-1/2\leq t \leq 1/2$. We deduce that
\begin{equation*}
a^2+a^{-2}=\Tr(g^t g)= \Im(x_0)+{\Im(x_0)}^{-1}(1+t^2)
\end{equation*}
Solving the above quadratic equation for $a^2$ and using standard calculus with the inequalities $t^2\leq 1/4$ and $\Im(x_0)\geq \sqrt{3}/2$ we deduce that
\begin{equation*}
a^2\leq 2 \Im(x_0)=\frac{\sqrt{|D|}}{\mathfrak{N}}
\end{equation*}
The claim follows by combining this inequality with \eqref{eq:Rmax-axis}.
\end{proof}

The next step is to split the sum from Lemma \ref{lem:shifted-to-polynomial} according to further congruence conditions to take into account the restrictions modulo $\Pic(\Lambda)^2$. We shall do that only for small odd primes dividing $D_E=\disc(E)$. Our sieve method will not be able to take into account large prime divisor. Fortunately, we will see later that not taking into account the genus congruence conditions for larger primes only changes the final upper bound by an absolute constant.

Let $C_\theta\geq 1$ be a constant such that for all $X\in\mathbb{N}$
\begin{equation}\label{eq:chebyshev-bounds}
C_\theta^{-1} X \leq \sum_{p\leq X} \log p \leq C_\theta X
\end{equation}
Such a $C_\theta$ exists due to the Chebyshev bounds on the prime counting function. We fix $1/2>\eta>0$ to be determined later. Write $D=D_\mathrm{small} D_\mathrm{large} $ where
\begin{equation*}
D_\mathrm{small}\coloneqq \prod_{\substack{p\parallel D,\; p\nmid \omega \\ 2<p\leq \eta/(4C_\theta)\log |D|}} p
\end{equation*}
Because of \eqref{eq:chebyshev-bounds} we know that $D_\mathrm{small}\leq |D|^{\eta/4}$. We are going to split the sum in Lemma \ref{lem:shifted-to-polynomial} according to congruence classes modulo $\upsilon p_1^{2n}$ and $p^2$ for any $p\mid D_\mathrm{small}$. It is exactly these congruence conditions that our sieve bound can take into account. Because we only seek upper bounds we can simply ignore any restrictions that the condition modulo $\Pic(\Lambda^2)$ implies for primes $p\mid D_\mathrm{large}$. Fortunately, ignoring the congruence conditions modulo large primes only changes the upper-bound by a fixed constant independent of $D$.

Thus our goal is to replace in each congruence class the functions $f_{[p_1^n\mathfrak{se}]^{-1}}$  and $r$ by the simpler functions $f$ and $r_0$ from the following definition.
\begin{defi}
Let $f\colon\mathbb{N}\to\mathbb{N}$ be the multiplicative function counting integral invertible $\Lambda$-ideals, i.e.\
\begin{equation*}
f(n)\coloneqq \# \left\{\mathfrak{a}\in\Ideals(\Lambda) \relmiddle{|} \mathfrak{a}\subseteq \Lambda ,\; \Nr\mathfrak{a}=n\right\}
\end{equation*}
Define also the multiplicative function $r_0\colon\mathbb{N}\to\mathbb{Z}$ by requiring that $r_0(p^k)=2$ if $p\mid D_\mathrm{large}$ and $k\geq \ord_p D$; and $r_0(p^k)=1$ otherwise.
\end{defi}
To take into account the condition modulo $\Pic(\Lambda^2)$ we need to add weights to the sums over different congruence classes for $p\mid D_\mathrm{small}$. We now define the correct weights as follows from principal genus theory.  
Define $k\coloneqq \upsilon p_1^{2n} D_\mathrm{small}^2$ and write
\begin{equation*}
\cyclic{k}=\cyclic{\upsilon}\times\cyclic{p_1^{2n}}\times\prod_{p\mid D_\mathrm{small}} \cyclic{p^2}
\end{equation*} 

For each prime $p\mid D_\mathrm{small}$ we partition $\cyclic{p^2}$ in the following way
\begin{align*}
\cyclic{p^2}&=C^{p^2}_{+0}\sqcup C^{p^2}_{-0}\sqcup C^{p^2}_{+1}\sqcup C^{p^2}_{-1}\sqcup C^{p^2}_2\\
C^{p^2}_{\pm 0}&\coloneqq \left\{u\in \left(\cyclic{p^2}\right)^\times \relmiddle{|} \left(\frac{u}{p}\right)=\pm 1\right\}\\
C^{p^2}_{\pm 1}&\coloneqq \left\{pu \relmiddle{|}u\in \left(\cyclic{p}\right)^\times,\; \left(\frac{u}{p}\right)=\pm 1\right\}\\
C^{p^2}_2&\coloneqq \{0\}
\end{align*}
We define a measure $w_p$ on $\cyclic{p^2}$. The measure $w_p$ is uniform on each atom of the partition above and assigns the following weights for each atom.
\begin{align*}
w_p(C^{p^2}_{\pm 0})&=\# C^{p^2}_{\pm 0}=\frac{p^2}{2}\left(1-\frac{1}{p}\right)\\
w_p(C^{p^2}_2)&=2\cdot \# C^{p^2}_2 = 2\\
w_p(C^{p^2}_{\epsilon})&=\# C^{p^2}_{\epsilon} \cdot \begin{cases}
2 & \ch_p\left(\Nr(p_1^n\mathfrak{se})\right)=-\epsilon\\
0 & \ch_p\left(\Nr(p_1^n\mathfrak{se})\right)\neq -\epsilon
\end{cases}\\
&=p\left(1-\frac{1}{p}\right) \delta_{\ch_p\left(\Nr(p_1^n\mathfrak{se})\right)=-\epsilon}
\end{align*}
where $\epsilon\in\{\pm 1\}$ and we denote by $\ch_p$ both the unique primitive real Dirichlet character of conductor $p>2$ and its adelic lift. See Definition \ref{def:Kronecker-characters} in the appendix for details. 

Each weight takes into account both the difference between $r_0$ and $r$ and the information from principal genus theory about the condition modulo $\Pic(\Lambda^2)$, cf. Proposition \ref{prop:principal-genus-thm} in Appendix \ref{appndx:principal-genus}. In particular, the factor of $2$ in the weights of all congruence classes modulo $p$ outside of $C_{\pm 0}^{p^2}$ is due to the contribution of $r$. The fact that one of the two sets $C_{\pm 1}^{p^2}$ has weight $0$ is due to the genus restriction.

These measures for $p\mid D_\mathrm{small}$ define a product measure $w_k$ on $\cyclic{k}$ by
\begin{equation*}
w_k\coloneqq \delta_{0 \bmod \upsilon}\times \delta_{0 \bmod p_1^{2n}} \times 
\prod_{p \mid D_\mathrm{small}} w_p
\end{equation*}

\begin{lem}\label{lem:final-sums-over-polynomial}
The following holds
\begin{align*}
\frac{1}{\# \Lambda^\times}
\sum_{\substack{(x,y)\in \mathscr{E} \cap \mathbb{Z}^2 \\ \upsilon p_1^{2n}\mid Q(x,y)}} &\left(f_{[p_1^n\mathfrak{se}]^{-1}}\cdot r\right)\left(\frac{Q(x,y)}{\upsilon p_1^{2n}}\right)\\
&\ll_{\mathbf{G}} \int
\sum_{\substack{(x,y)\in \mathscr{E} \cap \mathbb{Z}^2  
\\ Q(x,y) \equiv m \bmod k}}(f\cdot r_0)\left(\frac{Q(x,y)}{\upsilon p_1^{2n}}\right)
\dif w_k(m)
\end{align*}
\end{lem}
\begin{remark}
Unlike $r$ the mean value of the multiplicative function $r_0$ is bounded above only in terms of $\eta$ independently of $D$. This is why its contribution is of no significant effect. The contribution of $r$ which is not covered by $r_0$ is negated by the restriction to a fixed genus class whenever $p\mid D_\mathrm{small}$ and $p\parallel Q(x,y)$.
\end{remark}
\begin{proof}
Notice that $\#\Lambda^\times \geq 1$ hence the factor $\frac{1}{\# \Lambda^\times}$ is uniformly bounded. Moreover, using Proposition \ref{prop:principal-genus-thm} we deduce that
\begin{align*}
\frac{1}{\# \Lambda^\times}
&\sum_{\substack{(x,y)\in \mathscr{E} \cap \mathbb{Z}^2 \\ \upsilon p_1^{2n}\mid Q(x,y)}} \left(f_{[p_1^n\mathfrak{se}]^{-1}}\cdot r\right)\left(\frac{Q(x,y)}{\upsilon p_1^{2n}}\right)\\
\leq &\sum_{\substack{m \in \cyclic{k} \\ m\equiv 0 \mod \upsilon p_1^{2n} \\ \forall p\mid D_\mathrm{small}\colon m \bmod p^2 \not\in C^{p^2}_{\ch_p\left(\Nr(p_1^n\mathfrak{se})\right)}}} 
\sum_{\substack{(x,y)\in \mathscr{E}\cap \mathbb{Z}^2
\\ Q(x,y) \equiv m \bmod k}}(f\cdot r)\left(\frac{Q(x,y)}{\upsilon p_1^{2n}}\right)
\end{align*}
Notice that Proposition \ref{prop:principal-genus-thm} has only been applied to primes $p\mid D_\mathrm{small}$ and only in the case that $p\parallel Q(x,y)$. 
If $Q(x,y)$ is a unit modulo $p$ then $\frac{Q(x,y)}{\upsilon p_1^{2n}}\equiv \frac{q(x,y)}{\upsilon p_1^{2n}} \mod p$, where $q(x,y)$ is a norm of an ideal in the class $[\mathfrak{s}]$. Unwinding the definitions of $\upsilon$ and $\mathfrak{e}$, we see that the genus congruence class of $\frac{Q(x,y)}{\upsilon p_1^{2n}}$ modulo $p$ is equal to the genus congruence class modulo $p$ of $[\mathfrak{s}\mathfrak{e}^{-1}\mathfrak{p}_1^{-n}]\equiv [\mathfrak{p}_1^n\mathfrak{s}\mathfrak{e}]^{-1} \mod \Pic(\Lambda)^2$. This implies that principal genus theory in the form of Proposition \ref{prop:principal-genus-thm} provides no extra information in this case. We neglect also any information from principal genus theory if $\ord_p Q(x,y)\geq 2$ but this will only affects our final bound by multiplying it by a constant independent of all parameters.

Finally notice that if $Q(x,y)\equiv m \mod k$ then 
\begin{equation*}
r\left(\frac{Q(x,y)}{\upsilon p_1^{2n}}\right)=w_k(m) r_0\left(\frac{Q(x,y)}{\upsilon p_1^{2n}}\right)2^{\mu_{\mathrm{wild}}\delta_{2\mid D} } \prod_{\substack{p \parallel Q(x,y) \\ \mathbf{G} \textrm{ ramifies at } p}} 2
\ll_{\mathbf{G}} w_k(m) r_0\left(\frac{Q(x,y)}{\upsilon p_1^{2n}}\right)
\end{equation*}
\end{proof}

\subsection{The Sieved Upper Bound}
We are finally ready to apply Theorem \ref{thm:2d-sieve} in the form of Proposition \ref{prop:2d-sieve-q0q1q2} to bound the cross-correlation.

\begin{defi}
We say that an exponent $\theta_l>0$ is admissible if there is some $C_l>0$ depending on $\theta_l$ such that all ellipses defined by definite integral binary quadratic forms belong to $\mathcal{L}(C_l,\theta_l)$.

Van Der Corput's \cite{vanDerCorput} bound implies that any $\theta_l>2/3$ is admissible, while the bound of Huxley \cite{Huxley} implies that any $\theta_l>131/208>0.6298$ is admissible.
\end{defi}
\begin{defi}
For any $m\in\cyclic{k}$ define for $i\in\{0,1,2\}$
\begin{equation*}
D_i(m)=\prod_{\substack{p \mid D_\mathrm{small} \\ \ord_p m=i}} p
\end{equation*}
Then $D_\mathrm{small}=D_0(m)D_1(m)D_2(m)$.
\end{defi}
\begin{prop}\label{prop:applying-the-sieve}
Let $m\in\cyclic{k}$ with $w_k(m)>0$. Let $\theta_l>0$ be admissible, fix $0<\eta<1/2$ and assume $R_{\max}^{\theta_l}\leq A(\mathscr{E})^{1-4\eta}$. If $\upsilon p_1^{2n}\leq |D|^{\eta/2}$ then
\begin{align*}
\sum_{\substack{(x,y)\in \mathscr{E} \cap \mathbb{Z}^2  
\\ Q(x,y) \equiv m \bmod k}}&(f\cdot r_0)\left(\frac{Q(x,y)}{\upsilon p_1^{2n}}\right) \ll_{\mathcal{f}, \eta}
A(\mathscr{E}) \frac{\rho_Q\left(m;(D_0(m)D_1(m))^2\right)}{(D_0(m)D_1(m))^4}\\
&\cdot\prod_{\substack{2 < p\leq 2\kappa |D|^{1-\eta} \\ p\nmid \upsilon p_1D_\mathrm{small}}} \left(1-\frac{\rho_Q(p)}{p^2}\right)
\sum_{\substack{a\leq 2\kappa |D|\\\gcd\left(a,D_0(m)D_1(m)\right)=1}} \frac{f(a)r_0(a)\widetilde{\rho}_Q(\upsilon p_1^{2n} D_2(m)^2 a)}{(\upsilon p_1^{2n} D_2(m)^2 a)^2}
\end{align*}
\end{prop}
\begin{proof}
Write 
\begin{equation*}
m\equiv D_1(m) l \mod D_0(m)^2 D_1(m)^2
\end{equation*}
where $l\in\left(\cyclic{D_0(m)^2D_1(m)}\right)^\times$.

Notice that if $p \parallel D$ then $f(pn)=f(n)$ for all $n\in\mathbb{N}$. Using this we write
\begin{equation*}
\sum_{\substack{(x,y)\in \mathscr{E} \cap \mathbb{Z}^2  
\\ Q(x,y) \equiv m \bmod k}}(f\cdot r_0)\left(\frac{Q(x,y)}{\upsilon p_1^{2n}}\right)
=
\sum_{\substack{(x,y)\in \mathscr{E} \cap \mathbb{Z}^2  
\\ Q(x,y) \equiv m \bmod k}}(f\cdot r_0)\left(\frac{Q(x,y)}{\upsilon p_1^{2n}D_2(m)^2}\right)
\end{equation*}
We wish to apply Proposition \ref{prop:2d-sieve-q0q1q2}. We now define $k_0$, $k_1$, $k_2$, $X$ and $\delta$ and verify that the conditions of the Proposition hold.

Set $k_0=\upsilon p_1^{2n} D_2(m)^2$,
$k_1=D_1(m)$ and $k_2=D_0(m)^2 D_1(m)$. Notice that $k_0k_1k_2=\upsilon p_1^{2n} D_2(m)^2D_0(m)^2D_1(m)^2=\upsilon p_1^{2n} D_\mathrm{small}=k$. Because $D_\mathrm{small}\leq |D|^{\eta/2}$ and using Lemma \ref{lem:ellipse-area-rmax} we deduce that
$k \leq |D|^{\eta}\leq A(\mathscr{E})^{\eta/2}$. 

For any $(x,y)\in\mathscr{E}\cap\mathbb{Z}^2$ we know from the definition of $Q(x,y)$ that $Q(x,y)=\upsilon\Nr(\mathfrak{b})$ where $(\mathfrak{a},\mathfrak{b})$ is a pair of integral ideals satisfying the conclusions of Proposition \ref{prop:intersection-to-invariants} with $x=\ctr(\xi)$. We deduce the following using the explicit formulae for $\kappa$ and $\omega$ from Theorem \ref{thm:cross-correlation-shifted-convolution}
\begin{align*}
\max\left\{\left|Q(x,y)\right|\relmiddle{|} (x,y)\in\mathscr{E}\cap\mathbb{Z}^2\right\}&= \max_{\substack{\Ideals_0(\Lambda)\ni\mathfrak{a}\subseteq{\Lambda} \\ \Nr\mathfrak{a}\leq \kappa |D|}} |\Nr\mathfrak{a}-\omega D|\\
&\leq (\kappa+|\omega|) |D| \leq 2\kappa |D|\leq A(\mathscr{E})
\end{align*}
Hence we can take $X= 2\kappa |D|$ and $\delta=2$ in the conditions of Proposition \ref{prop:2d-sieve-q0q1q2}.

Moreover, using the standard Euler product for the Dedekind $\zeta$-function of $E$ with the necessary modifications at primes dividing the conductor we see that for every $\varepsilon>0$ there is some $1\leq A\ll_{\mathcal{f}} 1$ and $0<B\ll_{\varepsilon,\mathcal{f}} 1$ so that
$f\in\mathcal{M}(A,B,\varepsilon)$. Finally, to apply Proposition \ref{prop:2d-sieve-q0q1q2} we need a bound of the form $\widetilde{\rho}_Q(p^k)\leq C p^{k(2-r)}$ for some $C\geq 1$, $0<r<1$. Such a bound holds with $C=16$ and $r=1/2$ due to Corollary \ref{cor:rho_Q-C-r}.

Notice that $f(k_1)=1$ because $k_1$ is supported on ramified primes and it is coprime to $\mathcal{f}$.
After applying Proposition \ref{prop:2d-sieve-q0q1q2} we arrive at the necessary sum with product and summation up to $X/(k_0 k_1)$. The final result follows because $2\kappa|D|=X\geq X/(k_0 k_1)\geq X/k\geq 2\kappa |D|^{1-\eta}$.
\end{proof}

\begin{lem}\label{lem:rho-Q-final}
Let $m\in\cyclic{k}$ with $w_k(m)>0$.
For any $a\in \mathbb{N}$ such that $\gcd(a,D_0(m)D_1(m))=1$ the following inequality holds
\begin{equation*}
\frac{\widetilde{\rho}_Q(\upsilon p_1^{2n} D_2(m)^2 a)}{(\upsilon p_1^{2n} D_2(m)^2 a)^2} \ll_{\mathcal{f},\mathbf{G}} \frac{|\omega|^2}{p_1^{2n} D_2(m)^2 a} 
\left[\prod_{p \mid a} \left(1+\frac{1}{p}\right)\right] \left[\prod_{p \mid D_2(m)} 2\left(1-\frac{1}{p}\right)\right] r_1(a) 
\end{equation*}
where $r_1$ is a multiplicative function defined by $r_1(p^k)=2$ for any $p\mid D_\mathrm{large}$ and $k>\ord_p D_\mathrm{large}$ and $r_1(p^k)=1$ otherwise.
\end{lem}
\begin{remark}
Notice that by definition $r_1\leq r_0$.
\end{remark}
\begin{proof}
Recall that $\upsilon$ and $D_2(m)$ are square-free.
To prove the lemma we use the multiplicativity of $\widetilde{\rho}_Q$ to write 
\begin{align*}
\widetilde{\rho}_Q(\upsilon p_1^{2n} D_2(m)^2 a)&=\left[\prod_{p \mid \upsilon} \widetilde{\rho}_Q\left(p^{\ord_p a+1}\right)\right] \widetilde{\rho}_Q\left(p_1^{\ord_{p_1} a+2n}\right)
\left[\prod_{p \mid D_2(m)} \widetilde{\rho}_Q\left(p^{\ord_p a+2}\right)\right] \\
&\cdot\prod_{\substack{p \mid a \\ p\nmid \upsilon p_1 D_2(m)}} \widetilde{\rho}_Q\left(p^{\ord_p a}\right)
\end{align*}
We treat each term above separately. 

Any prime dividing $\upsilon$ is necessary coprime to the conductor, c.f.\ \S\ref{sec:ramified-local-order}, and is inert in $E/\mathbb{Q}$; thus according to Proposition \ref{prop:rho_Q-regular}
\begin{equation*}
\prod_{p \mid \upsilon} \widetilde{\rho}_Q\left(p^{\ord_p a+1}\right)=\prod_{p \mid \upsilon} (p+1)p^{\ord_p a}=\upsilon \gcd(a,\upsilon^\infty) \prod_{p \mid \upsilon} \left(1+\frac{1}{p}\right)\ll_{\mathbf{G}} \gcd(a,\upsilon^\infty)
\end{equation*}
The last inequality holds because $\upsilon$ is bounded above by the product of all primes where $\mathbf{B}$ ramifies.

The prime $p_1$ is split in $E/\mathbb{Q}$ and coprime to $\mathcal{f}$, hence by  Proposition \ref{prop:rho_Q-regular}
\begin{equation*}
\widetilde{\rho}_Q\left(p_1^{\ord_{p_1} a+2n}\right)=(p_1-1)p_1^{\ord_{p_1} a+2n -1}<p_1^{2n} \gcd(a,p_1^\infty)
\end{equation*}

Next we consider all primes $p$ dividing $D_2(m)$. These are ramified in $E/\mathbb{Q}$ and coprime to $4\mathcal{f}\omega$. Hence due to Corollary \ref{cor:rhoQtilde-final} we know that
\begin{align*}
\prod_{p \mid D_2(m)}\widetilde{\rho}_Q(p^{\ord_p a +2})&=\prod_{p \mid D_2(m)}\frac{\rho_Q^0(p)}{p}p^{\ord_p a+2}\left(1-\frac{1}{p}\right)\\
&\leq D_2(m)^2 \gcd(a, D_2(m)^\infty)\prod_{p \mid D_2(m)} 2\left(1-\frac{1}{p}\right)
\end{align*}

We are left dealing with primes $p\mid a$ which are coprime to $\upsilon p_1 D_2(m)$. Because we have assumed $\gcd(a, D_0(m)D_1(m))=1$ we know that $p\nmid D_\mathrm{small}$. If $2<p\mid \omega$ then because of Proposition \ref{prop:Q-rho-sing} 
\begin{equation*}
\widetilde{\rho}_Q(p^{\ord_p a})\leq 2 \ord_p \omega p^{\ord_p \mathcal{f}} p^{\ord_p a} r_1(p^{\ord_p a})
\end{equation*}
Applying Proposition \ref{prop:Q-rho-sing}  for $p=2$ we deduce
\begin{align*}
\prod_{p\mid \gcd(4\omega,a)} \widetilde{\rho}_Q(p^{\ord_p a}) &\ll_{\mathcal{f}} 
\gcd(a,\omega^\infty) r_1(\gcd(a,\omega^\infty))
\prod_{p\mid \omega} 2\ord_p \omega\\ 
&\leq \gcd(a,\omega^\infty) r_1(\gcd(a,\omega^\infty)) |\omega|^2
\end{align*}
in the last inequality we have used the facts $\prod_{p\mid \omega} 2 \leq |\omega|$ and $\prod_{p\mid \omega} \ord_p \omega \leq |\omega|\leq d(|\omega|)\leq |\omega|$, where $d(|\omega|)$ is the number of divisors of $\omega$.

For any prime $p\mid a$ coprime to $4D_\mathrm{small}\omega$ we can apply Proposition \ref{prop:rho_Q-regular} and Corollary \ref{cor:rhoQtilde-final} to deduce
\begin{equation*}
\prod_{\substack{p \mid a \\ p \nmid 4D_\mathrm{small}\omega}}  
\widetilde{\rho}_Q(p^{\ord_p a}) \ll_{\mathcal{f}} 
\prod_{\substack{p \mid a \\ p \nmid 4D_\mathrm{small}\omega}} 
p^{\ord_p a}r_1(p^{\ord_p a})\left(1+\frac{1}{p}\right)
\end{equation*}

The claim follows by combining all the inequalities above for the different cases of $p$.
\end{proof}

\begin{prop}\label{prop:log-sums-with-rhoQ}
Let $m\in\cyclic{k}$ with $w_k(m)>0$. Let $\theta_l>0$ be admissible, fix $0<\eta<1/2$ and assume $R_{\max}^{\theta_l}\leq A(\mathscr{E})^{1-4\eta}$. If $\upsilon p_1^{2n}\leq |D|^{\eta/2}$ then
\begin{align*}
\sum_{\substack{(x,y)\in \mathscr{E} \cap \mathbb{Z}^2  
\\ Q(x,y) \equiv m \bmod k}}&(f\cdot r_0)\left(\frac{Q(x,y)}{\upsilon p_1^{2n}}\right) \ll_{\mathcal{f}, \eta, \mathbf{G}}
A(\mathscr{E}) p_1^{-2n} \frac{\rho_Q\left(m;(D_0(m)D_1(m))^2\right)}{(D_0(m)D_1(m))^4} \frac{2^{\omega(D_2(m))}}{D_2(m)^2}\\
&\cdot  |\omega|^2 (\log\log (2|\omega|))^8 \prod_{\substack{2 < p\leq 2\kappa |D|^{1-\eta} \\ p\nmid D_0(m)D_1(m)}} \left(1-\frac{\rho_Q(p)}{p^2}\right)
\sum_{\substack{a\leq 2\kappa |D|\\\gcd\left(a,D_0(m)D_1(m)\right)=1}} \frac{f(a)}{a}
\end{align*}
\end{prop}
\begin{proof}
We begin by substituting the result of Lemma \ref{lem:rho-Q-final} into Proposition \ref{prop:applying-the-sieve} to see that
\begin{align*}
\sum_{\substack{(x,y)\in \mathscr{E} \cap \mathbb{Z}^2  
\\ Q(x,y) \equiv m \bmod k}}&(f\cdot r_0)\left(\frac{Q(x,y)}{\upsilon p_1^{2n}}\right) \ll_{\mathcal{f}, \eta}
A(\mathscr{E}) p_1^{-2n} \frac{\rho_Q\left(m;(D_0(m)D_1(m))^2\right)}{(D_0(m)D_1(m))^4} \frac{2^{\omega(D_2(m))}}{D_2(m)^2}\\
&\cdot\omega\log\omega \prod_{\substack{2 < p\leq 2\kappa |D|^{1-\eta} \\ p\nmid \upsilon p_1D_0(m)D_1(m)}} \left(1-\frac{\rho_Q(p)}{p^2}\right)\\
&\cdot\sum_{\substack{a\leq 2\kappa |D|\\\gcd\left(a,D_0(m)D_1(m)\right)=1}} \frac{f(a)r_0(a)r_1(a)}{a}
\left[\prod_{p \mid a} \left(1+\frac{1}{p}\right)\right] 
\end{align*}

Using Lemma \ref{lem:Schwartz-Zippel} we deduce
\begin{equation*}
\prod_{\substack{2 < p\leq 2\kappa |D|^{1-\eta} \\ p\nmid \upsilon p_1D_0(m)D_1(m)}} \left(1-\frac{\rho_Q(p)}{p^2}\right)\ll_{\mathbf{G}}
\prod_{\substack{2 < p\leq 2\kappa |D|^{1-\eta} \\ p\nmid D_0(m)D_1(m)}} \left(1-\frac{\rho_Q(p)}{p^2}\right)
\end{equation*} 

Because $r_1(a)\leq r_0(a)$ to prove the claim we need only to show that
\begin{equation*}
\sum_{\substack{a\leq 2\kappa |D|\\\gcd\left(a,D_0(m)D_1(m)\right)=1}} \frac{f(a)r_0(a)^2}{a}
\left[\prod_{p \mid a} \left(1+\frac{1}{p}\right)\right] 
\ll_{\mathcal{f}, \eta} (\log\log (2|\omega|))^8 \sum_{\substack{a\leq 2\kappa |D|\\\gcd\left(a,D_0(m)D_1(m)\right)=1}} \frac{f(a)}{a}
\end{equation*}

We prove this by applying the decoupling lemma \ref{lem:multiplicative-decoupling} twice.
In the first application let 
with $g(p^l)=f(p^l)r_0(p^l)^2/p^l$ if $p\nmid D_0(m)D_1(m)$ and $g(p^l)=1$ otherwise; and $h(p^l)=\left(1+\frac{1}{p}\right)$ for $l\geq 1$. We can write $h=1*\psi$ where $\psi(p^l)=0$ for $l\geq 2$ and $\psi(p)=1/p$.

We need to estimate $\mathfrak{M}_z(g,\psi)$ as follows.
\begin{align*}
\mathfrak{M}_z(g,\psi) &\leq \prod_{p \leq \infty} \left[1 +\psi(p)\sum_{j=1}^{\infty} \frac{f(p^l)r_0(p^l)^2}{p^l}\right]
\ll_{\mathcal{f}} \prod_{p \leq \infty} \left[1 +\frac{1}{p}\sum_{j=1}^{\infty} \frac{2(k+1)}{p^l}\right]\\
&= \prod_{p \leq \infty} \left[1 +\frac{2}{p}\frac{2p-1}{(p-1)^2}\right]
\leq  \prod_{p \leq \infty} \left[1 +\frac{4}{(p-1)^2}\right] \ll 1
\end{align*}
where we have used the trivial bound $f(p^l)\leq (k+1)/2$ for every $p\nmid \mathcal{f}$.
We have thus proved that 
\begin{equation*}
\sum_{\substack{a\leq 2\kappa |D|\\\gcd\left(a,D_0(m)D_1(m)\right)=1}} \frac{f(a)r_0(a)^2}{a}
\left[\prod_{p \mid a} \left(1+\frac{1}{p}\right)\right] 
\ll_{\mathcal{f}} \sum_{\substack{a\leq 2\kappa |D|\\\gcd\left(a,D_0(m)D_1(m)\right)=1}} \frac{f(a)r_0(a)^2}{a}
\end{equation*}

We continue by applying Lemma \ref{lem:multiplicative-decoupling} again. This time with $g(p^l)=f(p^l)/p^l$ whenever $p\nmid D_0(m)D_1(m)$ and $g(p^l)=1$ otherwise; and $h(p^l)=r_0(p^l)^2$. We have $h=1*\psi$ where $\psi(p^{\ord_p D})=4$ for $p\mid D_\mathrm{high}$ and $\psi(p^l)=0$ for all other prime powers with $l\geq 1$. We estimate $\mathfrak{M}_z(g,\psi)$ in the following way.
\begin{align*}
\mathfrak{M}_z(g,\psi)&\leq \prod_{p \mid D_\mathrm{high}} \left(1 +4 \sum_{j=\ord_p D}^{\infty} \frac{f(p^l)}{p^l}\right) 
\ll_{\mathcal{f}} \prod_{p \mid D_\mathrm{high}} \left(1 +4 \sum_{j=1}^{\infty} \frac{1}{p^l}\right) \\
&= \prod_{p \mid D_\mathrm{high}} \left(1 +\frac{4}{p-1}\right)\leq \prod_{p \mid D_\mathrm{high}} \left(1 +\frac{8}{p}\right)\leq \prod_{p \mid D_\mathrm{high}} \left(1 +\frac{1}{p}\right)^8\\
&\leq \prod_{p \mid \omega} \left(1 +\frac{1}{p}\right)^8 \prod_{\substack{p \mid D \\ p>\eta/(4 C_\theta)\log |D|}} \left(1 +\frac{1}{p}\right)^8 
\end{align*}
We bound the two factors above separately. The first one can be bounded because
\begin{equation*}
\prod_{p \mid \omega} \left(1 +\frac{1}{p}\right) \ll \log\log (2|\omega|)
\end{equation*}
For the second factor we have the following upper bound due to \eqref{eq:chebyshev-bounds}
\begin{equation*}
\prod_{\substack{p \mid D \\ p>\eta/(4 C_\theta)\log |D|}} \left(1 +\frac{1}{p}\right)
\leq
\prod_{\eta/(4 C_\theta)\log |D| <p\leq C_\theta \log |D|}  \left(1 +\frac{1}{p}\right)\ll_{\eta} 1
\end{equation*}
\end{proof}
The second inequality holds due to Mertens' theorem.

\begin{prop}\label{prop:ANT-bound}
Let $m\in\cyclic{k}$ with $w_k(m)>0$.
Fix $1/2>\eta>0$.
If $C>0$ satisfies
\begin{equation*}
\frac{L'(1,\ch_E)}{L(1,\ch_E)} \leq C \log |D_E|
\end{equation*}
then
\begin{align*}
\prod_{\substack{2 < p\leq 2\kappa |D|^{1-\eta} \\ p\nmid D_0(m)D_1(m)}} \left(1-\frac{\rho_Q(p)}{p^2}\right)
&\sum_{\substack{a\leq 2\kappa |D|\\\gcd\left(a,D_0(m)D_1(m)\right)=1}} \frac{f(a)}{a}\\
\\ &\ll_{C,\eta}
L(1,\ch_E)  \prod_{p\mid \mathcal{f}} \left(1-\left(\frac{D_E}{p}\right)\frac{1}{p}\right)
+|D|^{-2/3+o(1)}
\end{align*}
\end{prop}
\begin{proof}
We first estimate the product over primes appearing above. From Propositions \ref{prop:rho_Q-regular} and \ref{prop:Q-rho-sing} we deduce that $\rho_Q(p)\geq p-1$ for all primes $p$. Thus 
\begin{align}
\nonumber
\prod_{\substack{2 < p\leq 2\kappa |D|^{1-\eta} \\ p\nmid D_0(m)D_1(m)}} &\left(1-\frac{\rho_Q(p)}{p^2}\right) \leq
\prod_{\substack{2 < p\leq 2\kappa |D|^{1-\eta} \\ p\nmid D_0(m)D_1(m)}} \left(1-\frac{1}{p}+\frac{1}{p^2}\right) \\
\nonumber
&\leq
\prod_{\substack{2 < p\leq 2\kappa |D|^{1-\eta} \\ p\nmid D_0(m)D_1(m)}} \left(1-\frac{1}{p}\right) \prod_{p< \infty} \left(1+\frac{1}{p^2-p}\right)\\
\nonumber
&\leq
\prod_{\substack{2 < p\leq 2\kappa |D|^{1-\eta} \\ p\nmid D_0(m)D_1(m)}} \left(1-\frac{1}{p}\right) \prod_{p< \infty} \left(1+\frac{2}{p^2}\right)
\ll
\prod_{\substack{2 < p\leq 2\kappa |D|^{1-\eta} \\ p\nmid D_0(m)D_1(m)}} \left(1-\frac{1}{p}\right)\\
\label{eq:rhoQ-prod-bound}
&\ll
{\log(2\kappa |D|^{1-\eta})}^{-1} \prod_{p\mid D_0(m)D_1(m)} \left(1-\frac{1}{p}\right)^{-1} 
\end{align}
The last inequality above follows from Mertens' theorem.

The logarithmic mean 
\begin{equation*}
\sum_{\substack{a\leq 2\kappa |D|\\\gcd\left(a,D_0(m)D_1(m)\right)=1}} \frac{f(a)}{a}
\end{equation*}
can be estimated using standard tools from multiplicative number theory.
Consider the multiplicative function $g$ defined by $g(p^l)=f(p^l)$ if $p \nmid D_0(m)D_1(m)$ and $g(p^l)=1$ if $g\mid  D_0(m)D_1(m)$. Then because of the decomposition $\zeta_E(s)=\zeta(s) L(s,\ch_E)$ with $L(s,\ch_E)$ holomorphic the Dirichlet series of $g$ can be written as $L_g(s)=\zeta(s) \widetilde{L_g}(s)$  with $\widetilde{L_g}(s)$ holomorphic. 

Let $\varphi\colon [0,\infty)\to[0,\infty)$ be a compactly-supported smooth non-increasing function satisfying $\mathbb{1}_{[0,1]}\leq \varphi \leq \mathbb{1}_{[0,2]}$, i.e.\ $\varphi$ is a smooth approximation of the characteristic function of $[0,1]$. Notice that the Mellin transform satisfies $s\mathcal{M}(\varphi)(s)=\mathcal{M}(\Theta \varphi)(s)$, where $\Theta(\varphi)(x)=-x\varphi'(x)\geq 0$
is a smooth compactly-supported function vanishing outside of $[1,2]$. Hence $s\mathcal{M}(\varphi)(s)$ decays faster then any polynomial in the vertical direction. The decay is uniform in any strip of the form $\sigma_0\leq \Re(s)\leq \sigma_1$. The same property holds for $\mathcal{M}(\varphi)(s)$ outside a small neighborhood of $s=0$. Moreover, the the Laurent expansion of $\mathcal{M}(\varphi)(s)$ around $s=0$ is $\frac{1}{s}+\int_{1}^\infty\frac{\varphi(x)}{x}\dif x +O(|s|)$.
Using contour integration, the Perron formula and the decay of Dirichlet $L$-functions in the vertical direction we see that
\begin{align}
\label{eq:logarithmic-mean-bound}
\sum_{\substack{a\leq 2\kappa |D|\\\gcd\left(a,D_0(m)D_1(m)\right)=1}} \frac{f(a)}{a}
&\leq
 \widetilde{L_g}(1) \left(\log(2\kappa |D|) + \gamma + \frac{\widetilde{L_g}'(1)}{\widetilde{L_g}(1)}+\int_{1}^\infty\frac{\varphi(x)}{x}\dif x \right)\\
&+\frac{1}{2\pi}\int_{-\infty}^\infty \frac{|L_g(1/2+it)|}{(2\kappa |D|)^{1/2}} |\mathcal{M}\varphi(-1/2+it)|  \dif t
\nonumber
\end{align}
where $\gamma$ is the Euler-Mascheroni constant.

Because all the primes $p\mid D_0(m)D_1(m)$ are ramified in $E/\mathbb{Q}$ and coprime to $\mathcal{f}$ the following properties of $\widetilde{L_g}(s)$ are an immediate consequence of comparing the Euler product of $L_g(s)$ with that of $\zeta_E(s)$.
\begin{align*}
&\widetilde{L_g}(1)=L(1,\ch_E) \prod_{p\mid \mathcal{f}} \left(1-\left(\frac{D_E}{p}\right)\frac{1}{p}\right)
\prod_{p \mid D_0(m)D_1(m)} \left(1-\frac{1}{p}\right)\\
&\left|\frac{\widetilde{L_g}'(1)}{\widetilde{L_g}(1)}- C\log |D_E|\right|\ll 1\\
&|\widetilde{L_g}(1/2+it)|\ll_{\mathcal{f}} |L(1/2+it,\ch_E)|\ll_{\mathcal{f}}|D|^{1/6+o(1)} |1/2+it|^A
\end{align*}
The constant $A>0$ is absolute.
The last inequality for $|L(1/2,\ch_E)|$ is due to Conrey and Iwaniec \cite[Corollary 1.5]{ConreyIwaniec} strengthening the convexity breaking result of Brugess \cite{Burgess} for real characters.
Substituting these and \eqref{eq:rhoQ-prod-bound} into \eqref{eq:logarithmic-mean-bound} and using the super-polynomial decay of $|\mathcal{M}(\varphi)(1/2+it)|$ we deduce 
\begin{align*}
\prod_{\substack{2 < p\leq 2\kappa |D|^{1-\eta} \\ p\nmid D_0(m)D_1(m)}} &\left(1-\frac{\rho_Q(p)}{p^2}\right) \leq
{\log(2\kappa |D|^{1-\eta})}^{-1}  L(1,\ch_E) \prod_{p\mid \mathcal{f}} \left(1-\left(\frac{D_E}{p}\right)\frac{1}{p}\right)\\
&\cdot \left(\log(2\kappa |D|)+ C\log|D| + O(1)\right)+ O(1)\cdot\frac{|D|^{1/6+o(1)}}{(2\kappa |D|)^{1/2}}\\
&\ll_{\eta,C} L(1,\ch_E) \prod_{p\mid \mathcal{f}} \left(1-\left(\frac{D_E}{p}\right)\frac{1}{p}\right) +|D|^{-2/3+o(1)}
\end{align*}
\end{proof}

We can now combine all the results of this section to deduce a final bound on the shifted convolution sum.
\begin{prop}\label{prop:shifted-conv-final-bound} Let $\theta_l>0$ be admissible, fix $0<\eta<1/2$ and assume $R_{\max}^{\theta_l}\leq A(\mathscr{E})^{1-4\eta}$. Suppose $\upsilon p_1^{2n}\leq |D|^{\eta/2}$. If $C>0$ satisfies
\begin{equation*}
\frac{L'(1,\ch_E)}{L(1,\ch_E)} \leq C \log |D_E|
\end{equation*}
then
\begin{align*}
\sum_{\substack{0\leq x\leq \kappa |D| \\ x\equiv w|D| \mod \upsilon p_1^{2n}}} &g_{[\mathfrak{s}]}(x)
f_{[p_1^n\mathfrak{se}]^{-1}}\left(\frac{x-\omega D}{\upsilon p_1^{2n}}\right) r\left(\frac{x-\omega D}{\upsilon p_1^{2n}}\right)\\
&\ll_{\mathbf{G},\mathcal{f},\eta,C} 
\kappa |\omega| \log (2|\omega|) (\log\log (2|\omega|))^8
\frac{\sqrt{|D|}\left(L(1,\ch_E)+|D|^{-2/3+o(1)}\right)}{p_1^{2n}} 
\end{align*}
\end{prop}
\begin{proof}
In this proof only we allow all implicit constants to depend on $C,\eta,\mathcal{f},\mathbf{G}$ without specifying that further.

From Lemmata \ref{lem:shifted-to-polynomial}, \ref{lem:final-sums-over-polynomial} and Propositions \ref{prop:log-sums-with-rhoQ} and \ref{prop:ANT-bound} we deduce that
the shifted convolution sums is bounded above by
\begin{align*}
A(\mathscr{E}) p_1^{-2n}
&\left(L(1,\ch_E)+|D|^{-2/3+o(1)}\right)
|\omega|^2 (\log\log(2|\omega|))^8\\
&\cdot\int_{m\equiv 0 \bmod \upsilon p_1^{2n}} 
\frac{\rho_Q\left(m;(D_0(m)D_1(m))^2\right)}{(D_0(m)D_1(m))^4} \frac{2^{\omega(D_2(m))}}{D_2(m)^2}
\dif w_k(m)
\end{align*}
The claim would follow immediately from the formula for $A(\mathscr{E})$ in Lemma \ref{lem:ellipse-area-rmax} if we prove that
\begin{equation*}
\int_{m\equiv 0 \bmod \upsilon p_1^{2n}} 
\frac{\rho_Q\left(m;(D_0(m)D_1(m))^2\right)}{(D_0(m)D_1(m))^4} \frac{2^{\omega(D_2(m))}}{D_2(m)^2}
\dif w_k(m) 
\ll 1
\end{equation*}
The integrand decomposes as a product of functions on $\cyclic{p^2}$ for $p\mid D_\mathrm{small}$ and the measure is a product measure, thus we can use Fubini to write
\begin{align*}
&\int_{m\equiv 0 \bmod \upsilon p_1^{2n}} 
\frac{\rho_Q\left(m;(D_0(m)D_1(m))^2\right)}{(D_0(m)D_1(m))^4} \frac{2^{\omega(D_2(m))}}{D_2(m)^2}
\dif w_k(m)\\ 
&=
\prod_{p \mid D_\mathrm{small}}\left[
\frac{2}{p^2}w_p(0)+
\sum_{0\neq a \in \cyclic{p^2}} \frac{\rho_Q(a;p^2)}{p^4}w_p(a)
\right]
\end{align*}

We bound the term for each $p\mid D_\mathrm{small}$ using the definition of $w_p$ and Proposition \ref{prop:genus-mod-p^2}.
\begin{align*}
\frac{2}{p^2}w_p(0)+
\sum_{0\neq a \in \cyclic{p^2}} \frac{\rho_Q(a;p^2)}{p^4}w_p(a)
&=\frac{4}{p^2}+\left(1-\frac{1}{p}\right)+\frac{1}{p}\left(1-\frac{f}{p}\right)\\
&=1+\frac{4-f}{p^2}\leq 1+\frac{5}{p^2}
\end{align*}
where 
\begin{equation*}
f= 1+\epsilon\left(\frac{-D/p}{p}\right)\ch_p(\Nr\mathfrak{s}^{-1})+\left(\frac{\omega}{p}\right)\ch_p(\Nr\mathfrak{s}^{-1})\in\{-1,1,3\}
\end{equation*}

We conclude that the integral in question is bounded above by
\begin{equation*}
\prod_{p<\infty}\left(1+\frac{5}{p^2}\right)\ll 1
\end{equation*}
\end{proof}

\subsection{Conclusion of the Proof}
Let $\theta_l>0$ be an admissible exponent for lattice counting in ellipses.
If there is some $\eta_0>0$ such that for all $i\gg 1$
\begin{equation}\label{eq:N-big-vs-D}
\mathfrak{N}_i\geq |D_i|^{(2-\theta_l^{-1})/3+\eta_0}
\end{equation}
Then using Lemma \ref{lem:ellipse-area-rmax} we can deduce that the condition $R_{\max}^{\theta_l}\leq A(\mathscr{E})^{1-4\eta}$ holds for all $\mathcal{H}_i$ in the sequence where $1/2>\eta>0$ depends only on $\eta_0$ and $\theta_l$. Assume first that such $\eta_0>0$ exists. The condition that all fields $E_i/\mathbb{Q}$ have no exceptional zero implies that there is $C>0$ independent of $i$ such that
\begin{equation*}
\frac{L'(1,\ch_{E_i})}{L(1,\ch_{E_i})} \leq C \log |D_{E_i}|
\end{equation*}
This result has been attributed to Hecke by Landau \cite{Landau1918}.

Let $\xi\in\left(\mathbf{G}\times\mathbf{G}\right)(\mathbb{A})$.
Fix $n\in\mathbb{N}$ then for any $i\gg_{p_1,n,\varepsilon,\mathbf{G}} 1$ we have $\upsilon p_1^{2n}\leq |D_i|^{\eta/2}$. Moreover, the assumptions of Theorem \ref{thm:joinings-adelic} imply that
\begin{equation*}
g_i^{-1}\mathbf{T}_i(\mathbb{Q})s_i g_i \cap B^{(-n,n)}\ctr(\xi)B^{(-n,n)}=\emptyset
\end{equation*}
for all $i\gg_{\xi} 1$.

Thus for $i$ large enough we can use Proposition \ref{prop:shifted-conv-final-bound} and Theorem \ref{thm:cross-correlation-shifted-convolution} to deduce for any $n\in\mathbb{N}$ 
\begin{align*}
\Cor[\mu_i,\nu_{\xi}](B^{(-n,n)})&\ll_{\mathbf{G},\varepsilon,\mathcal{f}}
\vol\left(\left[\mathbf{T}(\mathbb{A})g\right]\right)^{-1} 
\vol\left(\left[\mathbf{G}^\Delta(\mathbb{A})^+\xi\right]\right)^{-1} p_1^{-2n}\\
&\kappa |\omega|^2 (\log\log (2|\omega|))^8
\frac{\sqrt{|D|}\left(L(1,\ch_E)+|D|^{-2/3+o(1)}\right)}{p_1^{2n}} \\
&\ll_{\mathcal{f}}
\vol\left(\left[\mathbf{G}^\Delta(\mathbb{A})^+\xi\right]\right)^{-1} \kappa|\omega|^2 (\log\log (2|\omega|))^8
p_1^{-4n}
\end{align*}
The last inequality follows from the computation of the volume of a homogeneous toral set using the analytic class number formula, c.f.\ \cite{ELMVCubic}.
The expression $\kappa |\omega|^2 (\log\log (2|\omega|))^8$ is a continuous function of $\ctr(\xi)$ as can be seen from the definition of $\kappa$ and $\omega$ in Theorem \ref{thm:cross-correlation-shifted-convolution}. Moreover, the definition of the volume implies immediately that $\vol\left(\left[\mathbf{G}^\Delta(\mathbb{A})^+\xi\right]\right)$ is a non-vanishing continuous function of $\xi\in\lfaktor{\mathbf{G}^\Delta(\mathbb{A})^+}{\left(\mathbf{G}\times\mathbf{G}\right)(\mathbb{A})}$. Because the fiber of the continuous map $\ctr\colon \lfaktor{\mathbf{G}^\Delta(\mathbb{A})^+} {\left(\mathbf{G}\times\mathbf{G}\right)(\mathbb{A})} \to \mathbf{G}(\mathbb{A})$ is compact\footnote{The fiber is isomorphic to $\faktor{\mathbf{G}(\mathbb{A})}{\mathbf{G}(\mathbb{A})^+}$} the function $\xi\mapsto \vol\left(\left[\mathbf{G}^\Delta(\mathbb{A})^+\xi\right]\right)$ is bounded below by a non-vanishing continuous function of $\ctr(\xi)$.

We deduce that if condition \eqref{eq:N-big-vs-D} holds then the proof is concluded by Lemma \ref{lem:cross-correlation-to-thm}.
If condition \eqref{eq:N-big-vs-D} fails then the theorem follows from the methods of Ellenberg, Michel and Venkatesh \cite[\S 3]{EMV} which is based on Linnik's method for equidistribution of CM points. Although the argument of \cite[\S 3]{EMV} applies verbatim only to the case of  $\mathbf{G}$ ramified at $\infty$ and functions invariant under an Iwahori at the place $p_1$, these restrictions are relaxed using the technical improvements presented in \cite[\S 5]{Kh17}.

The following discussion is a recap of \cite{EMV} with an emphasis on the required adaptation when removing the restriction $\gcd(\mathfrak{N}_i,p_1)=1$. Let $m\in\mathbb{N}$ to be determined later. Because we have assumed a splitting condition for two primes we can use the flow either at $p_1$ or at $p_2$.
The input required by \cite{EMV} and \cite{Kh17} is a norm gap for the Hecke operator $$T_{p_j^m}\colon \dfaktor{\mathbf{G}^\sc(\mathbb{Q})}{\mathbf{G}^\sc(\mathbb{A})}{U}\to \dfaktor{\mathbf{G}^\sc(\mathbb{Q})}{\mathbf{G}^\sc(\mathbb{A})}{U}$$ where $j\in\{1,2\}$ and $U=\prod_{p} U_p<\mathbf{G}^\sc(\mathbb{A}_f)$ is a compact-open subgroup such that $U_{p_j}$ is the intersection of two Iwahori in the apartment corresponding to $A_{p_j}$. Fix $K_{p_j}<\mathbf{G}^\sc(\mathbb{Q}_{p_j})$ a maximal compact subgroup containing $U_{p_j}$. The following norm gap for $T_{p_j^m}$ follows for any $\varepsilon>0$ from the decay of matrix coefficients of automorphic representations of $\mathbf{SL}_2$ \cite{Sarnak,BurgerSarnak,ClozelUllmo,ClozelOhUllmo} and the Jacquet-Langlands correspondence \cite{JacquetLangalnds}
\begin{equation}\label{eq:HeckeOp-Ramanujan}
\| T_{p_j^m}\|_{0} \ll_{U^{p_j},p_j,\epsilon} [K_{p_j} \colon U_{p_j}]p_j^{-m(1-\theta+\epsilon)/2}  
\end{equation}
where $\| T_{p_j^m}\|_{0}$ is the norm of $T_{p_j^m}$ restricted to the subspace of $L^2\left(\left[\mathbf{G}^\sc(\mathbb{A})\right]\right)^U$ orthogonal to the constant function, $U^{p_j}=\prod_{p\neq p_j} U_p$ and $\theta$ is the best bound towards the Generalized Ramanujan Conjecture for $\mathbf{SL}_2$ in the sense of \cite{ClozelUllmo}. The dependence of the constant on the  parameters $U^{p^j}$,$p_j$ and $\epsilon$ is effective and can be made explicit.

In the Ellenberg-Michel-Venkatesh argument we restrict a joint homogeneous toral set to an ambient Hecke correspondence of volume $\mathfrak{N}_i\prod_{p\mid \mathfrak{N}_i}\left(1+\frac{1}{p}\right)$ and apply an effective equidistribution argument in conjunction with Linnik's Basic Lemma to show that the joint period toral measure is close to the Haar measure on the ambient Hecke correspondence. 

Let $a_j\in A_{p_j}$ be as in Definition \ref{def:Bowen-ball}. We can use the effective equidistribution theorem from \cite{Kh17} which builds upon the work of \cite{LinnikBook,EMV} to deduce the necessary equidistribution result for the Ellenberg-Michel-Venkatesh argument as long as for each $i$ there is some $m$ such that $\|T_{p_j^m}\|_{0}<1/2$ and that Linnik's Basic Lemma is valid for $p_1^{m(1+\epsilon')}$ for some $\epsilon'>0$.

If $p_j\mid \mathfrak{N}_i$ we know that
$[K_{p_j}\colon U_{p_j}]=p_j^{\ord_{p_j}\mathfrak{N}_i} \left(1+\frac{1}{p_j}\right)$, cf.\ \S\ref{sec:Hecke}.
Because of the freedom to use either $p_1$ or $p_2$ we can assume without loss of generality that
\begin{equation*}
p_1^{\ord_{p_1}\mathfrak{N}_i}\leq \sqrt{\mathfrak{N}_i}
\end{equation*} 
For a fixed $U^{p_1}$ the bound $\|T_{p_1^m}\|_{0}<1/2$ would follow from \eqref{eq:HeckeOp-Ramanujan} for any $m\in\mathbb{N}$ satisfying
\begin{equation}\label{eq:EMV-m-from-above}
p_1^m\gg_{U^{p_1},p_1,\epsilon} \mathfrak{N}_i^{1/(1-\theta+\epsilon)}
\end{equation}
On the other hand Linnik's Basic Lemma  for $\emph{one-sided}$ Bowen balls in this setting, cf. \cite{EMV}, applies only for $m\in\mathbb{N}$ in the range
\begin{equation}\label{eq:EMV-m-from-below}
\mathfrak{N}_i p_1^{m}\leq|D|^{1/2+o(1)}
\end{equation}
where the $o(1)$ is ineffective as it is derived from Siegel's bound. There exists an $m\in\mathbb{N}$ satisfying both \eqref{eq:EMV-m-from-above} and \eqref{eq:EMV-m-from-below} if
\begin{equation*}
\mathfrak{N}_i^{1+1/(1-\theta+\epsilon)}\ll_{U^{p_1},p_1,\epsilon} |D|^{1/2+o(1)}
\end{equation*}
This condition is satisfied if we know that there is $\epsilon_1>0$ such that for all $i\gg 1$
\begin{equation}\label{eq:emv-range}
\mathfrak{N}_i\leq |D|^{\frac{1}{2+2/(1-\theta)}-\epsilon_1}
\end{equation}

In the range \eqref{eq:emv-range} the conclusion of the Ellenberg-Michel-Venkatesh argument is that for any limit measure $\mu$ one has $\int f\dif \mu =0$ for  any smooth compactly supported $f\in L^2_{00}\left(\left[\left(\mathbf{G}\times\mathbf{G}\right)(\mathbb{Q})\right], \meas_{\mathbf{G}\times\mathbf{G}}\right)$ which is invariant in the place $p_1$ under an intersection of two Iwahori subgroups stabilizing  edges in the aparatment of $A_{p_1}$.

We can now bootstrap this to deduce the conclusion of the theorem. Let $A_{p_1}^0<A_{p_1}$ be the maximal compact subgroup of the torus.
Using a decreasing sequence of intersections of two Iwahori we conclude that the push forward of the limit measure $\mu$ to $$\dfaktor{\left(\mathbf{G}\times\mathbf{G}\right)(\mathbb{Q})}{\left(\mathbf{G}\times\mathbf{G}\right)(\mathbb{A})}{A_{p_1}^0\times A_{p_1}^0}$$ is a measure of maximal entropy for the action of $a^\Delta$ for any element of $a\in A_{p_1}^+$ which is not contained in a compact subgroup. 
As this factor and the space $\left[\left(\mathbf{G}\times\mathbf{G}\right)(\mathbb{A})\right]$ have the same maximal entropy for $a^\Delta$ we deduce that any limit measure $\mu$ has maximal entropy for $a$ on $\left[\left(\mathbf{G}\times\mathbf{G}\right)(\mathbb{A})\right]$ which implies it is $\left(\mathbf{G}\times\mathbf{G}\right)(\mathbb{A})^+$-invariant.

To conclude the proof we need to verify that the range of \eqref{eq:N-big-vs-D} overlaps with the range of \eqref{eq:emv-range}. Taking $\theta_l>2/3$ from Van der Corput's bound \cite{vanDerCorput} and $\theta=1/2$ from Gelbart-Jacquet we see that any improvement to either bound would imply the necessary overlap in ranges. This can be achieved either by taking a smaller admissible value of $\theta_l$ such as provided by \cite{Huxley} or using any improvement towards Ramanujan beyond $\theta\leq 1/2$ as in \cite{Shahidi,LRS}. 

\appendix
\section{Principal Genus Theory}\label{appndx:principal-genus}
In this appendix we collect results related to the principal genus theory of quadratic orders. The results we discuss are classical when presented in an elementary form, going back to Gauss in the case of maximal orders \footnote{The reader interested in the history of the development of principal genus theory can consult the review \cite{Lemmermeyer} by Lemmermeyer.}. Unfortunately, the author is unfamiliar with a modern concise presentation treating the case of non-maximal orders. This appendix contains all the statements that are of use in this manuscript with complete proofs.

As is usually the case with topics in algebraic number theory 
the treatment is significantly streamlined by the use of ad\'eles. Noticeable features of the presentation below is that is uses class field theory only for quadratic extension and does not resort to the properties of ring class fields and genus fields. The main tools are  Hilbert's Satz 90 for quadratic global and local fields, the Hasse norm theorem, which for quadratic fields was proven by Hilbert and elementary Galois cohomology.
Except for treatment of the wild prime $2$ I have tried to circumvent explicit computations wherever possible.

\paragraph{Notations}
Let $\Lambda$ be an order in an imaginary quadratic extension $E<\mathbb{Q}$. As usual we denote by $D$ the discriminant of $\Lambda$ and define $\Lambda_v$ to be the closure of $\Lambda$ in $E_v\coloneqq \prod_{w|v} E_w$ for any rational place $v\neq\infty$.

\subsection{Adelic Form of \texorpdfstring{$\Pic(\Lambda)/\Pic(\Lambda)^2$}{Pic/Pic^2}}
The adelic interpretation below is used both for computing the structure of the group $\faktor{\Pic(\Lambda)}{\Pic(\Lambda)^2}$ and in describing the characters in the dual group $\widehat{\Pic(\Lambda)}$ vanishing on $\Pic(\Lambda)^2$ using Kronecker symbols of ideal norms.

\begin{prop}\label{prop:Pic/Pic2-adelic-short-exact}
The adelic norm map $\Nr\colon\mathbb{A}_E^\times\to\mathbb{A}^\times$ and the real adelic character $\ch_E\colon \lfaktor{\mathbb{Q}^\times}{\mathbb{A}^\times}\to\{\pm1\}$ attached to the quadratic extension $E/\mathbb{Q}$ by global class field theory descend to a short exact sequence
\begin{equation*}
1\to \faktor{\Pic(\Lambda)}{\Pic(\Lambda)^2} \xrightarrow{\Nr}
\dfaktor{\mathbb{Q}^\times}{\mathbb{A}^\times}{\mathbb{R}_{>0}\prod_{v\neq\infty}{\Nr\Lambda^\times}}\xrightarrow{\ch_E} \{\pm 1\}\to 1
\end{equation*}
\end{prop}
\begin{proof}
Recall that $\Pic(\Lambda)\simeq \dfaktor{E^\times}{\mathbb{A}_E^\times}{\mathbb{C}^\times \prod_{v\neq\infty}\Lambda_v^\times}$.
Because $\sigma$ acts by inversion on $\Pic(\Lambda)$ the group $\Pic(\Lambda)^2$ is equal to the group of coboundaries $\cbd\left(\Pic(\Lambda)\right)$. Hence $\Pic(\Lambda)^2$ is the image of $\prod_v\cbd(E_v)<\mathbb{A}_E^\times$ in $\dfaktor{E^\times}{\mathbb{A}_E^\times}{\mathbb{C}^\times \prod_{v\neq\infty}\Lambda_v^\times}$. Hilbert's Satz 90 implies $\cbd(E_v)=E_v^{(1)}$ for each $v$, hence $\prod_v\cbd(E_v)$ is the kernel of the norm map $\mathbb{A}_E^\times\to\mathbb{A}^\times$.

The norm map descends to a map 
\begin{equation*}
\Nr\colon\lfaktor{E^\times}{\mathbb{A}_E^\times}\to 
\lfaktor{\mathbb{Q}^\times}{\mathbb{A}^\times}
\end{equation*}
The Hasse norm theorem  implies that the kernel of this map is the projection of $\prod_v\cbd(E_v)$. Global class field theory states the the image is the kernel of $\ch_E$. It follows that there is a norm map
\begin{equation*}
\Nr\colon\Pic(\Lambda)\to 
\dfaktor{\mathbb{Q}^\times}{\mathbb{A}^\times}{\mathbb{R}_{>0}\prod_{v\neq\infty}\Nr\Lambda_v^\times}
\end{equation*}
with kernel $\Pic(\Lambda)^2$. Moreover, the conductor of the quadratic character $\ch_E$ contains $\Nr E_\infty^\times \prod_{v\neq\infty} \Nr \mathcal{O}_{E_v}^\times$ thus $\ch_E$ factors through the right hand side above and its kernel is the image of $\Nr$.
\end{proof}

\begin{cor}\label{cor:Pic-Pic2-index-local}
The index $\left[\Pic(\Lambda)\colon\Pic(\Lambda)^2\right]$ can be computed by
\begin{equation*}
2\left[\Pic(\Lambda)\colon\Pic(\Lambda)^2\right]=\prod_{v\neq\infty}\left[\mathbb{Z}_v^\times \colon \Nr\Lambda_v^\times\right]
\end{equation*}
\end{cor}
\begin{proof}
By Proposition \ref{prop:Pic/Pic2-adelic-short-exact} above the group $\dfaktor{\mathbb{Q}^\times}{\mathbb{A}^\times}{\mathbb{R}_{>0}\prod_{v\neq\infty}{\Nr\Lambda^\times}}$ is a $2$-cover of $\faktor{\Pic(\Lambda)}{\Pic(\Lambda)^2}$. Thus we need only to compute the size of this adelic quotient. 

We use the fact that $\mathbb{Q}$ has class number $1$ to conclude that the following sequence is exact
\begin{equation*}
1\to
\mathbb{Z}^\times \cdot \mathbb{R}_{>0}\prod_{v\neq\infty} \Nr\Lambda_v^\times\to
\mathbb{R}^\times\prod_{v\neq\infty} \mathbb{Z}_v^\times\to 
\dfaktor{\mathbb{Q}^\times}{\mathbb{A}^\times}{\mathbb{R}_{>0}\prod_{v\neq\infty} \Nr\Lambda_v^\times} 
\to 1
\end{equation*}
Moreover, the inclusion map descends to an isomorphism
\begin{equation*}
\prod_{v\neq\infty} \faktor{\mathbb{Z}_v^\times}{\Nr\Lambda_v^\times}\to \dfaktor{\mathbb{Z}^\times}{\mathbb{R}^\times\prod_{v\neq\infty} \mathbb{Z}_v^\times}{\mathbb{R}_{>0}\prod_{v\neq\infty} \Nr\Lambda_v^\times}
\end{equation*}
\end{proof}

The following two lemmata are necessary in order to understand non-maximal orders in terms of their reduction modulo the conductor.
\begin{lem}\label{lem:1-mod-conductor}
Assume $\Lambda_v<\mathcal{O}_{E_v}$ is a \emph{non}-maximal order.
Let $\mathcal{f}_v\mathcal{O}_{E_v}$ be the conductor of $\Lambda_v$ then $1+\mathcal{f}_v\mathcal{O}_{E_v}\subseteq \Lambda_v^\times$.
\end{lem}
\begin{proof}
Because $\Lambda_v$ is non-maximal $\ord_v\mathcal{f}_v\geq 1$ for $\mathcal{f}_v\in\mathbb{Z}_v$ as above. This implies that the Taylor series for $\frac{1}{1+x}$ converges for any $x\in\mathcal{f}_v\mathcal{O}_{E_v}$. As $\Lambda_v$ is a closed subset we deduce $(1+x)^{-1}\in\Lambda_v$.
\end{proof}

\begin{lem}\label{lem:Lambda-units-reduction}
Assume $\Lambda_v<\mathcal{O}_{E_v}$ is a non-maximal order. Consider the reduction map $\red_{\mathcal{f}_v}\colon \mathcal{O}_{E_v}\to \faktor{\mathcal{O}_{E_v}}{\mathcal{f}_v\mathcal{O}_{E_v}}$ then 
\begin{equation*}
\Lambda_v^\times=\red_{\mathcal{f}_v}^{-1}(\red_{\mathcal{f}_v}(\mathbb{Z})^\times)=\mathbb{Z}_v^\times+\mathcal{f}_v\mathcal{O}_{E_v}
\end{equation*}
\end{lem}
\begin{proof}
Notice that $\Lambda_v=\red_{\mathcal{f}_v}^{-1}(\red_{\mathcal{f}_v}(\mathbb{Z}_v))=\red_{\mathcal{f}_v}^{-1}(\red_{\mathcal{f}_v}(\mathbb{Z}))$. The first equality follows from Lemma \ref{lem:Lambda-Galois} and the second equality holds because $\mathbb{Z}$ is dense in $\mathbb{Z}_v$. It follows immediately that $\Lambda_v^\times\subseteq \red_{\mathcal{f}_v}^{-1}(\red_{\mathcal{f}_v}(\mathbb{Z})^\times)$. The reverse inclusion is a consequence of Lemma \ref{lem:1-mod-conductor} above.
\end{proof}

We now compute the groups $\Nr\Lambda_v^\times$ appearing in Proposition \ref{prop:Pic/Pic2-adelic-short-exact}.
\begin{lem}\label{lem:OEv-norms}
Fix $v\neq\infty$ and denote by $p$ the residue characteristic of $\mathbb{Q}_v$. Then
\begin{equation*}
\Nr\mathcal{O}_{E_v}^\times=\begin{cases}
\mathbb{Z}_v^\times & E_v/\mathbb{Q}_v \textrm{ is unramified}\\
{\mathbb{Z}_v^\times}^2 & E_v/\mathbb{Q}_v \textrm{ is ramified and } p>2\\
\end{cases}
\end{equation*}
If $p=2$ ramifies in $E/\mathbb{Q}$ then $\Nr\mathcal{O}_{E_v}^\times$ is one of the $3$ possible index $2$ subgroups of $\mathbb{Z}_2^\times$ containing the index $4$ subgroup ${\mathbb{Z}_2^\times}^2$, i.e.\ one of the index $2$ subgroups of 
\begin{equation*}
\faktor{\mathbb{Z}_v^\times}{{\mathbb{Z}_v^\times}^2}\simeq \left(\cyclic{8}\right)^\times\simeq \cyclic{2}\times\cyclic{2}
\end{equation*}
\end{lem}
\begin{proof}
Notice that always  $\mathbb{Z}_v^\times<\mathcal{O}_{E_v}^\times$ hence ${\mathbb{Z}_v^\times}^2<\Nr\mathcal{O}_{E_v}^\times$ for all $v\neq\infty$.

The claim now follows immediately from the local class field correspondence between degree $2$ extensions of $\mathbb{Q}_v$ and index $2$ subgroups of $\mathbb{Q}_v^\times$. 
\end{proof}

\begin{lem}\label{lem:NrLambda_v}
Assume $\Lambda_v\lneq \mathcal{O}_{E_v}$ is a \emph{non}-maximal order. If $p>2$ then
\begin{equation*}
\Nr\Lambda_v^\times={\mathbb{Z}_v^\times}^2
\end{equation*}
If $p=2$ then
\begin{equation*}
\Nr\Lambda_v^\times=\begin{cases}
\Nr\mathcal{O}_{E_v}^\times & 2\parallel \mathcal{f}_v\\
1+4\mathbb{Z}_2 & 4\parallel \mathcal{f}_v \textrm{ and } E_v/\mathbb{Q}_v \textrm{ is unramified or } \Nr\mathcal{O}_{E_v}^\times=1+4\mathbb{Z}_2\\
1+8\mathbb{Z}_2={\mathbb{Z}_v^\times}^2 & \textrm{otherwise}\\
\end{cases}
\end{equation*}
\end{lem}
\begin{proof}
For any $v\neq\infty$ we have $\mathbb{Z}_v^\times<\Lambda_v^\times$ and ${\mathbb{Z}_v^\times}^2<\Nr\Lambda_v^\times$. Also for all $v$ we know $\Nr\Lambda_v^\times<\Nr\mathcal{O}_{E_v}^\times$.
Hence if $E_v/\mathbb{Q}_v$ is ramified and $p>2$ Lemma \ref{lem:OEv-norms} above implies that $\Nr\Lambda_v^\times={\mathbb{Z}_v^\times}^2$.

Assume now $E_v/\mathbb{Q}_v$ is unramified or $p=2$.
We compute $\faktor{\Nr\mathcal{O}_{E_v}^\times}{\Nr\Lambda_v^\times}$ in the following way
\begin{align*}
\faktor{\Nr\mathcal{O}_{E_v}^\times}{\Nr\Lambda_v^\times}
&\simeq
\dfaktor{\mathcal{O}_{E_v}^\times}{\Lambda_v^\times}{\mathcal{O}_{E_v}^{(1)}}
\simeq 
\dfaktor{\red_{\mathcal{f}_v}\left(\mathcal{O}_{E_v}\right)^\times}{\red_{\mathcal{f}_v}\left(\Lambda_v\right)^\times}{\mathcal{O}_{E_v}^{(1)}}\\
&\simeq
\faktor{\Nr\red_{\mathcal{f}_v}\left(\mathcal{O}_{E_v}\right)^\times}{\Nr\red_{\mathcal{f}_v}\left(\Lambda_v\right)^\times}
=
\faktor{\red_{\mathcal{f}_v}\Nr\left(\mathcal{O}_{E_v}\right)^\times}{\red_{\mathcal{f}_v}\Nr\left(\Lambda_v\right)^\times}
\end{align*}
The first and the third equalities above hold because the kernels of the norm maps are the corresponding norm $1$ elements; the
second equality follows from Lemma \ref{lem:Lambda-units-reduction} and the fourth equality holds because the reduction map is equivariant for the Galois action. As all the isomorphisms above are canonical their composite is exactly the reduction map $\red_{\mathcal{f}_v}$.
We use Lemma \ref{lem:Lambda-units-reduction} once more to deduce $\red_{\mathcal{f}_v}\Nr\left(\Lambda_v\right)^\times\simeq {\left(\cyclic{\mathcal{f}_v}\right)^\times}^2$. We continue case by case.

If $E_v/\mathbb{Q}_v$ is unramified we use Lemma \ref{lem:OEv-norms} to deduce 
\begin{align*}
\faktor{\mathbb{Z}_v^\times}{\Nr\Lambda_v^\times} 
&=
\faktor{\Nr\mathcal{O}_{E_v}^\times}{\Nr\Lambda_v^\times} 
\simeq 
\left(\cyclic{\mathcal{f}_v}\right)^\times 
\Big\slash
{\left(\cyclic{\mathcal{f}_v}\right)^\times}^2\\
&\simeq
\begin{cases}
\cyclic{2} & p>2 \\
1 & p=2,2\parallel \mathcal{f}_v\\
\cyclic{2} & p=2, 4\parallel \mathcal{f}_v\\
\cyclic{2}\times \cyclic{2} & p=2, 8\mid \mathcal{f}_v\\
\end{cases}
\end{align*} 
This settles all the cases when $p>2$ and the unramified case when $p=2$.

Assume $p=2$ and $E_v/\mathbb{Q}_v$ is ramified.
If $2\parallel\mathcal{f}_v$ then $\Nr\mathcal{O}_{E_v}^\times=\Nr\Lambda_v^\times$ because $\mathbb{F}_2^\times=1$. If $4\parallel \mathcal{f}_v$ then $\red_{\mathcal{f}_v} \Nr\Lambda_v^\times\equiv \left\{1\right\}$  hence $\Nr\Lambda_v^\times= \Nr\mathcal{O}_{E_v}^\times\cap 1+4\mathbb{Z}_2$. If $\Nr\mathcal{O}_{E_v}^\times=1+4\mathbb{Z}_2=\{1,-3\}+8\mathbb{Z}_2$ then we deduce $\Nr\Lambda_v^\times= \Nr\mathcal{O}_{E_v}^\times$ otherwise $\Nr\Lambda_v^\times=1+8\mathbb{Z}_2$.

If $8\mid\mathcal{f}_v$ then $\red_{\mathcal{f}_v}\Nr\Lambda_v^\times=\{1\}$ and ${\mathbb{Z}_v^\times}^2<\Nr\Lambda_v^\times$ hence $\Nr\Lambda_v^\times={\mathbb{Z}_v^\times}^2=1+8\mathbb{Z}_2$.
\end{proof}

\begin{cor}\label{cor:Pic(Lambda)[2]-size}
Let $\mu_\mathrm{tame}$ be the number of \emph{odd} primes dividing $D$. Set
\begin{equation*}
\mu_\mathrm{wild}\coloneqq \ord_2 \left[\mathbb{Z}_2^\times\colon\Nr\Lambda_2^\times\right]
\in \{0,1,2\}
\end{equation*}
Notice that $\mu_\mathrm{wild}$ depends only on $\mathcal{f}_2$ and the ramification of $2$ in $E/\mathbb{Q}$.

The following equalities holds for any imaginary quadratic order $\Lambda$
\begin{equation*}
\#\Pic(\Lambda)[2]=\left[\Pic(\Lambda)\colon\Pic(\Lambda)^2\right]
=2^{\mu_\mathrm{tame}+\mu_\mathrm{wild}-1}
\asymp 2^{\omega(D)}
\end{equation*}
\end{cor}
\begin{cor}
The first equality holds because the squaring homomorphism fits in a short exact sequence
\begin{equation*}
1\to\Pic(\Lambda)[2]\to\Pic(\Lambda)\xrightarrow{x\mapsto x^2} \Pic(\Lambda)^2\to 1
\end{equation*}

The second equality follows from Corollary \ref{cor:Pic-Pic2-index-local} and Lemma \ref{lem:NrLambda_v} above.
\end{cor}

\subsection{Characters Orthogonal to \texorpdfstring{$\Pic(\Lambda)^2$}{Pic^2}}
\begin{defi}\label{def:Kronecker-characters}
For any prime $p>2$ define $p^*\coloneqq (-1)^{\frac{p-1}{2}}p\equiv 1 \mod 4$ and set for $n\in\mathbb{N}$ 
\begin{equation*}
\ch_p(n)=\left(\frac{p^*}{n}\right)
\end{equation*}
This is the unique non-trivial primitive real Dirichlet character of modulus $p$.

Define also
\begin{equation*}
\ch_4(n)=\left(\frac{-4}{n}\right),\;
\ch_8(n)=\left(\frac{8}{n}\right)
\end{equation*}
The unique non-trivial primitive real Dirichlet character of modulus $4$ is $\ch_4$ and the non-trivial primitive real Dirichlet characters of 
modulus $8$ are $\ch_8$ and $\ch_4\ch_8$.

We extend multiplicatively every Dirichlet character of modulus $q$ to the multiplicative group of rationals which are coprime to all prime divisors of $q$.

Moreover, we abuse the notation and denote by $\ch_q\colon \lfaktor{\mathbb{Q}^\times}{\mathbb{A}^\times}\to\{\pm 1\}$ the adelic lift of the corresponding Dirichlet character.  
\end{defi} 

\begin{prop}\label{prop:principal-genus-thm}
Let $\mathfrak{a}\in \Ideals(\Lambda)$ then $[\mathfrak{a}]\in\Pic(\Lambda)^2$ if and only if
\begin{equation*}
\ch\left(\frac{\Nr(\mathfrak{a})}{\gcd(\Nr(\mathfrak{a}),\operatorname{modulus}(\chi)^\infty)}\right)=1 
\end{equation*}
for all the following real Dirichlet characters $\ch$
\begin{enumerate}
\item $\ch_p$ for all \emph{odd} primes $p\mid D$,
\item one of\footnote{The specific character depends on the subgroup $\Nr \Lambda_2^\times <\mathbb{Z}_2^\times$.} $\ch_4,\ch_8,\ch_4\ch_8$ if $\left[\mathbb{Z}_2^\times\colon\Nr\Lambda_2^\times\right]=2$, 
\item $\ch_4$  and $\ch_8$ if $\Nr\Lambda_2^\times=1+8\mathbb{Z}_2$.
\end{enumerate}
\end{prop}
\begin{proof}
Our goal is to compute all the characters orthogonal to $\Pic(\Lambda)^2$.
Because $\left(\Pic(\Lambda)^2\right)^\perp=\widehat{\Pic(\Lambda)}[2]$ all these characters are real.

Consider the short exact sequence of character groups dual to the short exact sequence of Proposition \ref{prop:Pic/Pic2-adelic-short-exact}
\begin{equation*}
1\leftarrow 
\left(\Pic(\Lambda)^2\right)^\perp
\xleftarrow{\widehat{\Nr}}
\widehat{\left(\dfaktor{\mathbb{Q}^\times}{\mathbb{A}^\times}{\mathbb{R}_{>0}\prod_{v\neq\infty}{\Nr\Lambda^\times}}\right)}
\xleftarrow{\widehat{\ch_E}} \{\pm 1\}
\leftarrow 1
\end{equation*}
Exactness implies that every character in $\left(\Pic(\Lambda)^2\right)^\perp$ can be expressed as a composition with the norm map of a real rational Hecke grossencharacter $\lfaktor{\mathbb{Q}^\times}{\mathbb{A}^\times}\to\{\pm 1\}$ which vanishes on $\mathbb{R}_{>0}\prod_{v\neq\infty}\Nr\Lambda_v^\times$. Moreover, the only non-trivial relation is that $\chi_E$ is trivial in $\left(\Pic(\Lambda)^2\right)^\perp$. 

The translation between finite rational Hecke grossencharacters and Dirichlet characters implies that the relevant characters are adelic lifts of real Dirichlet characters with conductor containing $\mathbb{R}_{>0}\prod_{v\neq\infty}\Nr\Lambda_v^\times$. Using Lemma \ref{lem:NrLambda_v} and the fact that all primitive real Dirichlet characters are the Kronecker symbols described in Definition \ref{def:Kronecker-characters} we deduce that $\widehat{\left(\dfaktor{\mathbb{Q}^\times}{\mathbb{A}^\times}{\mathbb{R}_{>0}\prod_{v\neq\infty}{\Nr\Lambda^\times}}\right)}$ is generated by the adelic lifts of the characters listed in the claim.
The explicit expressions for the evaluation of a character at the norm of an ideal follows by unwinding the adelic lifting procedure.
\end{proof}

\subsection{2-Torsion in the Picard Group}
The cohomological interpretation of the $2$-torsion in the Picard group is used in the description of the fiber of the invariant map attaching pairs of fractional ideals to intersections.

\begin{prop}\label{prop:Pic(Lambda)[2]-H1(Lambda_v)}
The diagonal restriction map
\begin{equation*}
H^1(\mathfrak{G},\Lambda^\times)\to \prod_{v\neq\infty} H^1(\mathfrak{G},\Lambda_v^\times)
\end{equation*}
is injective and there is a canonical isomorphism
\begin{equation*}
\lfaktor{H^1(\mathfrak{G},\Lambda^\times)}
{\prod_{v\neq\infty} H^1(\mathfrak{G},\Lambda_v^\times)}
\simeq \Pic(\Lambda)[2]
\end{equation*}
\end{prop}
\begin{proof}
We construct the necessary isomorphism in several steps. On the way we also prove the claimed injectivity.
For any order $\mathcal{O}<\mathcal{O}_F$ in a global field $F$ denote by $\Princp(\mathcal{O})\coloneqq \lfaktor{\mathcal{O}^\times}{F^\times}$ --- the group of invertible principle fractional $\mathcal{O}$-ideals.

We begin by examining the following commuting diagram with exact rows
\begin{equation*}
\begin{tikzcd}
1 \arrow[r] & \Lambda^\times \arrow[r]\arrow[d] & E^\times \arrow[r]\arrow[d] & \Princp(\Lambda) \arrow[r]\arrow[d] & 1\\
1 \arrow[r] & \prod_{v\neq\infty} \Lambda_v^\times \arrow[r] & \prod_{v\neq\infty}' E_v^\times \arrow[r] & \Ideals(\Lambda) \arrow[r] & 1
\end{tikzcd}
\end{equation*}
This diagram induces a commuting diagram of $\mathfrak{G}$-cohomology with exact rows.
\begin{equation}\label{diag:local-global-Lambda-ideals}
\begin{tikzcd}
1 \arrow[r] & \mathbb{Z}^\times \arrow[r]\arrow[d] & \mathbb{Q}^\times \arrow[r]\arrow[d] & \Princp(\Lambda)^\mathfrak{G} \arrow[r]\arrow[d] & H^1(\mathfrak{G},\Lambda^\times)\arrow[r]\arrow[d] & 1\\
1 \arrow[r] & \prod_{v\neq\infty} \mathbb{Z}_v^\times \arrow[r] & \prod_{v\neq\infty}' \mathbb{Q}_v^\times \arrow[r] & \Ideals(\Lambda)^\mathfrak{G} \arrow[r] & \prod_{v\neq\infty} H^1(\mathbb{Q}_v,\Lambda_v^\times) \arrow[r] & 1
\end{tikzcd}
\end{equation}
The last terms are trivial due to Hilbert's Satz 90.
We can truncate the diagram above to the following commuting diagram with exact rows
\begin{equation}\label{diag:local-global-Lambda-ideals-truncated}
\begin{tikzcd}
1 \arrow[r] & \Princp(\mathbb{Z}) \arrow[r]\arrow[d] & \Princp(\Lambda)^\mathfrak{G} \arrow[r]\arrow[d] & H^1(\mathfrak{G},\Lambda^\times)\arrow[r]\arrow[d] & 1\\
1 \arrow[r] & \Ideals(\mathbb{Z}) \arrow[r] &  \Ideals(\Lambda)^\mathfrak{G} \arrow[r] & \prod_{v\neq\infty} H^1(\mathbb{Q}_v,\Lambda_v^\times) \arrow[r] & 1
\end{tikzcd}
\end{equation}
Because $\Princp(\Lambda)^\mathfrak{G}\to\Ideals(\Lambda)^\mathfrak{G}$ is injective the four-lemma implies that
\begin{equation*}
\ker\left[H^1(\mathfrak{G},\Lambda^\times)\to {\prod_{v\neq\infty} H^1(\mathbb{Q}_v,\Lambda_v^\times)}\right]\simeq 
\lfaktor{\Princp(\mathbb{Z})}{\Ideals(\mathbb{Z})}=\Pic(\mathbb{Z})=1
\end{equation*}
This proves the first claim.
Next we deduce from \eqref{diag:local-global-Lambda-ideals-truncated}
\begin{equation}\label{eq:P(Lambda)J(Lambda)-H1}
\lfaktor{\Princp(\Lambda)^\mathfrak{G}}{\Ideals(\Lambda)^\mathfrak{G}}\simeq \lfaktor{H^1(\mathfrak{G},\Lambda^\times)}{\prod_{v\neq\infty} H^1(\mathbb{Q}_v,\Lambda_v^\times)}
\end{equation}

We can also continue the long exact sequence in the first row of \eqref{diag:local-global-Lambda-ideals} 
\begin{equation}\label{eq:H1(Princ(Lambda))}
H^1(\mathfrak{G},\Lambda^\times)\to 1 \to H^1(\mathfrak{G},\Princp(\Lambda))
\to\faktor{\mathbb{Z}^\times}{\Nr\Lambda^\times}\to \faktor{\mathbb{Q}^\times}{\Nr E^\times}
\end{equation}
where we have computed the second cohomology groups using the formula $H^2(C,M)\simeq \faktor{M^C}{\Nr M}$ valid for any finite cyclic group $C$ acting on an abelian group $M$.
Because $E$ is an imaginary quadratic field $\mathbb{Z}^\times\cap\Nr E^\times=1$ thus the last map in \eqref{eq:H1(Princ(Lambda))} is injective. We deduce by exactness $H^1(\mathfrak{G},\Princp(\Lambda))=1$.

We finally consider the long exact sequence associated to the short exact sequence
\begin{equation*}
1\to \Princp(\Lambda) \to \Ideals(\Lambda)\to\Pic(\Lambda)\to 1
\end{equation*}
The equality $H^1(\mathfrak{G},\Princp(\Lambda))=1$ implies
\begin{equation}\label{eq:PicLambda-JLambda}
\Pic(\Lambda)[2]=\Pic(\Lambda)^{\mathfrak{G}}\simeq \lfaktor{\Princp(\Lambda)^\mathfrak{G}}{\Ideals(\Lambda)^{\mathfrak{G}}}
\end{equation}
The claimed  isomorphism is a composition of \eqref{eq:PicLambda-JLambda} and \eqref{eq:P(Lambda)J(Lambda)-H1}.
\end{proof}

\begin{cor}\label{cor:Pic(Lambda)[2]-H1(Lambda_v)-size}
Recall the definition $\cbd(x)=x/\tensor[^\sigma]{x}{}$ for $x\in E_v^\times$ for any $v$. The Proposition above implies
\begin{equation*}
\prod_{v\neq\infty} \left[\Lambda_v^{(1)} \colon \cbd(\Lambda_v^\times)\right]= 
\prod_{v\neq\infty} \# H^1(\mathfrak{G},\Lambda_v^\times)
=
2 \#\Pic(\Lambda)[2]
\end{equation*}
\end{cor}
\begin{proof}
The defintion of $H^1$ using cocycles and coboundaries implies $H^1(\mathfrak{G},L)\simeq\faktor{L^{(1)}}{\cbd(L^\times)}$ for $L=\Lambda$ and $L=\Lambda_v$ for all $v\neq\infty$. Hence by Proposition \ref{prop:Pic(Lambda)[2]-H1(Lambda_v)} the factor of proportionality between $\prod_{v\neq\infty} \left[\Lambda_v^{(1)} \colon \cbd(\Lambda_v^\times)\right]$ and $\#\Pic(\Lambda)[2]$ is $\left[\Lambda^{(1)} \colon \cbd(\Lambda^\times)\right]=2$.
\end{proof}

\begin{lem}\label{lem:units-cohomology-maximal}
Let $v\neq\infty$.
The first Galois cohomology group of the unit group of a maximal order is
\begin{equation*}
\begin{cases}
H^1(\mathbb{Q}_v, \mathcal{O}_{E_v}^\times)=1 & \text{if }E_v/\mathbb{Q}_v\text{ is unramified}\\
H^1(\mathbb{Q}_v, \mathcal{O}_{E_v}^\times)=\cyclic{2} & \text{if }E_v/\mathbb{Q}_v\text{ is ramified}
\end{cases}
\end{equation*}
If $E_v/\mathbb{Q}_v$ is tamely totally ramified, i.e.\ the residue characteristic is odd, then the non-trivial class of $H^1(\mathbb{Q}_v, \mathcal{O}_{E_v}^\times)$ is represented by the cocycle corresponding to $-1\in\mathcal{O}_{E_v}^{(1)}$.
\end{lem}
\begin{proof}
Denote by $g$ the number of places above $v$ in $E/\mathbb{Q}$ and let $e$ be the ramification index of $v$ in $E/\mathbb{Q}$.
Consider the short exact sequence
\begin{equation*}
1\to \mathcal{O}_{E_v}^\times\to E_v^\times \to \mathbb{Z}^{g}\to 1
\end{equation*}
The third map is the valuation map and if $g=2$ the Galois group acts on the value group $\mathbb{Z}^{g}$ by switching the coordinates.
The associated long exact cohomology sequence is
\begin{equation*}
1\to \mathbb{Z}_v^\times\to \mathbb{Q}_v^\times \to \mathbb{Z} \to H^1(\mathfrak{G},\mathcal{O}_{E_v}^\times)\to 1
\end{equation*}
where the last group is trivial by Hilbert's Satz 90 and the third map is the valuation map of $E_v^\times$ restricted to $\mathbb{Q}_v^\times$. The image of $\mathbb{Q}_v^\times \to \mathbb{Z}$ is $e\mathbb{Z}$ and we deduce that
\begin{equation*}
H^1(\mathfrak{G},\mathcal{O}_{E_v}^\times)\simeq \cyclic{e}
\end{equation*} 

Assume $E_v/\mathbb{Q}_v$ is ramified.
If $\Pi$ is a uniformizer of $\mathcal{O}_{E_v}$ then the map $\mathbb{Z} \to H^1(\mathfrak{G},\mathcal{O}_{E_v}^\times)$ can be written explicitly as $n\mapsto \cbd(\Pi^n)$.
If the ramification is tame we can choose a uniformizer so that $\tensor[^\sigma]{\Pi}{}=-\Pi$ and $H^1(\mathfrak{G},\mathcal{O}_{E_v}^\times)$ is generated by $-1$.
\end{proof}

\begin{lem}\label{lem:units-cohomology-non-maximal}
Assume $\Lambda_v<\mathcal{O}_{E_v}$ is a non-maximal order and denote the residue characteristic by $p$. Then 
\begin{equation*}
\# H^1(\mathfrak{G},\Lambda_v^\times)=\begin{cases}
2 & p>2 \\
2^{\mu_{\mathrm{wild}}} & p=2
\end{cases}
\end{equation*}

Moreover, when $p>2$ the non-trivial cocycle of $H^1(\mathfrak{G},\Lambda_v^\times)$ is represented by $-1\in\Lambda_v^{(1)}$.
\end{lem}
\begin{proof}
The short exact sequence of abelian $\mathfrak{G}$-modules
\begin{equation*}
1\to\Lambda_v^\times\to \mathcal{O}_{E_v}^\times\to \faktor{\mathcal{O}_{E_v}^\times}{\Lambda_v^\times}
\to 1
\end{equation*}
induces a long exact sequence of cohomology
\begin{align}
\label{eq:Lambdav-OEv-long-chomology}
1&\to \mathbb{Z}_v^\times 
\to\mathbb{Z}_v^\times
\to \left(\faktor{\mathcal{O}_{E_v}^\times}{\Lambda_v^\times}\right)^\mathfrak{G}\\
&\to H^1(\mathbb{Q}_v,\Lambda_v^\times)\to H^1(\mathbb{Q}_v,\mathcal{O}_{E_v}^\times)
\to H^1\left(\mathbb{Q}_v,\faktor{\mathcal{O}_{E_v}^\times}{\Lambda_v^\times}\right) 
\nonumber\\
&\to \faktor{\mathbb{Z}_v^\times}{\Nr\Lambda_v^\times}\to \faktor{\mathbb{Z}_v^\times}{\Nr\mathcal{O}_{E_v}^\times}
\nonumber
\end{align}
\end{proof}
Because $\mathfrak{G}$ acts on $\faktor{\mathcal{O}_{E_v}^\times}{\Lambda_v^\times}$ by inversion we see that
\begin{align*}
\left(\faktor{\mathcal{O}_{E_v}^\times}{\Lambda_v^\times}\right)^\mathfrak{G}&\simeq \left(\faktor{\mathcal{O}_{E_v}^\times}{\Lambda_v^\times}\right)[2]\\
H^1\left(\mathbb{Q}_v,\faktor{\mathcal{O}_{E_v}^\times}{\Lambda_v^\times}\right)&\simeq
\left(\faktor{\mathcal{O}_{E_v}^\times}{\Lambda_v^\times}\right) \Big\slash
\left(\faktor{\mathcal{O}_{E_v}^\times}{\Lambda_v^\times}\right)^2
\end{align*}

The second map in \eqref{eq:Lambdav-OEv-long-chomology} is simply the identity and we can truncate the sequence \eqref{eq:Lambdav-OEv-long-chomology} to an exact sequence
\begin{align*}
1&\to\left(\faktor{\mathcal{O}_{E_v}^\times}{\Lambda_v^\times}\right)[2]\\
&\to H^1(\mathbb{Q}_v,\Lambda_v^\times)\to H^1(\mathbb{Q}_v,\mathcal{O}_{E_v}^\times)
\to\left(\faktor{\mathcal{O}_{E_v}^\times}{\Lambda_v^\times}\right) \Big\slash
\left(\faktor{\mathcal{O}_{E_v}^\times}{\Lambda_v^\times}\right)^2\\
&\to \faktor{\Nr\mathcal{O}_{E_v}^\times}{\Nr \Lambda_v^\times}\to 1
\end{align*}

The second and the fifth groups above are non-canonically isomorphic and exactness implies 
\begin{equation*}
\# H^1(\mathbb{Q}_v,\Lambda_v^\times)=\#H^1(\mathbb{Q}_v,\mathcal{O}_{E_v}^\times)
\cdot \#\faktor{\Nr\mathcal{O}_{E_v}^\times}{\Nr \Lambda_v^\times}
=\begin{cases}
2 & p>2\\
2^{\mu_{\mathrm{wild}}} & p=2
\end{cases}
\end{equation*}
The second equality above follows from Lemmata \ref{lem:units-cohomology-maximal}, \ref{lem:OEv-norms} and \ref{lem:NrLambda_v}.

We need only show that the cocycle of $-1\in\Lambda_v^{(1)}$ is not a coboundary if $p>2$. Assume in the contrary the $-1=x/\tensor[^\sigma]{x}{}$ for some $x\in\Lambda_v^\times$. Then by Lemma \ref{lem:Lambda-units-reduction} we know that
\begin{equation*}
-1=\frac{x}{\tensor[^\sigma]{x}{}}\equiv \frac{x}{x} \mod \mathcal{f}_v\mathcal{O}_{E_v}\equiv 1 \mod \mathcal{f}_v\mathcal{O}_{E_v}
\end{equation*}
which is a contradiction because the residue characteristic is odd.

\section{Points on Conics}\label{appndx:conics}
\subsection{Notations}
In this section we denote by $q(x,y)$ a primitive binary integral quadratic form of discriminant $D<0$ and define $Q(x,y)=q(x,y)-\omega D$ for some $\omega\in\mathbb{Z}$. We denote by $X_Q\coloneqq \Spec \mathbb{Z}[x,y]/\left<Q(x,y)\right>$ the affine plane curve cutout by $Q$.

We shall also need the homogenized polynomial $\overline{Q}(x,y,z)=q(x,y)-\omega Dz^2$. Denote by $\overline{X_Q}\coloneqq \Proj \mathbb{Z}[x,y,z]/\left<\overline{Q}(x,y,z)\right>$ the projective completion of the curve $X_Q$.
The plane curve $X_Q$ is an affine conic and $\overline{X_Q}$ is a projective conic.

\subsection{Local Diagonlization of Binary Quadratic Forms}
\begin{lem}\label{lem:q-diagonal-Zp}
For any prime $p$ the form $q(x,y)$ is equivalent over $\mathbb{Z}_p$ to a form $q'(x,y)$ with  $\mathbb{Z}_p$ coefficients and satisfying the following.

If $p>2$ or $p=2$ and $D\equiv 0 \mod 4$  then $q'(x,y)=ux^2 +A y^2$ is diagonal. Moreover, we can assume $u\in \mathbb{Z}_p^{\times}$. If $p^l\parallel A$ then we write $A=u_A p^l$ for $u_A\in\mathbb{Z}_p^\times$.

For $p=2$ and  $D\equiv 1\mod 4$ 
\begin{equation*}
q'(x,y)=\begin{cases}
xy & D\equiv +1 \mod 8 \Longleftrightarrow \left(\frac{D}{2}\right)=+1\\
x^2+xy+y^2 & D \equiv -3 \mod 8 \Longleftrightarrow \left(\frac{D}{2}\right)=-1
\end{cases}
\end{equation*}
\end{lem}
\begin{proof}
This is classical, cf.\ \cite[Chapter 8]{CasselsQuadraticForms}.
\end{proof}
\begin{remark}
Assume $q$ corresponds to the ideal class $[\mathfrak{s}]\in\Pic(\Lambda)$ where $\Lambda$ is a quadratic order of discriminant $D$.
If $p>2$ and $q'(x,y)=ux^2+Ay^2$ as above then $\left(\frac{u}{p}\right)=\ch_p(\Nr [\mathfrak{s}])$ where $\ch_p$ is the genus class group  character from Proposition \ref{prop:principal-genus-thm}. In particular, the class of $u$ in $\faktor{\mathbb{Z}_p^\times}{{\mathbb{Z}_p^\times}^2}$ depends only on $[\mathfrak{s}] \mod \Pic(\Lambda)^2$. By abuse of notation we shall denote for odd $p\mid D$
\begin{equation*}
\ch_p(q)\coloneqq \ch_p(\Nr[s])=\left(\frac{u}{p}\right)
\end{equation*}

Moreover, because $D=-4u A$ for $p>2$ we have $\left(\frac{u_A}{p}\right)=\left(\frac{-D/p^l}{p}\right)\ch_p(q)$ where $p^l \parallel D$.
\end{remark}

\subsection{Regular Primes}
\begin{prop}\label{prop:rho_Q-regular}
If $p\nmid \omega D$ then $\widetilde{\rho}_Q(p^k)=\rho_Q(p^k)=\rho_Q(p)p^{k-1}$ and
\begin{equation*}
\rho_Q(p)=p-\left(\frac{D}{p}\right)
\end{equation*} 
\end{prop}

\begin{proof}
If $p\nmid D$ then $\overline{X_Q}$ and $X_Q$ have a smooth reduction modulo $p$. The first claim is an application of Hensel's lemma.

The reduction of $\overline{X_Q}$ is a smooth conic over a finite field and it is isomorphic to the projective line. In particular $\left|\overline{X_Q}\left(\cyclic{p}\right)\right|=p+1$.

To calculate $\rho_Q(p)=\left|X_Q\left(\cyclic{p}\right)\right|$ we need to subtract the points of $\overline{X_Q}$ on the line at infinity $z=0$. These are exactly the points on the projective variety cutout by $q(x,y)$, equivalently $q'(x,y)$ from Lemma \ref{lem:q-diagonal-Zp}. There are either $2$ or $0$ such points depending on the Kronecker symbol $\left(\frac{D}{p}\right)$.
\end{proof}

\subsection{Singular Primes}

\begin{lem}\label{lem:singular-conic}
Let $q_0(x,y)\in\mathbb{Z}_p[x,y]$ be a homogeneous binary quadratic form such that either $q_0(x,y)=u_1x^2+u_2y^2$ with $u_1,u_2\in\mathbb{Z}_p^\times$ or $p=2$ and $q_0(x,y)=xy$ or $q_0(x,y)=x^2+xy+y^2$. 

Fix $u_3\in\mathbb{Z}_p^\times$. For any integer $m\geq 0$ define
\begin{equation*}
Q_m(x,y)\coloneqq q_0(x,y)-u_3p^m
\end{equation*}
If $p>2$ or $p=2$ and $q_0$ is not diagonal then
\begin{equation*}
\rho_{Q_m}(p^n)=\begin{cases}
\left\lceil\frac{n}{2} \right\rceil p^{n-1}\left(1-\left(\frac{\disc(q_0)}{p}\right)\right)
+ p^n\delta_{n\equiv 0 \bmod 2}+p^{n-1}\delta_{n\equiv 1\bmod 2}
& n\leq m \\
\left(1+\left\lfloor\frac{m}{2} \right\rfloor \right)
p^{n-1}\left(1-\left(\frac{\disc(q_0)}{p}\right)\right)
+p^n\left(1-\frac{1}{p}\right)\delta_{m\equiv 0 \bmod 2} 
& n>m
\end{cases}
\end{equation*}
Otherwise if $p=2$ and $q_0$ is diagonal then
\begin{equation*}
\rho_{Q_m}(2^n)
\leq \min\left(\left\lceil\frac{n}{2} \right\rceil, 1+ \left\lfloor\frac{m}{2} \right\rfloor \right)2^{n+3} +2^n
\end{equation*}
\end{lem}
\begin{proof}
Denote by $\rho(p^n\colon p^m)$ the number of solutions to 
\begin{equation}\label{eq:Qm-p^n}
Q_m(x,y)\equiv 0 \mod p^n
\end{equation}
Similarly, denote by $\rho^0(p^n\colon p^m)$ the number of solutions modulo $p^n$ reducing to $(0,0)$ modulo $p$ and let $\rho^1(p^n\colon p^m)$ be the number of solutions not reducing to zero modulo $p$.

\paragraph{Case I: $m=0$}\hfill\\
If $p>2$ or $p=2$ and $q_0(x,y)$ is \emph{not} diagonal then $Q_0(x,y)$ defines a smooth affine conic modulo $p$. Subtracting the points on the line at infinity from the projective conic we deduce
\begin{equation*}
\rho(p \colon p^0)=\rho^1(p \colon p^0)=p-\left(\frac{\disc(q_0)}{p}\right)
\end{equation*}
Moreover, all these solutions are smooth and in this case 
\begin{align*}
\rho(p^n \colon p^0)=\rho^1(p^n \colon p^0)&=p^{n-1}\left(p-\left(\frac{\disc(q_0)}{p}\right)\right)\\
&= p^{n-1}\left(1-\left(\frac{\disc(q_0)}{p}\right)\right) +p^n\left(1-\frac{1}{p}\right)
\end{align*}

If $p=2$ and $q_0$ is diagonal then $\rho^0(2^n \colon 2^0 )=0$ and for $n\leq 3$ we use the trivial bound $\rho(2^n \colon 2^0)\leq 2^{2n}\leq 2^{n+3}$.
Hensel's lemma in the strong form implies for $n\geq 3$ that $\rho(2^n \colon 2^0)=2^{n-3}\rho(2^3 \colon p^0)\leq 2^{n+3}$. We deduce that for all $n\geq 1$
\begin{equation*}
\rho(2^n \colon 2^0)=\rho^1(2^n \colon 2^0) \leq 2^{n+3}
\end{equation*}

\paragraph{Case II: $n=1$, $m\geq 1$ and $\rho^1$ for $n\geq 1$}\hfill\\
If $p>2$ or $p=2$ and $q_0$ is not diagonal then $Q\equiv q_0 \mod p$ and $q_0$ is a non-degenerate quadratic form modulo $p$. We see that 
\begin{align*}
\rho^0(p \colon p^m)&=1\\
\rho^1(p \colon p^m)&=1-\left(\frac{\disc(q_0)}{p}\right)
\end{align*}
Moreover, in this case all the solution except $(0,0)$ are smooth, thus if $m\geq 1$ and $p>2$ or $p=2$ and $q_0$ is not diagonal then
\begin{equation*}
\rho^1(p^n \colon p^m)=p^{n-1}\left(1-\left(\frac{\disc(q_0)}{p}\right)\right)
\end{equation*}

If $m\geq 1$, $p=2$ and $q_0$ is diagonal then the same arguments as in Case I imply that 
\begin{align*}
\rho^0(2 \colon 2^m)&=1\\
\rho^1(2^n \colon 2^m)&\leq 2^{n+3}
\end{align*}

\paragraph{Case III: $\rho^0$ for $n\geq 2$, $m\geq 1$}\hfill\\
We proceed to compute $\rho^0(p^n\colon p^m)$ for $n\geq 2$, $m\geq 1$. We need to count solutions to \eqref{eq:Qm-p^n} of the form $(x,y)=(px_0,py_0)$ for $(x_0,y_0)\in\cyclic{p^{n-1}}\times \cyclic{p^{n-1}}$. The pertinent $(px_0,py_0$) solve \eqref{eq:Qm-p^n} if and only if
\begin{equation}\label{eq:Qm-p^n-0}
p^2q_0(x_0,y_0)-u_3p^m\equiv 0 \mod p^n
\end{equation}
The first case to consider is $n=2$, $m\geq 1$, then obviously 
\begin{equation*}
\rho^0(p^2\colon p^m)=\begin{cases}
0 & m=1 \\
p^2 & m>1
\end{cases}
\end{equation*}

If $n\geq 3$ and $m=1$ then \eqref{eq:Qm-p^n-0} implies that
\begin{equation*}
\rho^0(p^n \colon p^1)=0
\end{equation*}
If $n\geq 3$ and $m\geq 2$ then \eqref{eq:Qm-p^n-0} is equivalent to
\begin{equation*}
q_0(x_0,y_0)-u_3p^{m-2}\equiv 0 \mod p^{n-2}
\end{equation*}
and we have a recursion formula
\begin{equation*}
\rho^0(p^n\colon p^m)=p^2\rho(p^{n-2} \colon p^{m-2})
\end{equation*}

Finally we can use the recursion formula and the all the cases computed above to deduce the claim. 
\end{proof}

\begin{prop}\label{prop:Q-rho-sing} Fix a prime $p\mid \omega D$ and set 
$l\coloneqq\ord_p D$, $m\coloneqq \ord_p(4\omega)$. Define for $p>2$
\begin{equation*}
\rho_Q^0(p)\coloneqq \begin{cases}
p-\left(\frac{D/p^l}{p}\right) & l\equiv 0 \mod 2 \\
2p & l\equiv 1 \mod 2 \textrm{ and } \left(\frac{\omega}{p}\right)=+\ch_p(q)\\
0 & l\equiv 1 \mod 2 \textrm{ and } \left(\frac{\omega}{p}\right)=-\ch_p(q)\\
\end{cases}
\end{equation*}

If $n\leq l$ then for all primes $p$
\begin{align*}
\rho_Q(p^n)&=p^{n+\lfloor n/2 \rfloor}\\
\end{align*}

If $p>2$, $n>l$ and $p\nmid \omega$ then
\begin{equation*}
\rho_Q(p^l)=p^{n+\lfloor l/2 \rfloor}\frac{\rho_Q^0(p)}{p}
\end{equation*}

If  $p>2$, $n>l$ and $p\mid \omega$ then set $m=\ord_p (4\omega)$ 
\begin{align*}
\rho_Q(p^n)&=p^{n+\lceil l/2 \rceil}\\ 
&\cdot \begin{cases}
\left(\left\lceil \frac{n-l}{2}\right\rceil-(l\bmod 2) \right)
\frac{1}{p}\left(1-\left(\frac{D/p^l}{p}\right)\right) +\delta_{n\equiv l \mod 2} +\frac{1}{p}\delta_{n\not\equiv l \mod p}
& n-l\leq m\\
\left(1+\left\lfloor \frac{m}{2}\right\rfloor -(l\bmod 2) \right) 
\frac{1}{p}\left(1-\left(\frac{D/p^l}{p}\right)\right)
+\left(1-\frac{1}{p}\right)\delta_{m\equiv 0 \bmod 2} 
& n-l> m 
\end{cases}
\end{align*}

If $p=2$ and $n>l$ then
\begin{equation*}
\rho_Q(2^n)\leq 2^{n+\lceil l/2 \rceil} \left[
\min\left(\left\lceil\frac{n-l}{2} \right\rceil, 1+ \left\lfloor\frac{m}{2} \right\rfloor \right)2^3 +1\right]
\end{equation*}
\end{prop}
\begin{proof} 
Let $q'(x,y)\coloneqq ux^2+Ay^2$ as in Lemma \ref{lem:q-diagonal-Zp}.
We solve the equivalent equation $Q'(x,y)\coloneqq q'(x,y)-\omega D\equiv 0 \mod p^n$. 

\paragraph{Case I: $n\leq \ord_p D$}
Because $n\leq \ord_p D$ the equation $Q'(x,y)\equiv 0 \mod p^n$ is equivalent to 
\begin{equation*}
ux^2\equiv 0 \mod p^n \Longleftrightarrow x\equiv 0 \mod p^{\lceil n/2 \rceil}
\end{equation*}
Equivalently, $(x,y)\in \left(\cyclic{p^n}\right)$ is a solution to $Q(x,y)=0$ if and only if $x\equiv 0 \mod p^{\lceil n/2 \rceil}$.
The formula for $\rho_Q(p^n)$, $n\leq l$, follows immediately.

\paragraph{Case II: $n> \ord_p D$}
Any solution modulo $p^n$ must reduce to a solution modulo $p^l$, i.e.\ it must satisfy $x\equiv 0 \mod p^{\lceil l/2 \rceil}$. Write $x=p^{\lceil l/2 \rceil} x_0$ where $x_0\in \cyclic{p^{n-\lceil l/2 \rceil}}$ and denote $\varpi\coloneqq l \bmod 2 \in\{0,1\}$. 
Then the equation $Q'(x,y)\equiv 0 \mod p^n$ is equivalent to
\begin{align}
\nonumber
up^{l+2\varpi}x_0^2+Ay^2-\omega D&\equiv 0 \mod p^n\\
&\Longleftrightarrow
up^{2\varpi}x_0^2+u_Ay^2-4\omega u u_A\equiv 0 \mod p^{n-l}
\label{eq:q0-p^{n-l}}
\end{align}
where $A=u_A p^l$ and $D=-4uA$. 

Denote 
\begin{equation*}
q_0(x_0,y)\coloneqq up^{2\varpi}x_0^2+u_Ay^2-4\omega u u_A\in \mathbb{Z}_p[x,y]
\end{equation*}
We have shown that the solutions of $Q'(x,y)\equiv 0 \mod p^n$ are exactly the residue classes of the form $(p^{\lceil l/2 \rceil}x_0,y)$ where $(x_0,y)\in\cyclic{p^{n-\lceil l/2 \rceil}}\times \cyclic{p^n}$ reduces to a root of $q_0(x_0,y)$ modulo $p^{n-l}$. In particular,
\begin{equation*}
\rho_Q(p^n)=p^{l-\lceil l/2 \rceil}p^l \rho_{q_0}(p^{n-l})=p^{l+\lfloor l/2 \rfloor} \rho_{q_0}(p^{n-l})
\end{equation*}

If $\varpi\equiv 0 \mod 2$ then we can apply Lemma \ref{lem:singular-conic} directly to $q_0$ with $m=\ord_p (4\omega)$.

If $\varpi\equiv 1 \mod 2$ then we have two options. If $p\nmid 4\omega$ then $q_0$ defines a smooth affine conic modulo $p$. Computing explicitly and using Hensel's lemma we deduce then
\begin{align*}
\rho_{q_0}(p)&=\begin{cases}
2p & \left(\frac{\omega}{p}\right)=+\ch_p(q)\\
0 & \left(\frac{\omega}{p}\right)=-\ch_p(q) 
\end{cases}\\
\rho_{q_0}(p^{n-l})&=p^{n-l}\frac{\rho_{q_0}(p)}{p}
\end{align*}

Otherwise if $p\mid 4\omega$ then reducing \eqref{eq:q0-p^{n-l}} modulo $p$ we conclude that necessarily $y\equiv 0 \mod p$. Moreover, if $n-l=1$ then $\rho_{q_0}(p)=p$. Otherwise if $n\geq l+2$ we  write $y=py_0$ for $y_0\in \cyclic{p^{n-l-1}}$. 

Equation \eqref{eq:q0-p^{n-l}} is then equivalent to
\begin{equation*}
up^2x_0^2+u_Ap^2y_0^2-4\omega u u_A\equiv 0 \mod p^{n-l}
\end{equation*}
If $m=\ord_p (4\omega)=1$ then this equation has no solutions, i.e.\ $\rho_{q_0}(p)=0$. If $m\geq 2$ then define $q_1\coloneqq ux_0^2+u_Ay_0^2-(4\omega/p^2) u u_A$. Then $\rho_{q_0}(p^{n-l})=p^3\rho_{q_1}(p^{n-l-2})$ and apply Lemma \ref{lem:singular-conic} to $q_1$.
\end{proof}

\begin{cor}\label{cor:rho_Q-C-r}
The following bound holds for any prime power $p^n$
\begin{equation*}
\rho_Q(p^n)\leq 16 p^{n(2-1/2)}
\end{equation*}
\end{cor}
\begin{cor}\label{cor:rhoQtilde-final}
Let $p\mid D$ and set $l=\ord_p D$ then if $n\leq l$
\begin{equation*}
\widetilde{\rho}_Q(p^n)=\begin{cases}
p^{n+\lfloor n/2 \rfloor}\left(1-\frac{1}{p}\right) & n \equiv 0 \mod 2\\
0 & n\equiv 1 \mod 2
\end{cases}
\end{equation*}

Assume next that $p\nmid 4\omega$ then
\begin{equation*}
\widetilde{\rho}_Q(p^l)=p^{l+\lfloor l/2 \rfloor}\left(1-\frac{\rho_Q^0(p)}{p^2}\right)
\end{equation*}
and for $n>l$
\begin{equation*}
\widetilde{\rho}_Q(p^n)=p^{n+\lfloor l/2 \rfloor}\frac{\rho_Q^0(p)}{p}\left(1-\frac{1}{p}\right)
\end{equation*}
\end{cor}
\begin{proof}
Notice that if $p\mid D$ then $X_Q\left(\cyclic{p}\right)$ has no smooth points. The claim follows from Proposition \ref{prop:Q-rho-sing} and Lemma \ref{lem:rho-tilderho-formula}.
\end{proof}

\begin{prop}\label{prop:genus-mod-p^2}
Fix a prime $p\parallel D$ and $k\in\{0,1\}$. Assume $p\nmid 4\omega$. If $\epsilon\in\{\pm 1\}$ then
\begin{equation}\label{eq:p-genus-class}
\sum_{a\in\left(\cyclic{p^{2-k}}\right)^\times;\; \left(\frac{a}{p}\right)=\epsilon}
\rho_Q(p^k a;p^2)=\begin{cases}
p^4\left(1-\frac{1}{p}\right) & k=0,\; \ch_p(q)=+\epsilon\\
0 & k=0,\; \ch_p(q)=-\epsilon\\
\frac{p^3}{2}\left(1-\frac{f}{p}\right) & k=1
\end{cases}
\end{equation}
where
\begin{equation*}
f\coloneqq 1+\epsilon\left(\frac{-D/p}{p}\right)\ch_p(q)+\left(\frac{\omega}{p}\right)\ch_p(q)
\end{equation*}
\end{prop}
\begin{proof}
Let $q'(x,y)\coloneqq ux^2+Ay^2$ as in Lemma \ref{lem:q-diagonal-Zp} and write $A=u_Ap$ where $u_A\in\mathbb{Z}_p^\times$.
Let $u_\epsilon\in\mathbb{Z}_p^\times$ such that $\left(\frac{u_\epsilon}{p}\right)=\epsilon$. Define
\begin{equation*}
V(x,y,w)\coloneqq ux^2-u_Apy^2- 4 \omega u u_A p
-p^ku_\epsilon w^2 \in \mathbb{Z}_p[x,y,w]
\end{equation*}

Consider the equation
\begin{equation}\label{eq:p-genus-variety}
V(x,y,w)\equiv 0 \mod p^2 
\end{equation}
The left hand side of \eqref{eq:p-genus-class} is proportional to the number
of solutions to \eqref{eq:p-genus-variety} satisfying $w\neq 0 \mod p$. The proportionality constant is exactly the number of solutions to $p^k u_\epsilon w^2\equiv  p^k u_\varepsilon w_0^2 \mod p^2$ for any fixed unit $w_0$. The latter equation has $2p^k$ solutions. In conclusion 
\begin{align*}
\sum_{a\in\left(\cyclic{p^2}\right)^\times;\; \left(\frac{a}{p}\right)=\epsilon}
&\rho_Q(a;p^2)=  \\
&\frac{1}{2p^k}\# \left\{(x,y,w)\in\cyclic{p^2}\times
 \cyclic{p^2}
 \times\left(\cyclic{p^2}\right)^\times
\mid V(x,y,w)=0
 \right\}
\end{align*}

Equation \eqref{eq:p-genus-variety} reduces modulo $p$ to
\begin{equation}\label{eq:p-genus-variety-modp}
ux^2-p^k u_\varepsilon w^2\equiv 0 \mod p
\end{equation}

\paragraph{Case I: $k=0$}
All the solutions to equation \eqref{eq:p-genus-variety-modp} with $w\neq 0 \mod p$ are smooth. There $0$ such solutions if $\left(\frac{u u_\epsilon}{p}\right)=-1$ and $2p(p-1)$ solutions otherwise. Using Hensel's lemma we conclude that the number of solutions to \eqref{eq:p-genus-variety} with $w\neq 0 \mod p$ is $2p^4\left(1-\frac{1}{p}\right)$ if $\ch_p(q)=+\epsilon$ and $0$ otherwise.

\paragraph{Case II: $k=1$}
Equation \eqref{eq:p-genus-variety-modp} implies that necessarily $x\equiv 0 \mod p$. Hence equation \eqref{eq:p-genus-variety} is equivalent to
\begin{equation}\label{eq:p-k1-genus-variety-modp}
u_Ay^2-u_\epsilon w^2-4 \omega u u_A \equiv 0 \mod p
\end{equation}

This is an equation of a smooth conic with $p+1$ projective solutions. The are either $2$ or $0$ solutions on the line at infinity depending on the sign of $\epsilon\left(\frac{u_A}{p}\right)=\left(\frac{u_\epsilon u_A^{-1}}{p}\right)$. Hence the number of solutions to \eqref{eq:p-k1-genus-variety-modp} is $p-\epsilon\left(\frac{u_A}{p}\right)$. 

We also need to subtract from the solutions of \eqref{eq:p-k1-genus-variety-modp} the cases where $w\equiv 0 \mod p$. Substituting $0$ for $w$ in \eqref{eq:p-k1-genus-variety-modp} we see that there are either $2$ or $0$ such solution depending on the sign of $\left(\frac{\omega u}{p}\right)$.

We conclude that the number of relevant solutions to \eqref{eq:p-genus-variety} is
\begin{equation*}
p^4\left(1-\frac{f}{p}\right)
\end{equation*}
where $f\coloneqq 1+\epsilon\left(\frac{u_A}{p}\right)+\left(\frac{\omega u}{p}\right)$.
\end{proof}

\bibliographystyle{alpha}
\bibliography{joinings_duke}

\def\cprime{$'$}
\begin{thebibliography}{ELMV12}

\bibitem[And98]{Andre}
Yves Andr\'e.
\newblock Finitude des couples d'invariants modulaires singuliers sur une
  courbe alg\'ebrique plane non modulaire.
\newblock {\em J. Reine Angew. Math.}, 505:203--208, 1998.

\bibitem[BS91]{BurgerSarnak}
M.~Burger and P.~Sarnak.
\newblock Ramanujan duals. {II}.
\newblock {\em Invent. Math.}, 106(1):1--11, 1991.

\bibitem[BSR16]{BSR}
Jean Bourgain, Peter Sarnak, and Ze\'ev Rudnick.
\newblock Local statistics of lattice points on the sphere.
\newblock In {\em Modern trends in constructive function theory}, volume 661 of
  {\em Contemp. Math.}, pages 269--282. Amer. Math. Soc., Providence, RI, 2016.

\bibitem[Bur62]{Burgess}
D.~A. Burgess.
\newblock On character sums and {$L$}-series.
\newblock {\em Proc. London Math. Soc. (3)}, 12:193--206, 1962.

\bibitem[Cas78]{CasselsQuadraticForms}
J.~W.~S. Cassels.
\newblock {\em Rational quadratic forms}, volume~13 of {\em London Mathematical
  Society Monographs}.
\newblock Academic Press, Inc. [Harcourt Brace Jovanovich, Publishers],
  London-New York, 1978.

\bibitem[Che04]{Chelluri}
Th. Chelluri.
\newblock {\em Equidistribution of the Roots of Quadratic Congruences}.
\newblock PhD thesis, Rutgers University, 2004.

\bibitem[CI00]{ConreyIwaniec}
J.~B. Conrey and H.~Iwaniec.
\newblock The cubic moment of central values of automorphic {$L$}-functions.
\newblock {\em Ann. of Math. (2)}, 151(3):1175--1216, 2000.

\bibitem[Cor11]{Cornulier}
Yves Cornulier.
\newblock On the {C}habauty space of locally compact abelian groups.
\newblock {\em Algebr. Geom. Topol.}, 11(4):2007--2035, 2011.

\bibitem[COU01]{ClozelOhUllmo}
Laurent Clozel, Hee Oh, and Emmanuel Ullmo.
\newblock Hecke operators and equidistribution of {H}ecke points.
\newblock {\em Invent. Math.}, 144(2):327--351, 2001.

\bibitem[CU04]{ClozelUllmo}
Laurent Clozel and Emmanuel Ullmo.
\newblock \'equidistribution des points de {H}ecke.
\newblock In {\em Contributions to automorphic forms, geometry, and number
  theory}, pages 193--254. Johns Hopkins Univ. Press, Baltimore, MD, 2004.

\bibitem[Dic01]{DicksonPSLn}
Leonard~Eugene Dickson.
\newblock Theory of linear groups in an arbitrary field.
\newblock {\em Trans. Amer. Math. Soc.}, 2(4):363--394, 1901.

\bibitem[DL78]{DeMilloLipton}
Richard~A. Demillo and Richard~J. Lipton.
\newblock A probabilistic remark on algebraic program testing.
\newblock {\em Information Processing Letters}, 7(4):193 -- 195, 1978.

\bibitem[dlBB06]{BretecheBrowning}
R.~de~la Bret\`eche and T.~D. Browning.
\newblock Sums of arithmetic functions over values of binary forms.
\newblock {\em Acta Arith.}, 125(3):291--304, 2006.

\bibitem[dlBT12]{BretecheTenenbaum}
R\'egis de~la Bret\`eche and G\'erald Tenenbaum.
\newblock Moyennes de fonctions arithm\'etiques de formes binaires.
\newblock {\em Mathematika}, 58(2):290--304, 2012.

\bibitem[Duk88]{Duke}
W.~Duke.
\newblock Hyperbolic distribution problems and half-integral weight {M}aass
  forms.
\newblock {\em Invent. Math.}, 92(1):73--90, 1988.

\bibitem[Edi98]{Edixhoven2}
Bas Edixhoven.
\newblock Special points on the product of two modular curves.
\newblock {\em Compositio Math.}, 114(3):315--328, 1998.

\bibitem[Edi05]{EdixhovenArbitrary}
Bas Edixhoven.
\newblock Special points on products of modular curves.
\newblock {\em Duke Math. J.}, 126(2):325--348, 2005.

\bibitem[EKL06]{EKL}
Manfred Einsiedler, Anatole Katok, and Elon Lindenstrauss.
\newblock Invariant measures and the set of exceptions to {L}ittlewood's
  conjecture.
\newblock {\em Ann. of Math. (2)}, 164(2):513--560, 2006.

\bibitem[EL10]{Pisa}
M.~Einsiedler and E.~Lindenstrauss.
\newblock Diagonal actions on locally homogeneous spaces.
\newblock In {\em Homogeneous flows, moduli spaces and arithmetic}, volume~10
  of {\em Clay Math. Proc.}, pages 155--241. Amer. Math. Soc., Providence, RI,
  2010.

\bibitem[EL15]{EntropySArithmetic}
Manfred Einsiedler and Elon Lindenstrauss.
\newblock On measures invariant under tori on quotients of semisimple groups.
\newblock {\em Ann. of Math. (2)}, 181(3):993--1031, 2015.

\bibitem[EL17]{ELJoinings}
Manfred Einsiedler and Elon Lindenstrauss.
\newblock Joinings of higher rank torus actions on homogeneous spaces, 2017.
\newblock To appear, \textit{Publ. Math. Inst. Hautes \'{E}tudes Sci.}
  \url{https://arxiv.org/abs/1502.05133}.

\bibitem[ELMV09]{ELMVPeriodic}
Manfred Einsiedler, Elon Lindenstrauss, Philippe Michel, and Akshay Venkatesh.
\newblock Distribution of periodic torus orbits on homogeneous spaces.
\newblock {\em Duke Math. J.}, 148(1):119--174, 2009.

\bibitem[ELMV11]{ELMVCubic}
Manfred Einsiedler, Elon Lindenstrauss, Philippe Michel, and Akshay Venkatesh.
\newblock Distribution of periodic torus orbits and {D}uke's theorem for cubic
  fields.
\newblock {\em Ann. of Math. (2)}, 173(2):815--885, 2011.

\bibitem[ELMV12]{ELMVPGL2}
Manfred Einsiedler, Elon Lindenstrauss, Philippe Michel, and Akshay Venkatesh.
\newblock The distribution of closed geodesics on the modular surface, and
  {D}uke's theorem.
\newblock {\em Enseign. Math. (2)}, 58(3-4):249--313, 2012.

\bibitem[EMV13]{EMV}
Jordan~S. Ellenberg, Philippe Michel, and Akshay Venkatesh.
\newblock Linnik's ergodic method and the distribution of integer points on
  spheres.
\newblock In {\em Automorphic representations and {$L$}-functions}, volume~22
  of {\em Tata Inst. Fundam. Res. Stud. Math.}, pages 119--185. Tata Inst.
  Fund. Res., Mumbai, 2013.

\bibitem[GKM15]{GKM}
Andrew Granville, Dimitris Koukoulopoulos, and Kaisa Matom\"aki.
\newblock When the sieve works.
\newblock {\em Duke Math. J.}, 164(10):1935--1969, 2015.

\bibitem[GZ86]{GrossZagier}
Benedict~H. Gross and Don~B. Zagier.
\newblock Heegner points and derivatives of {$L$}-series.
\newblock {\em Invent. Math.}, 84(2):225--320, 1986.

\bibitem[HM79]{HoweMoore}
Roger~E. Howe and Calvin~C. Moore.
\newblock Asymptotic properties of unitary representations.
\newblock {\em J. Funct. Anal.}, 32(1):72--96, 1979.

\bibitem[Hol10]{Holowinsky}
Roman Holowinsky.
\newblock Sieving for mass equidistribution.
\newblock {\em Ann. of Math. (2)}, 172(2):1499--1516, 2010.

\bibitem[HS10]{HolowinnskySound}
Roman Holowinsky and Kannan Soundararajan.
\newblock Mass equidistribution for {H}ecke eigenforms.
\newblock {\em Ann. of Math. (2)}, 172(2):1517--1528, 2010.

\bibitem[Hux03]{Huxley}
M.~N. Huxley.
\newblock Exponential sums and lattice points. {III}.
\newblock {\em Proc. London Math. Soc. (3)}, 87(3):591--609, 2003.

\bibitem[Iwa87]{Iwaniec}
Henryk Iwaniec.
\newblock Fourier coefficients of modular forms of half-integral weight.
\newblock {\em Invent. Math.}, 87(2):385--401, 1987.

\bibitem[JL70]{JacquetLangalnds}
H.~Jacquet and R.~P. Langlands.
\newblock {\em Automorphic forms on {${\rm GL}(2)$}}.
\newblock Lecture Notes in Mathematics, Vol. 114. Springer-Verlag, Berlin-New
  York, 1970.

\bibitem[Kem78]{instability}
George~R. Kempf.
\newblock Instability in invariant theory.
\newblock {\em Ann. of Math. (2)}, 108(2):299--316, 1978.

\bibitem[Kha17]{Kh17}
Ilya Khayutin.
\newblock Large deviations and effective equidistribution.
\newblock {\em Int. Math. Res. Not. IMRN}, (10):3050--3106, 2017.

\bibitem[Kne65]{Kneser-p-adic}
Martin Kneser.
\newblock Galois-{K}ohomologie halbeinfacher algebraischer {G}ruppen \"uber
  {${\germ p}$}-adischen {K}\"orpern. {I}.
\newblock {\em Math. Z.}, 88:40--47, 1965.

\bibitem[Kow08]{Kowalski}
E.~Kowalski.
\newblock {\em The large sieve and its applications}, volume 175 of {\em
  Cambridge Tracts in Mathematics}.
\newblock Cambridge University Press, Cambridge, 2008.
\newblock Arithmetic geometry, random walks and discrete groups.

\bibitem[Lan18]{Landau1918}
E.~Landau.
\newblock Über imaginär- quadratischer zahlkörper.
\newblock {\em Nachrichten von der Gesellschaft der Wissenschaften zu
  Göttingen, Mathematisch-Physikalische Klasse}, 1918:285--295, 1918.

\bibitem[Lem07]{Lemmermeyer}
Franz Lemmermeyer.
\newblock The development of the principal genus theorem.
\newblock In {\em The shaping of arithmetic after {C}. {F}. {G}auss's {\it
  {D}isquisitiones arithmeticae}}, pages 529--561. Springer, Berlin, 2007.

\bibitem[Lin68]{LinnikBook}
Yu.~V. Linnik.
\newblock {\em Ergodic properties of algebraic fields}.
\newblock Translated from the Russian by M. S. Keane. Ergebnisse der Mathematik
  und ihrer Grenzgebiete, Band 45. Springer-Verlag New York Inc., New York,
  1968.

\bibitem[Lin06]{QuantumLindenstrauss}
Elon Lindenstrauss.
\newblock Invariant measures and arithmetic quantum unique ergodicity.
\newblock {\em Ann. of Math. (2)}, 163(1):165--219, 2006.

\bibitem[LRS99]{LRS}
Wenzhi Luo, Ze\'ev Rudnick, and Peter Sarnak.
\newblock On the generalized {R}amanujan conjecture for {${\rm GL}(n)$}.
\newblock In {\em Automorphic forms, automorphic representations, and
  arithmetic ({F}ort {W}orth, {TX}, 1996)}, volume~66 of {\em Proc. Sympos.
  Pure Math.}, pages 301--310. Amer. Math. Soc., Providence, RI, 1999.

\bibitem[MFK94]{GIT}
D.~Mumford, J.~Fogarty, and F.~Kirwan.
\newblock {\em Geometric invariant theory}, volume~34 of {\em Ergebnisse der
  Mathematik und ihrer Grenzgebiete (2) [Results in Mathematics and Related
  Areas (2)]}.
\newblock Springer-Verlag, Berlin, third edition, 1994.

\bibitem[Mic04]{Michel04}
P.~Michel.
\newblock The subconvexity problem for {R}ankin-{S}elberg {$L$}-functions and
  equidistribution of {H}eegner points.
\newblock {\em Ann. of Math. (2)}, 160(1):185--236, 2004.

\bibitem[Mil05]{Milne}
J.~S. Milne.
\newblock Introduction to {S}himura varieties.
\newblock In {\em Harmonic analysis, the trace formula, and {S}himura
  varieties}, volume~4 of {\em Clay Math. Proc.}, pages 265--378. Amer. Math.
  Soc., Providence, RI, 2005.
\newblock Available online at the author's homepage.

\bibitem[MT94]{MargulisTomanov}
G.~A. Margulis and G.~M. Tomanov.
\newblock Invariant measures for actions of unipotent groups over local fields
  on homogeneous spaces.
\newblock {\em Invent. Math.}, 116(1-3):347--392, 1994.

\bibitem[MV06]{MichelVenkateshRev}
Philippe Michel and Akshay Venkatesh.
\newblock Equidistribution, {$L$}-functions and ergodic theory: on some
  problems of {Y}u.\ {L}innik.
\newblock In {\em International {C}ongress of {M}athematicians. {V}ol. {II}},
  pages 421--457. Eur. Math. Soc., Z\"urich, 2006.

\bibitem[Nag64]{Nagata}
Masayoshi Nagata.
\newblock Invariants of a group in an affine ring.
\newblock {\em J. Math. Kyoto Univ.}, 3:369--377, 1963/1964.

\bibitem[Nai92]{Nair}
Mohan Nair.
\newblock Multiplicative functions of polynomial values in short intervals.
\newblock {\em Acta Arith.}, 62(3):257--269, 1992.

\bibitem[NT98]{NairTenenbaum}
Mohan Nair and G\'erald Tenenbaum.
\newblock Short sums of certain arithmetic functions.
\newblock {\em Acta Math.}, 180(1):119--144, 1998.

\bibitem[Pil11]{Pila}
Jonathan Pila.
\newblock O-minimality and the {A}ndr\'e-{O}ort conjecture for {$\Bbb C^n$}.
\newblock {\em Ann. of Math. (2)}, 173(3):1779--1840, 2011.

\bibitem[PT13]{PilaTsimerman}
Jonathan Pila and Jacob Tsimerman.
\newblock The {A}ndr\'e-{O}ort conjecture for the moduli space of abelian
  surfaces.
\newblock {\em Compos. Math.}, 149(2):204--216, 2013.

\bibitem[Rei75]{Reiner}
I.~Reiner.
\newblock {\em Maximal orders}.
\newblock Academic Press [A subsidiary of Harcourt Brace Jovanovich,
  Publishers], London-New York, 1975.
\newblock London Mathematical Society Monographs, No. 5.

\bibitem[Sar91]{Sarnak}
Peter~C. Sarnak.
\newblock Diophantine problems and linear groups.
\newblock In {\em Proceedings of the {I}nternational {C}ongress of
  {M}athematicians, {V}ol.\ {I}, {II} ({K}yoto, 1990)}, pages 459--471. Math.
  Soc. Japan, Tokyo, 1991.

\bibitem[Sch80]{Schwartz}
J.~T. Schwartz.
\newblock Fast probabilistic algorithms for verification of polynomial
  identities.
\newblock {\em J. Assoc. Comput. Mach.}, 27(4):701--717, 1980.

\bibitem[Sha88]{Shahidi}
Freydoon Shahidi.
\newblock On the {R}amanujan conjecture and finiteness of poles for certain
  {$L$}-functions.
\newblock {\em Ann. of Math. (2)}, 127(3):547--584, 1988.

\bibitem[Shi80]{Shiu}
P.~Shiu.
\newblock A {B}run-{T}itchmarsh theorem for multiplicative functions.
\newblock {\em J. Reine Angew. Math.}, 313:161--170, 1980.

\bibitem[ST17]{ShendeTsimerman}
Vivek Shende and Jacob Tsimerman.
\newblock Equidistribution in {$\mathrm{Bun}_2(\mathbb{P}^1)$}.
\newblock {\em Duke Math. J.}, 166(18):3461--3504, 2017.

\bibitem[Tao14]{TaoPolynomial}
Terence Tao.
\newblock Algebraic combinatorial geometry: the polynomial method in arithmetic
  combinatorics, incidence combinatorics, and number theory.
\newblock {\em EMS Surv. Math. Sci.}, 1(1):1--46, 2014.

\bibitem[Tsi18]{Tsimerman}
Jacob Tsimerman.
\newblock The {A}ndr\'{e}-{O}ort conjecture for {$\mathcal{A}_g$}.
\newblock {\em Ann. of Math. (2)}, 187(2):379--390, 2018.

\bibitem[vdC20]{vanDerCorput}
J.~G. van~der Corput.
\newblock \"uber {G}itterpunkte in der {E}bene.
\newblock {\em Math. Ann.}, 81(1):1--20, 1920.

\bibitem[Zha05]{Zhang}
Shou-Wu Zhang.
\newblock Equidistribution of {CM}-points on quaternion {S}himura varieties.
\newblock {\em Int. Math. Res. Not.}, (59):3657--3689, 2005.

\bibitem[Zip79]{Zippel}
Richard Zippel.
\newblock Probabilistic algorithms for sparse polynomials.
\newblock In {\em Symbolic and algebraic computation ({EUROSAM} '79,
  {I}nternat. {S}ympos., {M}arseille, 1979)}, volume~72 of {\em Lecture Notes
  in Comput. Sci.}, pages 216--226. Springer, Berlin-New York, 1979.

\end{thebibliography}
\end{document}